\documentclass[11pt,a4paper]{report}      %% LaTeX2e document.
\usepackage{warwickthesis,setspace,graphicx}     %!  setspace is used to control linepacing

\usepackage{amsmath,amsthm,amssymb,amsfonts}
\usepackage{microtype} %Takes care of better format/line breaks
\usepackage[all,arc]{xy}
\usepackage{enumerate}
\usepackage{mathrsfs}
\usepackage{multicol}
\usepackage{graphicx}
\usepackage{hyperref}
\usepackage{caption}
\usepackage{color}
\usepackage{dsfont}
\usepackage{tikz}
\usepackage{float}
\usepackage[shortlabels]{enumitem}

\usepackage[linesnumbered,ruled,noend,algosection]{algorithm2e}
\usepackage{tabularx}
\usepackage{ltablex}

\newtheorem{thm}{Theorem}[section]
\newtheorem{cor}[thm]{Corollary}

\newtheorem{lem}[thm]{Lemma}
\theoremstyle{definition}
\newtheorem{defn}[thm]{Definition}
\theoremstyle{remark}
\newtheorem{rem}[thm]{Remark}

\DeclareMathOperator{\ind}{index}
\DeclareMathOperator{\len}{len}
\DeclareMathOperator{\wind}{wind}
\DeclareMathOperator{\Hom}{Hom}
\DeclareMathOperator{\w}{weight}
\DeclareMathOperator{\rev}{rev}

\DeclareMathOperator{\carr}{carr}
\DeclareMathOperator{\dual}{dual}
\DeclareMathOperator{\DR}{d_R}

\makeatletter
\AtBeginDocument{%
  \expandafter\renewcommand\expandafter\section\expandafter{%
    \expandafter\@fb@secFB\section
  }%
}
\makeatother

\usepackage{subcaption}
\usepackage[section]{placeins}
\usepackage{xfrac}

 \leftchapter                       %% Uncomment one of these if you want
\renewcommand{\thesisdepartmentname}{Department of Mathematics}    %! The name of
\renewcommand{\thesissubmission}{Submitted to the University of Warwick\\
                        for the degree of}
\renewcommand{\thesisauthor} {Ronja Johanna Barbara Kuhne}    %% Your official name.
\renewcommand{\thesismonth}{August}     %% Your month of graduation.

\renewcommand{\thesisyear}{2019}      %% Your year of graduation.

\renewcommand{\thesistitle}{Polynomial-time efficient position}     %% The title of your thesis; use
\renewcommand{\thesisauthoraddress}{Mathematics Institute\\
University of Warwick\\
Coventry\\
CV4 7AL\\
United Kingdom}

\begin{document}

%\thesiscopyrightpage                 %! Generate the copyright page for the library.
%%%%%%%%%%% \thesiscopyrightpagehardcopyonly This only applies for a masters thesis that will not go online.

%%* Uncomment a ttitle page.
% \thesistitlepage                     %% Generate the title page.
\thesistitlecolourpage           %! Generates a COLOUR title page.

%%* Start roman page numbering here for contents, etc
\pagenumbering{roman} %! Begins roman numerals start from page i.

\tableofcontents                     %% Generate table of contents.
% \listoftables                      %% Uncomment this to generate list
                                     %% of tables.
\listoffigures                     %% Uncomment this to generate list
                                     %% of figures.
\newpage
\phantomsection
\addcontentsline{toc}{chapter}{List of Algorithms}
\listofalgorithms

\begin{thesisacknowledgments}        %% Use this to write your
%  \input ack.tex                    %% acknowledgments; it can be anything
                                     %% allowed in LaTeX2e par-mode.

                                     %! This following is not needed, but you may like to add it.
%This \lowercase\expandafter{\thesistype} was typeset with
%\LaTeXe\footnote{\LaTeXe{} is an extension of \LaTeX. \LaTeX{} is
%a collection of macros for \TeX. \TeX{} is a trademark of the
%American Mathematical Society. The style package {\em warwickthesis} was
%used.} by \thesistypist.

Firstly, I would like to thank my supervisors Saul Schleimer and Brian Bowditch. Saul, for his kind and patient guidance throughout my PhD, for being such an open, curious and engaging person, for always knowing when it was time to make a silly joke, and for using a GriGri when his eagerness to help was once again too great. Brian, for making sure that everything was going smoothly and for offering his help on innumerable occasions.

Secondly, I would like to thank the geometry and topology community in and beyond Warwick. The last four years have been full of wonderful people, stimulating conversations, and exciting adventures. My special thanks goes to Beatrice Pozzetti for her continuous guidance and support, for all the nice evenings spent together and for never allowing me to take myself too seriously.
I also like to thank Yoav Moriah, Thomas Vogel, and Stefan Friedl for their invitations and great hospitality. 

Thirdly, I am highly indebted to the University of Warwick and the mathematics department for providing such a welcoming and stimulating environment. The PhD community as well as the administrative staff have always offered a smile and comforting word and have helped whenever possible.

I am grateful to the Engineering and Physical Sciences Research
Council for funding my studies and for giving me the chance to live in the UK. It has been a pleasure and invaluable experience.

I am very lucky to have spent my time at Warwick in the company of
such amazing people. My special thanks goes to my housemate, Simon, and all my climbing partners for the many hours that were filled with laughter and great conversations. No words can describe how grateful I am for spending the last years with Esmee and Alex and knowing that they would always have my back.

Last, but certainly not least, I would like to express my gratitude for the unlimited support I have received from Bernhard and my parents. I can not thank you enough for all your love, trust, and encouragement.

\end{thesisacknowledgments}

\begin{thesisdeclaration}        %! Use this to declare the extent of
                 %! the original work,
                 %! collaboration, other published
                                 %! material etc.it can be anything
                                 %% allowed in LaTeX2e par-mode.
%Replace this text with a declaration of the extent of the original work,
%collaboration, other published material etc. You can use any \LaTeX\
%constructs.
I declare that the material in this thesis is, to the best of my knowledge, my own work except where otherwise indicated in the text, or where the material is widely known. No part has been submitted by me for any other degree.

This thesis was typeset in LaTeX2e using the style package \textit{warwickthesis}. All figures were created by the author using TikZ.

\end{thesisdeclaration}

\begin{thesisabstract}               %% Use this to write your thesis
                                     %% abstract; it can be anything
                                     %% allowed in LaTeX2e par-mode.
%  \begin{singlespace}       %! uncomment this if you need single spacing
%   \input abstract.tex       %!           don't forget the end spacing!
                                     %! It must fit on one page.
                                     %! single spacing and smaller
                                     %! font size
                                     %!  is allowed here.

Suppose that $S$ is a surface of positive complexity and $N\subset S$ a tie neighbourhood of a large train track $\tau$ in $S$. We present a polynomial-time algorithm that, given a properly immersed, essential, and non-peripheral arc or curve $\alpha \subset S$, homotopes $\alpha$ into efficient position with respect to the tie neighbourhood $N$.

Proofs for the existence of efficient position were previously given in \cite{Taka} and \cite{carrySaul}. In \cite{Taka}, a constructive proof for the existence of efficient position is given for immersed curves on closed surfaces of genus greater than or equal to two. There is no discussion of the complexity of the implied algorithm. In \cite{carrySaul}, the existence of efficient position is proved for embedded curves with respect to birecurrent train tracks on surfaces of positive complexity. The implied algorithm operates via an exhaustive search. No time bounds can be deduced.

We note that the algorithm presented in this thesis and the algorithm suggested by a careful reading of \cite{Taka} coincide in the case of closed surfaces. However, this thesis constitutes more than a time-complexity analysis of Takarajima's constructive proof. Firstly, we are more general as we allow surfaces with boundary, whereas Takarajima only considers closed surfaces. Secondly, our combinatorial set-up uses arcs and curves with transverse self-intersection, whereas the barycentric subdivision of complementary regions carried out in \cite{Taka} forces non-transverse self-intersection even for curves which are initially embedded. Thirdly, the algorithm in this thesis is formulated purely in terms of local homotopies, whereas \cite{Taka} requires semi-local arguments. Thus, we can, and do, give pseudocode for our algorithm as well as prove its correctness.

%Thus, one cannot recognize embedded arcs and curves or compute intersection numbers in the set-up of \cite{Taka}. This is contrary to the aim of using efficient position to provide a coordinate system for embedded curves on surfaces.

%   \end{singlespace}
\end{thesisabstract}

\begin{thesisabbreviations}       %! Use this to give a list of
                                   %! abbreviateons
                                   %! It can be anything
%%\end{thesisabbreviations}         %! allowed in LaTeX2e par-mode.
%                                   %!The following may be useful':
%                     %!\begin{itemize}
%                     %!     \item[symbol]descriptive text..
%                     %!\end{itemize}
\begin{center}
\begin{tabularx}{\linewidth}{lX}
\textbf{Standard notation} & \textbf{Description} \\
$S$ & A compact, connected, oriented surface\\
$\partial S$ & The boundary of $S$\\
$\widetilde{S}$ & The universal cover of $S$\\
$\chi(S)$ & The Euler characteristic of $S$\\
$\xi(S)$ & The complexity of $S$\\
$\alpha$ & An immersed arc or curve\\
$S^1$&The circle, identified with $[0,1]/_\sim$\\
$\square$&An indicator that the proof is complete, is omitted, or is given elsewhere\\
&\\
\textbf{Further notation} & \textbf{Description} (with page references) \\
$\tau$ & A train track (\pageref{Page:train track})\\
$N(\tau)$ & A tie neighbourhood of a train track $\tau$ (\pageref{Page:tie neighbourhood})\\
$T$ & A subsurface with corners (\pageref{Page:Index and Corners})\\
$\partial^2 T$ & The set of corners of $T$ (\pageref{Page:Index and Corners})\\
$\ind(T)$&The index of $T$ (\pageref{Page:Index and Corners})\\
$R$&A branch or switch rectangle or a complementary region (\pageref{Page: Rectangle}) \\
$\partial_h R$&The horizontal boundary of $R$ (\pageref{Page: Rectangle})\\
$\partial_v R$&The vertical boundary of $R$ (\pageref{Page: Rectangle})\\
$\mathcal{R}_\textrm{tie}$&The set of branch and switch rectangles of a tie neighbourhood (\pageref{Page: Rtie})\\
$\mathcal{R}_\textrm{comp}$&The set of complementary regions of a tie neighbourhood (\pageref{Page: Rcomp})\\
$\mathcal{R}$&The set of all branch and switch rectangles and complementary regions of a tie neighbourhood (\pageref{Page:R})\\
$\partial \mathcal{R}$&The union of all boundaries of elements in $\mathcal{R}$ (\pageref{Page: boundary R})\\
$\partial^2 \mathcal{R}$&The union of all corners of elements in $\mathcal{R}$ (\pageref{Page: boundary R})\\
$\wind(a)$&The winding number of a snippet $a$ inside a complementary region (\pageref{Page: winding number})\\
$\mathbb{B}(\bullet,\bullet)$&The type of a bad snippet inside a branch rectangle (\pageref{Page: Branch type})\\
$\mathbb{S}(\bullet,\bullet, \bullet)$&The type of a bad snippet inside a switch rectangle (\pageref{Page: Switch type})\\
$\mathbb{R}(\bullet,\bullet)$&The type of a bad snippet inside a complementary region (\pageref{Page: Comp type})\\
$\len(\alpha)$&The snippet length of $\alpha$ (\pageref{Page: Snippet length})\\
$\alpha[i+\epsilon:i+\delta]$&The subarc of the $i$-th snippet of $\alpha$ which is the image of the interval $[\epsilon,\delta]$ under the map $\alpha[i]$ (\pageref{Page: alpha subsnippet})\\
$\alpha_\textrm{trim}$&The inside of $\alpha$ (\pageref{Page: alpha trim})\\
$s_N$&The maximal side length of $N$ (\pageref{maximal side length})\\
$\len_\textrm{corn}(\alpha)$&The corner length of $\alpha$ (\pageref{Page: corner length})\\
$\len_\textrm{block}(\alpha)$&The number of blockers of $\alpha$ (\pageref{Page: reduced corner length and reduced corner length})\\
$\len_\textrm{red}(\alpha)$&The reduced corner length of $\alpha$ (\pageref{Page: reduced corner length and reduced corner length})\\
$\carr(\alpha)$&The number of carried snippets of $\alpha$ (\pageref{Page: carried and dual})\\
$\dual_\textrm{R}(\alpha)$&The number of right duals of $\alpha$ (\pageref{Page: carried and dual})\\
$\dual_\textrm{L}(\alpha)$&The number of left duals of $\alpha$ (\pageref{Page: carried and dual})\\
\end{tabularx}
\end{center}

\end{thesisabbreviations}
%!!!!!!!!!!!!!!!                     %% Begin your thesis text here; follow
                                     %% the report style and group your text
                                     %% in chapters, sections, etc. eg:
%%* don't need this with one-sided printing
%\newpage{\pagestyle{empty}\cleardoublepage} %! ensure that Chapter 1 starts on an odd
                                           %! page when using double sided printing.
%%* Start arabic numbering of main text here
\pagenumbering{arabic} %! Begins arabic numerals start from page 1.

%\chapter{Introduction}
%You would usually put the main content in separate files.
%% \input introduction.tex
%
%
%                            %% More chapters.
%%!
%%! There are a few variations of reference
%\begin{verbatim}\citet[chap. 2]{ballentine82}|
%\end{verbatim}
%for a textual one, as \citet[chap. 2]{ballentine82}.\\
% \\

\chapter{Introduction}

The compactification of Teichmüller space by the space of projective measured laminations, $\mathcal{PML}$, lies at the heart of Thurston's pioneering approach to geometric topology. Studying curves and their limits under iterations of a mapping class is essential to much of Thurston's work, including the Nielsen-Thurston classification of surface homeomorphisms, the topological characterization of rational maps, and the geometric classification of mapping tori \cite[Preface]{Hubbard}.

By introducing train tracks as a combinatorial tool to coordinatise $\mathcal{PML}$, Thurston opened up new ways of answering questions in topology, one-dimensional complex dynamics, and geometry. His ``tangential train track coordinate system'' provides a canonical piecewise integral projective structure for $\mathcal{PML}$ and is perfectly adapted to studying curves that are carried by the stable train track of a pseudo-Anosov map (\cite[Chapter~8.9]{Thurston},\cite[Chapter~15]{Primer}). However, this coordinate system of $\mathcal{PML}$ is not global. Therefore, comparing curves that lie in different charts, that is, curves that are not carried by the same maximal train track, is difficult.

A first step towards defining a global coordinate system for $\mathcal{PML}$ involving train tracks was made by Penner, a student of Thurston. He introduced the terminology of two train tracks ``hitting efficiently'' if they intersect transversely and do not bound any bigon between them \cite[Chapter~1.3]{Penner}. This leads to a dual concept of tangential train track coordinates called ``transverse train track coordinates'' \cite[Section~9]{ThurstonStretch}. Instead of counting carried arcs of the curve, we now count transverse intersections of the curve with the train track. However, this coordinate system for curves is not only coarser than the ``tangential train track coordinate system'', but suffers, by duality, from the same problem as tangential train track coordinates: it is not global.

Yet, there are global coordinate systems for curves on surfaces, with the two most prominent of the last century being normal coordinates and Dehn-Thurston coordinates.

\textit{Normal coordinates} count intersections of curves with the edges of an ideal triangulation of a punctured surface. For a formal definition and an extensive list of their applications, we refer the reader to \cite{Erickson}. To determine normal coordinates, curves are assumed to be in ``normal form'' with respect to the ideal triangulation. That is, they intersect the ideal triangles in standardized arcs, none of which cuts off a bigon. The existence of this normal form for every homotopy class of curves follows from a step-by-step elimination of bigons, each reducing by two the number of intersections between the curve and the edges of the triangulation. Normal coordinates are very well-suited to compute images of iterations of a mapping class on curves: mapping classes can be represented as paths in the flip graph of the surface \cite[Section~1.2]{BellWebb}, and computing the change of coordinates under a single flip is straightforward \cite[Theorem~2.2.1]{MarkThesis}. However, this global coordinate system has one significant drawback: the introduction of an artificial marked point on closed surfaces leads to a non-finite ambiguity. This requires special care when working with closed surfaces and aiming to apply the same algorithms as in the non-closed case \cite[Section~1.1]{BellWebb}.

Introduced by Dehn in 1922 \cite[Paper~7, Section~2]{Dehn} and later revived by Thurston \cite{ThurstonBulletin}, \textit{Dehn-Thurston coordinates} use a pants decomposition of the surface and a set of twisting parameters to provide a global coordinate system to the set of curves and laminations \cite[Theorem~3.1.1]{Penner}. Here, closed and non-closed surfaces can be treated alike. As in the case of normal coordinates, the ``normal position'' of a curve with respect to the pants decomposition can be achieved via a straightforward argument. Furthermore, explicit piecewise integral linear formulas for the change of coordinates under a single pants move are given by Penner in \cite[Section~3]{PennerCorrected}. Yet, these are quite complex and applications of Dehn-Thurston coordinates for computational purposes have been rather sparse \cite[Section~3]{PennerCorrected}.

In 1998, Takarajima combined the two versions of train track coordinates to prepare the ground for a new global coordinate system for curves: ``efficient coordinates''. He introduced a third normal position for curves with respect to a train track, called ``quasi-transversality'' \cite[Definition~3.2]{Taka}. This constitutes a simultaneous generalization of ``carried'' curves and ``dual'' curves. An almost identical concept was defined independently by Masur, Mosher, and Schleimer in 2010 under the name of ``efficient position'' \cite[Definition~2.3]{carrySaul}.
The existence of efficient position is much more subtle than that of normal position required for normal coordinates: a subdivision of the tie neighbourhood of the train track induces a tiling of the surface by polygons. If a curve is not in efficient position, a subarc of the curve cuts off a region of positive index inside a tile. Pushing the subarc across this region can increase the number of intersections of the curves with the boundary of the tiles. Furthermore, it might not reduce the overall amount of positive index cut off by subarcs in the tiles (see page \pageref{General Trigon Example 2}, Figure \ref{General Trigon Example 2}).

In 2000, Takarajima proved the existence of efficient position for immersed curves on closed surfaces using explicit semi-local homotopies that transform a given curve into efficient position \cite[Proposition~4.1]{Taka}. There is no discussion of the complexity of the implied algorithm in \cite{Taka}. In 2010, Masur, Mosher, and Schleimer gave an existence proof for embedded curves on surfaces of positive complexity with respect to birecurrent train tracks \cite[Theorem~4.1]{carrySaul}. We note that their proof is constructive but relies on an exhaustive search. Thus, no time bounds can be deduced from the work in \cite{carrySaul}. The results in this thesis imply that their algorithm halts in exponential time. We further remark that the existence of efficient position with respect to Reebless bigon tracks on the torus is proved in \cite[Lemma~14]{gueritaud}. The proof is constructive but no time bounds are discussed.

\section{Main result}

We give an algorithmic proof of the existence of efficient position for immersed curves on surfaces of positive complexity. The presented algorithm halts in quadratic-time in the length of the input curve. More precisely, suppose that $S$ is a surface of positive complexity. Fix a large train track $\tau \subset S$. A subdivision of a tie neighbourhood of $\tau$ into ``branch rectangles'' and ``switch rectangles'' provides a tiling of $S$ by polygons and peripheral annuli. We assume that curves are given as cutting sequences with respect to the one-skeleton of this tiling. Such curves are said to be \textit{snippet-decomposed}, and we refer to subarcs of the curve that lie inside a single tile as \textit{snippets}.

\setcounter{chapter}{7}
\setcounter{section}{0}
\setcounter{thm}{1}

\begin{thm}
There is an algorithm that takes as input
\begin{itemize}
\item a surface $S=S_{g,b}$ satisfying $\xi(S)=3g-3+b \geq 1$,
\item a tie neighbourhood $N=N(\tau)$ of a large train track $\tau$ in $S$, and
\item a properly immersed arc or curve $\alpha$ given via its snippet decomposition with respect to $N$,
\end{itemize}
and outputs an arc or curve $\alpha'$ homotopic to $\alpha$, relative to endpoints in the case of arcs, such that
\begin{itemize}
\item $\alpha'$ is in efficient position with respect to $N$ or
\item $\alpha'$ has snippet length one.
\end{itemize}
This algorithm terminates in $\mathit{O}(\chi(S)^4  \cdot \len(\alpha)^2)$ time. Moreover, the output $\alpha'$ has snippet length one if and only if $\alpha$ is inessential or peripheral. If $\alpha$ is peripheral, the boundary component that $\alpha$ is homotopic to, as well as the corresponding power, can be read off from the snippet-decomposition of $\alpha'$.
\end{thm}

As we work purely combinatorially, the notion of efficient position is defined with respect to the tie neighbourhood of a train track (see page \pageref{Def Efficient position}, Definition \ref{Def Efficient position}). We note that an arc or curve is in efficient position with respect to a tie neighbourhood of a large train track $\tau$ as defined here if and only if it is in efficient position with respect to $\tau$ following the definition given in \cite[Definition~2.3]{carrySaul}.
We remark that there already exist polynomial-time algorithms that decide if arcs or curves are essential and non-peripheral: the word problem for fundamental groups can be decided in linear time \cite{Dehn}, and normal coordinates allow the detection of peripheral curves.

The mathematical content of Theorem \ref{Polynomial time algorithm} can be rephrased as follows.

\begin{cor}
Suppose that $S=S_{g,b}$ is a surface satisfying $\xi(S)=3g-3+b \geq 1$. Suppose that $\tau \subset S$ is a large train track and $N=N(\tau)$ is a tie neighbourhood of $\tau$ in $S$. Let $\alpha \subset S$ be an essential, non-peripheral, properly immersed multiarc or multicurve given via its snippet decomposition. Then efficient position for $\alpha$ with respect to $N$ exists and can be obtained in $\mathit{O}(\chi(S)^4  \cdot \len(\alpha)^2)$ time.
\end{cor}

\setcounter{chapter}{1}
\setcounter{section}{1}
\setcounter{thm}{1}

\section{Comparison with previous work}

Even though the language and combinatorial set-up differs greatly from the one in \cite{Taka}, the algorithm presented in this thesis and the algorithm suggested by a careful reading of \cite{Taka} coincide in the case of closed surfaces. However, this thesis constitutes more than a time-complexity analysis of Takarajima's constructive proof. Firstly, we are more general as we allow surfaces with boundary, whereas Takarajima only considers closed surfaces. Secondly, our combinatorial set-up uses arcs and curves with transverse self-intersection, whereas the barycentric subdivision of complementary regions carried out in \cite{Taka} forces non-transverse self-intersection even for curves which are initially embedded. Thus, one cannot easily recognize embedded arcs and curves or compute intersection numbers in the set-up of \cite{Taka}. This is contrary to the aim of using efficient position to provide a coordinate system for embedded curves on surfaces. Thirdly, the algorithm in this thesis is formulated purely in terms of local homotopies, whereas \cite{Taka} requires semi-local arguments. Thus, we can, and do, give pseudocode for our algorithm as well as prove its correctness.

Studying Takarajima's proof for the existence of efficient position, one is tempted to look for a lexicographic complexity involving the length or curvature of a curve which decreases when applying local or semi-local homotopies to the arc or curve. All attempts by the author to find a simple linear or lexicographic measure decreasing for all types of local homotopies were unsuccessful. When considering arcs containing a single bad snippet of trigon type, the search for a lexicographic order is partially successful: the sum of \emph{right} or \emph{left} \textit{duals} and carried snippets decreases under all but one type of local trigon homotopies. The number of times this ``exceptional'' type of homotopy is applied when eliminating one trigon snippet from a subarc can be bounded. However, the number of snippets of the underlying arc might be multiplied by a constant in these cases. For an illustration, we refer the reader to Figure \ref{Bad growth of curve}.
Thus, if we apply such homotopies repeatedly to achieve efficient position for longer and longer subarcs of the given arc or curve, we could see an exponential growth of the curve. We also note that this discussion would not be sufficient for curves containing a single bad snippet.

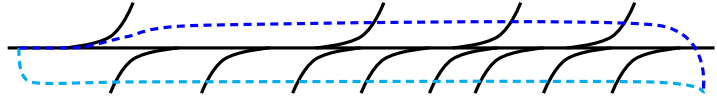
\begin{figure}[htbp]
	\centering
\begin{tikzpicture}[scale=0.6]
\draw [very thick](-12.75,-0.5) node (v1) {} -- (2.75,-0.5);
\draw [very thick] plot[smooth, tension=.7] coordinates {(-11.5,-0.5) (-10.5,-0.25) (-10,0.5)};
\draw [very thick] plot[smooth, tension=.7] coordinates {(-6,-0.5) (-5,-0.25) (-4.5,0.5)};
\draw [very thick] plot[smooth, tension=.7] coordinates {(-3,-0.5) (-2,-0.25) (-1.5,0.5)};
\draw [very thick] plot[smooth, tension=.7] coordinates {(-0.5,-0.5) (0.5,-0.25) (1,0.5)};
\draw [very thick] plot[smooth, tension=.7] coordinates {(-7,-0.5) (-8,-0.75) (-8.5,-1.5)};
\draw [very thick] plot[smooth, tension=.7] coordinates {(-5,-0.5) (-6,-0.75) (-6.5,-1.5)};
\draw [very thick] plot[smooth, tension=.7] coordinates {(-3.5,-0.5) (-4.5,-0.75) (-5,-1.5)};
\draw [very thick] plot[smooth, tension=.7] coordinates {(-2,-0.5) (-3,-0.75) (-3.5,-1.5)};
\draw [very thick] plot[smooth, tension=.7] coordinates {(-1,-0.5) (-2,-0.75) (-2.5,-1.5)};
\draw [very thick] plot[smooth, tension=.7] coordinates {(-9,-0.5) (-10,-0.75) (-10.5,-1.5)};
\draw [very thick] plot[smooth, tension=.7] coordinates {(0.5,-0.5) (-0.5,-0.75) (-1,-1.5)};
\draw [very thick] plot[smooth, tension=.7] coordinates {(2,-0.5) (1,-0.75) (0.5,-1.5)};
\draw [very thick, densely dashed, cyan] plot[smooth, tension=.4] coordinates {(-12.5,-0.5)(-12.25,-1.25) (-9.75,-1.25) (1.25,-1.25) (2.5,-1.5)};
\draw [very thick, densely dashed, blue] plot[smooth, tension=.4] coordinates { (-12.5,-0.5) (-11.142,-0.4865) (-10.125,-0.25) (-8.5,0)  (1.5,0) (2.5,-1.5)};
\end{tikzpicture}
	\caption{An arc whose snippet length is growing under a semi-local homotopy.}
	\label{Bad growth of curve}
\end{figure}

The crucial observation is that this multiplicative growth is caused by snippets cutting off index zero regions in complementary regions of the tie neighbourhood. However, subarcs created during such expanding homotopies cut off regions of index zero on one side, thus are ``blocked'' by a significant amount of index on the other side. Hence, their ``preferred'' way of moving under subsequent homotopies is back into their original ``shorter'' state. An index-argument shows that if we want to turn even a single snippet of these subarcs into a ``potentially growing'' snippet again, we first have to eliminate one bad snippet of the underlying curve. As the number of bad snippets is finite, the growth of the arc or curve under these trigon homotopies can thus be bounded. This interplay of index and length is captured in two combinatorial notions of length, the \emph{corner length} and the \emph{reduced corner length}. A careful analysis shows that the reduced corner length indeed only increases under the trigon homotopies if the trigon is ``eliminated''. As bigon homotopies naturally reduce to trigon homotopies, this will be sufficient to yield a proof of Theorem \ref{Polynomial time algorithm}.

\section{Structure of the thesis}

The layout of this thesis is as follows. In Chapter 2 we introduce the relevant notions needed for the proof of Theorem \ref{Polynomial time algorithm}. In particular, this includes the decomposition of the curve into \emph{snippets}, our first combinatorial notion of length, the \emph{snippet length} of an arc or curve, as well as the classification of snippets which are not in efficient position.
Chapter 3 provides two further combinatorial notions of length, the \emph{corner length} and the \emph{reduced corner length}. We define a family of local homotopies and study their effects on the type of bad snippets and combinatorial lengths of almost efficient arcs and curves of \textit{trigon type}. We end this chapter by presenting the algorithms \texttt{TrigArc} and \texttt{TrigCurve}, which can be used to achieve efficient position for the ``inside'' of almost efficient arcs and curves of trigon type. Both algorithms run in polynomial time in the reduced corner length and increase the latter only by an additive constant. 

In Chapter 4 we discuss how to reduce the case of almost efficient arcs of \textit{bigon type} to the case of almost efficient arcs of trigon type. We then present the algorithm \texttt{BigArc}, which achieves efficient position for the inside of almost efficient arcs of bigon and trigon type. Again, this algorithm runs in polynomial time in the reduced corner length of the arc and increases the latter only by a constant. We note that we restrict ourselves to the case of almost efficient arcs which have their bad snippets at the penultimate position. This simplifies subsequent running time analyses.

In Chapter 5 we give an algorithm that yields efficient position or reduces the number of bad snippets to at most one if $\alpha$ is an essential arc or an essential and non-peripheral curve respectively. Since the reduced corner length of an almost efficient arc or curve $\alpha$ differs from the snippet length $\len(\alpha)$ by at most a constant, the resulting algorithm halts in polynomial time in the length of $\alpha$.

In Chapter 6, we discuss how to achieve efficient position for an essential and non-peripheral curve that contains a unique bad snippet. We present a series of algorithms, each designed to deal with a certain type of bigon snippet, whose application lowers the reduced corner length of the underlying curve by at least one. Each of these algorithms terminates in polynomial time in the reduced corner length of the curve. 

In Chapter 7 we combine the algorithms presented in Chapter 5 and 6 to give a proof of Theorem \ref{Polynomial time algorithm} and Corollary \ref{Existence}.

\section{Future research directions}

Given the existence of efficient position for embedded curves, one can study ``efficient coordinates'' for curves on surfaces. We note that this yields, in particular, a global coordinate system for curves on closed surfaces. However, efficient position is only unique up to rectangle and annulus swaps (\cite[Proposition~5.4]{Taka}, \cite[Theorem~4.1]{carrySaul}). This implies a similar, though finite, ambiguity as in the case of normal coordinates on closed surfaces. 

First of all, one might ask how the graph of efficient positions of a homotopy class of curves looks like. We note that Takarajima suggests a preferred efficient position \cite[Theorem~6.7]{Taka}, which might be used to obtain a linear subspace or quotient space of efficient coordinates. This approach could provide a set of global, injective coordinates for curves on closed surfaces. 

In the context of studying mapping classes, efficient coordinates might prove useful in the following way: mapping classes can be encoded by paths in the train track graph \cite[Section~3]{Hamenstadt}. Two vertices in the train track graph are connected if one train track can be obtained from the other by a single split. Since efficient position can be achieved in polynomial time with the algorithm presented in this thesis, one can compute the change of efficient coordinates under a split in polynomial-time as well. This raises hopes that there is a formula for the coordinate change which yields a piecewise-linear map on $\mathcal{PML}$. However, this seems to require a delicate analysis of carried curves of weight at most two that are combed on one side.
Going further one can study how coordinates change under splits and folds if the curves are encoded in terms of weighted train tracks which are in efficient position with respect to a fixed train track. The latter could give a new way of finding minimal positions of curves that are given in normal coordinates, which might give an alternative approach to calculating distances in the curve graph \cite{BellWebb}. We remark that any algorithm computing intersection numbers and minimal positions of curves that builds directly on the work presented in this thesis will be dependent on the genus of the underlying surface and must require at least a quadratic number of operations in the lengths of the input curves. Using coordinates of curves that stem from surfaces decomposed into quadrilaterals, such an algorithm with quadratic running time was recently presented in \cite{DespreLazarus}.

With a detailed understanding of the changes of efficient coordinates under splits and folds, the proof of the polynomial-time recognition of mapping class types in \cite{BellWebb} could be reformulated in terms of (compressed) efficient coordinates. The efficient coordinate approach might eliminate exponential dependencies on the Euler characteristic in the running time and the need for introducing an artificial puncture in the closed case. Furthermore, an implementation of these algorithms in efficient coordinates could give a similar computer programme to ``Flipper'', which computes the action of mapping classes on laminations on a punctured surface using normal coordinates \cite{Flipper}.

\chapter{Background}
In this section, we provide the definitions needed for the proofs of Theorem \ref{Polynomial time algorithm} and Corollary \ref{Existence}.

\section{Surfaces, curves, and arcs}
Let $S = S_{g,b}$ be a compact, connected, oriented surface of genus $g$ with $b$ boundary components. The boundary components of $S$ receive the induced orientation. The \textit{complexity} of $S$ is defined as $\xi(S)=3g-3+b$. A \textit{curve} in $S$ is a proper immersion of the circle $S^1$ into $S$. We require curves to be in \textit{general position}; that is, we only consider smooth curves with transverse self-intersections. We say that an immersed curve $\alpha \subset S$ is \textit{essential} if it represents a non-trivial conjugacy class in the fundamental group of $S$. We further say that an immersed curve is \emph{peripheral} if it is homotopic to a power of a boundary component of $S$.

%Want polygonal curves in general position, otherwise infinetly many points of self-intersection, eg $C^\inf$ function $e^{-1/x^2} sin (1/x)$. Also, can consider curves as interval mod boundary identifactions, gives us parametrization hence an orientation.

An \textit{arc} in $S$ is an immersion of the interval $[0,1]$ into $S$. As for curves, we require arcs to be in general position. In particular, they intersect the boundary of $S$ transversely. We say that an arc in $S$ is \emph{proper} if it is a proper immersion of the interval $[0,1]$ into $S$. We say that a proper arc is \emph{essential} if it is not homotopic, relative its endpoints, into the boundary of $S$. An arc or curve $\alpha \subset S$ that does not have any self-intersections is said to be \textit{embedded}.

\section{Train tracks}

This brief introduction to train tracks is based on \cite[Chapter~1]{Penner}, \cite[Section~2.3]{carrySaul},\cite[Chapter~3]{Mosher_traintrack}, and \cite[Section~3.4]{SaulTieNeighbourhood}. 

\begin{defn} \label{Page:train track}
A \textit{pretrack} $\tau \subset S$ is a non-trivial, properly embedded, locally finite graph in $S$ with the following properties: for each vertex $v$ of $\tau$ there is a unique tangent line $L \subset T_v(S)$ such that for some neighbourhood $U$ of $v$ the intersection $\tau \cap U$ is a union of smooth open arcs in S, all of which are tangent to $L$ at $v$ (see Figure \ref{Branch ends}). The edges of $\tau$ are called \textit{branches} and are smoothly embedded in $S$. The vertices of $\tau$ are called \textit{switches}. We require every switch to have valence three. 
A smooth immersion $\rho: \mathbb{R} \longrightarrow S$ is a \textit{train-route} if
\begin{itemize}
\item $\rho (\mathbb{R}) \subset \tau$, and
\item $\rho (n)$ is a switch if and only if $n \in \mathbb{Z}$.
\end{itemize}
Every branch $b$ of a pretrack is required to have a train-route travelling along $b$.
\end{defn}

% A train track is called \textit{recurrent} if every branch has a closed train-route passing through itself.\\
If a train-route is periodic, it is called a \textit{closed train-route}. If $b$ is a branch of a pretrack $\tau$ and $p \in b$, a component of $b - p$ is called a \textit{half-branch} of $\tau$. If the intersection of two half-branches of a branch $b$ is again a half-branch of $b$, the two half-branches are regarded as being equivalent. An equivalence class of half-branches of $b$ is called an \textit{end} of $b$. An end $e$ of a branch $b$ is said to be $incident$ to a switch $v$ of $\tau$ if $v$ lies in the closure of a half-branch representing $e$.
For each switch $v$ of $\tau$, fix a direction in the tangent line $L \subset T_v(S)$ at $v$. We orient branch ends towards the switches they are incident to. A branch end $e$ incident to a switch $v$ is called \emph{incoming} if its orientation coincides with the chosen direction of the tangent line $L$ at $v$, \emph{outgoing} otherwise. Hence, the collection of branch ends incident to a switch can be split into two non-empty sets. If one of these sets consists of one branch end only, the corresponding branch end is called \textit{large}, otherwise \textit{small}. A branch with large ends on both sides is called \textit{large}, with a large end on one side only \textit{mixed}, and \textit{small} if both of its branch ends are small (see Figure \ref{Branch ends}).

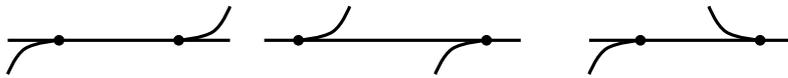
\begin{figure}[htbp]
	\centering
\begin{tikzpicture}[scale=0.45]
\draw  [very thick]plot[smooth, tension=.7] coordinates {(8.5,-0.5) (16,-0.5)};
\draw  [very thick]plot[smooth, tension=.7] coordinates {(9.5,-0.5) (10.5,-0.25) (11,0.5)};
\draw  [very thick]plot[smooth, tension=.7] coordinates {(15,-0.5) (14,-0.75) (13.5,-1.5)};
\draw [very thick](19.5,-0.5) -- (23,-0.5) node (v1) {};
\draw  [very thick]plot[smooth, tension=.7] coordinates {(18,-0.5) (18.5,-0.5) (19.5,-0.5)};
\draw [very thick] plot[smooth, tension=.7] coordinates {(18,-1.5) (18.5,-0.75) (19.5,-0.5)};
\draw [very thick] plot[smooth, tension=.7] coordinates {(v1) (22,-0.25) (21.5,0.5)};
\draw  [very thick]plot[smooth, tension=.7] coordinates {(23,-0.5) (23.5,-0.5) (24,-0.5)};
\draw  [very thick]plot[smooth, tension=.7] coordinates {(23,-0.5)  };
\draw [very thick](6,-0.5) node (v2) {} -- (2.5,-0.5);
\draw [very thick] plot[smooth, tension=.7] coordinates {(v2) (6.5,-0.5) (7.5,-0.5)};
\draw [very thick] plot[smooth, tension=.7] coordinates {(6,-0.5) (7,-0.25) (7.5,0.5)};
\draw [very thick] plot[smooth, tension=.7] coordinates {(2.5,-0.5) (2,-0.5) (1,-0.5)};
\draw [very thick] plot[smooth, tension=.7] coordinates {(2.5,-0.5) (1.5,-0.75) (1,-1.5)};
\filldraw [color=black] (2.5,-0.5) circle(4pt);
\filldraw [color=black] (6,-0.5) circle(4pt);
\filldraw [color=black]  (9.5,-0.5)circle(4pt);
\filldraw [color=black]  (15,-0.5) circle(4pt);
\filldraw [color=black]  (19.5,-0.5)circle(4pt);
\filldraw [color=black]  (23,-0.5) circle(4pt);
\end{tikzpicture}
	\caption{Examples of a large, small, and mixed branch.}
	\label{Branch ends}
\end{figure}

From a pretrack $\tau \subset S$ we build a \emph{tie neighbourhood} $N=N(\tau)$ as follows: Let $\mathcal{B} = \mathcal{B}(\tau)$ be the set of branches of $\tau$ and $\mathcal{S}= \mathcal{S}(\tau)$ its set of switches. We take one rectangle $R_b$ for each branch $b \in \mathcal{B}$ and one rectangle $R_s$ for each switch $s \in \mathcal{S}$. All rectangles are foliated by vertical arcs (the \emph{ties}). The boundary of each such rectangle $R$ consists of four edges: two edges that are parallel to the ties, and two that are perpendicular to the ties. We refer to the union of the first two edges as the \emph{vertical boundary} of $R$ and denote it by $\partial_v R$. The union of the perpendicular edges is called the \emph{horizontal boundary} of $R$ and is denoted by $\partial_h R$. For an illustration, we point the reader to Figure \ref{A tie neighbourhood rectangle}. \label{Page: Rectangle}
Suppose that $e$ is a branch end of $b \in \mathcal{B}$ that is incident to a switch $s \in \mathcal{S}$. We then glue the edge of $\partial_v R_b$ corresponding to $e$ to a subset of $\partial_v R_s$ as determined by the combinatorics of the train track (see Figure \ref{Tie neighbourhood}). The resulting surface $N=N(\tau)$ can be embedded into $S$ in such a way that it is disjoint from $\partial S$ and that the train track $\tau$ is properly embedded in it. We call this surface $N=N(\tau)$ the \emph{tie neighbourhood} of the train track $\tau$. \label{Page:tie neighbourhood}

\begin{rem}
In the literature, there are three common models for neighbourhoods of train tracks and branched surfaces: the \textit{cusp model} (\cite[page~90]{Penner}), the \textit{corner model} (introduced above, see also \cite[page ~11]{SaulTieNeighbourhood}), and the smooth model (\cite[page~28]{Mosher96laminationsand}). The combinatorial setting of this thesis asks for generic intersections of arcs and curves with themselves and the boundary of the tie neighbourhood. Thus, the corner model is the natural choice.
\end{rem}

\begin{figure}[htbp] 
  \begin{minipage}[b]{0.40\linewidth}
    \centering
\begin{tikzpicture}[scale=0.35]
\draw [very thick] (-3,1.5) rectangle (4,-3);
\node at (0.5,2.25) {\tiny{$\partial_h R$}};
\node at (-4,-0.75) {\tiny{$\partial_v R$}};
\node at (0.5,-3.75) {\tiny{$\partial_h R$}};
\node at (5,-0.75) {\tiny{$\partial_v R$}};
\draw [very thin, densely dotted] (-2.75,1.5) rectangle (3.75,-3);
\draw [very thin, densely dotted] (-2.5,1.5) rectangle (3.5,-3);
\draw [very thin, densely dotted] (-2.25,1.5) rectangle (3.25,-3);
\draw [very thin, densely dotted] (-2,1.5) rectangle (3,-3);
\draw [very thin, densely dotted] (-1.75,1.5) rectangle (2.75,-3);
\draw [very thin, densely dotted] (-1.5,1.5) rectangle (2.5,-3);
\draw [very thin, densely dotted] (-1.25,1.5) rectangle (2.25,-3);
\draw [very thin, densely dotted] (-1,1.5) rectangle (2,-3);
\draw [very thin, densely dotted] (-0.75,1.5) rectangle (1.75,-3);
\draw [very thin, densely dotted] (-0.5,1.5) rectangle (1.5,-3);
\draw [very thin, densely dotted] (-0.25,1.5) rectangle (1.25,-3);
\draw [very thin, densely dotted] (0,1.5) rectangle (1,-3);
\draw [very thin, densely dotted] (0.25,1.5) rectangle (0.75,-3);
\draw [very thin, densely dotted](0.5,1.5) -- (0.5,-3);
\end{tikzpicture}
    \caption{A tie neighbourhood rectangle.} 
    \label{A tie neighbourhood rectangle}
  \end{minipage} 
\begin{minipage}[b]{0.59\linewidth}
\centering
\begin{tikzpicture}[scale=0.45]
\draw [very thick] (-3.5,2) rectangle (0,0.5);
\draw [very thick] (-3.5,-1) rectangle (0,-2.5);
\draw [very thick] (0,2) rectangle (4.5,-2.5) node (v1) {};
%%%%%%%%%%%%%%%%%%%%%%%%%%%%%%%%
\draw [very thin, densely dotted] (-3.25,2) rectangle (-0.25,0.5);
\draw [very thin, densely dotted] (-3,2) rectangle (-0.5,0.5);
\draw [very thin, densely dotted] (-2.75,2) rectangle (-0.75,0.5);
\draw [very thin, densely dotted] (-2.5,2) rectangle (-1,0.5);
\draw [very thin, densely dotted] (-2.25,2) rectangle (-1.25,0.5);
\draw [very thin, densely dotted] (-2,2) rectangle (-1.5,0.5);
\draw [very thin, densely dotted](-1.75,2) -- (-1.75,0.5);
\draw [very thin, densely dotted](-1.75,-1) -- (-1.75,-2.5);
\draw [very thin, densely dotted](-3.25,-1) rectangle (-0.25,-2.5);
\draw [very thin, densely dotted] (-0.25,-2.5) rectangle (-0.25,-2.5);
\draw [very thin, densely dotted](-3,-1) rectangle (-0.5,-2.5);
\draw [very thin, densely dotted] (-2.75,-1) rectangle (-0.75,-2.5);
\draw [very thin, densely dotted] (-2.5,-1) rectangle (-1,-2.5);
\draw [very thin, densely dotted] (-2.25,-1) rectangle (-1.25,-2.5);
\draw [very thin, densely dotted] (-2,-1) rectangle (-1.5,-2.5);
\draw [very thin, densely dotted] (0.25,2) rectangle (4.25,-2.5);
\draw [very thin, densely dotted] (0.5,2) rectangle (4,-2.5);
\draw [very thin, densely dotted] (0.75,2) rectangle (3.75,-2.5);
\draw [very thin, densely dotted](1,2) rectangle (3.5,-2.5);
\draw [very thin, densely dotted] (1.25,2) rectangle (3.25,-2.5);
\draw [very thin, densely dotted] (1.5,2) rectangle (3,-2.5);
\draw [very thin, densely dotted] (1.75,2) rectangle (2.75,-2.5);
\draw [very thin, densely dotted](2,2) rectangle (2.5,-2.5);
\draw [very thin, densely dotted] (2.25,2) rectangle (v1);
\draw [very thick] plot[smooth, tension=.7] coordinates {(9,-0.25) (3.25,-0.25)};
\draw [very thick,fill] (3,-0.25) ellipse (0.125 and 0.125);
\draw [very thick] (4.5,2) rectangle (9,-2.5);
\draw [very thin, densely dotted] (4.75,2) rectangle (8.75,-2.5);
\draw [very thin, densely dotted] (5,2) rectangle (8.5,-2.5);
\draw [very thin, densely dotted] (5.25,2) rectangle (8.25,-2.5);
\draw [very thin, densely dotted] (5.5,2) rectangle (8,-2.5);
\draw [very thin, densely dotted] (5.75,2) rectangle (7.75,-2.5);
\draw [very thin, densely dotted] (6,2) rectangle (7.5,-2.5);
\draw [very thin, densely dotted] (6.25,2) rectangle (7.25,-2.5);
\draw [very thin, densely dotted] (6.5,2) rectangle (7,-2.5);
\draw [very thin, densely dotted](6.75,2) -- (6.75,-2.5);
\draw [very thick] plot[smooth, tension=.7] coordinates {(3.25,-0.25) (2.5,-0.45) (2,-1.5) (0.75,-1.75) (-3.5,-1.8)};
\draw [very thick] plot[smooth, tension=.7] coordinates {(3.25,-0.25) (2.5,-0.050) (2,1) (0.75,1.25) (-3.5,1.26)};
\end{tikzpicture}
\caption{Identifications of the vertical boundaries of branch and switch rectangles.}
	\label{Tie neighbourhood}
\end{minipage}
\end{figure}

The images of the rectangles $R_b$ for $b \in \mathcal{B}$ are called \emph{branch rectangles}, and, in an abuse of notation, are again denoted by $R_b$. In a similar abuse of notation we have \emph{switch rectangles} $R_s$ for $s \in \mathcal{S}$ and ties and vertical and horizontal boundaries of the embedded rectangles in $S$.
The collection of branch and switch rectangles of a tie neighbourhood $N=N(\tau)$ is denoted by $\mathcal{R}_\textrm{tie}=\mathcal{R}_\textrm{tie}(N)$. \label{Page: Rtie}
The \textit{horizontal boundary} $\partial_h N$ is the union of $\partial_h R$ for $R \in \mathcal{R}_\textrm{tie}$, while the \textit{vertical boundary} is $\partial_v N = \overline{\partial N - \partial_h N}$.

Let $N = N (\tau)$ be a tie neighbourhood of a pretrack $\tau$. Let $R$ be a \textit{complementary region} of $N$, that is, a component of $\overline{S - N}$. The collection of complementary regions of $N$ is denoted by $\mathcal{R}_\textrm{comp}=\mathcal{R}_\textrm{comp}(N)$. \label{Page: Rcomp} We set $\mathcal{R}=\mathcal{R}(N)=\mathcal{R}_\textrm{tie}(N) \cup \mathcal{R}_\textrm{comp}(N)$.  \label{Page:R} For each complementary component $R \in \mathcal{R}_\textrm{comp}$, the \textit{horizontal} and \textit{vertical boundary} of $R$ are defined as $\partial_h R = \partial R \cap \partial_h N$ and $\partial_v R = \partial R \cap \partial_v N$. So, for any $R \in \mathcal{R}_\textrm{comp}$ there is a, possibly empty, subset $B \subset \partial S$ such that $\partial R = \partial_h R \cup \partial_v R \cup B$. Connected subsets of the horizontal and vertical boundary of each complementary region $R \in \mathcal{R}_\textrm{comp}$ meet perpendicularly at their boundaries.

Suppose that $T \subset S$ is a subsurface with $\partial T$ being a finite union of one-dimensional smooth submanifolds meeting perpendicularly at their boundaries.\label{Page:Index and Corners} We say that $c$ is a \emph{corner} of $T$ if $c$ lies in the perpendicular intersection of two one-dimensional smooth submanifolds of $\partial T$. The set of all corners of $T$ is denoted by $\partial^2 T$.
Following \cite[Section~2.3]{carrySaul}, we set
\begin{align*}
\ind (T) = \chi (T) - \dfrac{\partial^2 T^-}{4} + \dfrac{\partial^2 T^+}{4},
\end{align*} 
where $\partial^2 T^-$ and $\partial^2 T^+$ denote the number of outward- and inward-pointing corners of $T$ respectively. We note that the index is additive under unions if the interiors of the two subsurfaces in question are disjoint.

%%%%%%%%%%%%%%%%%% For every point $x \in Int(G_1 \cap G_2)$, either no corners or one inside and one outside corner. Straightening in the last case doesn't change the overall index. For every point $x \notin Int(G_1 \cap G_2)$ the boundaries of $G_1$ and $G_2$ have to depart from each other at that point, so either one straight and one up, or one up and one down. In both cases, signed corner quarters plus $\frac{1}{2}$ times the number of points in the boundary of $G_1 \cap G_2$ (which is equal to 1), adds up to the correct signed corner. Remark: in the case of cusps there are four cases to consider!

Suppose that $N = N (\tau)$ is a tie neighbourhood of a pretrack $\tau$ and $R \in \mathcal{R}_\textrm{comp}$ is a finite-sided complementary region of $N$. Then all corners of $R$ are pointing outwards and the index simplifies to
\begin{align*}
\ind (R) = \chi (R) - \frac{|\partial^2 R|}{4}.
\end{align*} 

Since the tie neighbourhood of a train track is a union of rectangles whose interiors are disjoint, we know that $\ind(N)=0$.

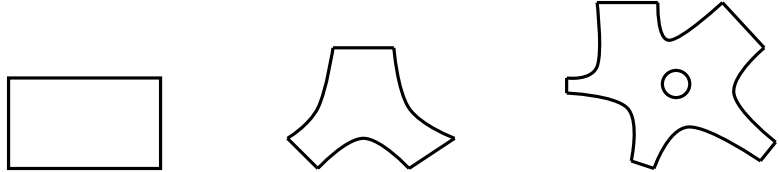
\begin{figure}[htbp]
	\centering
	\begin{subfigure}[t]{1in}
		\centering
		\begin{tikzpicture}[scale=0.4]
\draw [very thick](-0.5,1) node (v1) {} -- (-0.5,-2) -- (4.5,-2) -- (4.5,1) -- (-0.5,1) -- cycle;
\end{tikzpicture}	
	\end{subfigure}
\qquad \quad
	\begin{subfigure}[t]{1in}
		\centering
		\begin{tikzpicture}[scale=0.4]
\draw [very thick] plot[smooth, tension=.7] coordinates {(0.5,1.5) (0,-0.5) (-1,-1.5)};
\draw [very thick] plot[smooth, tension=.7] coordinates {(0,-2.5) (1.5,-1.5) (3,-2.5)};
\draw [very thick] plot[smooth, tension=.7] coordinates {(4.5,-1.5) (3,-0.5) (2.5,1.5)};
\draw [very thick](0.5,1.5) -- (2.5,1.5);
\draw [very thick](-1,-1.5) -- (0,-2.5);
\draw [very thick](3,-2.5) -- (4.5,-1.5);
\end{tikzpicture}
	\end{subfigure}
\qquad \quad
	\begin{subfigure}[t]{1in}
		\centering
		\begin{tikzpicture}[scale=0.4]
\draw [very thick] plot[smooth, tension=.7] coordinates {(-2.5,0) (-0.5,-0.5) (-0.375,-2.25)};
\draw [very thick] plot[smooth, tension=.7] coordinates {(0.375,-2.5) (1.5,-1.125) (3.875,-2.25)};
\draw [very thick] plot[smooth, tension=.7] coordinates {(4.375,-1.625) (3,0) (4,1.5)};
\draw [very thick] plot[smooth, tension=.7] coordinates {(2.625,3) (0.875,1.75) (0.5,3)};
\draw [very thick] plot[smooth, tension=.7] coordinates {(-1.5,3) (-1.5,0.875) (-2.5,0.5)};
\draw [very thick] (1.1,0.3) ellipse (0.45 and 0.45);
\draw [very thick, fill](-2.5,0.5) -- (-2.5,0);
\draw [very thick](-0.375,-2.25) -- (0.375,-2.5);
\draw [very thick] plot[smooth, tension=.7] coordinates {(-1.5,3) (0.5,3)};
\draw [very thick] plot[smooth, tension=.7] coordinates {(2.625,3) (4,1.5)};
\draw [very thick] plot[smooth, tension=.7] coordinates {(4.375,-1.625) (3.875,-2.25)};
\end{tikzpicture}
	\end{subfigure}
	\caption{These subsurfaces have indices $0$, $-2/4$, and $-10/4$ respectively.}
\end{figure}

\begin{defn}
Suppose that $\tau \subset S $ is a pretrack and $N = N(\tau)$ is a tie neighbourhood of $\tau$ in $S$. We say that $\tau$ is a \textit{train track} if $\tau$
is compact and every complementary region $R$ of $N$ has negative index.
\end{defn}

Following \cite[Section~3.1]{MasurMinskyI}, we say that a train track $\tau \subset S$ is \textit{large} if all complementary regions of $N(\tau)$ are discs or peripheral annuli. For the remainder of this thesis, we restrict ourselves to large train tracks.

\section{Snippets and winding numbers}

Let $S = S_{g,b}$ be a surface satisfying $\xi(S)=3g-3+b \geq 1$. Let $\tau \subset S$ be a large train track and $N=N(\tau)$ be a tie neighbourhood of $\tau$ in $S$.

\subsection{Snippets}
As seen in the previous section, $N$ is the union of branch and switch rectangles $R \in \mathcal{R}_\textrm{tie}$. Each rectangle is modelled on a rectangular box. Boxes for switch rectangles are equipped with two short horizontal dashes on their vertical boundary indicating the gluings with the adjacent branch rectangles. As $\tau$ is a large train track, each complementary region can be modelled on a, perhaps peripheral, polygon. Hence, $\mathcal{R}$ provides a tiling of the surface $S$ by rectangles, polygons, and peripheral annuli. For any region $R \in \mathcal{R}$, we label subsets of $\partial R$ by $h$ or $v$ if they are part of the horizontal or vertical boundary of $R$ respectively. All subsets of $\partial R$ that are (identified with) a tie of a branch rectangle are labelled by $t$. For an illustration see Figures \ref{branch rectangle} and \ref{switch rectangle}.

\begin{figure}[htbp] 
  \begin{minipage}[b]{0.49\linewidth}
    \centering
\begin{tikzpicture}[scale=0.35]
\draw [very thick] (-3,1.5) rectangle (4,-3);
\node at (0.5,2.25) {\tiny{$h$}};
\node at (-3.75,-0.75) {\tiny{$t$}};
\node at (0.5,-3.75) {\tiny{$h$}};
\node at (4.75,-0.75) {\tiny{$t$}};
\draw [very thin, densely dotted] (-2.75,1.5) rectangle (3.75,-3);
\draw [very thin, densely dotted] (-2.5,1.5) rectangle (3.5,-3);
\draw [very thin, densely dotted] (-2.25,1.5) rectangle (3.25,-3);
\draw [very thin, densely dotted] (-2,1.5) rectangle (3,-3);
\draw [very thin, densely dotted] (-1.75,1.5) rectangle (2.75,-3);
\draw [very thin, densely dotted] (-1.5,1.5) rectangle (2.5,-3);
\draw [very thin, densely dotted] (-1.25,1.5) rectangle (2.25,-3);
\draw [very thin, densely dotted] (-1,1.5) rectangle (2,-3);
\draw [very thin, densely dotted] (-0.75,1.5) rectangle (1.75,-3);
\draw [very thin, densely dotted] (-0.5,1.5) rectangle (1.5,-3);
\draw [very thin, densely dotted] (-0.25,1.5) rectangle (1.25,-3);
\draw [very thin, densely dotted] (0,1.5) rectangle (1,-3);
\draw [very thin, densely dotted] (0.25,1.5) rectangle (0.75,-3);
\draw [very thin, densely dotted](0.5,1.5) -- (0.5,-3);
\end{tikzpicture}
    \caption{A branch rectangle} 
    \label{branch rectangle}
  \end{minipage} 
\begin{minipage}[b]{0.49\linewidth}
\centering
\begin{tikzpicture}[scale=0.35]
\draw [very thick] (-3,1.5) rectangle (4,-3);
\node at (0.5,2.25) {\tiny{$h$}};
\node at (-3.75,-0.75) {\tiny{$t$}};
\node at (0.5,-3.75) {\tiny{$h$}};
\node at (4.75,-2.25) {\tiny{$t$}};
\draw [very thick](3.5,0) -- (4.5,0);
\draw [very thick](3.5,-1.5) -- (4.5,-1.5);
\node at (4.75,-0.75) {\tiny{$v$}};
\node at (4.75,0.75) {\tiny{$t$}};
\draw [very thin, densely dotted] (-2.75,1.5) rectangle (3.75,-3);
\draw [very thin, densely dotted] (-2.5,1.5) rectangle (3.5,-3);
\draw [very thin, densely dotted] (-2.25,1.5) rectangle (3.25,-3);
\draw [very thin, densely dotted] (-2,1.5) rectangle (3,-3);
\draw [very thin, densely dotted] (-1.75,1.5) rectangle (2.75,-3);
\draw [very thin, densely dotted] (-1.5,1.5) rectangle (2.5,-3);
\draw [very thin, densely dotted] (-1.25,1.5) rectangle (2.25,-3);
\draw [very thin, densely dotted] (-1,1.5) rectangle (2,-3);
\draw [very thin, densely dotted] (-0.75,1.5) rectangle (1.75,-3);
\draw [very thin, densely dotted] (-0.5,1.5) rectangle (1.5,-3);
\draw [very thin, densely dotted] (-0.25,1.5) rectangle (1.25,-3);
\draw [very thin, densely dotted] (0,1.5) rectangle (1,-3);
\draw [very thin, densely dotted] (0.25,1.5) rectangle (0.75,-3);
\draw [very thin, densely dotted](0.5,1.5) -- (0.5,-3);
\end{tikzpicture}
 \caption{A switch rectangle} 
    \label{switch rectangle}
\end{minipage}
\end{figure}

We define $\partial \mathcal{R}$ \label{Page: boundary R} to be the union of all boundaries of the regions $R \in \mathcal{R}$. That is, $\partial \mathcal{R} = \bigcup_{R \in \mathcal{R}} \partial R$ is a trivalent graph whose edges meet at an angle of $\pi$ or $\pi/2$. We further set $\partial^2 \mathcal{R}$ to be the union of all points in $\partial^2 R$ for $R \in \mathcal{R}$. Suppose that $R \in \mathcal{R}$ is a rectangle or polygon. In the following, we refer to the closure of a component of $\partial R - \partial^2 R$ as a \emph{side} of $R$. Hence, for $R \in \mathcal{R}_\textrm{comp}$, a side of $R$ can contain more than two points of $\partial^2 \mathcal{R}$ if it is part of the horizontal boundary of $R$ (see Figure \ref{Long horizontal boundary side}). If $R \in \mathcal{R}_\textrm{tie}$ is a switch rectangle, then one vertical boundary side of $R$ contains two points of $\partial^2 \mathcal{R}$ in its interior. In our figures of switch rectangles, these two extra corners are indicated by short horizontal dashes.

\begin{figure}[htbp] 
\begin{minipage}[b]{0.99\linewidth}
    \centering
\begin{tikzpicture}[scale=0.55]
\draw [very thick](-4,-1.5) {} -- (-4,-3) {};
\draw [very thick] plot[smooth, tension=.7] coordinates { (-4,-3) (1,-2.5) (6,-3)};
\draw [very thick](6,-1.5) node (v3) {} -- (6,-3);
\draw [very thick] plot[smooth, tension=.7] coordinates {(-4,-1.5) (-1,-0.5) (0,2)};
\draw [very thick] plot[smooth, tension=.7] coordinates {(v3) (3,-0.5) (2,2)};
\draw [very thick] (-4,-4) node (v1) {} rectangle (-6,-0.5);
\draw [very thick] plot[smooth, tension=.7] coordinates {(v1)  (-2.4,-3.8)};
\draw [very thick] plot[smooth, tension=.7] coordinates {(-4,-0.5) (-3,-0.25) (-2.5,0)};
\draw [very thick](-2.5,0) -- (-2.1915,-1.0747);
\node at (-6.5,-2) {{\tiny{$t$}}};
\node at (-5,0) {{\tiny{$h$}}};
\draw [very thick] (0.9752,-0.8739) ellipse (0.35 and 0.35);
\draw [very thick](-2.4,-2.78) -- (-2.4,-5.4) -- (-0.6,-5.4) -- (-0.6,-2.58);
\draw [very thick](-0.6,-5.4)-- (0.9,-5.4)  -- (0.9,-2.5);
\draw [very thick](0.9,-5.4) -- (2.4,-5.4) -- (2.4,-2.56);
\draw [very thick](-2.6,-4.588) -- (-2.2,-4.588);
\draw [very thick](2.2,-3.686) -- (2.6,-3.686);
\draw [very thick](2.2,-4.588) -- (2.6,-4.588);
\node at (2.94,-4.16) {{\tiny{$v$}}};
\node at (2.86,-3.1) {{\tiny{$t$}}};
\node at (0.2,-5.9) {{\tiny{$h$}}};
\node at (0.08,-1.4) {{\tiny{$\partial S$}}};
\draw [very thick](0,2) -- (2,2);
\node at (0.947,2.496) {{\tiny{$v$}}};
\node at (6.5175,-2.3319) {{\tiny{$v$}}};
\end{tikzpicture}
    \caption{A horizontal boundary component of $R \in \mathcal{R}_\textrm{comp}$ containing more than two points of $\partial^2 \mathcal{R}$.}
    \label{Long horizontal boundary side}
  \end{minipage}
\end{figure}
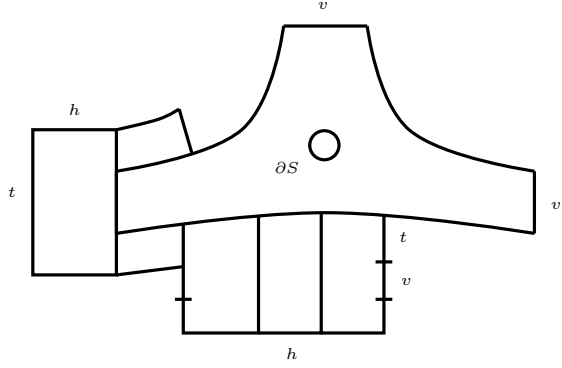

\begin{defn}
Suppose that $R \in \mathcal{R}$. Suppose that $D=(I,\partial I)$ or $D=(S^1,\emptyset)$. We say that $a:D \rightarrow (R,\partial R)$ is a \emph{snippet} if $a$ is a transverse and self-transverse immersion of pairs.
\end{defn}

We note that the transversality of the map $D \rightarrow (R,\partial R)$ implies that $\partial D$ misses the corners of $R$.
In an abuse of notation, we denote by $a$ the snippet as well as its image in the surface $S$. For some examples of snippets we refer the reader to Figure \ref{Examples of snippets}.

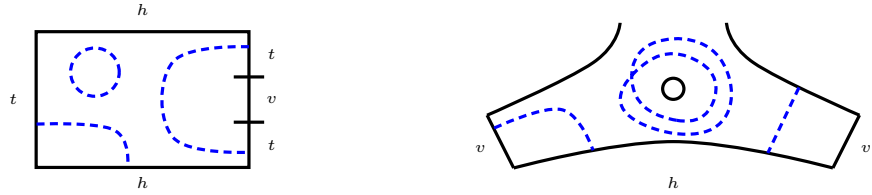
\begin{figure}[htbp] 
\begin{minipage}[b]{0.49\linewidth}
    \centering
\begin{tikzpicture}[scale=0.4]
\draw [very thick] (-3,1.5) rectangle (4,-3);
\node at (0.5,2.25) {\tiny{$h$}};
\node at (-3.75,-0.75) {\tiny{$t$}};
\node at (0.5,-3.5) {\tiny{$h$}};
\node at (4.75,-2.25) {\tiny{$t$}};
\draw [very thick](3.5,0) -- (4.5,0);
\draw [very thick](3.5,-1.5) -- (4.5,-1.5);
\node at (4.75,-0.75) {\tiny{$v$}};
\node at (4.75,0.75) {\tiny{$t$}};
\draw [densely dashed, blue, very thick] plot[smooth, tension=.7] coordinates {(4,1) (1.5,0.5) (1.5,-2) (4,-2.5)};
\draw [very thick, densely dashed, blue] (-1.0666,0.1589) ellipse (0.8 and 0.8);
\draw [very thick, densely dashed, blue] plot[smooth, tension=.7] coordinates {(-2.9873,-1.5637) (-0.4588,-1.7504) (0.0093,-2.9745)};
\end{tikzpicture}
  \end{minipage}
\begin{minipage}[b]{0.49\linewidth}
    \centering
\begin{tikzpicture}[scale=0.35]
\draw [very thick](-6,-3) node (v1) {} -- (-5,-5) node (v2) {};
\draw [very thick] plot[smooth, tension=.7] coordinates {(v1) (-2,-1) (-1,0.5)};
\draw [very thick] plot[smooth, tension=.7] coordinates {(v2) (1,-4) (7,-5)};
\draw [very thick](7,-5) -- (8,-3) node (v3) {};
\draw [very thick] plot[smooth, tension=.7] coordinates {(v3) (4,-1) (3,0.5)};
\node at (1,-5.55) {\tiny{$h$}};
\node at (-6.25,-4.25) {\tiny{$v$}};
\node at (8.25,-4.25) {\tiny{$v$}};
\draw [very thick] (1,-2) ellipse (0.4 and 0.4);
\draw [very thick, densely dashed, blue] plot[smooth cycle, tension=.9] coordinates {(-0.5,-2) (1.1295,-3.1939) (2.5,-2.5) (2,-1) (0.0096,-0.9903) (-0.7489,-2.9427) (2.5,-3.5) (2.7999,-0.9149) (0.3619,-0.2246) };
\draw [very thick, densely dashed, blue] plot[smooth, tension=.7] coordinates {(-5.7345,-3.4896) (-3.2004,-2.7892) (-2.0162,-4.33)};
\draw [very thick, densely dashed, blue] plot[smooth, tension=.7] coordinates {(5.726,-1.9488) (4.5545,-4.4319)};
\end{tikzpicture}
\end{minipage}
\caption{Snippets in a switch rectangle and complementary region.}
\label{Examples of snippets}
\end{figure}

For the remainder of this thesis, we assume that $S^1$ is parametrized as $[0,1]/_\sim$. Hence, every snippet carries a canonical orientation induced by its parametrization.

\begin{defn}
Suppose that $a_1,a_2:D \rightarrow R$ are two snippets contained in the same region $R \in \mathcal{R}$. We say that $a_1$ and $a_2$ are \emph{strongly snippet homotopic} if $a_1$ and $a_2$ are homotopic via a transverse homotopy of pairs $D \rightarrow (R,\partial R)$ which keeps the $k$-skeleton of $[0,1]$ or $S^1$ inside the $(k+1)$-skeleton of $\mathcal{R}$ at all times. The strong snippet homotopy class of a snippet $a \subset R \in \mathcal{R}$ will be denoted by $[a]$.
\end{defn}

We note that strong snippet homotopies preserve the orientation of the snippets.

\begin{rem}
In the further course of this thesis, a coarser equivalence relation on snippets will turn out to be important, too. This will be called \emph{weak snippet homotopy}. Homotopies of this kind are only required to avoid the corners of the region that contains the snippet, and not all corners of $\partial^2 \mathcal{R}$. For a formal statement, we point the reader to Definition \ref{Weak snippet homotopy} on page \pageref{Weak snippet homotopy}.
\end{rem}

\begin{lem}\label{classification of snippets in simply connected regions}
Suppose that $a \subset R$ is a snippet that lies inside a simply connected region $R \in \mathcal{R}$. Then $[a]$ is, up to orientation, uniquely determined by the components of $\partial \mathcal{R}-\partial^2 \mathcal{R}$ containing $\partial a$. \qed
\end{lem}

For peripheral regions $R \in \mathcal{R}$, recording the intersection with the components of $\partial \mathcal{R}-\partial^2 \mathcal{R}$ is not sufficient to determine a strong snippet homotopy class. In addition to the intersection data we have to provide information about the winding behaviour of the snippet around the boundary component of $S$.

\subsection{Winding number}

Traditionally, the winding number of an arc or curve around a point is an integer representing the total number of times that the arc or curve travels around that point. However, the combinatorial set-up of this thesis makes it convenient to take a different approach: suppose that $R \in \mathcal{R}_\textrm{comp}$ is a peripheral annulus and $a \subset R$ is a snippet. We set the winding number of $a$ to equal the number of outward-pointing corners of $R$ that $a$ passes while travelling around the component of $\partial S$. Hence, winding numbers in this thesis are integers but do not count an integer number of winding.

Suppose that $R \in \mathcal{R}$. Recall that we refer to the closure of a component of $\partial R - \partial^2 R$ as a side of $R$. By construction of the tie neighbourhood, the sides of $R$ are smoothly embedded arcs meeting perpendicularly at their endpoints. For any snippet $a \subset R$ we may assume that $a$ meets $\partial R \cup \partial S$ perpendicularly.

\begin{defn}\label{Page: winding number}
Suppose that $R \in \mathcal{R}_\textrm{comp}$ is a $(2n+1)$-sided peripheral annulus. Suppose that $a \subset R$ is a snippet of minimal self-intersection. We assign a \emph{winding number} to $a$ as follows:
\begin{itemize}
\item  If $\partial a \cap \partial S \neq \emptyset$, we set $\wind(a)=0$.
\item If $\partial a \neq \emptyset$ and $\partial a \subset \partial_h R \cup \partial_v R$, we consider a lift $\widetilde{a}$ of $a$ to the universal cover $\widetilde{R}$ of $R$. This lift bounds a region $Q \subset \widetilde{R}$ of finite index. In this case, the winding number of $a$ is defined as 
\begin{align*}
\wind (a) \ = \ -4 \cdot \ind(Q) + 2.
\end{align*}
\item If $\partial a = \emptyset$, then $a$ is a curve inside a peripheral annulus. As ${\pi}_1(R)= \mathbb{Z}$, $a$ is homotopic to the $k$-th power of a component of $\partial S$, where $k \in \mathbb{Z}$, and we say that $\wind(a)=k \cdot 2n$.
\end{itemize}
We further decorate $\wind (a)$ with a sign obtained as follows: The winding number is equipped with a positive sign if $\widetilde{a} \subset \widetilde{R}$ with its induced orientation from $a$ bounds the finite index region $Q$ or a strip with the lift of $\partial R$ on its right-hand side. Else, we equip $\wind (a)$ with a negative sign.
\end{defn}

%Universal cover of $T$ is homeomorphic to $\mathbb{R}xI$ with decorations on one side$
For examples of snippets and their winding numbers we refer the reader to Figures \ref{Example of winding number I}-\ref{Example of winding number II}.

\begin{figure}[htbp] 
  \begin{minipage}[b]{0.46\linewidth}
    \centering
\begin{tikzpicture}[scale=0.4]
\draw [very thick] plot[smooth, tension=.7] coordinates {(-2.5,0) (-0.5,-0.5) (-0.375,-2.25)};
\draw [very thick] plot[smooth, tension=.7] coordinates {(0.375,-2.5) (1.5,-1.125) (3.875,-2.25)};
\draw [very thick] plot[smooth, tension=.7] coordinates {(4.375,-1.625) (3,0) (4,1.5)};
\draw [very thick] plot[smooth, tension=.7] coordinates {(2.625,3) (0.875,1.75) (0.5,3)};
\draw [very thick] plot[smooth, tension=.7] coordinates {(-1.5,3) (-1.5,0.875) (-2.5,0.5)};
\draw [very thick] (1.25,0.375) ellipse (0.3 and 0.3);
\draw [very thick, fill](-2.5,0.5) -- (-2.5,0);
\draw [very thick](-0.375,-2.25) -- (0.375,-2.5);
\draw [very thick] plot[smooth, tension=.7] coordinates {(-1.5,3) (0.5,3)};
\draw [very thick] plot[smooth, tension=.7] coordinates {(2.625,3) (4,1.5)};
\draw [very thick] plot[smooth, tension=.7] coordinates {(4.375,-1.625) (3.875,-2.25)};
\draw [densely dashed, blue, very thick] plot[smooth, tension=.7] coordinates {(1.5,-1.125) (2,-0.375) (2.375,0.875) (1.5,1.5) (0,1.5) (-0.375,3)};
\draw [fill, blue, very thick](2,-0.375) -- (1.75,-0.375) -- (2.125,-0.625) -- cycle;
\end{tikzpicture}
    \caption{Winding number $5$.} 
    \label{Example of winding number I}
    \vspace{2ex}
  \end{minipage} 
\begin{minipage}[b]{0.53\linewidth}
\centering
\begin{tikzpicture}[scale=0.4]
\draw [very thick] plot[smooth, tension=.7] coordinates {(-2.5,0) (-0.5,-0.5) (-0.375,-2.25)};
\draw [very thick] plot[smooth, tension=.7] coordinates {(0.375,-2.5) (1.5,-1.125) (3.875,-2.25)};
\draw [very thick] plot[smooth, tension=.7] coordinates {(4.375,-1.625) (3,0) (4,1.5)};
\draw [very thick] plot[smooth, tension=.7] coordinates {(2.625,3) (0.875,1.75) (0.5,3)};
\draw [very thick] plot[smooth, tension=.7] coordinates {(-1.5,3) (-1.5,0.875) (-2.5,0.5)};
\draw [very thick] (1.25,0.375) ellipse (0.3 and 0.3);
\draw [very thick, fill](-2.5,0.5) -- (-2.5,0);
\draw [very thick](-0.375,-2.25) -- (0.375,-2.5);
\draw [very thick] plot[smooth, tension=.7] coordinates {(-1.5,3) (0.5,3)};
\draw [very thick] plot[smooth, tension=.7] coordinates {(2.625,3) (4,1.5)};
\draw [very thick] plot[smooth, tension=.7] coordinates {(4.375,-1.625) (3.875,-2.25)};
\draw [densely dashed, blue, very thick] plot[smooth, tension=.7] coordinates {(1.375,-1.125) (0.75,-0.625) (0.375,0.125) (0.875,0.875) (1.625,1) (2.125,0.375) (1.875,-0.25) (1.125,-0.625) (0.25,-0.5) (0,0.375) (0.625,1) (1.5,1.25) (2.15,1.125) (2.475,0.6625) (2.5,0.125) (2.375,-0.75) (2.25,-1.25)};
\draw [fill, blue, very thick](2.2125,-0.8375) -- (2.3375,-0.975) -- (2.5,-0.875) -- cycle;
\end{tikzpicture}
 \caption{Winding number $-20$.} 
\vspace{2ex}
\end{minipage}
  \begin{minipage}[b]{0.49\linewidth}
    \centering
\begin{tikzpicture}[scale=0.4]
\draw [very thick] plot[smooth, tension=.7] coordinates {(-2.5,0) (-0.5,-0.5) (-0.375,-2.25)};
\draw [very thick] plot[smooth, tension=.7] coordinates {(0.375,-2.5) (1.5,-1.125) (3.875,-2.25)};
\draw [very thick] plot[smooth, tension=.7] coordinates {(4.375,-1.625) (3,0) (4,1.5)};
\draw [very thick] plot[smooth, tension=.7] coordinates {(2.625,3) (0.875,1.75) (0.5,3)};
\draw [very thick] plot[smooth, tension=.7] coordinates {(-1.5,3) (-1.5,0.875) (-2.5,0.5)};
\draw [very thick] (1.25,0.375) ellipse (0.3 and 0.3);
\draw [very thick, fill](-2.5,0.5) -- (-2.5,0);
\draw [very thick](-0.375,-2.25) -- (0.375,-2.5);
\draw [very thick] plot[smooth, tension=.7] coordinates {(-1.5,3) (0.5,3)};
\draw [very thick] plot[smooth, tension=.7] coordinates {(2.625,3) (4,1.5)};
\draw [very thick] plot[smooth, tension=.7] coordinates {(4.375,-1.625) (3.875,-2.25)};
\draw [densely dashed, blue, very thick] plot[smooth, tension=.7] coordinates {(1.1559,0.0858) (0.6223,-0.3665) (-0.3274,0.2305) (0.1701,1.3249) (1.2193,1.1259) (1.3278,0.6556)};
\draw [densely dashed, blue, very thick] plot[smooth, tension=.7] coordinates {(1.4815,0.1491) (2.0785,-0.7282) (2.1328,-1.2709)};
\draw [densely dashed, blue, very thick] plot[smooth, tension=.7] coordinates {(-1.2,3) (-1,2) (0,2) (0.2,3)};
\draw [very thick, densely dashed, blue] (2.5,1.5) ellipse (0.5 and 0.5);
\end{tikzpicture}
    \caption{Winding number $0$.} 
  \end{minipage}
  \begin{minipage}[b]{0.49\linewidth}
    \centering
\begin{tikzpicture}[scale=0.4]
\draw [very thick] plot[smooth, tension=.7] coordinates {(-2.5,0) (-0.5,-0.5) (-0.375,-2.25)};
\draw [very thick] plot[smooth, tension=.7] coordinates {(0.375,-2.5) (1.5,-1.125) (3.875,-2.25)};
\draw [very thick] plot[smooth, tension=.7] coordinates {(4.375,-1.625) (3,0) (4,1.5)};
\draw [very thick] plot[smooth, tension=.7] coordinates {(2.625,3) (0.875,1.75) (0.5,3)};
\draw [very thick] plot[smooth, tension=.7] coordinates {(-1.5,3) (-1.5,0.875) (-2.5,0.5)};
\draw [very thick] (1.25,0.375) ellipse (0.3 and 0.3);
\draw [very thick, fill](-2.5,0.5) -- (-2.5,0);
\draw [very thick](-0.375,-2.25) -- (0.375,-2.5);
\draw [very thick] plot[smooth, tension=.7] coordinates {(-1.5,3) (0.5,3)};
\draw [very thick] plot[smooth, tension=.7] coordinates {(2.625,3) (4,1.5)};
\draw [very thick] plot[smooth, tension=.7] coordinates {(4.375,-1.625) (3.875,-2.25)};
\draw [very thick, densely dashed, blue] (1.25,0.375) ellipse (1 and 1);
\draw [very thick, fill, blue](2.25,0.375) -- (2.125,0.25) -- (2.375,0.25) -- cycle;
\end{tikzpicture}
    \caption{Winding number $10$.} 
    \label{Example of winding number II}
  \end{minipage} 
\end{figure}

Any two snippets $a_1, a_2 \in \mathcal{R}_\textrm{comp}$ that have minimal self-intersection and lie in the same snippet homotopy class satisfy $\wind(a_1)=\wind(a_2)$. Hence, we can set $\wind([a])=\wind(a)$ where $a \in [a]$ is any snippet of minimal self-intersection inside the peripheral annulus $R$. According to our set-up, $\wind (a)$ counts the minimal number of corners of $R$ that $a$ passes as it travels within the region $R$, possibly winding around $\partial S$. Here, winding around the component of $\partial S$ in accordance with its orientation induced by the orientation of $S$, results in a positive winding number.

\begin{lem}
Suppose that $a \subset R$ is a snippet that lies inside a peripheral annulus $R \in \mathcal{R}$. Then $[a]$ is uniquely determined by 
\begin{itemize}
\item the sides of $R$ (in order) containing $a(0)$ and $a(1)$ and
\item its winding number.
\end{itemize}
\end{lem}

\begin{proof}
This follows from the fact that strong snippet homotopies cannot move the points of $\partial a$ out of their containing sides. We remark that the winding number is redundant in case of snippets that have non-empty intersection with $\partial S$.
\end{proof}

\section{Efficient position and classification of bad snippets}
Let $S = S_{g,b}$ be a surface satisfying $\xi(S)=3g-3+b \geq 1$. Let $\tau \subset S$ be a large train track and $N=N(\tau)$ be a tie neighbourhood of $\tau$ in $S$. From this point on we no longer distinguish between snippets and their strong homotopy classes unless otherwise stated. For any snippet $a \subset R \in \mathcal{R}$, we always assume that $a$ has minimal self-intersection, is perpendicular to $\partial R$, and misses $\partial^2 \mathcal{R}$. Thus, snippets contained in simply connected regions are assumed to be embedded.

\begin{defn}
Suppose that $a \subset R \in \mathcal{R}_\textrm{tie}$ is a snippet. We say that $a$ is \textit{carried} by $N$ if $a$ is transverse to the ties of $R$.
\end{defn}

\begin{defn}\label{Def Efficient position}
Suppose that $\tau \subset S$ is a large train track and $N=N(\tau)$ is a tie neighbourhood of $\tau$ in $S$. Suppose that $a \subset R$ is a snippet, where $R \in \mathcal{R}(N)=\mathcal{R}$. We say that $a$ is in \textit{efficient position with respect to $N$} if exactly one of the following conditions holds:
\begin{itemize}
\item $[$\textit{Track}$]$: $R \in \mathcal{R}_\textrm{tie}$ and $a$ is a tie of $N$ or $a$ is carried by $N$.
\item $[$\textit{Disc}$]$: $R \in \mathcal{R}_\textrm{comp}$ is a disc and $a$ divides $R$ into two regions each of which has non-positive index.
% exclude peripheral components...
\item $[$\textit{Passing through peripheral annulus}$]$: $R \in \mathcal{R}_\textrm{comp}$ is a peripheral annulus, $\emptyset \neq \partial a \subset \partial \mathcal{R}-\partial S$ and $| \wind(a) | \geq 2$.
\item $[$\textit{Start or end in peripheral annulus}$]$: $R \in \mathcal{R}_\textrm{comp}$ is a peripheral annulus, $\partial a \neq \emptyset$ and exactly one of the points in $\partial a$ lies on the boundary of $S$.
\end{itemize}
Otherwise, we say that $a$ is a \emph{bad snippet with respect to $N$}.
\end{defn}

Suppose that $a \subset R$ is a snippet, where $R \in \mathcal{R}$. We say that $a$ is \textit{dual} to $N$ if $a$ is in efficient position with respect to $N$ but not carried by $N$. In other words, snippets in efficient position are either carried or dual. 

For examples of snippets in efficient position we refer the reader to Figures \ref{Examples of carried snippets in switch}-\ref{Examples of efficient snippets in complementary snippets}.

\begin{figure}[htbp]
\begin{minipage}[b]{0.49\linewidth}
\centering
\begin{tikzpicture}[scale=0.5]
\draw [very thick] (-3,1.5) rectangle (4,-3);
\node at (0.5,2.25) {\tiny{$h$}};
\node at (-3.75,-0.75) {\tiny{$t$}};
\node at (0.5,-3.75) {\tiny{$h$}};
\node at (4.75,-0.75) {\tiny{$v$}};
\draw [very thick](3.5,0) -- (4.5,0);
\draw [very thick](3.5,-1.5) -- (4.5,-1.5);
\draw [very thick, densely dashed, blue](-3,0.75) -- (4,0.75);
\draw [very thick, densely dashed, blue](-3,-0.75) -- (4,-0.75);
\draw [very thick, densely dashed, blue](-3,-2.25) -- (4,-2.25);
\end{tikzpicture}
 \caption{Examples of carried snippets in a switch rectangle.} 
 \label{Examples of carried snippets in switch}
  \vspace{12pt}
\end{minipage}
\begin{minipage}[b]{0.49\linewidth}
\centering
\begin{tikzpicture}[scale=0.5]
\draw [very thick] (-3,1.5) rectangle (4,-3);
\node at (0.5,2.25) {\tiny{$h$}};
\node at (-3.75,-0.75) {\tiny{$t$}};
\node at (0.5,-3.75) {\tiny{$h$}};
\node at (4.75,-0.75) {\tiny{$t$}};
\draw [very thick, densely dashed, blue](0.5,1.5) -- (0.5,-3);
\end{tikzpicture}
 \caption{A dual snippet inside a branch rectangle.} 
 \vspace{12pt}
\end{minipage}
\begin{minipage}[b]{0.49\linewidth}
    \centering
\begin{tikzpicture}[scale=0.6]

\draw [very thick](-4,-1.5) {} -- (-4,-3) {};
\draw [very thick] plot[smooth, tension=.7] coordinates { (-4,-3) (1,-2.5) (6,-3)};
\draw [very thick](6,-1.5) node (v3) {} -- (6,-3);
\draw [very thick] plot[smooth, tension=.7] coordinates {(-4,-1.5) (-1,-0.5) (-0.25,0.75)};
\draw [very thick] plot[smooth, tension=.7] coordinates {(v3) (3,-0.5) (2.25,0.75)};
\node at (-4.625,-2.25) {{\tiny{$v$}}};
\node at (1,-3.5) {{\tiny{$h$}}};
\node at (6.5175,-2.3319) {{\tiny{$v$}}};
\node at (-1.3541,0.0119) {{\tiny{$h$}}};
\node at (3.4897,-0.0798) {{\tiny{$h$}}};
\draw [very thick](-0.25,0.75) -- (2.25,0.75);
\draw [very thick, densely dashed, blue] plot[smooth, tension=.7] coordinates {(-4,-2.5) (1,-1.75) (6,-2.5)};
\draw [very thick, densely dashed, blue] plot[smooth, tension=.7] coordinates {(-3.9729,-2.1538) (0.25,-1) (1.0866,0.7459)};
\end{tikzpicture}
    \caption{Examples of snippets in efficient position that lie inside a simply connected complementary region.}
  \end{minipage}
\begin{minipage}[b]{0.49\linewidth}
    \centering
\begin{tikzpicture}[scale=0.6]
\draw [very thick](-4,-1.5) {} -- (-4,-3) {};
\draw [very thick] plot[smooth, tension=.7] coordinates { (-4,-3) (1,-2.5) (6,-3)};
\draw [very thick](6,-1.5) node (v3) {} -- (6,-3);
\draw [very thick] plot[smooth, tension=.7] coordinates {(-4,-1.5) (-1,-0.5) (-0.25,0.75)};
\draw [very thick] plot[smooth, tension=.7] coordinates {(v3) (3,-0.5) (2.25,0.75)};
\node at (-4.625,-2.25) {{\tiny{$v$}}};
\node at (1,-3.5) {{\tiny{$h$}}};
\node at (6.5175,-2.3319) {{\tiny{$v$}}};
\draw [very thick] (1,-0.75) ellipse (0.5 and 0.5);
\node at (-1.3541,0.0119) {{\tiny{$h$}}};
\node at (3.4897,-0.0798) {{\tiny{$h$}}};
\node at (-0.5,-1.25) {{\tiny{$\partial S$}}};
\draw [very thick, densely dashed, blue](5.25,-1.3099) -- (5.2579,-2.8905);
\draw [very thick, densely dashed, blue] plot[smooth, tension=.7] coordinates {(2.8237,-0.3819) (1.0951,-1.7859) (-0.0622,-0.3821) (2.25,0.5)};
\draw [very thick, densely dashed, blue] plot[smooth, tension=.7] coordinates {(4.5429,-1.1002) (3.1732,-2.0847) (-1.6551,-1.8964) (-3.975,-2.273)};
\draw [very thick, densely dashed, blue] plot[smooth, tension=.7] coordinates {(1.4439,-0.5266) (2.4369,0.0555)};
\draw [very thick](-0.25,0.75) -- (2.25,0.75);
\end{tikzpicture}
    \caption{Examples of snippets in efficient position that lie inside a peripheral complementary region.}
    \label{Examples of efficient snippets in complementary snippets}
  \end{minipage}
\end{figure}

Rather than considering examples of bad snippets, we provide a complete classification of them up to symmetries and orientation in the next section. We close the current section by pointing out two observations following the definition of efficient position.

\begin{lem} \label{Bad snippets in the tie neighbourhood}
Suppose that $a \subset R \in \mathcal{R}_\textrm{tie}$ is a snippet. Then $a$ is a bad snippet with respect to $N$ if and only if $a$ cuts off a region in $R$ that has positive index. \qed
\end{lem}

\begin{lem} \label{Bad snippets in complementary regions}
Suppose that $a \subset R \in \mathcal{R}_\textrm{comp}$ is a snippet. Then $a$ is a bad snippet with respect to $N$ if and only if $a$ is a peripheral curve or $a$ is embedded and cuts off a region of $R$ that has positive index.
\end{lem}

\begin{proof}
If $R$ is homeomorphic to a disc, the statement follows directly from the definition. Hence, let us assume that $R$ is a peripheral annulus. As $a$ is a snippet, it is an immersion of $S^1$ or $[0,1]$ into $R$. 

If $a$ is an immersion of the circle into $R$ it follows from the definition that $a$ must be bad as its boundary is empty. We note that any immersion of $S^1$ into a peripheral annulus is an inessential or peripheral curve. Since we assume that our snippets have minimal self-intersection, any inessential curve cuts off a region of index one. Hence, snippets that are an immersion of $S^1$ into $R$ are bad if and only if they are a peripheral curve or cut off a region of $R$ of positive index.

On the other hand, assume that $a$ is a proper immersion of $[0,1]$ into $R$. Hence, $|\partial a|=2$, and one of the following three statements holds:
\begin{itemize}
\item $\partial a \subset \partial S$.
\item $\partial a \subset \partial R - \partial S$.
\item Exactly one of the points of $\partial a$ lies in $\partial S$.
\end{itemize}

If $a$ is a snippet satisfying $\partial a \subset \partial S$, then the definition of efficient position implies that $a$ is a bad snippet.
Since we assume that snippets have minimal self-intersection, are smooth, and intersect $\partial S \subset \partial \mathcal{R}$ transversely, $a$ must be embedded and divides $R$ into two regions, one of which has exactly two outward-pointing corners occurring at the points $a(0)$ and $a(1)$. Hence, this region has index $1/2$. Thus, if $\partial a \subset \partial S$, then $a$ is a bad snippet if and only if it is embedded and cuts off a region of $R$ of positive index.

If $a$ is a snippet satisfying $\partial a \subset \partial R - \partial S$, then $a$ is bad if and only if $|\wind(a)|< 2$. This is the case if and only if $a$ cuts off a region of $R$ that contains at most one point of $\partial R$. This is true if and only if $a$ cuts off a region of index $1/4$ or $1/2$. The last equivalence follows from the fact that the region cut off by $a$ must be a disc with at least two outward-pointing corners occurring at the points $a(0)$ and $a(1)$. Thus, if $\partial a \subset \partial R - \partial S$, then we also see that $a$ is a bad snippet if and only if it is embedded and cuts off a region of positive index in $R$.

Lastly, if exactly one of the points of $\partial a$ lies in $\partial S$, then the snippet $a$ is in efficient position according to the definition. We remark that $a$ does not cut off a region of positive index in this case as complementary regions of train tracks have negative index.
Thus, a snippet $a \subset R \in \mathcal{R}_\textrm{comp}$ is bad if and only if it is a peripheral curve or cuts off a region of $R$ of positive index.
\end{proof}

\section{Classification of bad snippets} \label{Section classification of bad snippets}

\begin{defn}
Suppose that $R \in \mathcal{R}$ is a region. We say that a self-homeomorphism $\phi:R \rightarrow R$ is a \textit{symmetry} of $R$ if $\phi$ preserves $\partial_h R$ and $\partial_v N \cap R$ setwise.
\end{defn}

\begin{rem}
It may seem more natural to define a symmetry of a region as a self-homeomorphism that preserves $\partial_h R$ and $\partial_v R$ setwise. However, we need to remember if the boundary of a snippet lies inside $\partial_v N$ or on a tie of an adjacent branch rectangle, as adjacent snippets in these two cases are of very different types. Thus, we require symmetries to preserve $\partial_h R$ and $\partial_v N \cap R$ setwise.
\end{rem}

In the following, we are classifying bad snippets up to symmetry and orientation.

\subsection{Classification of bad snippets in the tie neighbourhood}

According to Lemma \ref{Bad snippets in the tie neighbourhood}, a snippet inside a branch or switch rectangle is bad if and only if it cuts off a region $T$ of positive index. We claim that the index of $T$ must be $1/4$, 1/2, or 1. If the index of $T$ is $1/4$ (respectively $1/2$ or 1), we say that $T$ is a \textit{trigon} (respectively a \textit{bigon} or a \textit{disk}).

\begin{lem}\label{Positive index snippet in tie neighbourhood}
Suppose that $R \in \mathcal{R}_\textrm{tie}$ is a branch or switch rectangle. Suppose further that $a \subset R$ is a snippet that cuts off a region $T$ of $R$ of strictly positive index. Then $T$ is a trigon, bigon, or disk.
\end{lem}

\begin{proof}
Since $a \subset R$ is a snippet, we know that $a$ is a properly embedded arc in $R$ or an embedded curve in $R$. As $R$ is a rectangle, any embedded curve in $R$ is trivial and cuts off a disk in $R$. If $a$ is a properly embedded arc, then $a$ intersects $\partial R$ transversely. Thus, the region $T$ of positive index cut off by $a$ of $R$ must have at least two outward-pointing corners. As $T$ must be homeomorphic to a disk, it has Euler characteristic one, so $0 \leq \ind(T) \leq 1/2$ and the claim follows.
\end{proof}

\begin{lem}\label{Classification of snippets in branch}
Suppose that $R \in \mathcal{R}_\textrm{tie}$ is a branch rectangle. Up to symmetries of $R$ and orientation of the snippet, there are four strong snippet homotopy classes of bad snippets in $R$.
\end{lem}

\begin{proof}
Suppose that $a \subset R$ is a bad snippet inside a branch rectangle $R$. Following Lemma \ref{Bad snippets in the tie neighbourhood} and Lemma \ref{Positive index snippet in tie neighbourhood}, $a$ cuts off a region $T \subset R$ which is a trigon, bigon, or disk.
If $T$ is a disk, then $a$ is a trivial curve in $R$. If $T$ is a bigon, then $a(0)$ and $a(1)$ must lie on the same side of $R$. Up to symmetries of $R$, there a two different kinds of snippets of this type: snippets where $a(0)$ and $a(1)$ lie on a horizontal boundary side of $R$ and snippets where $a(0)$ and $a(1)$ lie on a vertical boundary side of $R$.
If $T$ is a trigon, then $\partial T$ contains exactly one corner of $R$. Hence, $a(0)$ and $a(1)$ lie on adjacent sides of $R$, that is, one lies on a horizontal boundary side of $R$ and one on a vertical boundary side of $R$. All such snippets are equivalent up to symmetries of $R$ and a choice of orientation of the snippets.

As snippets of these four kinds have a distinct intersection pattern with $\partial R$, they are non-equivalent under symmetries of $R$ which concludes the proof of the lemma.
\end{proof}

For examples of bad snippets inside branch rectangles, we refer the reader to Figure \ref{Branch snippets I}.

\begin{figure}[htbp] 
\begin{minipage}[b]{0.99\linewidth}
\centering
\begin{tikzpicture}[scale=0.5]
\draw [very thick] (-3,1.5) rectangle (4,-3);
\node at (0.5,2.25) {\tiny{$h$}};
\node at (-3.75,-0.75) {\tiny{$t$}};
\node at (0.5,-3.75) {\tiny{$h$}};
\node at (4.75,-0.75) {\tiny{$t$}};
\draw [very thick, densely dashed, blue] plot[smooth, tension=.7] coordinates {(-3,1) (-1.25,0.75) (-1.25,-0.75) (-3,-1)};
\draw [very thick, densely dashed, blue] (-0.375,-1.75) ellipse (0.6 and 0.6);
\draw [very thick, densely dashed, blue] plot[smooth, tension=.7] coordinates {(1.375,-3) (1.75,-1.25) (4,-0.875)};
\draw [very thick, densely dashed, blue] plot[smooth, tension=.7] coordinates {(0.125,1.5) (0.5,0) (2.25,0) (2.625,1.5)};
\end{tikzpicture}
 \caption{All possible types of bad snippets inside a branch rectangle.} 
 \label{Branch snippets I}
\end{minipage} 
\end{figure}
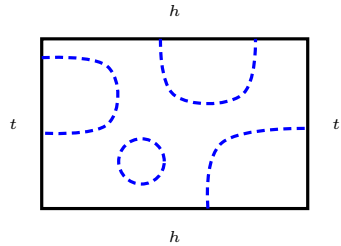

\begin{defn}\label{Page: Branch type}
Suppose that $a \subset R$ is a bad snippet inside a branch rectangle $R \in \mathcal{R}$. We say that $a$ is \emph{of type $\mathbb{B}(x,y)$} where
\begin{itemize}
\item $(x,y)=(\mathring{R},\mathring{R})$ if $\partial a = \emptyset$.
\item $x,y \in \{h,t\}$ if $\partial a \neq \emptyset$, with $x=y=h$ or $x=y=t$ if and only if $\partial a \subset \partial_h R$ or $\partial_v R$ respectively. Else, $(x,y)=(h,t)$.
\end{itemize}
\end{defn}

\begin{cor}
Suppose that $R\in \mathcal{R}_\textrm{tie}$ is a branch rectangle. Suppose further that $a \subset R$ is a bad snippet. Then $a$ is of type $\mathbb{B}(\mathring{R},\mathring{R})$, $\mathbb{B}(h,h)$, $\mathbb{B}(t,t)$, or $\mathbb{B}(h,t)$. Two bad snippets in $R$ are of the same type if and only if they are equivalent up to symmetries of $R$ and orientation. Hence, if two bad snippets $a_1, a_2 \subset R$ are strongly snippet homotopic, then they are of the same type. Moreover, if two bad snippets $a_1, a_2 \subset R$ are of the same type and intersect the same components of $\partial R - \partial^2 R$, they must be strongly snippet homotopic up to orientation. \qed
\end{cor}

Thus, up to orientation and symmetries of $R$, there are four different types of bad snippets inside a branch rectangle.

We now proceed by classifying bad snippets inside switch rectangles. As one vertical side of each switch rectangle contains a component of $\partial_v N$ in its interior, not all snippets intersecting the horizontal and vertical boundary of the switch rectangle turn out to be equivalent under symmetries of $R$. To aid the classification of bad snippets in switch rectangles, we introduce the notion of \emph{weight} for such a bad snippet.

\begin{defn}
Suppose that $R \in\mathcal{R}_\textrm{tie}$ is a switch rectangle. Suppose further that $a \subset R$ is a snippet that cuts off a region $T$ of $R$ of positive index. We set $\w(a)=|\partial^2 \mathcal{R} \cap \partial T|$ and say that $\w(a)$ is the \emph{weight} of the snippet $a \subset R$.
\end{defn}
For examples of snippets and their weight, we refer the reader to Figure \ref{Weight examples I}.

\begin{figure}[htbp] 
\begin{minipage}[b]{0.99\linewidth}
\centering
\begin{tikzpicture}[scale=0.6]
\draw [very thick] (-3,1.5) rectangle (4,-3);
\node at (0.5,2.25) {\tiny{$h$}};
\node at (-3.75,-0.75) {\tiny{$t$}};
\node at (0.5,-3.5) {\tiny{$h$}};
\node at (4.75,-2.25) {\tiny{$t$}};
\draw [very thick](3.5,0) -- (4.5,0);
\draw [very thick](3.5,-1.5) -- (4.5,-1.5);
\node at (4.75,-0.75) {\tiny{$v$}};
\node at (4.75,0.75) {\tiny{$t$}};
\draw [very thick, densely dashed, blue] (-0.75,-0.125) ellipse (0.6 and 0.6);
\draw [very thick, densely dashed, blue] plot[smooth, tension=.7] coordinates {(4,0.375) (2,0.125) (2,-1.625) (4,-1.875)};
\draw [very thick, blue, densely dashed] plot[smooth, tension=.7] coordinates {(-3,-1) (-1.75,-1.25) (-1.5,-3)};
\draw [very thick, blue, densely dashed] plot[smooth, tension=.7] coordinates {(4,1) (1.25,0.375) (0.625,-3)};
\end{tikzpicture}
 \caption{Snippets inside a switch rectangle of weight zero, one, two, and three respectively.} 
 \label{Weight examples I}
\end{minipage}
\end{figure}
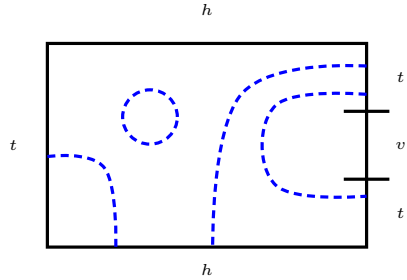

\begin{rem} \label{Rem on bounds on geom length}
As switch rectangles are simply connected, any two bad snippets belonging to the same strong snippet homotopy class have equal weight. We further remark that the weight of any bad snippet inside a switch rectangle is bounded by three.
\end{rem}

\begin{lem}\label{classification of snippets in switch}
Suppose that $R \in \mathcal{R}_\textrm{tie}$ is a switch rectangle. Up to symmetries of $R$ and orientation of the snippet, there are eleven strong snippet homotopy classes of bad snippets in $R$. Each class $[a]$ is uniquely determined by 
\begin{itemize}
\item its weight and
\item the information whether $a(0)$ and $a(1)$ each are contained in a horizontal boundary side of $R$ that has empty or non-empty intersection with $\partial_v N$, a vertical boundary side of $R$, or a component of $\partial_v N$.
\end{itemize}
\end{lem}

\begin{proof}
Suppose that $a \subset R$ is a bad snippet inside a switch rectangle $R$. Following Lemma \ref{Bad snippets in the tie neighbourhood} and Lemma \ref{Positive index snippet in tie neighbourhood}, $a$ cuts off a region $T$ of $R$ that is a trigon, bigon, or disk.
If $T$ is a disk, then $a$ is a trivial curve in $R$.

If $T$ is a bigon, then $a(0)$ and $a(1)$ lie on the same side of $R$. Recalling that symmetries of $R$ preserve $\partial_v N$ and $\partial_h N$ setwise, we see that they also preserve the single vertical boundary sides of $R$ setwise. Hence, up to symmetries of $R$, we obtain the following possibilities for a snippet $a$ to cut off a bigon:
\begin{itemize}
\item $a(0)$ and $a(1)$ lie on a horizontal boundary side of $R$.
\item $a(0)$ and $a(1)$ lie on a vertical boundary side of $R$ which has empty intersection with $\partial_v N$.
\item $a(0)$ and $a(1)$ lie on a vertical boundary side of $R$ which has non-empty intersection with $\partial_v N$. In this case, the region $T$ can contain up to two points of $\partial^2 \mathcal{R}$ in its boundary. As we want to classify snippets up to strong snippet homotopy, we further distinguish snippets according to their weight:
\begin{itemize}
\item $\w(a)=0$. Then $a(0)$ and $a(1)$ both either lie in $\partial_v N$ or in the same component of $\partial_v R -\partial_v N$.
\item $\w(a)=1$. Then one point of $\partial a$ lies in $\partial_v N$ and one point of $\partial a$ lies in $\partial_v R -\partial_v N$.
\item $\w(a)=2$. Then $a(0)$ and $a(1)$ lie in different components of $\partial_v R -\partial_v N$.
\end{itemize}
\end{itemize}
Thus, up to symmetries of $R$ and the orientation of the snippet, there are six different strong snippet homotopy classes of bad snippets that cut off a bigon of $R$. For an illustration we refer the reader to Figure \ref{Switch bigons}.

If $T$ is a trigon, then $a(0)$ and $a(1)$ lie on different sides of $R$. Up to orientation of the snippet, we may assume that $a(0)$ lies on a horizontal side of $R$. Thus, $a(1)$ either lies on the vertical side of $R$ that has empty intersection with $\partial_v N$ or on the vertical side of $R$ that has non-empty intersection with $\partial_v N$. In the latter case, the snippet is uniquely defined by its weight $w(a) \in \{1,2,3\}$. Thus, up to symmetries of $R$ and the orientation of the snippet, there are four different strong snippet homotopy classes of bad snippets that cut off a trigon of $R$. For an illustration we refer the reader to Figure \ref{Switch trigons}.

\begin{figure}[htbp] 
   \begin{minipage}[b]{0.49\linewidth}
    \centering
\begin{tikzpicture}[scale=0.65]
\draw [very thick] (-3,1.5) rectangle (4,-3);
\node at (0.5,2.25) {\tiny{$h$}};
\node at (-3.75,-0.75) {\tiny{$t$}};
\node at (0.5,-3.5) {\tiny{$h$}};
\node at (4.75,-2.25) {\tiny{$t$}};
\draw [very thick](3.5,0) -- (4.5,0);
\draw [very thick](3.5,-1.5) -- (4.5,-1.5);
\node at (4.75,-0.75) {\tiny{$v$}};
\node at (4.75,0.75) {\tiny{$t$}};
\draw [very thick, densely dashed, blue] plot[smooth, tension=.7] coordinates {(4,1.25) (1,0.875) (1,-1.375) (4,-1.75)};
\draw [very thick, densely dashed, blue] plot[smooth, tension=.7] coordinates {(4,-2) (2.5,-2.125) (2.5,-2.75) (4,-2.875)};
\draw [very thick, densely dashed, blue] plot[smooth, tension=.7] coordinates {(-3,-1) (-1,-1.25) (-1,-2.375) (-3,-2.625)};
\draw [very thick, densely dashed, blue] plot[smooth, tension=.7] coordinates {(-2.125,1.5) (-1.875,0.25) (-0.75,0.25) (-0.5,1.5)};
\draw [very thick, densely dashed, blue] plot[smooth, tension=.7] coordinates {(4,-0.5) (2.5,-0.625) (2.5,-1.125) (4,-1.25)};
\draw [very thick, densely dashed, blue] plot[smooth, tension=.7] coordinates {(4,1) (2.5,0.875) (2.5,-0.125) (4,-0.25)};
\end{tikzpicture}
    \caption{The six different strong snippet homotopy classes of bigon snippets inside a switch rectangle, up to symmetry.}
    \label{Switch bigons} 
    \vspace{2ex}
  \end{minipage}
  \begin{minipage}[b]{0.49\linewidth}
    \centering
\begin{tikzpicture}[scale=0.65]
\draw [very thick] (-3,1.5) rectangle (4,-3);
\node at (0.5,2.25) {\tiny{$h$}};
\node at (-3.75,-0.75) {\tiny{$t$}};
\node at (0.5,-3.5) {\tiny{$h$}};
\node at (4.75,-2.25) {\tiny{$t$}};
\draw [very thick](3.5,0) -- (4.5,0);
\draw [very thick](3.5,-1.5) -- (4.5,-1.5);
\node at (4.75,-0.75) {\tiny{$v$}};
\node at (4.75,0.75) {\tiny{$t$}};
\draw [very thick, densely dashed, blue] plot[smooth, tension=.7] coordinates {(-1.25,1.5) (-1.5,-0.25) (-3,-0.5)};
\draw [very thick, densely dashed, blue] plot[smooth, tension=.7] coordinates {(0.25,1.5) (1,-1.75) (4,-2.25)};
\draw [very thick, densely dashed, blue] plot[smooth, tension=.7] coordinates {(1.25,1.5) (1.75,-0.375) (4,-0.75)};
\draw [very thick, densely dashed, blue] plot[smooth, tension=.7] coordinates {(2.125,1.5) (2.5,0.625) (4,0.375)};
\end{tikzpicture}
    \caption{The four different strong snippet homotopy classes of trigon snippets inside a switch rectangle, up to symmetry.}
    \label{Switch trigons} 
    \vspace{2ex}
  \end{minipage} 
\end{figure}
\end{proof}

\begin{defn} \label{Page: Switch type}
Suppose that $a \subset R$ is a bad snippet inside a switch rectangle $R \in \mathcal{R}$. We say that $a$ is \emph{of type $\mathbb{S}(x,y,k)$} where
\begin{itemize}
\item $(x,y)=(\mathring{R},\mathring{R})$ if $\partial a = \emptyset$.
\item $x=h$ (respectively $x=v$ or $x=t$) if $a(0) \subset \partial_h R$ (respectively $a(0) \subset \partial_v N$ or $a(0) \subset (\partial_v R - \partial_v N)$).
\item $y=h$ (respectively $x=v$ or $x=t$) if $a(1) \subset \partial_h R$ (respectively $a(1) \subset \partial_v N$ or $a(1) \subset (\partial_v R - \partial_v N)$).
\item $k=\w(a) \in \{0,1,2,3\}$.
\end{itemize}
In the following, we will not distinguish between the types $\mathbb{S}(x,y,k)$ and $\mathbb{S}(y,x,k)$.
\end{defn}

\begin{cor}
Suppose that $R\in \mathcal{R}_\textrm{tie}$ is a switch rectangle. Suppose further that $a \subset R$ is a bad snippet. Then $a$ is of one of the following nine types: $\mathbb{S}(\mathring{R},\mathring{R},0)$, $\mathbb{S}(h,h,0)$, $\mathbb{S}(t,t,0)$, $\mathbb{S}(v,v,0)$, $\mathbb{S}(t,v,1)$, $\mathbb{S}(t,t,2)$, $\mathbb{S}(h,t,1)$, $\mathbb{S}(h,v,2)$, or $\mathbb{S}(h,t,3)$. If two snippets in $R$ are not of type $\mathbb{S}(t,t,0)$ or $\mathbb{S}(h,t,1)$, then they are equivalent up to symmetries of $R$ and orientation of the snippets if and only if they are of the same type. If two bad snippets $a_1, a_2 \subset R$ are strongly snippet homotopic, then they are of the same type. Moreover, if two bad snippets $a_1, a_2 \subset R$ are of the same type and intersect the same components of $\partial R - \partial^2 R$, then they must be strongly snippet homotopic up to orientation. \qed
\end{cor}

Summarizing, we see that up to symmetries of the switch rectangle and orientation of the snippet, there are eleven different strong snippet homotopy classes of bad snippets inside switch rectangles. However, our combinatorial set-up requires us to only distinguish between nine different types of these: two bad snippets of type $\mathbb{S}(t,t,0)$ or $\mathbb{S}(h,t,1)$ might not be related by a symmetry of $R$, but turn out to behave very similarly in all relevant situations.

\subsection{Bad snippets in complementary regions}

\begin{defn}\label{Page: Comp type}
Suppose that $a \subset R$ is a non-peripheral bad snippet inside a complementary region $R  \in \mathcal{R}_\textrm{comp}$. We say that $a$ is \emph{of type $\mathbb{R}(x,y)$} where
\begin{itemize}
\item $(x,y)=(\mathring{R},\mathring{R})$ if $\partial a = \emptyset$.
\item $(x,y)=(\partial S,\partial S)$ if $\partial a \subset \partial S$.
\item $x,y \in \{h,v\}$ with $x=y=h$ or $x=y=v$ if and only if $\partial a \subset \partial_h R$ or $\partial_v R$ respectively. Else, $(x,y)=(h,v)$.
\end{itemize}
\end{defn}

For an illustration of the different types of non-peripheral snippets inside complementary regions we point the reader to Figure \ref{Complementary snippets}.

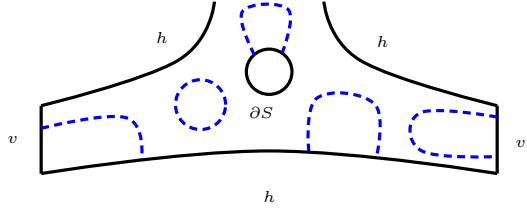
\begin{figure}[htbp] 
\begin{minipage}[b]{0.99\linewidth}
    \centering
\begin{tikzpicture}[scale=0.6]
\draw [very thick](-4,-1.5) {} -- (-4,-3) {};
\draw [very thick] plot[smooth, tension=.7] coordinates { (-4,-3) (1,-2.5) (6,-3)};
\draw [very thick](6,-1.5) node (v3) {} -- (6,-3);
\draw [very thick] plot[smooth, tension=.7] coordinates {(-4,-1.5) (-1,-0.5) (-0.2,0.8)};
\draw [very thick] plot[smooth, tension=.7] coordinates {(v3) (3,-0.5) (2.2,0.8)};
\node at (-4.625,-2.25) {{\tiny{$v$}}};
\node at (1,-3.5) {{\tiny{$h$}}};
\node at (6.5175,-2.3319) {{\tiny{$v$}}};
\draw [very thick] (1,-0.75) ellipse (0.5 and 0.5);
\draw [very thick, blue, densely dashed] plot[smooth, tension=.7] coordinates {(6,-1.75) (4.3249,-1.6525) (4.3351,-2.4934) (6,-2.625)};
\draw [very thick, blue, densely dashed] plot[smooth, tension=.7] coordinates {(1.8689,-2.5138) (2.0498,-1.3315) (3.3058,-1.4588) (3.3594,-2.6131)};
\draw [very thick, blue, densely dashed] plot[smooth, tension=.7] coordinates {(0.6759,-0.3803) (0.4172,0.5847) (1.4184,0.622) (1.2832,-0.3395)};
\draw [very thick, blue, densely dashed] plot[smooth, tension=.7] coordinates {(-4,-2) (-2.1332,-1.7673) (-1.797,-2.6845)};
\draw [very thick, blue, densely dashed] (-0.5052,-1.4769) ellipse (0.55 and 0.55);
\node at (-1.3541,0.0119) {{\tiny{$h$}}};
\node at (3.4897,-0.0798) {{\tiny{$h$}}};
\node at (0.8282,-1.6605) {{\tiny{$\partial S$}}};
\end{tikzpicture}
    \caption{The five types of non-peripheral bad snippets inside a complementary region.}
    \label{Complementary snippets}
  \end{minipage}
\end{figure}

As in the previous section, we see that every non-peripheral bad snippet inside a complementary region is of one of the just defined five types:

\begin{lem}\label{classification of snippets in complementary}
Suppose that $a \subset R$ is a non-peripheral bad snippet inside a complementary region $R \in \mathcal{R}_\textrm{tie}$. Then $a$ is of type $\mathbb{S}(\mathring{R},\mathring{R})$, $\mathbb{S}(\partial S,\partial S)$, $\mathbb{S}(h,h)$, $\mathbb{S}(v,v)$, or $\mathbb{S}(h,v)$. Disregarding orientation, two non-peripheral bad snippets in $R$ are of the same type if and only if they are equivalent up to symmetries of $R$.
\end{lem}

\begin{proof}
Suppose that $a \subset R$ is a non-peripheral bad snippet inside a complementary region $R$. By Lemma \ref{Bad snippets in complementary regions} we know that $a$ is embedded and cuts off a region $T$ of $R$ of positive index. If $a:S^1 \rightarrow R$, it, therefore, bounds a disc inside $R$ and is of type $\mathbb{S}(\mathring{R},\mathring{R})$. Any two discs in $R$ are homeomorphic to each other. Thus, two non-peripheral bad snippets $a:S^1 \rightarrow R$ of the same type are equivalent up to symmetries of $R$.

If $a:[0,1]\rightarrow R$ is a proper immersion of the interval into $R$, then either $\partial a \subset \partial S$ or $\partial a \subset \partial R - \partial S$. This follows from the fact that snippets with exactly one of their boundary points in $\partial S$ are in efficient position.

A snippet satisfying $\partial a \subset \partial S$ is, by definition, of type $\mathbb{S}(\partial S,\partial S)$. As any two such snippets equivalent under symmetries of $R$, the claim of the lemma follows in this case.

Suppose that $a:[0,1]\rightarrow R$ is a bad snippet satisfying $\partial a \subset \partial R - \partial S$. As $a$ meets $\partial R$ perpendicularly, the region $T$ cut off by $a$ in $R$ must have at least two outward-pointing corners. Since $a$ is smooth, $T$ has positive index and $R$ has outward-pointing corners only, $T$ must be a disk whose boundary contains either none or exactly one corner of $R$. If $T$ contains no corner of $R$, then $a(0)$ and $a(1)$ lie on the same component of $\partial R - \partial^2 R$. Hence, $\ind(T)=1/2$ and $a$ is of type $\mathbb{S}(h,h)$ or $\mathbb{S}(v,v)$. Up to orientation, any two snippets of type $\mathbb{S}(h,h)$ (respectively $\mathbb{S}(v,v)$) are equivalent under symmetries of $R$. If $T$ contains one corner of $R$, then $a(0)$ and $a(1)$ lie on different components of $\partial R - \partial^2 R$. Hence, $\ind(T)=1/4$ and $a$ is of type $\mathbb{S}(h,v)$. Up to orientation, any two snippets of type $\mathbb{S}(h,v)$ are equivalent under symmetries of $R$.
\end{proof}

We remind ourselves that horizontal sides of complementary regions can contain several points of $\partial^2 \mathcal{R}$ in their interiors. Thus, in contrary to the case of branch and switch rectangles, two non-peripheral bad snippets inside a complementary region $R$ that are of the same type and intersect the same sides of $R$ need not be strongly snippet homotopic. For an example of two snippets that intersect the same sides of a complementary region but are not strongly snippet homotopic we refer the reader to Figure \ref{Weakly but not strongly homotopic}.

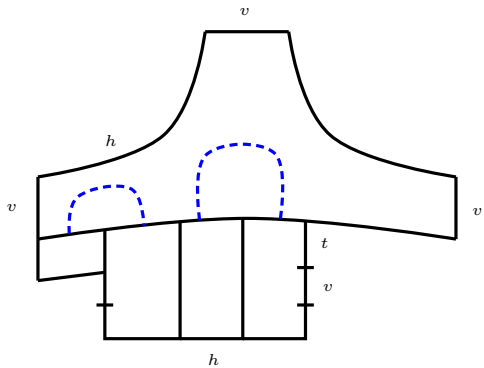
\begin{figure}[htbp] 
\begin{minipage}[b]{0.99\linewidth}
    \centering
\begin{tikzpicture}[scale=0.55]
\draw [very thick](-4,-1.5) {} -- (-4,-3) {};
\draw [very thick] plot[smooth, tension=.7] coordinates { (-4,-3) (1,-2.5) (6,-3)};
\draw [very thick](6,-1.5) node (v3) {} -- (6,-3);
\draw [very thick] plot[smooth, tension=.7] coordinates {(-4,-1.5) (-1,-0.5) (0,2)};
\draw [very thick] plot[smooth, tension=.7] coordinates {(v3) (3,-0.5) (2,2)};
\draw [very thick] plot[smooth, tension=.7] coordinates {(-4,-4)  (-2.4,-3.8)};
\node at (-4.625,-2.25) {{\tiny{$v$}}};
\node at (-2.25,-0.625) {{\tiny{$h$}}};
\draw [very thick](-2.4,-2.78) -- (-2.4,-5.4) -- (-0.6,-5.4) -- (-0.6,-2.58);
\draw [very thick](-0.6,-5.4)-- (0.9,-5.4)  -- (0.9,-2.5);
\draw [very thick](0.9,-5.4) -- (2.4,-5.4) -- (2.4,-2.56);
\draw [very thick](-2.6,-4.588) -- (-2.2,-4.588);
\draw [very thick](2.2,-3.686) -- (2.6,-3.686);
\draw [very thick](2.2,-4.588) -- (2.6,-4.588);
\node at (2.94,-4.16) {{\tiny{$v$}}};
\node at (2.86,-3.1) {{\tiny{$t$}}};
\node at (0.2,-5.9) {{\tiny{$h$}}};
\draw [very thick](0,2) -- (2,2);
\node at (0.947,2.496) {{\tiny{$v$}}};
\node at (6.5175,-2.3319) {{\tiny{$v$}}};
\draw [very thick, densely dashed, blue] plot[smooth, tension=.7] coordinates {(-3.25,-2.875) (-3.0348,-1.9713) (-1.7894,-1.7806) (-1.4663,-2.6789)};
\draw [very thick, densely dashed, blue] plot[smooth, tension=.7] coordinates {(-0.1322,-2.5439) (0.0263,-0.9656) (1.6526,-0.9236) (1.7977,-2.5227)};
\draw [very thick](-4,-3) -- (-4,-4);
\end{tikzpicture}
    \caption{Weakly but not strongly snippet homotopic snippets.}
    \label{Weakly but not strongly homotopic}
  \end{minipage}
\end{figure}

However, snippets in complementary regions that are homotopic up to ``moving their boundary points'' along the sides of the region play an important role in this thesis. We, therefore, close this section by introducing the following coarser equivalence relation on snippets:

\begin{defn}\label{Weak snippet homotopy}
Fix a region $R \in \mathcal{R}$. Suppose that $a_1,a_2:D \rightarrow R$ are two snippets in $R$. We say that $a_1$ and $a_2$ are \emph{weakly snippet homotopic} if $a_1$ and $a_2$ are homotopic via a transverse homotopy of pairs $D \rightarrow (R,\partial R)$  which keeps the $k$-skeleton of $[0,1]$ or $S^1$ inside the $(k+1)$-skeleton of $R$ at all time.
\end{defn}

\begin{lem}\label{Weakly snippet homotopic implies same type}
Suppose that $a_1,a_2:D \rightarrow R$ are two snippets inside a region $R \in \mathcal{R}$. If $a_1$ and $a_2$ are weakly snippet homotopic and $a_1$ is in efficient position, then $a_2$ is in efficient position, too. \qed
\end{lem}

\subsection{Classification of bad snippets - an overview}

The following corollary is an immediate consequence of the discussion in the previous section:

\begin{cor}\label{classification of bad snippets}
Suppose that $a\subset R \in \mathcal{R}$ is a bad snippet. Then $a$ is either a peripheral curve or it is of type
\begin{itemize}
\item $\mathbb{B}(\mathring{R},\mathring{R})$, $\mathbb{S}(\mathring{R},\mathring{R},0)$, or $\mathbb{R}(\mathring{R},\mathring{R})$, that is, $a$ is an inessential curve. In this case we say that $a$ is a \textit{trivial snippet}.
\item $\mathbb{R}(\partial S,\partial S)$. In this case we say that $a$ is an \emph{inessential bigon snippet}.
\item $\mathbb{B}(h,h)$, $\mathbb{B}(t,t)$, $\mathbb{S}(h,h,0)$, $\mathbb{S}(t,t,0)$, $\mathbb{S}(v,v,0)$, $\mathbb{S}(t,v,1)$, $\mathbb{S}(t,t,2)$, $\mathbb{R}(h,h)$, or $\mathbb{R}(v,v)$. That is, $a$ cuts off a bigon of $R$. In this case we say that $a$ is a \textit{bigon snippet}.
\item $\mathbb{B}(h,t)$, $\mathbb{S}(h,t,1)$, $\mathbb{S}(h,v,2)$, $\mathbb{S}(h,t,3)$, or $\mathbb{R}(h,v)$. That is, $a$ cuts off a trigon of $R$. In this case we say that $a$ is a \textit{trigon snippet}.
\end{itemize}
Hence, there are four and nine types of bad snippets inside branch and switch rectangles respectively. There are five types of non-peripheral bad snippets inside complementary regions. \qed
\end{cor}

\section{Efficient position for arcs and curves}
Let $S = S_{g,b}$ be a surface satisfying $\xi(S)=3g-3+b \geq 1$. Let $\tau \subset S$ be a large train track and $N=N(\tau)$ be a tie neighbourhood of $\tau$ in $S$.

Suppose that $\alpha \subset S$ is an immersed arc or curve in $S$. For the remainder of this thesis, we always assume that $\alpha$ is self-transverse and transverse to $\partial \mathcal{R}$. Moreover, if $\alpha$ is an arc, we assume that $\alpha:(I,\partial I)  \rightarrow (S, \partial \mathcal{R})$ is an immersion of pairs. Hence, $\alpha(0)$ and $\alpha(1)$ lie in $\partial \mathcal{R}$ and $\alpha$ admits a canonical decomposition into snippets.

\begin{defn}\label{Page: Snippet length}
Suppose that $\alpha \subset S$ is a snippet-decomposed immersed arc or curve in $S$. The number of snippets occurring in the snippet decomposition of $\alpha$ is denoted by $\len(\alpha)$ and is called the \emph{snippet length of $\alpha$}.
\end{defn}

A snippet-decomposed arc or curve $\alpha \subset S$ is said to be \emph{carried} by $N$ or \emph{dual} to $N$ if all its snippets are carried or dual respectively.

\begin{defn}
Suppose that $\tau \subset S$ is a large train track and $N=N(\tau)$ is a tie neighbourhood of $\tau$ in $S$. Suppose that $\alpha \subset S$ is an immersed arc or curve in $S$. We say that $\alpha$ is in \textit{efficient position with respect to $N$} if all snippets in its snippet-decomposition are in efficient position with respect to $N$.
\end{defn}

We remark that the intersection of carried and dual subarcs of an arc or curve in efficient position must lie in $\partial_v N$ (see Figure \ref{Transition from carried to dual}).

\begin{figure}[htbp]
	\centering
	\begin{tikzpicture}[scale=0.6]
\draw [very thick](10,2) -- (10,-1);
\draw [very thick](7.5,0) -- (7.5,-1) node (v1) {};
\draw [very thick] plot[smooth, tension=.7] coordinates {(v1)  (12.5,-1)};
\draw [very thick] plot[smooth, tension=.7] coordinates {(12.5,-1) (12.5,-1)};
\draw [very thick](12.5,-1) -- (12.5,1) -- (12.5,2) -- (7.5,2) -- (7.5,1) -- (10,1) node (v2) {} -- (10,0) node (v3) {} -- (7.5,0);
\draw [blue, very thick,densely dashed] plot[smooth, tension=.7] coordinates {(7.5,0.5) (12.5,0.5)};
\draw [very thin, densely dotted] (7.75,2) rectangle (9.75,1);
\draw [very thin, densely dotted] (8,2) rectangle (9.5,1);
\draw [very thin, densely dotted] (8.25,2) rectangle (9.25,1);
\draw [very thin, densely dotted] (8.5,2) rectangle (9,1);
\draw [very thin, densely dotted] (10.25,2) rectangle (12.25,-1);
\draw [very thin, densely dotted] (10.5,2) rectangle (12,-1);
\draw [very thin, densely dotted] (10.75,2) rectangle (11.75,-1);
\draw [very thin, densely dotted] (11,2) rectangle (11.5,-1);
\draw [very thin, densely dotted] (7.75,0) rectangle (9.75,-1);
\draw [very thin, densely dotted] (8,0) rectangle (9.5,-1);
\draw [very thin, densely dotted] (8.25,0) rectangle (9.25,-1);
\draw [very thin, densely dotted] (8.5,0) rectangle (9,-1);
\draw [very thin, densely dotted](8.75,2) -- (8.75,1);
\draw [very thin, densely dotted](8.75,0) -- (8.75,-1);
\draw [very thin, densely dotted](11.25,2) -- (11.25,-1);
\end{tikzpicture}
	\caption{Adjacent snippets of different types in an arc or curve in efficient position.}
	\label{Transition from carried to dual}
\end{figure}
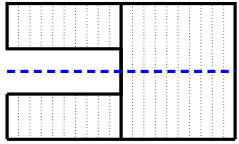

The following observations on short properly immersed arcs and curves follow directly from the classification of snippets in the previous sections:

\begin{lem}\label{arcs or curves of length one}
Suppose that $\alpha \subset S$ is a properly immersed, snippet-decomposed arc or curve in $S$. If $\len(\alpha)=1$, then $\alpha$ is inessential or peripheral.
\end{lem}

\begin{proof}
As $\len(\alpha)=1$, the arc or curve $\alpha$ consists of a single snippet $a \subset R$ for some region $R \in \mathcal{R}$. If $\partial R \cap \partial S = \emptyset$, then $a=\alpha$ must be an inessential curve. If $R$ is a peripheral region, either $\partial a = \emptyset$ and hence $a=\alpha$ is an inessential or peripheral curve, or $\partial a \subset \partial S$ as $\alpha$ is proper. The latter implies that $a$ is of type $\mathbb{R}(\partial S, \partial S)$. Hence, $\alpha$ is an inessential arc.
\end{proof}

\begin{lem}\label{No arcs of length 2}
Suppose that $\alpha \subset S$ is a properly immersed, snippet-decomposed arc in $S$. Then $\len(\alpha)\neq 2$. If $\len(\alpha)\geq 3$, then the snippets containing $\alpha(0)$ and $\alpha(1)$ are in efficient position.
\end{lem}

\begin{proof}
For any properly immersed arc $\alpha \subset S$, $\alpha(0)$ and $\alpha(1)$ lie in $\partial S$. As any component of $\partial S$ is contained in the boundary of a peripheral complementary region $R \in \mathcal{R}_\textrm{comp}$, this implies that the first and last snippet of $\alpha$ must lie in peripheral complementary regions. As $\alpha$ is transverse to $\partial \mathcal{R}$, no two adjacent snippets of $\alpha$ can lie inside the same region $R \in \mathcal{R}$. As any two complementary regions of $N$ are separated by the tie neighbourhood, this implies that $\len(\alpha)=1$ or $\len(\alpha)\geq 3$.
If $\len(\alpha)\geq 3$, then the snippets containing $\alpha(0)$ and $\alpha(1)$ must both intersect $\partial N$ and $\partial S$, and hence are in efficient position.
\end{proof}

\begin{lem}\label{Bad snippet types in long enough curves and arcs}
Suppose that $\alpha \subset S$ is a properly immersed, snippet-decomposed arc or curve in $S$. If $\len(\alpha)>1$, then any bad snippet of $\alpha$ is of one of the following two kinds:
\begin{itemize}
\item a bigon snippet, that is of type $\mathbb{B}(h,h)$, $\mathbb{B}(t,t)$, $\mathbb{S}(h,h,0)$, $\mathbb{S}(t,t,0)$, $\mathbb{S}(v,v,0)$, $\mathbb{S}(t,v,1)$, $\mathbb{S}(t,t,2)$, $\mathbb{R}(h,h)$, or $\mathbb{R}(v,v)$.
\item a trigon snippet, that is of type $\mathbb{B}(h,t)$, $\mathbb{S}(h,t,1)$, $\mathbb{S}(h,v,2)$, $\mathbb{S}(h,t,3)$, or $\mathbb{R}(h,v)$.
\end{itemize}
\end{lem}

\begin{proof}
This follows from the classification of bad snippets in Corollary \ref{classification of bad snippets}.
\end{proof}

\begin{lem}\label{Efficient position for curves implies essential and non-peripheral}
Suppose that $N=N(\tau)$ is a tie neighbourhood of a large train track $\tau \subset S$. If $\alpha$ is an immersed curve which is in efficient position with respect to $N$, then $\alpha$ is essential and nonperipheral in $S$.
\end{lem}

\begin{proof}
Suppose that $\alpha$ is a snippet-decomposed curve that is in efficient position with respect to $N$. Let $\widetilde{\alpha} \subset \widetilde{S}$ be a preimage of $\alpha$ in the universal cover of $S$.

First, assume that $\alpha$ is inessential.  As $\alpha$ represents a trivial conjugacy class in $\pi_1(S)$, we know that $\widetilde{\alpha}$ is a closed curve in $\widetilde{S}$. 
Let us first assume that $\widetilde{\alpha}$ is embedded. Thus, $\widetilde{\alpha}$ bounds an embedded disk $D \subset \widetilde{S}$. We recall that $\mathcal{R}$ provides a tiling of $S$ by rectangles, polygons, and peripheral annuli. Hence, $\widetilde{\mathcal{R}}$ provides a tiling of $\widetilde{S}$. The induced tiling of $D$ consists of lifts of tiles in $\mathcal{R}$ and lifts of subsets of tiles in $\mathcal{R}$ cut off by snippets of $\alpha$. We know that $\ind(D)=1$. Therefore, at least one of the tiles of $D$, in the following called $T$, must have positive index. Since tiles $R \in \mathcal{R}$ have non-positive index, they lift to non-compact tiles or tiles of non-positive index. Thus, $T$ must be a subtile of some $\widetilde{R} \in \widetilde{\mathcal{R}}$ that is cut off by the lift of a snippet $a \subset \alpha$. Since $T$ is embedded and has positive index, we know that $a$ must be embedded and cuts off a region of positive index inside the tile $R \in \mathcal{R}$. This contradicts our assumption on $\alpha$ being in efficient position.

Let us now assume that $\widetilde{\alpha}$ is not embedded. Since $\widetilde{S}$ is simply connected, we know that one of the regions bounded by $\widetilde{\alpha}$ must be an embedded monogon following \cite[Lemma 1.1.]{HassScott}. Since the index of this monogon is greater than or equal to $3/4$, a contradiction is derived as in the previous paragraph.

Secondly, let us assume that $\alpha$ is peripheral. If $\widetilde{\alpha}$ is not embedded, \cite[Lemma 1.1.]{HassScott} implies that $\widetilde{\alpha}$ bounds an embedded monogon. We derive a contradiction as in the previous cases. Therefore, let us assume that $\widetilde{\alpha}$ is embedded. Thus, together with a lift $\widetilde{\beta}$ of a component $\beta$ of $\partial S$, the arc $\widetilde{\alpha}$ bounds a bi-infinite strip. Let $\delta \subset \widetilde{S}$ be an embedded arc connecting $\widetilde{\alpha}$ with $\widetilde{\beta}$. Without loss of generality, we may assume that $\delta$ is contained inside the strip bounded by $\widetilde{\beta}$ and $\widetilde{\alpha}$. After applying a small isotopy, we can further assume that $\delta$ meets $\widetilde{\alpha}$, $\widetilde{\beta}$, as well as $\widetilde{\partial \mathcal{R}}$ perpendicularly. Then $\delta$ and its image under the covering transformation corresponding to $\alpha$ cut an embedded rectangle $A \subset \widetilde{S}$ out of the bi-infinite strip bounded by $\widetilde{\alpha}$ and $\widetilde{\beta}$. As $\widetilde{\alpha}$ and $\widetilde{\beta}$ do not have any corners, $A$ has exactly four outward-pointing corners. Thus, we know that $\ind(A)=0$. Let $R \in \mathcal{R}_\textrm{com}$ be the peripheral complementary region of $N$ whose boundary contains $\beta$. If $\widetilde{\alpha} \subset \widetilde{R}$, then $\alpha$ is a snippet-decomposed curve of length one, thus is not in efficient position. 
If $\widetilde{\alpha} \cap \partial \widetilde{R} \neq \emptyset$, then $\widetilde{\alpha}$ and $\partial \widetilde{R}$ co-bound an embedded bigon lying outside of $\widetilde{R}$ of index greater than or equal to $1/2$. 
Hence, this bigon must have at least one positive-index tile, which contradicts the assumption that all tiles cut off by snippets of $\widetilde{\alpha}$ have non-positive index. Thus, we can assume that $\widetilde{\alpha} \cap \widetilde{R} = \emptyset$. Therefore, $A$ must contain a tile of index smaller or equal to the index of $R$. 
Snippets of $\widetilde{\alpha}$ cut off non-positive tiles. By additivity of the index, the index of tiles of $A$ cut off by the arc $\delta$ and the index of tiles of $A$ cut off by the image of $\delta$ under the covering transformation add up to the index of the entire tile (potentially cut off by $\widetilde{\alpha}$ and $\widetilde{\beta}$), which must be non-positive. Since $R$ has strictly negative index, this implies that $A$ must have strictly negative index as well. This yields a contradiction.
\end{proof}

\begin{lem}\label{Efficient position for arcs implies essential}
Suppose that $N=N(\tau)$ is a tie neighbourhood of a large train track $\tau \subset S$. If $\alpha$ is a properly immersed arc which is in efficient position with respect to $N$, then $\alpha$ is essential in $S$.
\end{lem}

\begin{proof}
Suppose that $\alpha$ is a properly immersed, snippet-decomposed arc in efficient position. Let $\widetilde{\alpha} \subset \widetilde{S}$ be one preimage of $\alpha$ in the universal cover of $S$. Suppose that $\alpha$ is inessential. If $\widetilde{\alpha}$ is not embedded, then $\widetilde{\alpha}$ bounds an embedded monogon following \cite[Lemma 1.1.]{HassScott}. As in the proof of the previous lemma, this contradicts our assumption that $\alpha$ is in efficient position.
Hence, let us assume that $\widetilde{\alpha}$ is embedded. Suppose that $R \in \mathcal{R}_\textrm{com}$ is the peripheral complementary region that contains the boundary component which $\alpha$ is homotopic into. Since $\len(\alpha)>2$, we know that $\widetilde{\alpha}$ bounds an embedded region of index at least $1/2$ with a subset $\widetilde{\partial R}$. This implies that there must be a positive-index subtile of $\widetilde{\mathcal{R}}$ cut off by a lift of a snippet $a \subset \alpha$. Again, this contradicts our assumption that all snippets of $\alpha$ are in efficient position and thus cut off only non-positive regions.
\end{proof}

We recall that the goal of this thesis is to show the converse to the previous two lemmas, that is, that essential and non-peripheral arcs and curves in $S$ can be homotoped into efficient position. We do this by constructing an explicit algorithm that homotopes essential and non-peripheral arcs and curves into efficient position and runs in polynomial-time in the length of its input.

\section{Algorithms and Python notation for snippet-decomposed arcs and curves}

We end this chapter by introducing algorithms and giving some further notation for snippet-decomposed arcs and curves. This notation is based on standard conventions within the Python programming language. For a reference, we refer the reader to \cite{Python}.

\subsection{Conventions on algorithms}
In this thesis, an algorithm consists of
\begin{itemize}
\item the name of the algorithm,
\item a brief comment on the purpose of the algorithm,
\item a specification of the input and output of the algorithm, and
\item a block of pseudo-code.
\end{itemize}
For an example we refer the reader to Algorithm \ref{TrigArc} on page \pageref{TrigArc}.
Following \cite[1]{IntroToAlgo}, we say that an algorithm is \emph{correct} if, for every input instance, it halts with the correct output. 

\begin{rem}\label{Standard conventions algor}
By convention, an algorithm should verify that the input meets the specified requirements for input instances of the algorithm. For all algorithms presented in this thesis, the input consists of a surface $S$, a tie neighbourhood $N=N(\tau)\subset S$ of a large train track $\tau$ in $S$, and a snippet-decomposed arc or curve $\alpha$. From a computer science point of view, it is most convenient to assume that the tie neighbourhood is given via the collection of branch and switch rectangles with corresponding identifications. Likewise, the surface is assumed to be given via its polygonal decomposition, that is, as a collection of rectangles, polygons, and peripheral annuli with the corresponding gluings. Then, verifying the requirements for input instances are routine checks and can be done in $\mathit{O}(|\chi(S)|\cdot \len(\alpha))$ time. Thus, these checks will be omitted from the pseudocode of the algorithm. 

In the course of this thesis, algorithms often perform ``look-ups'' in data structures of size $\mathit{O}(|\chi(S)|)$ (see for example page \pageref{Required time for trigon homotopies}, Remark \ref{Required time for trigon homotopies}). Examples of such data structures include the list of all sides of regions in $\mathcal{R}$ or lists of tuples of sides that yield bad snippets. In general, the length of the arc or curve is a greater contributor to the running time of the algorithms than $\mathit{O}(|\chi(S)|)$. Therefore, we do not aim to give best-possible bounds on the number of operations required to perform such ``look-up'' operations. Instead, we simply assume that they can be done in $\mathit{O}(|\chi(S)|)$ time without worrying about the exact computational models.
\end{rem}

\subsection{Python notation}

Let $\alpha$ be a snippet-decomposed arc or curve of length $k=\len(\alpha)$. Recall that we parametrize $S^1$ as $[0,1]/_{\sim}$. We always consider snippet-decomposed arcs and curves to be oriented. We say that the \textit{first snippet} of $\alpha$ is the snippet that begins with $\alpha(0)$ and that the \textit{last snippet} of $\alpha$ is the snippet that ends with $\alpha(1)$. Suppose that $0 \leq i < j \leq k$. We set $\alpha[i:j]$ to be the subarc of $\alpha$ that begins with the $(i+1)$-th snippet of $\alpha$ and ends with the $j$-th snippet of $\alpha$.
For $ 0 \leq i \leq k$ we set $\alpha[i:i]$ to be the empty subarc. We abbreviate $\alpha[i:i+1]$ by $\alpha[i]$, and note that this gives us the $(i+1)$-th snippet of $\alpha$.
By convention, negative indices count from the end of $\alpha$, that is $\alpha[-j:-i]=\alpha[k-j:k-i]$.

Suppose that $\beta, \gamma \subset S$ are two adjacent subarcs of $\alpha$ such that $\beta(1)=\gamma(0)$. Then, we denote their \textit{concatenation} by $\beta \cdot \gamma$. Hence, $\alpha = \alpha[0 : i] \cdot \alpha[i : k]$ for  any $ 0 \leq i \leq k$. To address all snippets of $\alpha$ from its beginning up to the $i$-th snippet, we use the notation $\alpha[:i]=\alpha[0:i]$. Similarly, $\alpha[i:]=\alpha[i:k]$ denotes the subarc of $\alpha$ starting with the $(i+1)$-th snippet and ending with the last snippet of $\alpha$. If $\alpha$ is a curve, we consider indexing to be \textit{circular}. That is, for $0 \leq i,j < \len(\alpha)$, we set $\alpha[i:\len(\alpha)+j]=\alpha[i:] \cdot \alpha[:j]$.

Recall that snippets are parametrized by the unit interval. Let $0 \leq \epsilon \leq \delta \leq 1$. As $\alpha[i]=\alpha[i:i+1]$, we set $\alpha[i+\epsilon:i+\delta]=\alpha[i]\big|[\epsilon,\delta]$ in an abuse of notation.\label{Page: alpha subsnippet} We remark that with this convention, $\alpha[i]=\alpha[i:i+\epsilon] \cdot \alpha[i+\epsilon:i+\delta] \cdot \alpha[i+\delta:i+1]$.

\begin{defn}\label{Page: alpha trim}
Suppose that $\alpha \subset S$ is a snippet-decomposed arc or curve in $S$. If $\alpha$ is an arc, we say that the \textit{inside} of $\alpha$ consists of the snippet-decomposed arc $\alpha_\textrm{trim}=[1:\len(\alpha)-1]$. Thus, if $\alpha$ is an arc and $\len(\alpha) < 3$, the inside of $\alpha$ is empty. If $\alpha$ is a curve, we set $\alpha_\textrm{trim}=\alpha$.
\end{defn}

\newpage
\chapter{Local homotopies and trigon arcs}

In this chapter, we present the most important algorithm of this thesis, the routine \texttt{TrigArc}. Almost every further algorithm is built upon it, in particular, the algorithm we use for the proof of Theorem \ref{Polynomial time algorithm}. As input, the algorithm \texttt{TrigArc} takes an arc $\alpha$ which contains exactly one bad snippet in its inside. This bad snippet is required to be a trigon snippet. The output is an arc $\alpha'$ homotopic, relative its endpoints, to $\alpha$ such that $\alpha'$ does not contain any bad snippets in its inside. In other words, $\alpha'[1:-1]$ is in efficient position. This is achieved by applying a series of local homotopies depending on the type of the initial trigon.

The outline of this chapter is a follows: First, we introduce the notion of a ``turning direction'' for certain snippets and define \textit{horizontal} and \emph{vertical duals}. Both concepts are important throughout the remainder of this thesis. Secondly, we define two further notions of length for arcs and curves, the \emph{corner length} and the \emph{reduced corner length}, and study how they are related to the snippet length of an arc or curve. Thirdly, we introduce a family of local homotopies and discuss the effects they have on arcs or curves with unique trigons in their insides. Here, we put a special emphasis on the various notions of length and their changes under the trigon homotopies. We then present the algorithm \texttt{TrigArc} and use our prior observations to give bounds on its running time and the length of its output arc. We close this chapter by defining and analysing the algorithm \texttt{TrigCurve}, which homotopes snippet-decomposed curves with a single bad snippet of trigon type into efficient position.

Throughout this chapter, we fix a surface $S = S_{g,b}$ satisfying $\xi(S)=3g-3+b \geq 1$. We further fix a large train track $\tau \subset S$ and a tie neighbourhood $N=N(\tau)$ of $\tau$ in $S$. We recall that we do not distinguish between snippets and their strong snippet homotopy classes unless otherwise stated. For any snippet $a \subset R \in \mathcal{R}$, we always assume that $a$ has minimal self-intersection, intersects $\partial R$ perpendicularly and misses $\partial^2 \mathcal{R}$. We further recall that arcs and curves in $S$ are assumed to be self-transverse and transverse to $\partial \mathcal{R}$. Moreover, if $\alpha$ is an arc, we assume that $\alpha(0)$ and $\alpha(1)$ lie in $\partial \mathcal{R}$. Thus, arcs and curves in $S$ admit canonical decompositions into snippets.

\section{Right/left turning snippets and horizontal/vertical duals}

Recall that snippets are parametrized and are therefore canonically oriented. Hence, we can talk about regions cut off by the snippet on its right- or left-hand side.

\begin{defn}
Fix a region $R \in \mathcal{R}=\mathcal{R}_\textrm{tie} \cup \mathcal{R}_\textrm{comp}$. Suppose that $a \subset R$ is an embedded snippet that cuts off a region $T$ of $R$ of positive or non-negative index, where $R \in \mathcal{R}_\textrm{tie}$ or $R \in \mathcal{R}_\textrm{comp}$ respectively. We further suppose that the boundary of this region contains at least one point of $\partial^2 \mathcal{R}$. We say that the snippet $a$ is \textit{turning right} if it cuts off the region $T$ on its right-hand side. Similarly, we say that $a$ is \textit{turning left} if it cuts off the region $T$ on its left-hand side (see Figures \ref{Right turning trigon}-\ref{Left turning trigon}).
\end{defn}

\begin{figure}[htbp] 
  \begin{minipage}[b]{0.49\linewidth}
    \centering
\begin{tikzpicture}[scale=0.35]
\draw [very thick] (-3,1.5) rectangle (4,-3);
\draw [very thick](3.5,0) -- (4.5,0);
\draw [very thick](3.5,-1.5) -- (4.5,-1.5);
\draw [densely dashed, blue, very thick] plot[smooth, tension=.7] coordinates {(0.18,1.48)  (0.96,-1.83) (4,-2.5)};
\draw [blue, fill](0.26,0.32) -- (0.105,0.03) -- (0.51,0.07) -- cycle;
\end{tikzpicture}
    \caption{A right-turning trigon inside a switch rectangle.}
    \label{Right turning trigon}
  \end{minipage} 
\begin{minipage}[b]{0.49\linewidth}
\centering
\begin{tikzpicture}[scale=0.25]
\draw [very thick](-6,-3) node (v1) {} -- (-5,-5) node (v2) {};
\draw [very thick] plot[smooth, tension=.7] coordinates {(v1) (-2,-1) (-1,0.5)};
\draw [very thick] plot[smooth, tension=.7] coordinates {(v2) (1,-4) (7,-5)};
\draw [very thick](7,-5) -- (8,-3) node (v3) {};
\draw [very thick] plot[smooth, tension=.7] coordinates {(v3) (4,-1) (3,0.5)};
\draw [very thick, densely dashed, blue] plot[smooth, tension=.7] coordinates {(-5.6,-3.75) (-1.55,-2.35) (0.6,-2.6) (1,-4)};
\draw [blue, fill](-3.26,-2.42) -- (-3.7,-3.02) -- (-2.98,-3.1) -- cycle;
\end{tikzpicture}
 \caption{A left-turning trigon inside a complementary region.} 
 \label{Left turning trigon}
\end{minipage}
\end{figure}

It follows from the additivity of the index and the definition of a large train track that a snippet cannot be turning left and right at the same time.

\begin{defn}
Suppose that $a \subset R$ is a dual snippet inside a complementary region $R \in \mathcal{R}_\textrm{comp}$. Suppose further that $a$ is embedded and cuts off a region of $R$ of index zero. We say that $a$ is a \textit{horizontal dual} if $\partial a \subset \partial_v N$. Similarly, we say that $a$ is a \textit{vertical dual} if $\partial a \subset \partial_h N$ (see Figures \ref{Right horizontal dual}-\ref{Left vertical dual}).
\end{defn}

\begin{figure}[htbp] 
\begin{minipage}[b]{0.49\linewidth}
\centering
\begin{tikzpicture}[scale=0.25]
\draw [very thick](-6,-3) node (v1) {} -- (-5,-5) node (v2) {};
\draw [very thick] plot[smooth, tension=.7] coordinates {(v1) (-2,-1) (-1,0.5)};
\draw [very thick] plot[smooth, tension=.7] coordinates {(v2) (1,-4) (7,-5)};
\draw [very thick](7,-5) -- (8,-3) node (v3) {};
\draw [very thick] plot[smooth, tension=.7] coordinates {(v3) (4,-1) (3,0.5)};
\draw [very thick, densely dashed, blue] plot[smooth, tension=.7] coordinates {(-5.46,-4) (-1.55,-2.35) (2.84,-2.38) (7.5,-3.94)};
\node at (-6.25,-4.25) {\tiny{$v$}};
\draw [blue, fill](1.2,-1.82) -- (1.12,-2.44) -- (1.9,-2.22) -- cycle;
\node at (-3.86,-1.22) {\tiny{$h$}};
\end{tikzpicture}
 \caption{A right horizontal dual.} 
 \label{Right horizontal dual}
\end{minipage}
  \begin{minipage}[b]{0.49\linewidth}
    \centering
\begin{tikzpicture}[scale=0.25]
\draw [very thick](-6,-3) node (v1) {} -- (-5,-5) node (v2) {};
\draw [very thick] plot[smooth, tension=.7] coordinates {(v1) (-2,-1) (-1,0.5)};
\draw [very thick] plot[smooth, tension=.7] coordinates {(v2) (1,-4) (7,-5)};
\draw [very thick](7,-5) -- (8,-3) node (v3) {};
\draw [very thick] plot[smooth, tension=.7] coordinates {(v3) (4,-1) (3,0.5)};
\draw [very thick, densely dashed, blue] plot[smooth, tension=.7] coordinates {(-2.28,-1.32) (-1.55,-2.35) (-1.16,-3.02) (-0.78,-4.1)};
\node at (-6.25,-4.25) {\tiny{$v$}};
\draw [blue, fill](-1.78,-1.96) -- (-1.9,-2.66) -- (-1.14,-2.24) -- cycle;
\node at (-3.86,-1) {\tiny{$h$}};
\end{tikzpicture}
    \caption{A left vertical dual.} 
    \label{Left vertical dual}
  \end{minipage} 
\end{figure}

Combining these two concepts, we obtain \emph{right- or left-turning horizontal or vertical duals}. In the following, we simply call them \emph{right or left horizontal or vertical duals}.

\section{Corner length and reduced corner length}

\begin{defn} \label{maximal side length}
Suppose that $C$ is a component of $\partial_h N$. We set 
\begin{align*}
s(C)= (|C\cap \partial^2 \mathcal{R}|-2) \cdot 2 + 1.
\end{align*}
In other words, $s(c)$ is a weighted count of the components of $C-\partial ^2 \mathcal{R}$, where a component is attributed weight one or three if it is contained in the boundary of a branch or switch rectangle respectively. We further set
\begin{align*}
s=s_N= \max \left\{ s(c) \bigm |
C \textrm{ a component of } \partial_h N \right\},
\end{align*}
and call $s$ the \emph{maximal side length of $N$}.

%\begin{align*}
%s=s_N= \max \left\{ s(c) \Bigm | \begin{array}{c}
%C \textrm{ a component of } \partial_h N \\ 
%\textrm{ of non-negative index}
%\end{array} \right\},
%\end{align*}
\end{defn}

\begin{rem} \label{s is O(Euler)}
As $S$ is compact, $\tau$ has a finite number of branches and switches. Therefore, the tie neighbourhood $N=N(\tau)$ is compact and $\partial_h N$ has a finite number of components. Hence, the maximal side length of $N$ is well-defined. We remark that no horizontal side of $R \in \mathcal{R}_\textrm{comp}$ is homeomorphic to a circle. This follows from the fact that complementary regions of tie neighbourhoods of large train tracks must have at least two outwards pointing corners to satisfy the requirements on index. As branch and switch rectangles alternate in the tie neighbourhood and the first and last component of $C-\partial^2 \mathcal{R}$ belongs to the horizontal boundaries of branch rectangles, $s(C)$ is indeed a weighted count of the boundaries of branch and switch rectangles contained in $C$.
We further remark that $s= \mathit{O}(|\chi(S)|)$ by \cite[Corollary~1.1.3]{Penner}. This will be of importance when we wish to obtain complexity bounds for our algorithms.
\end{rem}

\begin{lem}
Suppose that $N=N(\tau)$ is a tie neighbourhood of a large train track $\tau \subset S$. Then $s_N \geq 5$.
\end{lem}

\begin{proof}
As train tracks are non-trivial by assumption, the tie neighbourhood $N$ contains at least one switch rectangle. Each component of the horizontal boundary of this switch rectangle must be contained in a component $C$ of $\partial_h R_C$ for some complementary region $R_C \in \mathcal{R}_\textrm{comp}$. As no horizontal boundary side is homeomorphic to a circle, branch and switch rectangles alternate in the tie neighbourhood and the first and last component of $C-\partial^2 \mathcal{R}$ must be contained in the horizontal boundaries of branch rectangles, this implies that $|C-\partial^2 \mathcal{R}|\geq 4$. Hence, $s(C) \geq 5$ and the claim follows. 
\end{proof}

\begin{defn}\label{Page: corner length}
Suppose that $a \subset R$ is a snippet inside a complementary region $R \in \mathcal{R}_\textrm{comp}$.
If $a$ is non-peripheral, embedded, and cuts off a region $T$ of $R$ of non-negative index, then $\partial T = a \cup b$ for some set $b \subset (\partial \mathcal{R}-\partial S)$. By $|b|_v$ (respectively $|b|_B$ or $|b|_S$) we denote the number of components of $b- \partial^2 \mathcal{R}$ whose closure is a component of $\partial_v N$ (respectively of the horizontal boundary of a branch or switch rectangle of $N$). We set
\begin{align*}
\len_\textrm{corn}(a)=|b|_v + |b|_B + 3 \cdot |b|_S.
\end{align*}
If $a \subset R$ is peripheral, not embedded, or does not cut off a region of $R$ of non-negative index, we set
\begin{align*}
\len_\textrm{corn}(b)=2s_N.
\end{align*}
For any snippet $a \subset R \in \mathcal{R}_\textrm{comp} $, we call $\len_\textrm{corn}(a)$ the \emph{corner length of the snippet $a$}.
\end{defn}
For some examples of snippets and their respective corner length we refer the reader to Figure \ref{corner length examples}.

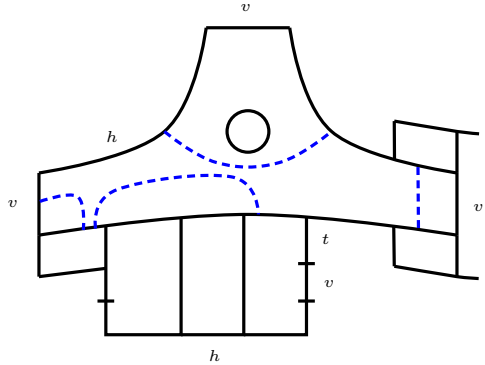
\begin{figure}[htbp] 
\begin{minipage}[b]{0.99\linewidth}
    \centering
\begin{tikzpicture}[scale=0.55]
\draw [very thick](-4,-1.5) {} -- (-4,-3) {};
\draw [very thick] plot[smooth, tension=.7] coordinates { (-4,-3) (1,-2.5) (6,-3)};
\draw [very thick](6,-1.5) node (v3) {} -- (6,-3) node (v1) {};
\draw [very thick] plot[smooth, tension=.7] coordinates {(-4,-1.5) (-1,-0.5) (0,2)};
\draw [very thick] plot[smooth, tension=.7] coordinates {(v3) (3,-0.5) (2,2)};
\draw [very thick] plot[smooth, tension=.7] coordinates {(-4,-4)  (-2.4,-3.8)};
\node at (-4.625,-2.25) {{\tiny{$v$}}};
\node at (-2.25,-0.625) {{\tiny{$h$}}};
\draw [very thick](-2.4,-2.78) -- (-2.4,-5.4) -- (-0.6,-5.4) -- (-0.6,-2.58);
\draw [very thick](-0.6,-5.4)-- (0.9,-5.4)  -- (0.9,-2.5);
\draw [very thick](0.9,-5.4) -- (2.4,-5.4) -- (2.4,-2.56);
\draw [very thick](-2.6,-4.588) -- (-2.2,-4.588);
\draw [very thick](2.2,-3.686) -- (2.6,-3.686);
\draw [very thick](2.2,-4.588) -- (2.6,-4.588);
\node at (2.94,-4.16) {{\tiny{$v$}}};
\node at (2.86,-3.1) {{\tiny{$t$}}};
\node at (0.2,-5.9) {{\tiny{$h$}}};
\draw [very thick](0,2) -- (2,2);
\node at (0.947,2.496) {{\tiny{$v$}}};
\node at (6.5175,-2.3319) {{\tiny{$v$}}};
\draw [very thick](-4,-3) -- (-4,-4);
\draw [very thick] (1,-0.5) ellipse (0.5 and 0.5);
\draw [very thick](6,-1.5) -- (6,-0.5) node (v4) {};
\draw [very thick](4.5057,-1.191) -- (4.5,-0.25) node (v2) {};
\draw [very thick](4.4876,-2.8023) -- (4.5,-3.75) node (v5) {};
\draw [very thick, densely dashed, blue] plot[smooth, tension=.7] coordinates {(-3.9787,-2.1881) (-3.0787,-2.0712) (-2.9384,-2.8777)};
\draw [very thick, densely dashed, blue] plot[smooth, tension=.7] coordinates {(-2.6463,-2.8426) (-2.1437,-1.9076) (0.5912,-1.5921) (1.25,-2.5)};
\draw [very thick, densely dashed, blue] plot[smooth, tension=.7] coordinates {(5.0674,-1.3583) (5.0791,-2.8893)};
\draw [very thick, densely dashed, blue] plot[smooth, tension=.7] coordinates {(-1,-0.5) (0.25,-1.25) (1.75,-1.25) (3,-0.5)};
\draw [very thick] plot[smooth, tension=.4] coordinates {(v2) (v4) (6.5288,-0.5571)};
\draw [very thick](6,-3) -- (6,-4) node (v6) {};
\draw [very thick] plot[smooth, tension=.4] coordinates { (4.5,-3.75) (6,-4)(6.5192,-4.0457)};
\end{tikzpicture}
    \caption{These snippets have corner length zero, one, four, and $2s_N$ respectively.}
    \label{corner length examples}
  \end{minipage}
\end{figure}

\begin{rem}\label{Bounds on corner length}
Following an index argument, any embedded snippet $a \subset R \in \mathcal{R}_\textrm{comp}$ cutting off a region $T$ of $R$ of non-negative index cuts off exactly one such region of $R$. If $\len_\textrm{corn}(a)=0$, then $a$ must be a bigon snippet, trigon snippet, or an inessential curve. This stems from the fact that every simply connected region of non-positive index has at least four outward facing corners, of which at least two must belong to $\partial^2 \mathcal{R}$. Hence, its boundary contains at least one component of $\partial \mathcal{R}- \partial^2 \mathcal{R}$. For any trigon or bigon snippet $a \subset R$, $\len_\textrm{corn}(a)$ is bounded by $s_N -1$. Suppose that $a \subset R$ is a horizontal dual. Then $a$ is parallel to a horizontal boundary side $C$ of $\partial_h R$ and $\len_\textrm{corn}(a)=s(C) \leq s_N $. Similarly, the corner length of any vertical dual is bounded by $2s_N -1$. Thus, $\len_\textrm{corn}(a) \leq 2s_N$ for any snippet $a \subset R \in \mathcal{R}_\textrm{comp}$.
\end{rem}

\begin{defn}
Suppose that $a \subset R$ is a snippet inside a branch or switch rectangle $R \in \mathcal{R}_\textrm{tie}$. We set $\len_\textrm{corn}(a)=1$ or $\len_\textrm{corn}(a)=3$ respectively and call $\len_\textrm{corn}(a)$ the \emph{corner length of the snippet $a$}. The \textit{corner length} $\len_\textrm{corn}(\alpha)$ of any snippet-decomposed arc or curve $\alpha \subset S$ is the sum over the corner lengths of all its snippets. In other words,
\begin{align*}
\len_\textrm{corn}(\alpha)=\sum_{\substack{a \; \textrm{snippet} \\  \textrm{of} \; \alpha }} \len_\textrm{corn}(a).
\end{align*}
\end{defn}

\begin{lem}\label{Snippet vs corner length}
Suppose that $\alpha \subset S$ is a snippet-decomposed arc or curve. Then
\begin{align*}
\len(\alpha) \leq \len_\textrm{corn}(\alpha)+ m_\alpha \leq 2s_N \cdot \len(\alpha) + m_\alpha,
\end{align*}
where $m_\alpha$ is the number of bad snippets of $\alpha$.
\end{lem}

\begin{proof}
Since $s_N \geq 4$, both bounds follow from the definition of corner length and Remark \ref{Bounds on corner length}.
\end{proof}

Throughout the various algorithms presented in this thesis, one type of subarc is occurring again and again. Its presence turns out to be a valuable indicator for the growth of the arc or curve under local homotopies.

\begin{defn}
Suppose that $\alpha \subset S$ is an arc in efficient position which contains exactly three snippets. Suppose further that $\alpha[0]$ and $\alpha[2]$ are both right or left vertical duals and $\alpha[1]$ is a dual snippet inside a branch rectangle. We then call $\alpha$ a \textit{right or left blocker} respectively.
\end{defn}

For an example of a blocker we refer the reader to Figure \ref{Example of blocker}.

\begin{figure}[htbp] 
\begin{minipage}[b]{0.99\linewidth}
    \centering
\begin{tikzpicture}[scale=0.55]
\draw [very thick](-4,-1.5) {} -- (-4,-3) {};
\draw [very thick] plot[smooth, tension=.7] coordinates { (-4,-3) (1,-2.5) (6,-3)};
\draw [very thick](6,-1.5) node (v3) {} -- (6,-3) node (v1) {};
\draw [very thick] plot[smooth, tension=.7] coordinates {(-4,-1.5) (-1,-0.5) (0,2)};
\draw [very thick] plot[smooth, tension=.7] coordinates {(v3) (3,-0.5) (2,2)};
\draw [very thick] plot[smooth, tension=.7] coordinates {(-4,-4)  (-2.4,-3.8)};
\node at (-7.5771,0.3297) {{\tiny{$v$}}};
\node at (-1.4691,-0.2399) {{\tiny{$h$}}};
\draw [very thick](-2.4,-2.78) -- (-2.4,-5.4) -- (-0.6,-5.4) -- (-0.6,-2.58);
\draw [very thick](-0.6,-5.4)-- (0.9,-5.4)  -- (0.9,-2.5);
\draw [very thick](-2.6,-4.588) -- (-2.2,-4.588);
\node at (1.3937,-3.916) {{\tiny{$t$}}};
\node at (0.2,-5.9) {{\tiny{$h$}}};
\draw [very thick](0,2) -- (2,2);
\node at (0.947,2.496) {{\tiny{$v$}}};
\node at (6.5175,-2.3319) {{\tiny{$v$}}};
\draw [very thick](-4,-1.5) -- (-4,-0.5) node (v5) {} -- (-2.4876,-0.1221) -- (-2.4831,-1.1652);
\draw [very thick](-4,-0.5) -- (-5.5,-0.5) node (v8) {} -- (-5.5,-4) -- (-4,-4) node (v2) {};
\draw [thick, dotted](-5.5,-0.5)
 -- (-8,-0.5) -- (-8,1) node (v9) {};
\draw [thick, dotted] plot[smooth, tension=.7] coordinates {(v9) (-6,1.5) (-4,3)};
\draw [very thick](-4,-3) -- (-4,-4);
\draw [very thick, densely dashed, blue] plot[smooth, tension=.4] coordinates {(-4.92,2.2063) (-3.292,-0.3172) (-2.7132,-1.9091) (-0.0902,-1.665) (0.3259,-2.5281) };
\node at (-3.5474,-2.2808) {{\tiny{$v$}}};
\node at (-6.0952,-0.1654) {{\tiny{$h$}}};
\draw [very thick, fill, blue](-4,0.8) -- (-4.0402,0.5886) -- (-3.7518,0.8046) -- cycle;
\end{tikzpicture}
    \caption{A left blocker.}
    \label{Example of blocker}
  \end{minipage}
\end{figure}

\begin{lem}\label{Length of blocker}
Suppose that $\alpha \subset S$ is a right or left blocker. Then $\len_\textrm{corn}(\alpha) \geq 7$.
\end{lem}

\begin{proof}
As $\alpha$ is in efficient position, we know that $\len_\textrm{corn}(\alpha[0])$ and $\len_\textrm{corn}(\alpha[2])$ are greater than zero. We further know that $\len_\textrm{corn}(\alpha[1])=1$ as $\alpha[1]$ is a snippet inside a branch rectangle. Let us assume that $\len_\textrm{corn}(\alpha[0])=1$. Then $\alpha[2]$ cuts off a region of index zero whose boundary contains at least one component of the horizontal boundary of the switch rectangle that is adjacent to the branch rectangle containing $\alpha[1]$ (see Figure \ref{Example of blocker}). As branch and switch rectangles alternate, this implies that the boundary of the region cut off by $\alpha[2]$ contains also at least one component of $\partial_h R$ for a branch rectangle $R \in \mathcal{R}_\textrm{tie}$ as well as one component of $\partial_v N$. Hence, $\len_\textrm{corn}(\alpha[2]) \geq 3+1+1 = 5$ and the claim of the lemma follows. 
\end{proof}

\begin{lem}\label{Overlap of blockers}
Suppose that $\alpha \subset S$ is an arc or curve containing two blockers $\beta, \gamma \subset \alpha$. If $\beta\neq \gamma$, then $\beta$ and $\gamma$ overlap in at most one snippet. This is only possible if both are right or left blockers. Furthermore, no two blockers can intersect in exactly one endpoint.
\end{lem}

\begin{proof}
As the second snippet of any blocker lies inside the tie neighbourhood and the first and last snippet of any blocker lies inside a complementary region, two different blockers can overlap along at most one of their vertical duals. Since any vertical dual cuts off a region of index zero on exactly one side, any two blockers of the same curve that share one vertical dual must be turning into the same direction.
As the first and last snippets of blockers lie in complementary regions, the arc or curve $\alpha$ is transverse to $\partial \mathcal{R}$, and any two complementary regions of $N$ are separated from each other by the tie neighbourhood, no two blockers can be directly adjacent to each other in $\alpha$. Hence, no two blockers intersect in exactly one endpoint.
\end{proof}

\begin{defn}\label{Page: reduced corner length and reduced corner length}
Suppose that $\alpha \subset S$ is an arc or curve. By $\len_\textrm{block}(\alpha)$ we denote the number of blockers contained in $\alpha$. We set
\begin{align*}
\len_\textrm{red}(\alpha)=\len_\textrm{corn}(\alpha)-2 \cdot \len_\textrm{block}(\alpha)
\end{align*}
and call $\len_\textrm{red}(\alpha)$ the \emph{reduced corner length of $\alpha$}.
\end{defn}

\begin{lem}\label{SnippetSequenceIOnly}
Suppose that $\alpha$ is an arc such that each snippet of $\alpha$ is contained in at least one blocker. Then $\len_\textrm{red}(\alpha) \geq \len(\alpha)$.
\end{lem}

\begin{proof}
By Lemma \ref{Overlap of blockers}, if every snippet of $\alpha$ is contained in at least one blocker, $\alpha$ consists of blockers only. These overlap along their vertical duals. As every blocker consists of three snippets, $\len(\alpha)$ must be odd. 

We prove the desired statement by induction on the snippet length of $\alpha$. For $\len(\alpha)=3$, the claim follows from Lemma \ref{Length of blocker}.
Now suppose that $\len(\alpha)>3$. Then $\alpha[0:5]$ consists of two blockers, hence $\alpha[4:]$ contains two fewer blockers than $\alpha$. We remind ourselves that for any dual snippet $a \subset \alpha$ we know that $\len_\textrm{red}(a)=\len_\textrm{corn}(a)\geq \len(a)$. As $\alpha$ must be in efficient position, induction and Lemma \ref{Length of blocker} give us
\begin{align*}
\len_\textrm{red}(\alpha)&= \len_\textrm{corn}(\alpha[0:4]) + \len_\textrm{red}(\alpha[4:])-4\\
&\geq 8 + (\len(\alpha[4:])) -4\\
& \geq \len(\alpha[:4]) + 4 + \len(\alpha[4:]) -4\\
& \geq \len(\alpha)
\end{align*}
and the claim follows.
\end{proof}

\begin{lem}\label{reduced corner length vs snippet length}
Suppose that $\alpha$ is an arc. Then $\len_\textrm{red}(\alpha) \geq \len(\alpha) - m_\alpha$, where $m_\alpha$ denotes the number of bad snippets of $\alpha$.
\end{lem}

\begin{proof}
We begin by splitting the arc $\alpha$ into maximal subarcs such that each subarc either consists of snippets that are all contained in at least one blocker or that no snippet is contained in any blocker. Then, the reduced corner length of $\alpha$ equals the sum of the reduced corner lengths of these subarcs. Furthermore, the reduced corner length of subarcs not containing any blockers is equal to their corner length. Lemma \ref{Snippet vs corner length} implies that $\len_\textrm{corn}(\alpha) \geq \len(\alpha)-m_\alpha$. Hence, it remains to prove that the reduced corner length provides an upper bound for subarcs of the first type, which follows from Lemma \ref{SnippetSequenceIOnly}.
\end{proof}

\section{Local homotopies}

Suppose that $\alpha \subset S$ is a snippet-decomposed arc or curve. As seen in the previous chapter, arcs or curves consisting of more than one snippet may contain bad snippets of up to fourteen different types. They all have one characteristic in common: they cut off a region of positive index. In the following, we define a family of local homotopies that can be applied to (neighbourhoods of) such bad snippets. In the course of this thesis, we will see that those are sufficient to achieve efficient position for arcs and curves.

\begin{defn}
Suppose that $\alpha \subset S$ is a snippet-decomposed arc or curve consisting of at least two snippets, of which one, $a=\alpha[k]$, is a bad snippet. If $\alpha$ is an arc, we furthermore suppose that $0 < k< \len(\alpha)-1$. As $a$ is a bad snippet meeting $\partial \mathcal{R}- \partial S$, it lies inside a region $R \in \mathcal{R}$ and cuts off a trigon or bigon. Let $N(T)$ be a small regular neighbourhood of $T$ in $S$. We assume that $\alpha$ meets $\partial N(T)$ perpendicularly. Then a regular neighbourhood of $a \subset \alpha$, in the following called $N(a)$, divides $N(T)$ into two regions, exactly one of which, $A$, contains $T$. Replacing the subarc $N(a)\cap \partial A$ of $\alpha$ by the arc $\partial A -  N(a)$ and smoothing out the resulting two corners yields a smooth arc or curve $\alpha'=:\Hom (\alpha, k)$ homotopic to $\alpha$. We refer to this as \textit{applying a local homotopy to $\alpha$ at $a=\alpha[k]$}. If $T\subset R$ is a trigon or a bigon, we refer to this as \textit{applying a local trigon or bigon homotopy to $\alpha$ at $\alpha[k]$} respectively.
\end{defn}

For examples of local homotopies we refer the reader to Figures \ref{Branch trigon no marked point}-\ref{Homotopy of type R(h,v) picture} and Figures \ref{Homotopies of type B(S,t,t,0)}-\ref{Homotopy of type B(B,h,h) picture}.

\begin{rem}\label{local trigon homotopy fixes boundary}
Suppose that $\alpha[k]$ is a bad snippet of a snippet-decomposed arc or curve $\alpha$. Applying a local homotopy to $\alpha$ at $\alpha[k]$ only alters a small neighbourhood of the snippet $\alpha[k]$ in $\alpha$. Without loss of generality, we may always assume that this neighbourhood is contained in the subarc $\alpha[k-1/3:k+4/3]$. For a reminder on the latter notation we refer the reader to page \pageref{Page: alpha subsnippet}. If $\len(\alpha)>2$, this implies that the boundary of the subarc $\alpha[k-1:k+2]$ is fixed, and any local homotopy affects at most these three snippets.
\end{rem}

In Section \ref{Section:Trigon homotopies}, we study the effects of trigon homotopies on the number and types of snippets of the underlying arc or curve. We begin with a series of general observations and then draw our conclusions for each of the five trigon types separately.

%%%%%%%%%%%%%%%%%%%%%%%%%%%%%%%%%%%%%%%%%%%%%%%%%%%%%%%%%%%%%%%%%%%%%%%%%%%%%%%%%%%%%%%%%%%%%%%%%%%%%%%%%%%%%%%%%%%%%%%%%%%%%%%%%%%%%%%%%%%%%%%%%%%%
\section{Trigon homotopies}\label{Section:Trigon homotopies}

We recall that $\partial \mathcal{R}$ is a trivalent graph whose edges meet at angles $\pi/2$ or $\pi$.

\begin{lem}\label{First ever observations trigons}
Suppose that $\alpha \subset S$ is a snippet-decomposed arc or curve satisfying $\len(\alpha)>2$. Suppose further that $\alpha_\textrm{trim}$ contains a trigon snippet $\alpha[k] \subset R \in \mathcal{R}$. That is, $\alpha[k]$ cuts off a trigon $T \subset R$ and $\partial T$ contains a unique corner $x$ of $R$. Set $b=\partial T - \alpha[k]$. Let $E$ be the set of (half-)edges of the trivalent graph $\partial^2 \mathcal{R}$ that meet $b$ but are not contained in $b$. Then the following statements hold.
\begin{enumerate}
\item $b-x$ consists of two components $b_1$ and $b_2$, where $b_1 \cap\partial^2 \mathcal{R} = \emptyset$.
\item Either $|E|=1$ or all (half-)edges of $E$ meet $b_2$ at angle $\pi/2$.
\item $\len(\Hom(\alpha,k))= \len(\alpha)-2+|\partial T \cap \partial^2 \mathcal{R}|$.
\end{enumerate}
\end{lem}

\begin{proof}
We prove the statements of the lemma in order.
First, recall that horizontal sides of branch and switch rectangles and vertical sides of complementary regions do not contain any points of $\partial^2 \mathcal{R}$ in their interiors. As vertical and horizontal boundary sides of any region alternate, one of the components $b_1$ and $b_2$ of $b-x$ must have empty intersection with $\partial^2 \mathcal{R}$. In the following, we assume that $b_1 \cap\partial^2 \mathcal{R} = \emptyset$.

Secondly, if $|E|>1$, that is, if $b_2$ has non-empty intersection with $\partial^2 \mathcal{R}$, $b_2$ must be a subarc of a vertical side of a switch rectangle or a subarc of a horizontal side of a complementary region. Set $e$ to be the edge of $E$ at the corner $x$ of the trigon. If $b_2$ is a subarc of a vertical side of a switch rectangle, then $e$ and $b_1$ are subarcs of $\partial_h N$ (see Figure \ref{Switch rectangle edge}). Thus, they meet at an angle of $\pi$ at $x$. As $b_1$ and $b_2$ meet at an angle of $\pi/2$ at $x$, this implies that $e$ and $b_2$ meet at an angle of $\pi/2$ at $x$ as well. If $b_2$ is a subarc of a horizontal side of a complementary region, then $e$ and $b_2$ are subarcs of a vertical boundary side of a switch rectangle (see Figure \ref{Complementary region edge}). Hence, they meet at an angle of $\pi$ at $x$ and $e$ and $b_2$ must therefore meet at an angle of $\pi/2$ at $x$. As $b_2$ is a smooth subarc, all edges of $E-e$ must be perpendicular to $b_2$, which proves the claim.

\begin{figure}[htbp]
\begin{minipage}[b]{0.49\linewidth}
\centering
\begin{tikzpicture}[scale=0.45]
\draw [very thick] (-1.5,1.5) rectangle (4,-3);
\node at (1.5,2.25) {\tiny{$h$}};
\node at (-2.25,-0.75) {\tiny{$t$}};
\node at (1.5,-3.5) {\tiny{$h$}};
\node at (4.75,-0.75) {\tiny{$v$}};
\draw [very thick, densely dashed, blue] plot[smooth, tension=.7] coordinates {(0.375,-3) (1,-1.125) (4,-0.75)};
\draw [very thick] (4,1.5) rectangle (6.5,0);
\draw [very thick] (4,-1.5) rectangle (6.5,-3);
\end{tikzpicture}
 \caption{A trigon inside a switch rectangle whose boundary contains more than one point of $\partial^2 \mathcal{R}$.} 
 \label{Switch rectangle edge}
    \vspace{2ex}
\end{minipage}
\begin{minipage}[b]{0.49\linewidth}
\centering
\begin{tikzpicture}[scale=0.45]
\draw [very thick](-4,-1.5) {} -- (-4,-3) {};
\draw [very thick] plot[smooth, tension=.7] coordinates { (-4,-3) (1,-2.5) (6,-3)};
\draw [very thick](6,-1.5) node (v3) {} -- (6,-3) node (v1) {};
\draw [very thick] plot[smooth, tension=.7] coordinates {(-4,-1.5) (-1,-0.5) (0,1)};
\draw [very thick] plot[smooth, tension=.7] coordinates {(v3) (3,-0.5) (2,1)};
\draw [very thick] plot[smooth, tension=.7] coordinates {(-4,-4)  (-2.4,-3.8)};
\node at (-2.25,-0.55) {{\tiny{$h$}}};
\draw [very thick](-2.4,-2.78) -- (-2.4,-5.4) -- (-0.6,-5.4) -- (-0.6,-2.58);
\draw [very thick](-0.6,-5.4)-- (0.9,-5.4)  -- (0.9,-2.5);
\draw [very thick](0.9,-5.4) -- (2.4,-5.4) -- (2.4,-2.56);
\draw [very thick](-2.6,-4.588) -- (-2.2,-4.588);
\draw [very thick](2.2,-3.686) -- (2.6,-3.686);
\draw [very thick](2.2,-4.588) -- (2.6,-4.588);
\node at (2.94,-4.16) {{\tiny{$v$}}};
\node at (2.86,-3.1) {{\tiny{$t$}}};
\node at (0.2,-5.9) {{\tiny{$h$}}};
\node at (6.5175,-2.3319) {{\tiny{$v$}}};
\draw [very thick](-4,-3) -- (-4,-4);
\draw [very thick, densely dashed, blue] plot[smooth, tension=.7] coordinates {(-4,-2.25) (0.8,-1.6) (1.7642,-2.5387)};
\draw [very thick] (-4,-0.5) rectangle (-5.5,-4);
\end{tikzpicture}
 \caption{A trigon inside a complementary region whose boundary contains more than one point of $\partial^2 \mathcal{R}$.} 
 \label{Complementary region edge}
    \vspace{2ex}
\end{minipage}
\end{figure}

Thirdly, let $\beta$ be the image of the subarc $\alpha[k-1/3,k+4/3]$ under the local homotopy applied to $\alpha[k]$ (see Figures \ref{General Trigon Example 1}-\ref{General Trigon Example 2}). For each point in $\partial T \cap \partial^2 \mathcal{R}$, $\beta$ intersects an edge of $\partial \mathcal{R}$ once. Thus, the arc $\alpha[k-1:k+2]$ is replaced by $\len(\alpha[k-1:k+2])-1+ |\partial T \cap \partial^2 \mathcal{R}|-1$ many snippets. This implies that $\len(\alpha')= \len(\alpha)-1+|\partial T \cap \partial^2 \mathcal{R}|-1=\len(\alpha)-2+|\partial T \cap \partial^2 \mathcal{R}|$.
\begin{figure}[htbp] 
\begin{minipage}[b]{0.49\linewidth}
    \centering
\begin{tikzpicture}[scale=0.7]
\draw [very thick](-3.5,2) -- (-3.5,0) -- (2,0);
\draw [thick, densely dotted] (-3.5,0) -- (-4.5,0);
\draw [very thick, densely dashed, blue] plot[smooth, tension=.7] coordinates {(-4.5,1) (0.5,0.75) (1.6,-1.2)};
\draw [very thick](-3.5,2) node (v1) {} -- (-1,2);
\draw [very thick, densely dashed, cyan] plot[smooth, tension=.4] coordinates {(-4.5,0.75) (-3.9944,0.592) (-3.5782,-0.6238) (0.8,-0.6) (1.2,-1.2)};
\draw [thick, densely dotted] (-3.5,0) -- (-3.5,-1);
\end{tikzpicture}
    \caption{A trigon that contains a unique point of $\partial ^2 \mathcal{R}$.}
    \label{General Trigon Example 1}
  \end{minipage}
\begin{minipage}[b]{0.49\linewidth}
    \centering
\begin{tikzpicture}[scale=0.7]
\draw [very thick](-3.5,2) -- (-3.5,0) -- (2,0);
\draw [very thick](-2.5,0) -- (-2.5,-1);
\draw [very thick](-1.5,0) -- (-1.5,-1);
\draw [very thick](-0.5,0) -- (-0.5,-1);
\draw [very thick](0.5,0) -- (0.5,-1);
\draw [very thick] (-3.5,0) -- (-3.5,-1);
\draw [very thick, densely dashed, blue] plot[smooth, tension=.7] coordinates {(-4.5,1) (0.5,0.75) (1.6,-1.2)};
\draw [very thick](-3.5,2) node (v1) {} -- (-1,2);
\draw [very thick, densely dashed, cyan] plot[smooth, tension=.4] coordinates {(-4.5,0.75) (-3.9944,0.592) (-3.5782,-0.6238) (0.8,-0.6) (1.2,-1.2)};
\draw [thick, densely dotted] (-3.5,0) -- (-3.5,-1);
\end{tikzpicture}
    \caption{A trigon that contains more than one point of $\partial ^2 \mathcal{R}$.}
    \label{General Trigon Example 2}
  \end{minipage}
\end{figure}
\end{proof}

\begin{defn}
Suppose that $\alpha \subset S$ is a snippet-decomposed arc or curve satisfying $\len(\alpha) \geq 2$. Suppose further that $\alpha_\textrm{trim}$ contains exactly one bad snippet $\alpha[k]$. If $\alpha[k]$ is a trigon or bigon snippet, we say that $\alpha$ is an \emph{almost efficient arc or curve of trigon or bigon type} respectively.
\end{defn}

Suppose that $\alpha \subset S$ is an almost efficient arc or curve of trigon type. The following lemma shows that applying a local homotopy at the unique trigon snippet yields an arc or curve $\alpha'$ which is in efficient position in its inside or is again an almost efficient arc or curve of trigon type. In the latter case, the trigon snippet of $\alpha'_\textrm{trim}$ turns the same way as the trigon snippet of $\alpha_\textrm{trim}$. In addition, the lemma determines the position of the trigon snippet of $\alpha'$ as well as gives restrictions on the types of snippets that the homotopy gives rise to.

\begin{lem} \label{General observations trigon homotopies}
Suppose that $\alpha \subset S$ is an almost efficient arc or curve of trigon type. Set $0 \leq k \leq \len(\alpha)-1$ such that $\alpha[k]$ is the unique bad snippet of $\alpha_\textrm{trim}$. We further set $\alpha'=\Hom(\alpha,k)$ as well as $\beta=\alpha[k-1:k+2]$ and $\beta'=\Hom(\beta,1)$. Then the following statements hold:
\begin{enumerate}
\item The subarc $\beta'[1:-1]$ is in efficient position. If $\alpha[k]$ is turning right, then all snippets of $\beta[1:-1]$ cut off a region of index zero on their left-hand side.
\item If $\len(\alpha)>2$, then either $\beta[0]$ is weakly snippet homotopic to $\beta'[0]$ or $\beta[-1]$ is weakly snippet homotopic to $\beta'[-1]$.
\item Suppose that $\alpha[k]$ is turning right and that $\beta[0]$ is embedded and cuts off a simply connected region on its right-hand side that is not a bigon. If $\beta[0]$ and $\beta'[0]$ are not weakly snippet homotopic, then the region cut off by $\beta'[0]$ on its right-hand side has one outward-pointing corner less than the region cut off by $\beta[0]$ on its right-hand side.
\item Suppose that $\beta[0]$ lies inside a peripheral complementary region $R \in \mathcal{R}_\textrm{comp}$. If $\beta[0]$ is not weakly snippet homotopic to $\beta'[0]$, then \[\wind(\beta[0]) = \wind(\beta'[0]) \pm 1.\] Else, $\wind(\beta[0]) =\wind(\beta'[0])$.
\item $\alpha'_\textrm{trim}$ contains at most one bad snippet, which must be a trigon turning into the same direction as $\alpha[k]$.
\item Suppose that $\alpha'_\textrm{trim}$ is not in efficient position and that $\alpha[k]$ is turning right. If $\alpha'_\textrm{trim}$ contains a trigon snippet of type $\mathbb{R}(h,v)$, then $\alpha'_\textrm{trim}$ contains one right dual less than $\alpha_\textrm{trim}$. 
\end{enumerate}
\end{lem}

\begin{proof}
We prove the statements of the lemma in order. Without loss of generality, we assume that $\alpha[k] \subset R \in \mathcal{R}$ is a right-turning trigon snippet. By $T$ we denote the trigon cut off by $\alpha[k]$. As discussed before, $b=\partial T - \alpha[k]$ contains a unique corner $x$ of $R$ and $b-x$ consists of two components $b_1$, $b_2$, where we assume that $b_1 \cap \partial^2 R = \emptyset$. Let $E$ be the set of (half-)edges of the trivalent graph $\partial^2 \mathcal{R}$ that meet $b$ but are not contained in $b$.

First, we want to show that $\beta'[1:-1]$ is in efficient position. If $|E|=1$, then $\beta'[1:-1]$ is empty. Else, its snippets are parallel to $b_2 \subset \partial R$ (see page \pageref{General Trigon Example 2}, Figure \ref{General Trigon Example 2}). As $\alpha[k]$ is cutting off a trigon of $R$ on its right-hand side, the snippets of $\beta[1:-1]$ cut off rectangles on their left-hand sides.

Secondly, assume that $\len(\alpha)>2$. Thus, we have that $\alpha[k-1] \neq \alpha[k+1]$. Without loss of generality, we may assume that $\beta[0]$ is the snippet that intersects $\alpha[k]$ in $\overline{b_1}$. Let $e \in E$ be the edge of $\partial \mathcal{R}$ that is adjacent to the corner $x$. If $|E|>1$, then $e$ meets $b_1$ at an angle of $\pi$ (see Figure \ref{General Trigon Example 2}). Thus, $x$ is not a corner of the region containing the snippet $\beta[0]$ and $\beta[0]$ is weakly snippet homotopic to $\beta'[0]$. However, as $e$ meets $b_1$ at an angle of $\pi/2$, $\beta[-1]$ and $\beta'[-1]$ are not weakly snippet homotopic.
If $|E|=1$, then either $e$ meets $b_1$ at an angle of $\pi$ or $e$ meets $b_2$ at an angle of $\pi$. Hence, $x$ is a corner of either the region containing $\beta[0]$ or the region containing $\beta[-1]$, which determines whether $\beta[0]$ and $\beta'[0]$, or $\beta[-1]$ and $\beta'[-1]$, are weakly snippet homotopic.

Thirdly, assume that $\beta[0]$ is embedded and cuts off a simply connected region on its right-hand side that is not a bigon. Suppose further that $\beta[0]$ and $\beta'[0]$ are not weakly snippet homotopic. Then $\beta'[0]$ cuts off a region on its right-hand side that has one outward-pointing corner less than the region cut off by $\beta[0]$, namely the corner $x$. Thus, if $\beta[0]$ is in efficient position, then $\beta'[0]$ is in efficient position or cuts off a trigon on its-right hand side. 

Fourthly, assume that $\beta[0]$ lies inside a peripheral complementary region $R \in \mathcal{R}_\textrm{com}$. If $\beta[0]$ is not weakly snippet homotopic, then $\beta[0](1)$ is moved ``across'' the corner $x$ under the homotopy. This implies that $\wind(\beta[0]) = \wind(\beta'[0]) \pm 1$. Else, that is if $\beta[0]$ and $\beta'[0]$ are weakly snippet homotopic, then their winding numbers coincide by definition of the winding number.

To prove the fifth claim we first assume that $\len(\alpha)>2$ and that $\beta[-1]$ and $\beta'[-1]$ are weakly snippet homotopic. Thus, if $\beta[0] \subset \alpha_\textrm{trim}$, we know that $\beta[0]$ is in efficient position, that $\beta'[0] \subset \alpha'_\textrm{trim}$ and that $\beta'[0]$ is in efficient position or cuts off a trigon on its-right hand side. If $\beta[-1] \subset \alpha_\textrm{trim}$, then Lemma \ref{Weakly snippet homotopic implies same type} implies that $\beta'[-1]$ is in efficient position as $\beta[-1]$ and $\beta'[-1]$ are weakly snippet homotopic and $\beta[-1]$ is in efficient position. Since all snippets of $\beta'[1:-1]$ are in efficient position, this implies that $\alpha'_\textrm{trim}$ is in efficient position or contains exactly one bad snippet which must be a right-turning trigon snippet.
If $\len(\alpha)=2$, then $\alpha[k-1]$ is in efficient position and $\alpha'[k-1]$ is in efficient position or a right-turning trigon snippet. Since all other snippets of $\alpha'_\textrm{trim}$ are parallel to $\partial \mathcal{R}$, they must be in efficient position. Thus, $\alpha'$ is in efficient position or contains a unique right-turning trigon snippet $\alpha'[k-1]$.

The sixth claim follows from statements $3$ and $4$ and the fact that each corner of a region contributes $-1/4$ towards the index of that region. So, if $\alpha'_\textrm{trim}$ contains a right-turning trigon snippet inside a complementary region, then $\alpha[k-1]$ or $\alpha[k+1]$ must lie inside a complementary region and cut off a right dual.
\end{proof}

\begin{rem}
We note that we obtain the equivalent statements of Lemma \ref{General observations trigon homotopies} for left-turning trigons by replacing every occurrence of the word ``right'' by the word ``left''.
\end{rem}

\begin{rem}\label{Required time for trigon homotopies}
We recall that any properly immersed arc or curve $\alpha \subset S$ is uniquely determined by its cutting sequence and the winding number of the respective snippets. For the remainder of this thesis we assume that arithmetic on winding numbers can be done in constant time. We claim that this implies that trigon homotopies can be computed in $\mathit{O}(|\chi(S)|)$ time: Suppose that $\alpha \subset S$ is a snippet-decomposed arc or curve that contains a snippet $\alpha[k]$ of trigon type in its inside. Let $T$ be the trigon cut off by the snippet $\alpha[k]$. Lemma \ref{First ever observations trigons} implies that the snippets of the subarc $\alpha[k-1:k+2]$ are replaced by $(|\partial T \cap \partial^2 \mathcal{R}|+1) \leq s $ many snippets. Remark \ref{local trigon homotopy fixes boundary} implies that all other snippets of $\alpha$ remain unchanged. Following the conventions agreed upon in Remark \ref{Standard conventions algor}, this implies that the cutting sequence can be adjusted in $\mathit{O}(|\chi(S)|)$ time. Lemma \ref{General observations trigon homotopies} implies that the winding number of the first and last snippet of $\alpha[k-1:k+2]$ change by at most one. Furthermore, if a snippet of $\Hom(\alpha[k-1:k+2],1)[1:-1]$ lies inside a peripheral annulus region, it is a horizontal or vertical dual so the modulus of its winding number equals two. As $\chi(S)\neq 0$ and $s=\mathit{O}(|\chi(S)|)$, this implies that the winding numbers of $\Hom(\alpha,k)$ can be computed in $\mathit{O}(|\chi(S)|)$ time. Thus, $\Hom(\alpha,k)$ can be computed in $\mathit{O}(|\chi(S)|)$ time.
\end{rem}

For each of the five trigon types we are now going to draw our conclusion from the previous lemma, putting special emphasis on the occurring types of snippets and their length. We remark that this is a mere application of the results in Lemma \ref{General observations trigon homotopies}. However, stating the results for the single trigon types separately help us to understand the implications that repeated applications of local homotopies have on an almost efficient arc or curve of trigon type.
For the remainder of this thesis, we employ the following notation \label{Page: carried and dual}
\begin{itemize}
\item $\carr(\alpha)$ denotes the number of carried snippets of $\alpha$.
\item $\dual_R(\alpha)$ denotes the number of right duals of $\alpha$.
\item $\dual_L(\alpha)$ denotes the number of left duals of $\alpha$.
\end{itemize}
We note that all these quantities are bounded from above by $\len(\alpha)$ and $\len_\textrm{corn}(\alpha)$.

\subsection{Trigons of type $\mathbb{B}(h,t)$}

\begin{lem} \label{Homotopy of type B(h,t,0)}
Suppose that $\alpha \subset S$ is an almost efficient arc or curve of trigon type. Set $0 \leq k \leq \len(\alpha)-1$ such that $\alpha[k]$ is the unique bad snippet of $\alpha_\textrm{trim}$. Set $\alpha'=\Hom(\alpha,k)$. If $\alpha[k]$ is a trigon snippet of type $\mathbb{B}(h,t)$, then the following statements hold.
\begin{enumerate}
\item $\len(\alpha')= \len(\alpha)-1$.
\item $\alpha'_\textrm{trim}$ contains at most one bad snippet. This bad snippet is a trigon snippet inside a switch rectangle or complementary region and turns the same way as $\alpha[k]$.
\item If $\alpha'_\textrm{trim}$ contains a trigon snippet inside a switch rectangle, then $\carr(\alpha'_\textrm{trim})=\carr(\alpha_\textrm{trim})-1$ and $\alpha'_\textrm{trim}$ contains the same number of right and left duals as $\alpha_\textrm{trim}$.
\item If $\alpha'_\textrm{trim}$ contains a right-turning trigon of type $\mathbb{R}(h,v)$, then $\dual_R(\alpha'_\textrm{trim})=\dual_R(\alpha_\textrm{trim})-1$ and $\carr(\alpha'_\textrm{trim})=\carr(\alpha_\textrm{trim})$.
\item If $\alpha'_\textrm{trim}$ is not in efficient position, then $\len_\textrm{red}(\alpha'_\textrm{trim})\leq \len_\textrm{red}(\alpha_\textrm{trim})$.
\item If $\alpha'_\textrm{trim}$ is in efficient position, then $\len_\textrm{red}(\alpha'_\textrm{trim})\leq \len_\textrm{red}(\alpha_\textrm{trim})+2s$.
\end{enumerate}
\end{lem}

\begin{proof}
We prove the statements of the lemma in order.
First, suppose that $T$ is the trigon cut off by $\alpha[k]$. As $\alpha[k]$ is of type $\mathbb{B}(h,t)$, we know that $|\partial T \cap \partial^2 \mathcal{R}|=1$. Hence, Lemma \ref{First ever observations trigons} implies that $\len(\alpha')=\len(\alpha)-2+1=\len(\alpha)-1$, giving $\textit{1}$.

For the remainder of this proof we assume that $\alpha[k-1]$ and $\alpha[k+1]$ lie inside a switch rectangle and complementary region respectively. Thus, $\alpha'[k-1]$ and $\alpha'[k]$ lie inside a switch rectangle and complementary region respectively. Hence, if one of these is a trigon snippet, it must be of type $\mathbb{S}(h,t,1)$, $\mathbb{S}(h,v,2)$, $\mathbb{S}(h,t,3)$ or $\mathbb{R}(h,v)$. Combining this analysis with Lemma \ref{General observations trigon homotopies} gives \textit{2}.

Thirdly, if $\alpha'_\textrm{trim}$ contains a bad snippet of type $\mathbb{S}(h,t,1)$, $\mathbb{S}(h,v,2)$, or $\mathbb{S}(h,t,3)$, then this must be the snippet $\alpha'[k-1]$. Hence, $\alpha[k-1] \subset \alpha_\textrm{trim}$ and $\alpha[k-1]$ must be carried. As no other carried snippets of $\alpha$ are affected by the homotopy, we see that $\carr(\alpha'_\textrm{trim})=\carr(\alpha_\textrm{trim})-1$. Since $\alpha[k+1]$ and $\alpha'[k]$ must be weakly snippet homotopic in this case, $\alpha'_\textrm{trim}$ contains the same number of right and left duals as $\alpha_\textrm{trim}$, giving \textit{3}.

Fourthly, if $\alpha'_\textrm{trim}$ contains a right-turning trigon of type $\mathbb{R}(h,v)$, this must be the snippet $\alpha'[k]$. Hence, $\alpha[k+1]$ lies in $\alpha_\textrm{trim}$ and must be a right dual. As no duals of $\alpha$ but $\alpha[k+1]$ are affected by the homotopy, this implies that \[\dual_R(\alpha'_\textrm{trim})=\dual_R(\alpha_\textrm{trim})-1.\] Since $\alpha[k-1]$ and $\alpha'[k-1]$ must be weakly snippet homotopic in this case, we know that $\carr(\alpha'_\textrm{trim})=\carr(\alpha_\textrm{trim})$, giving \textit{4}.

\begin{figure}[htbp]
\begin{minipage}[b]{0.99\linewidth}
    \centering
\begin{tikzpicture}[scale=0.35]
\draw [very thick] (-3,1.5) rectangle (4,-3);
\node at (-3.75,-0.75) {\tiny{$t$}};
\node at (0.5,-3.75) {\tiny{$h$}};
\draw [very thick] (4,-3) rectangle (11,1.5);
\draw [very thick](10.5,-1.5) -- (11.5,-1.5);
\draw [very thick, densely dashed, blue] plot[smooth, tension=.7] coordinates {(-1,3.5) (0.7268,-0.5327) (8.0494,-1.305) };
\node at (-0.5493,-1.6153) {\tiny{$\alpha[k]$}};
\draw [very thick](10.5,0) -- (11.5,0);
\node at (-2.853,2.5984) {\tiny{$\alpha[k+1]$}};
\node at (6.9936,-2.1505) {\tiny{$\alpha[k-1]$}};
\draw [very thick, dashed, cyan] plot[smooth, tension=.4] coordinates {(-0.5101,3.4581) (0.0455,1.9173) (4.3848,1.6971)   (4.6889,-0.7451) (8.0116,-0.9756) };
\end{tikzpicture}
    \caption{A homotopy of type $\mathbb{B}(h,t)$ applied to a trigon snippet that intersects the large end of a switch rectangle.} 
    \label{Branch trigon no marked point}
    \vspace{3 ex}
  \end{minipage}
  
  \begin{minipage}[b]{0.49\linewidth}
    \centering
\begin{tikzpicture}[scale=0.3]
\draw [very thick] (-3,1.5) rectangle (4,-3);
\node at (12,-4.5) {\tiny{$t$}};
\draw [very thick] (4,1.5) rectangle (11,-10);
\draw [very thick](3.5,-6) -- (4.5,-6);
\draw [very thick, densely dashed, blue] plot[smooth, tension=.7] coordinates {(-1,3) (0.5,-1.5) (8.5,-2.5) };
\node at (-1.31,-1.795) {\tiny{$\alpha[k]$}};
\draw [very thick, dashed, cyan] plot[smooth, tension=.4] coordinates {(-0.55,2.985) (-0.275,2.315)  (4.3,1.85)  (4.81,-1.55) (8.5,-2) };
\node at (7.925,-4.185) {\tiny{$\alpha[k-1]$}};
\node at (-3.62,2.53) {\tiny{$\alpha[k+1]$}};
\node at (7.235,2.23) {\tiny{$h$}};
\end{tikzpicture}
    \caption{A homotopy of type $\mathbb{B}(h,t)$ applied to a trigon snippet that intersects a small end of a switch rectangle such that the trigon does not contain any corners of the adjacent complementary region in its boundary.} 
    \label{Branch trigon with no marked point but marked tie}
  \end{minipage}
\begin{minipage}[b]{0.49\linewidth}
    \centering
\begin{tikzpicture}[scale=0.3]
\draw [very thick] (-3,1.5) rectangle (4,-3);
\node at (12,3) {\tiny{$t$}};
\draw [very thick] (4,-3) rectangle (11,8.5);
\draw [very thick](3.5,4.5) -- (4.5,4.5);
\draw [very thick, densely dashed, blue] plot[smooth, tension=.7] coordinates {(-1,3.5) (0.77,-0.075) (8.1,-0.665) };
\node at (-0.725,-1.51) {\tiny{$\alpha[k]$}};
\node at (7.045,9.225) {\tiny{$h$}};
\draw [very thick, dashed, cyan] plot[smooth, tension=.4] coordinates {(-0.655,3.44) (-0.32,2.2) (4.38,1.985) (4.77,-0.11) (8.105,-0.18)};
\node at (-3.36,3.06) {\tiny{$\alpha[k+1]$}};
\node at (7.605,-2.14) {\tiny{$\alpha[k-1]$}};
\end{tikzpicture}
    \caption{A homotopy of type $\mathbb{B}(h,t)$ applied to a trigon snippet that intersects a small end of a switch rectangle such that the trigon contains a corner of the adjacent complementary region in its boundary.}
    \label{Branch trigon with marked point}
  \end{minipage}
\end{figure}

We prove the remaining two statements at the same time.
Since $\alpha[k-1]$ and hence $\Hom(\alpha[k-1:k+2],1)[0]$ lie inside a switch rectangle, we know that $\len_\textrm{corn}(\alpha'[k-1])=\len_\textrm{corn}(\alpha[k-1])$. We further know that $\len_\textrm{corn}(\alpha[k])=1$. 

If $\alpha[k+1] \not\subset \alpha_\textrm{trim}$, this implies that $\len_\textrm{corn}(\alpha'_\textrm{trim}) \leq \len_\textrm{corn}(\alpha_\textrm{trim})-1$. As no blockers of $\alpha_\textrm{trim}$ are affected by the homotopy in this case, we see that $\len_\textrm{red}(\alpha'_\textrm{trim})\leq \len_\textrm{red}(\alpha_\textrm{trim})-1$.

For the remainder of this proof we assume that $\alpha[k+1] \subset \alpha_\textrm{trim}$. We have to distinguish two cases: either $\alpha[k+1]$ is weakly snippet homotopic to $\alpha'[k]$ or it is not.
If $\alpha[k+1]$ is weakly snippet homotopic to $\alpha'[k]$, then either \[\len_\textrm{corn}(\alpha'[k]) = \len_\textrm{corn}(\alpha[k+1])=2s\] or \[\len_\textrm{corn}(\alpha'[k]) \leq \len_\textrm{corn}(\alpha[k+1])+1.\] The last inequality follows from the fact that the boundary of the region cut off by $\alpha'[k]$ contains at most one horizontal side of a branch rectangle more than the boundary of the region cut off by $\alpha[k+1]$. Since $\alpha[k+1]$ is weakly snippet homotopic to $\alpha'[k]$, the number of blockers of $\alpha_\textrm{trim}$ and $\alpha'_\textrm{trim}$ coincides, so $\len_\textrm{red}(\alpha'_\textrm{trim})\leq \len_\textrm{red}(\alpha_\textrm{trim})$.

If $\alpha[k+1]$ is not weakly snippet homotopic to $\alpha'[k]$, then we have to distinguish the following two cases: either $\len_\textrm{corn}(\alpha'[k])=2s$ or $\len_\textrm{corn}(\alpha'[k])<2s$. 
In the first case we know that $\alpha'_\textrm{trim}$ must be in efficient position. As $\alpha[k+1]$ is in efficient position, we further know that $\len_\textrm{corn}(\alpha[k+1])\geq 1$. This implies that \[\len_\textrm{corn}(\alpha'[k])\leq \len_\textrm{corn}(\alpha[k+1])+2s-1.\] Thus, we see that $\len_\textrm{corn}(\alpha'_\textrm{trim}) \leq \len_\textrm{corn}(\alpha_\textrm{trim})+2s-2.$ As the homotopy decreases the number of blockers by at most one, this implies that \[\len_\textrm{red}(\alpha'_\textrm{trim})\leq \len_\textrm{red}(\alpha_\textrm{trim})+2s.\]

In the second case, that is if $\len_\textrm{corn}(\alpha'[k])<2s$, an index-argument shows that $\len_\textrm{corn}(\alpha[k+1])=2s$. Thus, we see that $\len_\textrm{corn}(\alpha'_\textrm{trim}) \leq \len_\textrm{corn}(\alpha_\textrm{trim})-2.$ As the homotopy decreases the number of blockers again by at most one, this implies that $\len_\textrm{red}(\alpha'_\textrm{trim})\leq \len_\textrm{red}(\alpha_\textrm{trim})$, which concludes the proof of the lemma.
\end{proof}

\subsection{Trigons of type $\mathbb{S}(h,t,1)$}

\begin{lem} \label{Homotopy of type S(h,t,0)}
Suppose that $\alpha \subset S$ is an almost efficient arc or curve of trigon type. Set $0 \leq k \leq \len(\alpha)-1$ such that $\alpha[k]$ is the unique bad snippet of $\alpha_\textrm{trim}$. Set $\alpha'=\Hom(\alpha,k)$. If $\alpha[k]$ is a trigon snippet of type $\mathbb{S}(h,t,1)$, then the following statements hold.
\begin{enumerate}
\item $\len(\alpha')= \len(\alpha)-1$.
\item $\alpha'_\textrm{trim}$ contains at most one bad snippet. This snippet is of type $\mathbb{B}(h,t)$ and turns the same way as $\alpha[k]$.
\item If $\alpha'_\textrm{trim}$ is not in efficient position, then $\carr(\alpha'_\textrm{trim})=\carr(\alpha_\textrm{trim})-1$ and $\alpha'_\textrm{trim}$ contains the same number of right and left duals as $\alpha_\textrm{trim}$.
\item $\len_\textrm{red}(\alpha'_\textrm{trim})\leq \len_\textrm{red}(\alpha_\textrm{trim})$.
\end{enumerate}
\end{lem}

\begin{proof}
We prove the statements of the lemma in order.
First, suppose that $T$ is the trigon cut off by $\alpha[k]$. As $\alpha[k]$ is of type $\mathbb{S}(h,t,1)$, we know that $|\partial T \cap \partial^2 \mathcal{R}|=1$. Hence, Lemma \ref{First ever observations trigons} implies that $\len(\alpha')=\len(\alpha)-2+1=\len(\alpha)-1$, giving $\textit{1}$.

For the remainder of this proof we assume that $\alpha[k-1]$ lies inside a branch rectangle. Denote by $R \in \mathcal{R}_\textrm{com}$ the complementary region that contains the snippet $\alpha[k+1]$. As $\alpha[k]$ is of type $\mathbb{S}(h,t,1)$, the corner of $\partial^2 \mathcal{R}$ that lies in the boundary of the trigon $T$ is not a corner of the complementary region $R$ (see Figures \ref{Homotopy of type S(h,t,0) tie with marked point}-\ref{Homotopy of type S(h,t,0) tie without marked point}). Thus, $\alpha[k+1]$ and $\alpha'[k]$ are weakly snippet homotopic. Lemma \ref{General observations trigon homotopies} then implies that if $\alpha'_\textrm{trim}$ contains a bad snippet, this must be the snippet $\alpha'[k-1]$, which lies inside a branch rectangle. This gives \textit{2}.

\begin{figure}[htbp] 
  \begin{minipage}[b]{0.49\linewidth}
    \centering
\begin{tikzpicture}[scale=0.3]
\draw [very thick] (-3,1.5) rectangle (4,-3);
\node at (-3.75,-0.75) {\tiny{$t$}};
\node at (0.5,2.525) {\tiny{$h$}};
\draw [very thick] (4,1.5) rectangle (11,-10);
\draw [very thick](3.5,-6) -- (4.5,-6);
\draw [very thick, densely dashed, blue] plot[smooth, tension=.7] coordinates {(8.75,3) (7.16,-0.585) (-0.495,-1.205) };
\node at (8.5,-3) {\tiny{$\alpha[k]$}};
\draw [very thick, dashed, cyan] plot[smooth, tension=.4] coordinates {(8.395,3.01) (8.08,2.07) (3.885,1.89) (3.47,-0.7) (-0.505,-0.87)};
\node at (0.6,-2.045) {\tiny{$\alpha[k-1]$}};
\node at (11.2,2.525){\tiny{$\alpha[k+1]$}};
\end{tikzpicture}
    \caption{A homotopy of type $\mathbb{S}(h,t,1)$ applied to a trigon snippet that intersects a small end of a switch rectangle.} 
    \label{Homotopy of type S(h,t,0) tie with marked point}
  \end{minipage}
\begin{minipage}[b]{0.49\linewidth}
    \centering
\begin{tikzpicture}[scale=0.35]
\draw [very thick, densely dashed, blue] plot[smooth, tension=.7] coordinates {(1.885,2.505) (0.125,-0.545) (-8.035,-1.04) };
\node at (1,-1.565) {\tiny{$\alpha[k]$}};
\draw [very thick] (-3,1.5) rectangle (4,-3);
\node at (-6.53,2.46) {\tiny{$h$}};
\draw [very thick](3.5,0) -- (4.5,0);
\draw [very thick](3.5,-1.5) -- (4.5,-1.5);
\node at (-11.005,-0.62) {\tiny{$t$}};
\draw [very thick] (-3,-3) rectangle (-10,1.5);
\draw [very thick, dashed, cyan] plot[smooth, tension=.4] coordinates {(1.48,2.53) (1.13,1.98) (-3.15,1.865) (-3.655,-0.565) (-8.015,-0.695)};
\node at (-6.935,-2.265) {\tiny{$\alpha[k-1]$}};
\node at (3.9,2.46){\tiny{$\alpha[k+1]$}};
\end{tikzpicture}
    \caption{A homotopy of type $\mathbb{S}(h,t,1)$ applied to a trigon snippet that intersects the large end of a switch rectangle.}
    \label{Homotopy of type S(h,t,0) tie without marked point}
  \end{minipage}
\end{figure}

If $\alpha'[k-1] \subset \alpha'_\textrm{trim}$, then we know that $\alpha[k-1]$ lies in $\alpha_\textrm{trim}$ and must therefore be carried. As no other carried snippets of $\alpha_\textrm{trim}$ are affected by the homotopy, we see that $\carr(\alpha'_\textrm{trim})=\carr(\alpha_\textrm{trim})-1$. Furthermore, since $\alpha[k+1]$ and $\alpha'[k]$ are weakly snippet homotopic and no other duals of $\alpha_\textrm{trim}$ are affected by the homotopy, $\alpha'_\textrm{trim}$ contains the same number of right and left duals as $\alpha_\textrm{trim}$, giving \textit{3}.

Lastly, we note that $\len_\textrm{corn}(\alpha[k-1])=\len_\textrm{corn}(\alpha'[k-1])$ since both snippets lie inside a branch rectangle. As $\alpha[k]$ lies inside a switch rectangle, we know that $\len_\textrm{corn}(\alpha[k])=3$. Unless $\alpha[k+1]$ and $\alpha'[k]$ both have corner length $2s$, the boundary of the region cut off by $\alpha'[k]$ of $R$ either contains one side of a branch rectangle less or one side of a switch rectangle more than the boundary of the region cut off by $\alpha[k+1]$ of $R$. Thus, we see that $\len_\textrm{corn}(\alpha'[k])\leq \len_\textrm{corn}(\alpha[k+1])+3$, which implies that $\len_\textrm{corn}(\alpha'_\textrm{trim})\leq \len_\textrm{red}(\alpha_\textrm{corn})$. Since $\len_\textrm{block}(\alpha'_\textrm{trim})=\len_\textrm{block}(\alpha_\textrm{trim})$, statement \textit{4} follows.
\end{proof}

\subsection{Trigons of type $\mathbb{S}(h,v,2)$}

\begin{lem} \label{Homotopy of type T(S,h,v)}
Suppose that $\alpha \subset S$ is an almost efficient arc or curve of trigon type. Set $0 \leq k \leq \len(\alpha)-1$ such that $\alpha[k]$ is the unique bad snippet of $\alpha_\textrm{trim}$. Set $\alpha'=\Hom(\alpha,k)$. If $\alpha[k]$ is a trigon snippet of type $\mathbb{S}(h,v,2)$, then the following statements hold.
\begin{enumerate}
\item $\len(\alpha')= \len(\alpha)$.
\item $\alpha'_\textrm{trim}$ contains at most one bad snippet. This snippet is of type $\mathbb{R}(h,v)$ and turns the same way as $\alpha[k]$.
\item $\alpha'_\textrm{trim}$ contains a right-turning trigon of type $\mathbb{R}(h,v)$ if and only if \[\dual_R(\alpha'_\textrm{trim})=\dual_R(\alpha_\textrm{trim})-1.\]
\item If $\alpha'_\textrm{trim}$ is not in efficient position, then $\len_\textrm{red}(\alpha'_\textrm{trim})\leq \len_\textrm{red}(\alpha_\textrm{trim})$ and $\carr(\alpha'_\textrm{trim})=\carr(\alpha_\textrm{trim})$.
\item If $\alpha'_\textrm{trim}$ is in efficient position, then $\len_\textrm{red}(\alpha'_\textrm{trim})\leq \len_\textrm{red}(\alpha_\textrm{trim})+2s$.
\end{enumerate}
\end{lem}

\begin{proof}
We prove the statements of the lemma in order.
First, suppose that $T$ is the trigon cut off by $\alpha[k]$. As $\alpha[k]$ is of type $\mathbb{S}(h,v,2)$, we know that $|\partial T \cap \partial^2 \mathcal{R}|=2$. Hence, Lemma \ref{First ever observations trigons} implies that $\len(\alpha')=\len(\alpha)-2+2=\len(\alpha)$, giving $\textit{1}$.

For the remainder of this proof we assume that $\alpha[k-1]$ lies inside a complementary region $R \in \mathcal{R}_\textrm{com}$ and intersects $\alpha[k]$ in $\partial_v N$ (see Figure \ref{Homotopy of type T(S,h,v) picture}). Let $R \in \mathcal{R}_\textrm{com}$ be the complementary region that contains the snippet $\alpha[k+1]$. We note that it is possible that $R=R'$. This holds especially true if $\len(\alpha)=2$. Since $\alpha[k]$ is of type $\mathbb{S}
(h,v,2)$, the boundary of $T$ contains two corners of $\partial^2 \mathcal{R}$. One of them is neither a corner of $R$ nor $R'$, whereas the other is a corner of $R$. Thus, if $\alpha'_\textrm{trim}$ contains a bad snippet, this is the snippet $\alpha'[k-1]$. Since $\alpha'[k-1]$ lies inside a complementary region, Lemma \ref{General observations trigon homotopies} implies that $\alpha'[k-1]$ is of type $\mathbb{R}(h,v)$ and turns the same way as $\alpha[k]$, giving \textit{2}.

\begin{figure}[htbp] 
  \begin{minipage}[b]{0.39\linewidth}
    \centering
\begin{tikzpicture}[scale=0.25]
\draw [very thick] (-5,2) rectangle (4,-3);
\node at (-1.335,2.71) {\tiny{$h$}};
\node at (-6.05,-0.445) {\tiny{$t$}};
\draw [very thick] (4,2) rectangle (11,-12);
\draw [very thick, densely dashed, blue] plot[smooth, tension=.7] coordinates {(9.18,3.91) (7.23,-3.905) (-2.065,-5.25) };
\node at (8.105,-5.775) {\tiny{$\alpha[k]$}};
\draw [very thick] (4,-12) rectangle (-5,-8);
\draw [very thick, dashed, cyan] plot[smooth, tension=.4] coordinates {(8.595,3.88) (8.215,2.55) (3.68,2.25) (3.24,-3.88) (-2.14,-4.4)};
\node at (-0.25,-6.78) {\tiny{$\alpha[k-1]$}};
\node at (12.485,3.895){\tiny{$\alpha[k+1]$}};
\end{tikzpicture}
    \caption{A homotopy of type $\mathbb{S}(h,v,2)$.} 
    \label{Homotopy of type T(S,h,v) picture}
  \end{minipage}
\begin{minipage}[b]{0.59\linewidth}
    \centering
\begin{tikzpicture}[scale=0.55]
\draw [very thick](-4,-1.5) {} -- (-4,-3) {};
\draw [very thick] plot[smooth, tension=.7] coordinates { (-4,-3) (1,-2.5) (6,-3)};
\draw [very thick](6,-1.5) node (v3) {} -- (6,-3);
\draw [very thick] plot[smooth, tension=.7] coordinates {(-4,-1.5) (-1.5,0) (-1,2)};
\draw [very thick] plot[smooth, tension=.7] coordinates {(v3) (3.5,0) (3,2)};
\draw [very thick] (-4,-4) node (v1) {} rectangle (-6,-0.5);
\draw [very thick] plot[smooth, tension=.7] coordinates {(v1) (-2.1,-3.7) (-1,-3.6)};
\draw [very thick] plot[smooth, tension=.7] coordinates {(-4,-0.5) (-3,-0.1) (-2.1,0.6)};
\draw [very thick](-2.1,0.6) -- (-1.4,0.1);
\draw [very thick, densely dashed, blue] plot[smooth, tension=.7] coordinates {(-5.295,-5.025) (-4.275,-2.43) (1.285,-1.345)  (5.985,-2.165) };
\node at (-6.5,-2) {{\tiny{$t$}}};
\node at (-5,0) {{\tiny{$h$}}};
\node at (-5.12,-1.925) {\tiny{$\alpha[k]$}};
\node at (1.09,-0.37) {\tiny{$\alpha[k-1]$}};
\draw [very thick, densely dashed, cyan] plot[smooth, tension=.7] coordinates {(-5.03,-5.035) (-4.91,-4.605) (-3.625,-4.305) (-3.155,-2.37) (1.59,-1.655) (5.99,-2.46)};
\end{tikzpicture}
    \caption{A homotopy of type $\mathbb{S}(h,v,2)$ creating a trigon of type $\mathbb{R}(h,v)$.}
    \label{Homotopy of type T(S,h,v) causing trigon}
  \end{minipage}
\end{figure}

Let us suppose that $\alpha[k]$ turns right.
We already know that $\alpha'_\textrm{trim}$ is not in efficient position if and only if $\alpha[k-1] \subset \alpha_\textrm{trim}$ and $\alpha'[k-1]$ is a trigon of type $\mathbb{R}(h,v)$ (see Figure \ref{Homotopy of type T(S,h,v) causing trigon}). Lemma \ref{General observations trigon homotopies} implies that this is the case if and only if $\alpha[k-1] \subset \alpha_\textrm{trim}$ and $\alpha'_\textrm{trim}$ contains one right horizontal dual less than $\alpha_\textrm{trim}$, giving \textit{3}. We remark that this is the dual $\alpha[k-1]$.

No carried snippets are affected by the homotopy. Thus, we know that $\carr(\alpha'_\textrm{trim})=\carr(\alpha_\textrm{trim})$. By definition of the corner length, we further know that $\len_\textrm{corn}(\alpha[k])=3$ and $\len_\textrm{corn}(\alpha'[k])=1$. If $\alpha'[k-1]$ is a trigon snippet and $\len(\alpha)>2$, we see that $\len_\textrm{corn}(\alpha'[k-1])=\len_\textrm{corn}(\alpha[k-1])-1$. We note that $\len(\alpha)=2$ implies that $\alpha'[k-1]$ is in efficient position. This follows from the observations that both its endpoints lie on the horizontal boundary and the index is changed by at most $1/4$.
Thus, we see that $\len_\textrm{corn}(\alpha'_\textrm{trim})\leq \len_\textrm{red}(\alpha_\textrm{corn})$ if $\alpha'[k-1]$ is a trigon snippet. Since blockers do not contain horizontal duals, the number of blockers of $\alpha_\textrm{trim}$ and $\alpha'_\textrm{trim}$ coincides, giving \textit{4}.

If $\alpha'_\textrm{trim}$ is in efficient position, then either $\alpha'[k-1] \not\subset \alpha'_\textrm{trim}$ or $\alpha'[k-1] \subset \alpha'_\textrm{trim}$ and it is in efficient position. In the first case, the number of blockers of $\alpha_\textrm{trim}$ and $\alpha'_\textrm{trim}$ coincides and $\len_\textrm{red}(\alpha'_\textrm{trim})\leq \len_\textrm{red}(\alpha_\textrm{corn})+1$. In the second case, we see that \[\len_\textrm{corn}(\alpha'[k-1]) \leq 2s \leq 2s + \len_\textrm{corn}(\alpha[k-1])-1.\] This implies that $\len_\textrm{corn}(\alpha'_\textrm{trim})\leq \len_\textrm{corn}(\alpha_\textrm{trim})+2s$. The number of blockers can only grow under the homotopy, which gives \textit{5}.
\end{proof}

%%%%%%%%%%%%%%%%%%%%%%%%%%%%%%%%%%%%%%%%%%%%%%%%%%%%%%%%%%%%%%%%%%%%%%%%%%%%%%%%%%%%%%%%%%%%%%%%%%%%%%%%%%%%%%%%%%%%%%%%%%%%%%%%%%%%%%%%%%%%%%%%%%%%
\subsection{Trigons of type $\mathbb{S}(h,t,3)$}

\begin{lem} \label{Homotopy of type T(S,h,t,2)}
Suppose that $\alpha \subset S$ is an almost efficient arc or curve of trigon type. Set $0 \leq k \leq \len(\alpha)-1$ such that $\alpha[k]$ is the unique bad snippet of $\alpha_\textrm{trim}$. Set $\alpha'=\Hom(\alpha,k)$. If $\alpha[k]$ is a trigon snippet of type $\mathbb{S}(h,t,3)$, then the following statements hold.
\begin{enumerate}
\item $\len(\alpha')= \len(\alpha)+1$.
\item $\alpha'_\textrm{trim}$ contains at most one bad snippet. This snippet is of type $\mathbb{B}(h,t)$ and turns the same way as $\alpha[k]$.
\item If $\alpha'_\textrm{trim}$ is not in efficient position, then $\carr(\alpha'_\textrm{trim})=\carr(\alpha_\textrm{trim})-1$.
\item If $\alpha[k]$ is turning right, then $\dual_R(\alpha'_\textrm{trim})=\dual_R(\alpha_\textrm{trim})$.
\item $\len_\textrm{red}(\alpha'_\textrm{trim})\leq \len_\textrm{red}(\alpha_\textrm{trim})$.
\end{enumerate}
\end{lem}

\begin{proof}
We prove the statements of the lemma in order.
First, suppose that $T$ is the trigon cut off by $\alpha[k]$. As $\alpha[k]$ is of type $\mathbb{S}(h,t,3)$, we know that $|\partial T \cap \partial^2 \mathcal{R}|=3$. Hence, Lemma \ref{First ever observations trigons} implies that $\len(\alpha')=\len(\alpha)-2+3=\len(\alpha)+1$, giving $\textit{1}$.

For the remainder of this proof we assume that $\alpha[k-1]$ lies inside a branch rectangle. Denote by $R \in \mathcal{R}_\textrm{comp}$ the complementary region that contains the snippet $\alpha[k+1]$. As $\alpha[k]$ is of type $\mathbb{S}(h,t,3)$, the corner of $\partial^2 \mathcal{R}$ that lies in the boundary of the trigon $T$ is not a corner of the complementary region $R$ (see Figures \ref{Homotopy of type T(S,h,t,2) picture}). Thus, $\alpha[k+1]$ and $\alpha'[k+2]$ are weakly snippet homotopic. Lemma \ref{General observations trigon homotopies} then implies that if $\alpha'_\textrm{trim}$ contains a bad snippet, this must be the snippet $\alpha'[k-1]$, which lies inside a branch rectangle. This gives \textit{2}.

\begin{figure}[htbp] 
  \begin{minipage}[b]{0.99\linewidth}
    \centering
\begin{tikzpicture}[scale=0.25]
\draw [very thick] (-5,2) rectangle (4,-2.5);
\node at (-1.75,2.5) {\tiny{$h$}};
\node at (12,-4.5) {\tiny{$t$}};
\draw [very thick] (4,2) rectangle (11,-12);
\draw [very thick, densely dashed, blue] plot[smooth, tension=.7] coordinates {(9.355,3.81) (7.44,-7.71) (-1.645,-9.985) };
\node at (6.565,-4.135) {\tiny{$\alpha[k]$}};
\draw [very thick] (4,-12) rectangle (-5,-7.5);
\draw [very thick, dashed, cyan] plot[smooth, tension=.4] coordinates {(8.645,3.765) (8.31,2.84) (3.61,2.08) (3.17,-8.555) (-1.61,-9.49)};
\node at (-0.675,-11.12) {\tiny{$\alpha[k-1]$}};
\node at (12.515,3.795) {\tiny{$\alpha[k+1]$}};
\end{tikzpicture}
    \caption{A homotopy of type $\mathbb{S}(h,t,3)$.} 
    \label{Homotopy of type T(S,h,t,2) picture}
  \end{minipage}
\end{figure}
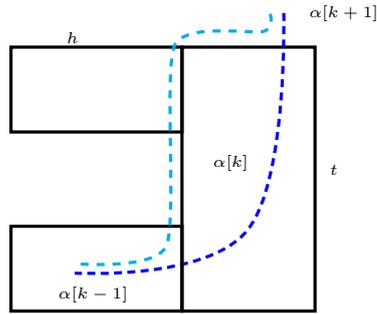

If $\alpha'[k-1] \subset \alpha'_\textrm{trim}$, then we know that $\alpha[k-1]$ lies in $\alpha_\textrm{trim}$ and must therefore be carried. As no other carried snippets of $\alpha_\textrm{trim}$ are affected by the homotopy, we see that $\carr(\alpha'_\textrm{trim})=\carr(\alpha_\textrm{trim})-1$, giving \textit{3}.

Suppose that $\alpha[k]$ turns right. We recall that $\alpha[k+1]$ and $\alpha'[k+2]$ are weakly snippet homotopic and that $\alpha'[k]$ must be a left vertical dual by Lemma \ref{General observations trigon homotopies}. As no further duals of $\alpha_\textrm{trim}$ are created or affected by the homotopy, $\alpha'_\textrm{trim}$ contains the same number of right duals as $\alpha_\textrm{trim}$, giving \textit{4}.

Lastly, we note that $\len_\textrm{corn}(\alpha[k-1])=\len_\textrm{corn}(\alpha'[k-1])$ since both snippets lie inside a branch rectangle. As $\alpha[k]$ lies inside a switch rectangle, we know that $\len_\textrm{corn}(\alpha[k])=3$. We further know that $\len_\textrm{corn}(\alpha'[k:k+2])=2$. Unless $\alpha[k+1]$ and $\alpha'[k]$ both have corner length $2s$, the boundary of the region cut off by $\alpha'[k]$ of $R$ either contains one side of a branch rectangle less or one side of a switch rectangle more than the boundary of the region cut off by $\alpha[k+1]$ of $R$. Thus, we see that $\len_\textrm{corn}(\alpha'[k+2])\leq \len_\textrm{corn}(\alpha[k+1])+3$. This implies that \[\len_\textrm{corn}(\alpha'_\textrm{trim})\leq \len_\textrm{red}(\alpha_\textrm{corn})+2.\] However, we note that $\len_\textrm{corn}(\alpha'[k+2])= \len_\textrm{corn}(\alpha[k+1])+3$ if and only if $\alpha'[k+2]$ is a left vertical dual. In this case, $\alpha'_\textrm{trim}$ contains one blocker more than $\alpha_\textrm{trim}$, which implies that $\len_\textrm{block}(\alpha'_\textrm{trim})=\len_\textrm{block}(\alpha_\textrm{trim})$.
Else, we know that $\len_\textrm{corn}(\alpha'[k])\leq \len_\textrm{corn}(\alpha[k+1])$. As the homotopy does not decrease the number of blockers, this implies that $\len_\textrm{block}(\alpha'_\textrm{trim})=\len_\textrm{block}(\alpha_\textrm{trim})$, giving \textit{5}.
\end{proof}

%%%%%%%%%%%%%%%%%%%%%%%%%%%%%%%%%%%%%%%%%%%%%%%%%%%%%%%%%%%%%%%%%%%%%%%%%%%%%%%%%%%%%%%%%%%%%%%%%%%%%%%%%%%%%%%%%%%%%%%%%%%%%%%%%%%%%%%%%%%%%%%%%%%%
\subsection{Trigons of type $\mathbb{R}(h,v)$}

\begin{lem} \label{Homotopy of type R(h,v)}
Suppose that $\alpha \subset S$ is an almost efficient arc or curve of trigon type. Set $0 \leq k \leq \len(\alpha)-1$ such that $\alpha[k]$ is the unique bad snippet of $\alpha_\textrm{trim}$. Set $\alpha'=\Hom(\alpha,k)$. If $\alpha[k]$ is a trigon snippet of type $\mathbb{R}(h,v)$, then the following statements hold.
\begin{enumerate}
\item $\len(\alpha') \leq \len(\alpha)+s_N-2$.
\item $\alpha'_\textrm{trim}$ contains at most one bad snippet. This snippet is of type $\mathbb{B}(h,t)$, $\mathbb{S}(h,t,1)$, or $\mathbb{S}(h,t,3)$ and turns the same way as $\alpha[k]$.
\item If $\alpha'_\textrm{trim}$ is not in efficient position, then $\carr(\alpha'_\textrm{trim})\leq \carr(\alpha_\textrm{trim})+s_N-1$.
\item The numbers of right and left duals of $\alpha_\textrm{trim}$ and $\alpha'_\textrm{trim}$ coincide.
\item $\len_\textrm{red}(\alpha'_\textrm{trim})\leq \len_\textrm{red}(\alpha_\textrm{trim})$.
\end{enumerate}
\end{lem}

\begin{proof}
We prove the statements of the lemma in order.
First, suppose that $T$ is the trigon cut off by $\alpha[k]$. As $\alpha[k]$ is of type $\mathbb{R}(h,v)$, we know that $|\partial T \cap \partial^2 \mathcal{R}|\leq s$. Hence, Lemma \ref{First ever observations trigons} implies that $\len(\alpha')\leq \len(\alpha)-2+s$, giving $\textit{1}$.

For the remainder of this proof we assume that $\alpha[k-1]$ lies inside a switch rectangle $R \in \mathcal{R}_\textrm{tie}$ and intersects $\alpha[k]$ in $\partial_v N$ (see Figure \ref{Homotopy of type R(h,v) picture}). We further assume that $\alpha[k]$ turns right. We remark that $\len(\alpha)>2$ as $\alpha[k-1]$ is in efficient position. Set $\beta=\alpha[k-1:k+2]$ and $\beta'=\Hom(\beta,1)$.
As the corner of $T$ that is also a point of $\partial^2 \mathcal{R}$ is no outward-pointing corner of $R$ we see that $\beta[0]$ and $\beta'[0]$ are weakly snippet homotopic. Lemma \ref{First ever observations trigons} implies the boundary of the region cut off by $\beta[-1]$ on its right-hand side contains exactly one more point of $\partial^2 \mathcal{R}$ than the boundary of the region cut off by $\beta'[-1]$ on its right-hand side. Since $\beta'[-1]$ lies inside a branch or switch rectangle, Lemma \ref{General observations trigon homotopies} therefore implies that $\beta'[-1]$ can only be of type $\mathbb{B}(h,t)$, $\mathbb{S}(h,t,1)$, or $\mathbb{S}(h,t,3)$ and if so, turns the same way as $\alpha[k]$, giving \textit{2}.

\begin{figure}[htbp] 
  \begin{minipage}[b]{0.99\linewidth}
    \centering
\begin{tikzpicture}[scale=0.55]
\draw [very thick](-4,-1.5) {} -- (-4,-4) {};
\draw [very thick] plot[smooth, tension=.7] coordinates { (-4,-4) (1,-3) (6,-4)};
\draw [very thick](6,-1.5) node (v3) {} -- (6,-4);
\draw [very thick] plot[smooth, tension=.7] coordinates {(-4,-1.5) (-1.5,0) (-1,1.5)};
\draw [very thick] plot[smooth, tension=.7] coordinates {(v3) (3.5,0) (3,1.5)};
\draw [very thick, blue, densely dashed] plot[smooth, tension=.7] coordinates {(-5.445,-2.895)  (3.165,-1.84) (3.54,-4.58)};
\draw [very thick] (-4,-0.25) rectangle (-6.83,-5.625);
\draw [very thick] plot[smooth, tension=.7] coordinates {(-4,-5.625) (-2.25,-5.075)};
\draw [very thick] (2,-5) -- (3.9,-5.375) -- (4.35,-3.55);
\node at (-5.4,-6.1) {{\tiny{$h$}}};
\node at (-7.41,-3) {{\tiny{$t$}}};
\draw [very thick, dashed, cyan] plot[smooth, tension=.6] coordinates {(-5.475,-3.22) (-4.3,-3.195) (-3.995,-4.51) (0,-3.6) (3.195,-3.745) (3.15,-4.525)};
\node at (-0.07,-1.305) {\tiny{$\alpha[k]$}};
\node at (-5.5,-2.16)  {\tiny{$\alpha[k-1]$}};
\node at (5.26,-4.74) {\tiny{$\alpha[k+1]$}};
\end{tikzpicture}
    \caption{A homotopy of type $\mathbb{R}(h,v)$.} 
    \label{Homotopy of type R(h,v) picture}
  \end{minipage}
\end{figure}
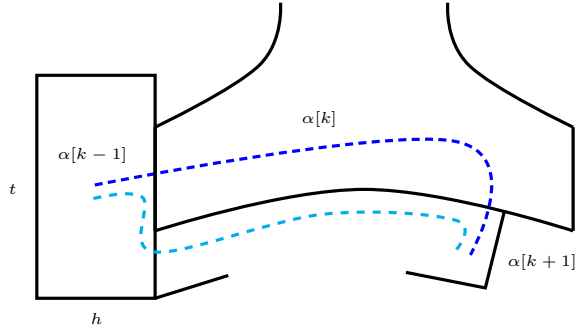

If $\beta[0] \subset \alpha_\textrm{trim}$, then $\beta[0]$ as well as $\beta'[0]$ are carried. Lemma \ref{General observations trigon homotopies} implies that all snippets of $\beta'[1:-1]$ are carried. As $\len(\beta'[1:-1])\leq s-1$ and $\beta'[-1]$ is not carried if it lies in $\alpha'_\textrm{trim}$, this implies that $\carr(\alpha'_\textrm{trim})\leq \carr(\alpha_\textrm{trim})+s_N-1$, giving \textit{3}.

Since no duals of $\alpha$ are affected by or created under the homotopy, statement \textit{4} follows.

Lastly, we note that $\len_\textrm{corn}(\beta[0])=\len_\textrm{corn}(\beta'[0])$ since both snippets lie inside a switch rectangle. By definition of the corner length of the snippet $\beta[1]$ we know that $\len_\textrm{corn}(\beta[1])=\len_\textrm{corn}(\beta'[1:-1])$. We further know that \[\len_\textrm{corn}(\beta[-1])=\len_\textrm{corn}(\beta'[-1])\] and that no duals of $\alpha$ are affected by or created under the homotopy, so statement \textit{5} follows.
\end{proof}

\subsection{Trigon homotopies - a summary}

We can summarise our findings of trigon homotopies in the following lemma:

\begin{lem}\label{Omnibus Trigons}
Suppose that $\alpha \subset S$ is an almost efficient arc or curve of trigon type. Set $0 \leq k \leq \len(\alpha)-1$ such that $\alpha[k]$ is the unique bad snippet of $\alpha_\textrm{trim}$. Set $\alpha'=\Hom(\alpha,k)$. If $\alpha[k]$ turns right, then the following statements hold.
\begin{enumerate}
\item $\alpha'_\textrm{trim}$ contains at most one bad snippet. This is a right-turning trigon.
\item If $\alpha'_\textrm{trim}$ contains a bad snippet, then $\len_\textrm{red}(\alpha'_\textrm{trim}) \leq \len_\textrm{red}(\alpha_\textrm{trim})$ and one of the following statements holds.
\begin{itemize}
\item $\dual_R(\alpha_\textrm{trim})=\dual_R(\alpha'_\textrm{trim})$ and $\carr(\alpha'_\textrm{trim}) < \carr(\alpha_\textrm{trim})$.
\item $\dual_R(\alpha'_\textrm{trim})< \dual_R(\alpha_\textrm{trim})$ and $\carr(\alpha'_\textrm{trim}) \leq \carr(\alpha_\textrm{trim})$.
\item $\alpha[k]$ is a trigon of type $\mathbb{R}(h,v)$. This implies that \\ $\dual_R(\alpha_\textrm{trim})=\dual_R(\alpha'_\textrm{trim})$ and $\carr(\alpha'_\textrm{trim}) \leq \carr(\alpha_\textrm{trim})+s -1$.
\end{itemize}
\item If $\alpha'_\textrm{trim}$ contains a trigon of type $\mathbb{R}(h,v)$, then $\dual_R(\alpha'_\textrm{trim})< \dual_R(\alpha_\textrm{trim})$.
\item If $\alpha'_\textrm{trim}$ is in efficient position, then $\len_\textrm{red}(\alpha'_\textrm{trim}) \leq \len_\textrm{red}(\alpha_\textrm{trim})+2s$.
\end{enumerate} 
\end{lem}

\begin{proof}
This is an immediate consequence of the Lemmas \ref{Homotopy of type B(h,t,0)}, \ref{Homotopy of type S(h,t,0)}, \ref{Homotopy of type T(S,h,v)}, \ref{Homotopy of type T(S,h,t,2)} and \ref{Homotopy of type R(h,v)}.
\end{proof}

\begin{rem}
We note that the statement of Lemma \ref{Omnibus Trigons} also holds for left-turning trigons if we replace all occurrences of right-turning trigons and duals by left turning trigons and duals.
\end{rem}

To make the content of Lemmas \ref{Homotopy of type B(h,t,0)}, \ref{Homotopy of type S(h,t,0)}, \ref{Homotopy of type T(S,h,v)}, \ref{Homotopy of type T(S,h,t,2)}, \ref{Homotopy of type R(h,v)} and \ref{Omnibus Trigons} more accessible, the directed graph in Figure \ref{Trigon homotopy graph} visualizes the connections between the different types of trigons. We note that we abbreviated $\carr(\alpha_\textrm{trim})$ by $c$, $\carr(\alpha'_\textrm{trim})$ by $c'$, $\dual_R(\alpha_\textrm{trim})$ by $\DR$, and $\dual_R(\alpha'_\textrm{trim})$ by $\DR'$. This graph displays the changes in the number of duals of $\alpha_\textrm{trim}$ in the case of a right-turning trigon snippet. Next to its five vertices corresponding to the five different types of trigon snippets, there should be a vertex corresponding to arcs that are in efficient position. For clarity, this has been omitted. Directed edges between a source and a target vertex indicate that under a trigon homotopy, trigon snippets of the source type might turn into trigon snippets of the target type. Situations in which neither the number of carried snippets nor the number of right duals decreases have been highlighted in red. We notice that these only occur for trigon snippets of type $\mathbb{R}(h,v)$ and that one only reaches the vertex labelled $\mathbb{R}(h,v)$ if one travels along an edge that indicates a decrease in right duals as highlighted in the second to last point of Lemma \ref{Omnibus Trigons}.

\begin{figure}[htbp] 
  \begin{minipage}[b]{0.99\linewidth}
    \centering
\begin{tikzpicture}[scale=0.55]
\draw [very thick] (-2,1) rectangle node{$\mathbb{B}(h,t)$} (2,-2);
\draw [very thick] (-5,-9) rectangle node{$\mathbb{S}(h,t,3)$} (3,-12);
\draw [very thick] (-10,0) rectangle node{$\mathbb{S}(h,t,1)$} (-14,-4);
\draw [very thick] (8,-1) rectangle node{$\mathbb{S}(h,v,2)$} (12,-4);
\draw [very thick] (-7,11) rectangle node{$\mathbb{R}(h,v)$} (3,8) node (v1) {};
\draw [very thick, ->](-7,8) --  node[sloped, below, red] {\small{$c' \leq c +s-1$}} node[sloped, above, red]  {\small{$\DR'=\DR$}} (-12,0);
\draw [very thick, ->](-3.5,8) -- node[sloped, near start, below, red] {\small{$c' \leq c +s -1$}} node[sloped, near start, above, red]  {\small{$\DR'=\DR$}}(-3.5,-9);
\draw [very thick, ->](-10,-1) -- node[sloped, below] {\small{$c'=c-1$}} (-2,0);
\draw [very thick, ->](-2,-2) -- node[sloped, below]  {\small{$c'=c-1$}} (-10,-3);
\draw [very thick, ->](-1,-2) -- node[sloped, below] {\small{$c'=c-1$}} (-1,-9);
\draw [very thick, ->](1,-9) -- node[sloped, below] {\small{$c'=c-1$}} (1,-2);
\draw [very thick, ->](2,-1) -- node[sloped, below]  {\small{$c'=c-1$}} (8,-2);
\draw [very thick, <-](3,8) -- node[sloped, below] {\small{$\DR'=\DR-1$}} (9,-1);
\draw [very thick, ->](-1,8) -- node[sloped, below, red] {\small{$c' \leq c +s-1$}} node[sloped, above, red]  {\small{$\DR'=\DR$}} (-1,1);
\draw [very thick, ->](1,1) -- node[sloped, below]  {\small{$\DR'=\DR-1$}} (1,8);
\end{tikzpicture}
 \caption{The trigon homotopy graph for right-turning trigons.}
 \label{Trigon homotopy graph}
    \vspace{6.2 ex}
  \end{minipage}
\end{figure}
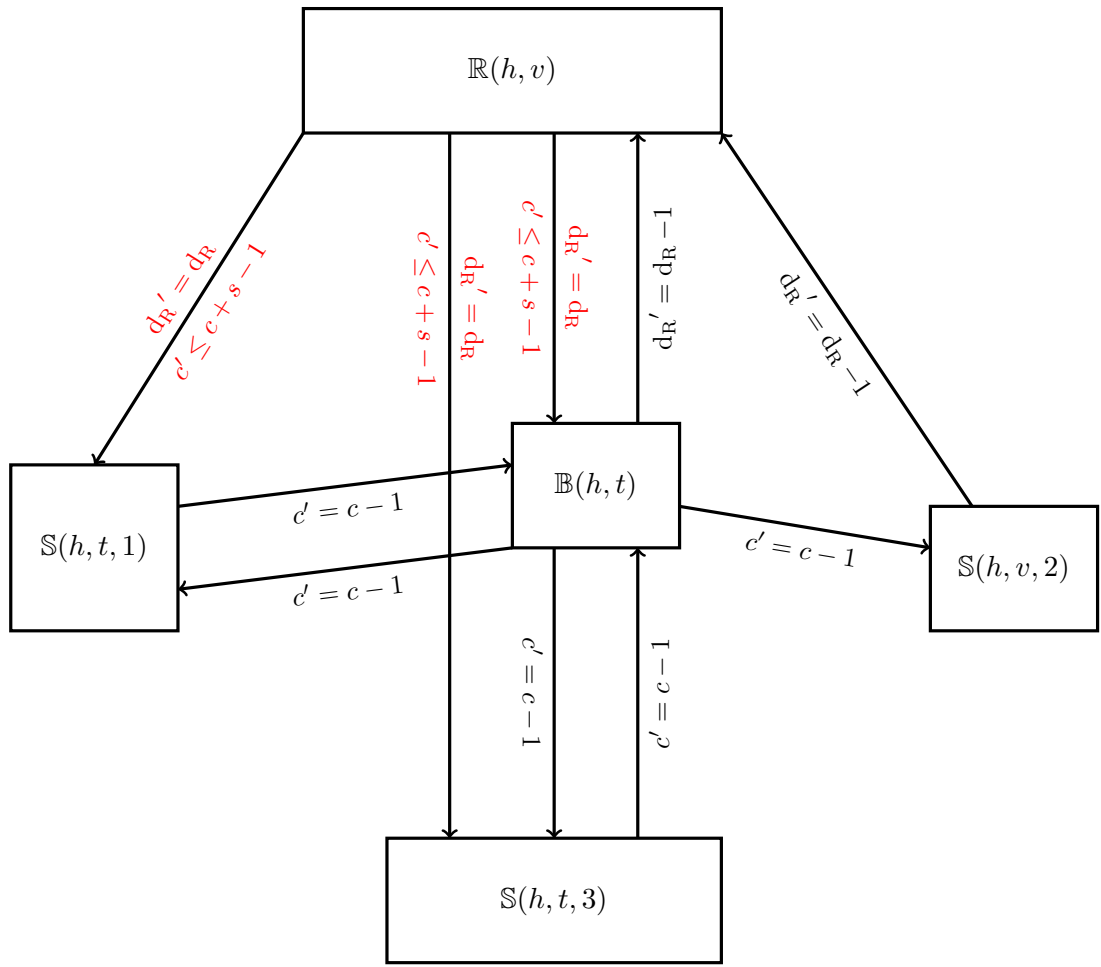

%%%%%%%%%%%%%%%%%%%%%%%%%%%%%%%%%%%%%%%%%%%%%%%%%%%%%%%%%%%%%%%%%%%%%%%%%%%%%%%%%%%%%%%%%%%%%%%%%%%%%%%%%%%%%%%%%%%%%%%%%%%%%%%%%%%%%%%%%%%%%%%%%%%%

\section{The algorithm \texttt{TrigArc}}

Building on the previous sections, we now present the algorithm \texttt{TrigArc}. This algorithm
\begin{itemize}
\item takes as input a surface $S$ of positive complexity, a tie neighbourhood $N \subset S$ of a large train track in $S$, as well as an arc $\alpha \subset S$ which is almost efficient of trigon type or does not contain any bad snippets in its inside, and
\item outputs an arc $\alpha'$ that is homotopic to $\alpha$ and does not contain any bad snippets in $\alpha'[1:-1]$.
\end{itemize}
Lemma \ref{Omnibus Trigons} shows that applying one trigon homotopy might not be sufficient to achieve efficient position for almost efficient arcs. The goal of this section is to prove that the process of repeatedly applying trigon homotopies terminates and yields an arc that is in efficient position in its inside. The formal statement of \texttt{TrigArc} is given in Algorithm \ref{TrigArc}.

%%%%%%%%%%%%%%%%%%%%%%% TrigArc $$$$$$$$$$$$$$$$$$$$$$$$$$$$$$$$

\begin{algorithm}
    \SetKwInOut{Input}{Input}
    \SetKwInOut{Output}{Output}
	\DontPrintSemicolon
%    \textbf{Function} \texttt{TrigArc}($\alpha$):\;
%    \BlankLine
    \Input{A surface $S$ of positive complexity, a tie neighbourhood $N \subset S$ of a large train track in $S$, and an arc $\alpha \subset S$ which is almost efficient of trigon type or does not contain any bad snippets in its inside.}
    \Output{An arc $\alpha'$ homotopic, relative its endpoints, to $\alpha$ such that $\alpha'_\textrm{trim}$ is in efficient position.}
    \BlankLine
    $\alpha'= \alpha$\;
    \While{there is some $k$ satisfying $0<k<\len(\alpha')-1$ such that $\alpha'[k]$ is a bad snippet}
    {$\alpha'=\Hom(\alpha',k)$}
    \Return $\alpha'$    
    \caption{\texttt{TrigArc} - Homotoping the inside of arcs
with at most one bad snippet into efficient position, where the bad snippet is of trigon type.}
    \label{TrigArc}
\end{algorithm}

\begin{lem} \label{TrigArc funtionally correct}
The algorithm \texttt{TrigArc} is correct. On an input $(S,N,\alpha)$, the algorithm halts in $\mathit{O}(\chi(S)^2 (\len(\alpha_\textrm{trim})+1))$ time. For $\alpha'=\texttt{TrigArc}(S,N,\alpha)$, we have that
\[
\len_\textrm{red}(\alpha'_\textrm{trim}) \leq \len_\textrm{red}(\alpha_\textrm{trim})+2s.\]
\end{lem}

\begin{proof}
By assumption on the input of the algorithm \texttt{TrigArc}, we are given an arc $\alpha$ whose inside either is in efficient position and, therefore, meets the requirements on the output of \texttt{TrigArc}, or contains a unique bad snippet $\alpha[k]$. In the latter case we further have that $\alpha[k]$ is a trigon snippet. Without loss of generality, we may assume that $\alpha[k]$ is turning right. By Lemma \ref{Omnibus Trigons}, applying a local homotopy to $\alpha$ at $\alpha[k]$ yields an arc whose inside contains at most one bad snippet of trigon type that is again turning right. By induction, repeatedly applying a local homotopy to the unique right-turning trigon snippet of the arc yields again an arc with at most one right-turning trigon snippet in its inside. We claim that after having applied at most $(\len(\alpha_\textrm{trim})+1) \cdot (s+2)$ many homotopies, the inside of the arc will be in efficient position.

This can be seen as follows: by Lemma \ref{Omnibus Trigons}, one of the following two statements is true when applying a local trigon homotopy to an arc $\alpha$:
\begin{itemize}
\item the number of carried snippets or right duals of $\alpha_\textrm{trim}$ decreases.
\item the trigon homotopy is of type $\mathbb{R}(h,v)$, the number of carried snippets of $\alpha_\textrm{trim}$ increases by at most $s-1$, and the number of right duals of $\alpha_\textrm{trim}$ does not change.
\end{itemize}

Moreover, Lemma \ref{Omnibus Trigons} implies that $\alpha_\textrm{trim}$ contains a trigon of type $\mathbb{R}(h,v)$ only if $\alpha$ is the original input arc or $\alpha$ is the image under a local homotopy of an arc whose inside has one right dual more than $\alpha_\textrm{trim}$. Thus, there can be at most $\dual_R(\alpha_\textrm{trim})+1$ many homotopies of type $\mathbb{R}(h,v)$. Therefore, under the process of applying trigon homotopies repeatedly, the total number of carried snippets is bounded by 
\begin{align*}
& \carr(\alpha_\textrm{trim})+\big( \dual_R(\alpha_\textrm{trim})+1\big)\cdot (s-1) \\
\leq & \len(\alpha_\textrm{trim})+ \big( \len(\alpha_\textrm{trim})+1\big)\cdot (s-1)\\
\leq & (\len(\alpha_\textrm{trim})+1) \cdot s,
\end{align*}
and any local homotopy which is not of type $\mathbb{R}(h,v)$ reduces this number or the number of right duals by at least one.

In summary, if we take into account one final homotopy which might be required once there are no carried snippets or right duals left, the total number of local homotopies that can be applied to an arc $\alpha$ before efficient position in the inside must be achieved is bounded by
\begin{align*}
& \big( \dual_R(\alpha_\textrm{trim})+1\big) +  \dual_R(\alpha_\textrm{trim}) + (\len(\alpha_\textrm{trim})+1) \cdot s + 1 \\
\leq & \big( \len(\alpha_\textrm{trim})+1 \big) + \len(\alpha_\textrm{trim}) + (\len(\alpha_\textrm{trim})+1) \cdot s + 1 \\
\leq & (\len(\alpha_\textrm{trim})+1) \cdot (s+2).
\end{align*}
Hence, the algorithm exits the while-loop after at most $(\len(\alpha_\textrm{trim})+1) \cdot (s+2)$ many iterations and the output arc $\alpha'$ must be in efficient position in its inside by Lemma \ref{Omnibus Trigons}. This implies correctness of the algorithm \texttt{TrigArc}.

We now analyse the running time of the algorithm \texttt{TrigArc} by going through its pseudocode line by line.
If we want to determine whether there is indeed a bad snippet in the inside of the arc $\alpha$ in the very first iteration of the while-loop, we potentially have to check each snippet of $\alpha_\textrm{trim}$ and determine if it is in efficient position. This takes $\mathit{O}(|\chi(S)|\len(\alpha_\textrm{trim}))$ time. However, in further iterations of the while-loop, we only have to check at most $s$ many snippets. This follows from the fact that all snippets but the subarc $\alpha[k-1:k+2]$ remain unchanged when applying a local homotopy to $\alpha$ at $\alpha[k]$ and that $\len(\Hom(\alpha'[k-1:k+2],k)) \leq s$ by Lemma \ref{Omnibus Trigons}. Thus, for all but one iteration of the while-loop, checking the while-condition takes $\mathit{O}(|\chi(S)|)$ time.
Following Remark \ref{Required time for trigon homotopies}, $\Hom(\alpha,k)$ can be computed in $\mathit{O}(|\chi(S)|)$ time. Therefore, the homotopy in an iteration of the while-loop can be computed in $\mathit{O}(|\chi(S)|)$ time as well.
Thus, the algorithm \texttt{TrigArc} executes at most one iteration of the while-loop in time $\mathit{O}(|\chi(S)|\len(\alpha_\textrm{trim}))$ and at most $(\len(\alpha_\textrm{trim})+1) \cdot (s+1)$ many iterations of the while loop in $\mathit{O}(|\chi(S)|)$ time.
Hence, the algorithm \texttt{TrigArc} terminates in $\mathit{O}(\chi(S)^2 (\len(\alpha_\textrm{trim})+1))$ time.

By Lemma \ref{Omnibus Trigons}, the reduced corner length of the arc does not increase under a local trigon homotopy unless one obtains efficient position for the inside of the arc. For $\alpha'=\texttt{TrigArc}(S,N,\alpha)$, this implies that $\len_\textrm{red}(\alpha'_\textrm{trim}) \leq \len_\textrm{red}(\alpha_\textrm{trim})+2s$.
\end{proof}

\begin{cor} \label{TrigArc in reduced corner length}
For any input $(S,N,\alpha)$, the algorithm $\texttt{TrigArc}$ halts in \\ $\mathit{O}(\chi(S)^2 \cdot (\len_\textrm{red}(\alpha_\textrm{trim})+1))$ time.
\end{cor}

\begin{proof}
This follows from Lemma \ref{TrigArc funtionally correct} and \ref{reduced corner length vs snippet length} if we keep in mind that $\alpha_\textrm{trim}$ contains at most one bad snippet, so $\len(\alpha_\textrm{trim})+1 \leq \len_\textrm{red}(\alpha_\textrm{trim}) +2$. Since $\mathit{O}(\chi(S)^2 \cdot (\len_\textrm{red}(\alpha_\textrm{trim})+2))=\mathit{O}(\chi(S)^2 \cdot (\len_\textrm{red}(\alpha_\textrm{trim})+1))$, the desired bound follows.
\end{proof}

Suppose that $S$ is a surface of positive complexity and $N \subset S$ is a tie neighbourhood of a large train track in $S$.
The following two corollaries will be of use later on:

\begin{cor}\label{TrigArc terminates as soon as changes weak snippet type of beginning or end}
Suppose that $\alpha \subset S$ is an almost efficient arc of trigon type. Set $0 <k<\len(\alpha)-1$ such that $\alpha[k]$ is the unique bad snippet of $\alpha_\textrm{trim}$. If $\Hom(\alpha,k)[0]$ or $\Hom(\alpha,k)[-1]$ is not weakly snippet homotopic to $\alpha[0]$ or $\alpha[-1]$ respectively, then $\Hom(\alpha,k)_\textrm{trim}$ is in efficient position. Hence, the algorithm $\texttt{TrigArc}$ terminates immediately once the weak snippet homotopy type of the first or last snippet of the given arc has changed.
\end{cor}

\begin{proof}
We recall that any local homotopy applied to $\alpha[k]$ affects the subarc $\alpha[k-1:k+2]$ only. By Lemma \ref{General observations trigon homotopies} we further know that the homotopy alters the weak snippet homotopy type of either $\alpha[k-1]$ or $\alpha[k+1]$ and that all snippets in $\Hom(\alpha[k-1:k+2],1)[1:-1]$ are in efficient position. Hence, if $\Hom(\alpha,k)[0]$ is not weakly snippet homotopic to $\alpha[0]$, we know that $k=1$ and that $\Hom(\alpha,1)[1:-1]$ must be in efficient position. If $\Hom(\alpha,k)[-1]$ is not weakly snippet homotopic to $\alpha[-1]$, we know that $k=\len(\alpha)-2$ and that again $\Hom(\alpha,1)[1:-1]$ must be in efficient position. Thus, we see that $\Hom(\alpha,k)_\textrm{trim}$ is in efficient position in both cases. As the algorithm \texttt{TrigArc} terminates as soon as the inside of the underlying arc is in efficient position, the second claim of the lemma follows.
\end{proof}

\begin{cor}\label{Cor for bigons of weight 2}
Suppose that $\alpha \subset S$ is a snippet-decomposed arc. Suppose further that $\len(\alpha)>2$, that $\alpha[0]$ is a left vertical dual and, that $\alpha[1]$ and $\alpha[-1]$ are right-turning trigon snippets of type $\mathbb{B}(h,t)$. Then $\alpha'=\texttt{TrigArc}(S,N,\alpha)$ contains at most one bad snippet, $\alpha'[-1]$, which is a snippet of type $\mathbb{B}(h,t)$ or $\mathbb{B}(h,h)$. If $\alpha'[-1]$ is a bigon snippet, then $\alpha'[0]$ and $\alpha'[-2]$ are left vertical duals.
\end{cor}

\begin{proof}
For clarity, we set $\texttt{TrigArc}(\alpha):=\texttt{TrigArc}(S,N,\alpha)$ for the remainder of this proof.
By Corollary \ref{TrigArc terminates as soon as changes weak snippet type of beginning or end} we know that either both or exactly one of the pairs $\{\texttt{TrigArc}(\alpha)[0],\alpha[0]\}$ and $\{\texttt{TrigArc}(\alpha)[-1],\alpha[-1]\}$ are weakly snippet homotopic. 

If both are weakly snippet homotopic, then $\texttt{TrigArc}(\alpha)[:-1]$ is in efficient position and $\texttt{TrigArc}(\alpha)[-1]=\alpha[-1]$ is a right-turning trigon snippet inside a branch rectangle. 

If $\texttt{TrigArc}(\alpha)[0]$ and $\alpha[0]$ are not weakly snippet homotopic, then $\texttt{TrigArc}(\alpha)[0]$ cuts off a region of index $-1/4$ on its left-hand side by Lemma \ref{General observations trigon homotopies}. Thus, it is in efficient position following an index-argument. As $\texttt{TrigArc}(\alpha)[-1]$ and $\alpha[-1]$ are weakly snippet homotopic in this case, $\texttt{TrigArc}(\alpha)$ contains one bad snippet, which is the trigon snippet $\texttt{TrigArc}(\alpha)[-1]=\alpha[-1]$.

If $\texttt{TrigArc}(\alpha)[0]$ and $\alpha[0]$ are weakly snippet homotopic but $\texttt{TrigArc}(\alpha)[-1]$ and $\alpha[-1]$ are not, then $\texttt{TrigArc}(\alpha)[:-1]$ is in efficient position and $\texttt{TrigArc}(\alpha)[-1]$ is a bigon snippet of type $\mathbb{B}(h,h)$. The last claim follows from the fact that $\alpha[-1]$ is embedded and cuts off a trigon on its right-hand side. Thus, Lemma \ref{General observations trigon homotopies} implies that $\texttt{TrigArc}(\alpha)[-1]$ is embedded and cuts off a bigon on its right-hand side. As $\alpha[-1](1)=\texttt{TrigArc}(\alpha)[-1](1) \subset \partial_h N$, the snippet $\texttt{TrigArc}(\alpha)[-1]$ is a bigon snippet of type $\mathbb{B}(h,h)$. We recall that $\alpha[1]$ is adjacent to the vertical dual $\alpha[0]$. Hence, as long as the homotopy does not change the weak snippet homotopy type of $\alpha[0]$, the subarc $\Hom(\alpha[0:2],1)[:-1]$ consists of snippets that cut off regions of index zero on their left-hand side. So, if $\Hom(\alpha[0:2],1)[1:-1]$ is empty, then the ``new'' trigon $\Hom(\alpha[0:2],1)[-1]$ is adjacent to a snippet that is weakly snippet homotopic to $\alpha[0]$, thus is a left vertical dual. Else, it is adjacent to the snippet $\Hom(\alpha[0:2],1)[-2]$, which cuts off a region of index zero on its left-hand side. By induction, if $\texttt{TrigArc}(\alpha)[-1]$ and $\alpha[-1]$ are not weakly snippet homotopic, then $\texttt{TrigArc}(\alpha)[-2]$ cuts off a region of index zero on its left-hand side. As $\texttt{TrigArc}(\alpha)[-1](0) \subset \partial_h N$, $\texttt{TrigArc}(\alpha)[-2]$ is a left vertical dual.
\end{proof}

\section{The algorithm \texttt{TrigCurve}}

In the last section we saw that the process of repeatedly applying trigon homotopies to an almost efficient arc $\alpha$ of trigon type eventually yields a homotopic arc $\alpha'$ whose inside is in efficient position. In this chapter, we are applying a similar line of reasoning to see that the repeated application of local homotopies to the unique trigon snippet of an almost efficient curve of trigon type yields efficient position or shortens the curve until it consists of a single snippet only. For a formal statement of the corresponding algorithm \texttt{TrigCurve} we refer the reader to Algorithm \ref{TrigCurve}.

%%%%%%%%%%%%%%%%%%%%%%% TRIG\_Snip_Curve $$$$$$$$$$$$$$$$$$$$$$$$$$$$$$$$

\begin{algorithm} 
    \SetKwInOut{Input}{Input}
    \SetKwInOut{Output}{Output}
	\DontPrintSemicolon
%    \textbf{Function} \texttt{TrigCurve}$(\alpha)$:\;
%    \BlankLine
    \Input{A surface $S$ of positive complexity, a tie neighbourhood $N \subset S$ of a large train track in $S$, and a curve $\alpha \subset S$ with at most one bad snippet, which must be a trigon snippet.}
    \Output{A curve $\alpha'$ homotopic to $\alpha$ such that $\alpha'$ is in efficient position or $\len(\alpha')=1$.}
    \BlankLine
    $\alpha'= \alpha$\;
    \While{$\len(\alpha')>1$ and there is some $k$ satisfying $0\leq k \leq \len(\alpha')-1$ such that $\alpha[k]$ is a bad snippet}
    {$\alpha'=\Hom(\alpha',k)$}
    \Return $\alpha'$ 
    \caption{\texttt{TrigCurve} - Homotoping a curve with at most one bad snippet into efficient position or shortening it to consist of a single snippet. The bad snippet of the input curve is required to be a trigon snippet.}
    \label{TrigCurve}
\end{algorithm}

\begin{lem} \label{TrigCurve funtionally correct}
The algorithm \texttt{TrigCurve} is correct. On an input $(S,N,\alpha)$, the algorithm halts in $\mathit{O}(\chi(S)^2 \cdot \len(\alpha))$ time.
\end{lem}

\begin{proof}
By assumption on the input of the algorithm \texttt{TrigCurve}, we are given a surface $S$ of positive complexity, a tie neighbourhood $N \subset S$ of a large train track in $S$, and a curve $\alpha$ that contains at most one bad snippet $\alpha[k]$. This is a trigon snippet.
We note that $\len(\alpha_\textrm{trim})=\len(\alpha)$ as $\alpha=\alpha_\textrm{trim}$.
Arguing as in the case of the algorithm \texttt{TrigArc}, Lemma \ref{Omnibus Trigons} tells us that the algorithm exits the while-loop after at most $(\len(\alpha)+1)\cdot(s+2)$ many iterations. Thus, checking the validity of the while-loop condition and applying the trigon homotopy in each iteration of the while-loop sums up to at most $\mathit{O}(\chi(S)^2 \cdot (\len(\alpha)+1))$ operation being required before efficient position is achieved or the curve is shortened to consist of a single snippet. As $\len(\alpha) \geq 1$, also $2 \cdot \len(\alpha) \geq (\len(\alpha)+1)$, so the algorithm \texttt{TrigCurve} halts in $\mathit{O}(\chi(S)^2 \cdot \len(\alpha))$ time.
\end{proof}

\begin{cor}\label{Trig Curve Corollary}
For any input $(S,N,\alpha)$, the algorithm \texttt{TrigCurve} terminates in \\ $\mathit{O}(\chi(S)^2 \cdot (\len_\textrm{red}(\alpha) + 1))$ time.
\end{cor}

\begin{proof}
Again, this follows from Lemma \ref{TrigArc funtionally correct} and \ref{reduced corner length vs snippet length} if we keep in mind that $\alpha$ contains at most one bad snippet.
\end{proof}

\newpage

\chapter{Local bigon homotopies}

In the previous chapter, we studied one family of bad snippets in detail: trigon snippets. There is one other important family of bad snippets: bigon snippets. These are the focus of this chapter. 
We will see that we have to alter our approach from the trigon snippet case: Suppose that we are given an almost efficient arc or curve of bigon type. Applying one local homotopy might yield another almost efficient arc or curve of bigon type. However, it might also yield an arc or curve that contains up to two trigon snippets. Therefore, we can not simply apply bigon homotopies repeatedly, and if we wish to make use of the algorithm \texttt{TrigArc} we have to be careful to select appropriate almost efficient subarcs.

These issues can be bypassed as follows: suppose that $\alpha$ is an almost efficient arc of bigon type, with $\alpha[k]$ being the unique bad snippet in $\alpha_\textrm{trim}$. As in the case of trigon snippets we will see that $\Hom(\alpha[k-1:k+2],1)_\textrm{trim}$ is in efficient position. Thus, if $k=-2$, that is, if the bigon snippet $\alpha[k]$ is the second to last snippet of the arc, then applying one local homotopy yields an arc that contains at most one bad snippet in its inside. A careful analysis will show that this snippet has to be a trigon snippet. This allows us to employ the algorithm \texttt{TrigArc} to achieve efficient position for the inside of the arc.

The layout of this chapter is as follows: First, we discuss the effects that local homotopies have on almost efficient arcs and curves of bigon type. Again, we put special emphasis on the various notions of length and their changes under the bigon homotopies. Secondly, we present the algorithm \texttt{BigArc} which yields efficient position for the inside of an arc $\alpha$ containing a unique bad snippet, $\alpha[-2]$, in its inside. We close this chapter by giving bounds on the running time of \texttt{BigArc} and the snippet length of its output. We remark that most of the very detailed observations made about bigon homotopies are not required to obtain the bounds on the running time and length of the output of \texttt{BigArc}. However, they will be essential when working with curves containing a single bad snippet (see Chapter \ref{Chapter last snippet}).

\section{Local homotopies for bigon arcs}

Throughout this section, we fix a surface $S = S_{g,b}$ satisfying $\xi(S)=3g-3+b \geq 1$. We further fix a large train track $\tau \subset S$ and a tie neighbourhood $N=N(\tau) \subset S$. As usual, we do not distinguish between snippets and their strong snippet homotopy classes unless otherwise stated. For any snippet $a \subset R \in \mathcal{R}$, we always assume that $a$ has minimal self-intersection, intersects the one-skeleton of $\mathcal{R}$ perpendicularly and misses the corners of $\mathcal{R}$. We further recall that arcs and curves in $S$ are assumed to be self-transverse and transverse to $\partial \mathcal{R}$. Moreover, if $\alpha$ is an arc, we assume that $\alpha(0)$ and $\alpha(1)$ lie in $\partial \mathcal{R}$. Thus, arcs and curves in $S$ admit canonical decompositions into snippets.

\begin{lem}\label{Easy Bigon bounds}
Suppose that $\alpha\subset S$ is a snippet-decomposed arc or curve satisfying $\len(\alpha)\geq 2$. Suppose that $\alpha[k] \subset \alpha_\textrm{trim}$ is a bad snippet of bigon type $\mathbb{B}(t,t)$, $\mathbb{S}(h,h,0)$, $\mathbb{S}(t,t,0)$, $\mathbb{S}(v,v,0)$, or $\mathbb{R}(v,v)$. If $\len(\alpha)=2$, then $\len(\Hom(\alpha,k))=1$. Else, that is if $\len(\alpha)\geq 3$, we have that $\len(\Hom(\alpha,k))=\len(\alpha)-2$.
Furthermore, if $\alpha[k-1]$ and $\alpha[k+1]$ are in efficient position, then one of the following three statements holds.
\begin{enumerate}
\item $\Hom(\alpha[k-1:k+2],1)$ is a bigon snippet or a trigon snippet and \[\len_\textrm{red}(\Hom(\alpha,k))\leq \len_\textrm{red}(\alpha)-2.\]
\item $\Hom(\alpha,k)$ is in efficient position.
\item $\Hom(\alpha,k)$ is a peripheral curve consisting of a single snippet only.
\end{enumerate}
\end{lem}

\begin{proof}
We prove the statements of the lemma in order. As $\alpha[k]$ is a bigon, $\alpha[k](0)$ and $\alpha[k](1)$ lie on the same component of $\partial_h R$ or $\partial_v R$ of some rectangle $R \in \mathcal{R}_\textrm{tie}$. By assumption on the specific bigon snippet types, the boundary of the bigon does not contain any points of $\partial^2 \mathcal{R}$ (see Figures \ref{Homotopy of type B(S,v,v,0) picture}-\ref{Homotopy of type B(S,h,h,0) picture}). Hence, $\alpha[k-1]$ and $\alpha[k+1]$ lie inside the same region $R' \in \mathcal{R}$ and the arc $\alpha[k-1:k+1]$ is replaced by a single snippet $\alpha'[k-1]$.

\begin{figure}[htbp] 
    \begin{minipage}[b]{0.49\linewidth}
    \centering
\begin{tikzpicture}[scale=0.35]
\draw [very thick] (-3,2) rectangle (2,-5.5);
\node at (-0.5,3) {\tiny{$h$}};
\node at (-9,-1.75) {\tiny{$t$}};
\draw [very thick](1.5,-3) -- (2.5,-3);
\node at (3,-1.75) {\tiny{$v$}};
\draw [very thick] (-3,-5.5) rectangle (-8,2);
\draw [very thick, blue, densely dashed] plot[smooth, tension=.7] coordinates {(-6,-0.5)  (-1.5,-1) (-1.5,-4)  (-6,-4.5)};
\draw [very thick, cyan, densely dashed] plot[smooth, tension=.4] coordinates {(-6,-0.75)  (-3.5,-1) (-3.5,-4)  (-6,-4.25)};
\draw [very thick] (2,2) rectangle (5,-0.5);
\draw [very thick, blue, densely dashed] plot[smooth, tension=.7] coordinates {(4,1.5) (0,1.25) (0,0.25) (4,0)};
\draw [very thick, cyan, densely dashed] plot[smooth, tension=.4] coordinates {(4,1.25) (2.375,1.125) (2.375,0.375) (4,0.25)};
\end{tikzpicture}
    \caption{Homotopies of type $\mathbb{S}(t,t,0)$.} 
    \label{Homotopies of type B(S,t,t,0)}
  \end{minipage}
  \begin{minipage}[b]{0.49\linewidth}
    \centering
\begin{tikzpicture}[scale=0.35]
\draw [very thick] (-3,2) rectangle (2,-5.5);
\node at (-0.5,3) {\tiny{$h$}};
\node at (-9,-1.75) {\tiny{$t$}};
\draw [very thick](1.5,-3) -- (2.5,-3);
\node at (3,-1.75) {\tiny{$v$}};
\draw [very thick] (-3,-5.5) rectangle (-8,2);
\draw [very thick] (2,2) rectangle (5,-0.5);
\draw [very thick, blue, densely dashed] plot[smooth, tension=.7] coordinates {(-0.5,1.75) (3.625,1.5) (3.625,0) (-0.5,-0.25)};
\draw [very thick, blue, densely dashed] plot[smooth, tension=.7] coordinates {(-1,-1.5) (-6,-2) (-6,-4.5) (-1,-5)};
\draw [very thick, cyan, densely dashed] plot[smooth, tension=.4] coordinates {(-1,-1.75) (-2.625,-1.875) (-2.625,-4.625) (-1,-4.75)};
\draw [very thick, cyan, densely dashed] plot[smooth, tension=.4] coordinates {(-0.5,1.5) (1.625,1.375) (1.625,0.125) (-0.5,0)};
\end{tikzpicture}
    \caption{Homotopies of type $\mathbb{B}(t,t)$.} 
    \label{Homotopies of type B(B,t,t)}
    \vspace{12pt}
  \end{minipage}
  \begin{minipage}[b]{0.49\linewidth}
    \centering
\begin{tikzpicture}[scale=0.4]
\node at (5,-9) {\tiny{$v$}};
\node at (1.5,-13) {\tiny{$h$}};
\node at (-3,-9) {\tiny{$t$}};
\draw [very thick] (-1.5,-6) rectangle (4,-12) node (v3) {};
\draw [very thick] plot[smooth, tension=.7] coordinates {(4,-7.5) (7.5,-7) (8.5,-6.5)};
\draw [very thick] plot[smooth, tension=.7] coordinates {(4,-6) (5.75,-5.75) (6.25,-5.5)};
\draw [very thick] plot[smooth, tension=.7] coordinates {(4,-10.5) (7.5,-11) (8.5,-11.5)};
\draw [very thick] plot[smooth, tension=.7] coordinates {(v3) (5.5,-12.25) (6.25,-12.5)};
\draw [very thick](6.25,-5.5) -- (6.25,-7.25);
\draw [very thick](6.25,-12.5) -- (6.25,-10.75);
\draw [very thick, densely dashed, blue] plot[smooth, tension=.7] coordinates {(0.5,-8) (6.5,-8.25) (6.5,-9.75) (0.5,-10)};
\draw [very thick, densely dashed, cyan] plot[smooth, tension=.4] coordinates {(0.5,-8.5) (3.5,-8.5) (3.5,-9.5) (0.5,-9.5)};
\end{tikzpicture}
    \caption{A homotopy of type $\mathbb{R}(v,v)$.} 
    \label{Homotopy of type B(R,v,v) picture}
    \vspace{2 ex}
  \end{minipage}
\begin{minipage}[b]{0.49\linewidth}
    \centering
\begin{tikzpicture}[scale=0.4]
\node at (3,-9) {\tiny{$v$}};
\node at (1.5,-13) {\tiny{$h$}};
\node at (-3,-9) {\tiny{$t$}};
\draw [very thick] (-1.5,-6) rectangle (4,-12) node (v3) {};
\draw [very thick] plot[smooth, tension=.7] coordinates {(4,-7.5) (7.5,-7) (8.5,-6.5)};
\draw [very thick] plot[smooth, tension=.7] coordinates {(4,-6) (5.75,-5.75) (6.25,-5.5)};
\draw [very thick] plot[smooth, tension=.7] coordinates {(4,-10.5) (7.5,-11) (8.5,-11.5)};
\draw [very thick] plot[smooth, tension=.7] coordinates {(v3) (5.5,-12.25) (6.25,-12.5)};
\draw [very thick](6.25,-5.5) -- (6.25,-7.25);
\draw [very thick](6.25,-12.5) -- (6.25,-10.75);
\draw [very thick, densely dashed, blue] plot[smooth, tension=.7] coordinates {(8,-7.25) (4.5,-8) (0.75,-8.25)     (0.75,-10) (4.5,-10.25) (8,-10.75)};
\draw [very thick, densely dashed, cyan] plot[smooth, tension=.4] coordinates {(8,-7.75) (4.5,-8.5) (4.5,-9.75) (8,-10.25)};
\end{tikzpicture}
    \caption{A homotopy of type $\mathbb{S}(v,v,0)$.}
    \label{Homotopy of type B(S,v,v,0) picture}
    \vspace{2 ex}
  \end{minipage}
\begin{minipage}[b]{0.99\linewidth}
    \centering
\begin{tikzpicture}[scale=0.25]
\draw [very thick] (-3.5,1.5) rectangle (4,-3);
\node at (-6.5,6.5) {\tiny{$h$}};
\draw [very thick](-6.5,1.5) node (v1) {} -- (7,1.5) node (v2) {};
\draw [very thick] plot[smooth, tension=.7] coordinates {(v1) (-7,3) (-8,4.5)};
\draw [very thick] plot[smooth, tension=.7] coordinates {(v2) (7.5,3) (8.5,4.5)};
\draw [very thick] plot[smooth, tension=.7] coordinates {(-8,4.5) (-5,6) (-4,7.5)};
\draw [very thick] plot[smooth, tension=.7] coordinates {(8.5,4.5) (5.5,6) (4.5,7.5)};
\draw [blue, densely dashed, very thick] plot[smooth, tension=.7] coordinates {(-3.25,5.5) (-2,3) (-1,-1.5)   (1.5,-1.5) (2.5,3) (4,5.5)};
\draw [cyan, densely dashed, very thick] plot[smooth, tension=.4] coordinates {(-2.5,5.5) (-1.5,3.5) (-1,2) (1.5,2) (2,3.5) (3.5,5.5)};
\node at (-8,2.5) {\tiny{$v$}};
\node at (5.24,-0.86) {\tiny{$t$}};
\draw [very thick](-3,0) -- (-4,0);
\draw [very thick](-3,-1.5) -- (-4,-1.5);
\end{tikzpicture}
    \caption{A homotopy of type $\mathbb{S}(h,h,0)$.}
    \label{Homotopy of type B(S,h,h,0) picture}
    \vspace{2 ex}
  \end{minipage}
\end{figure}

To prove the remaining claims of the lemma, we split the five possible types of bigon snippets into two groups and consider them separately.

Firstly, let us assume that $\alpha[k]$ is a bigon snippet of type $\mathbb{S}(t,t,0)$, $\mathbb{B}(t,t)$ or $\mathbb{R}(v,v)$. If $\alpha[k-1]$ and $\alpha[k+1]$ are in efficient position, we claim that $\Hom(\alpha[k-1:k+2],1)$ is a bigon snippet and satisfies $\len_\textrm{red}(\Hom(\alpha,k))\leq \len_\textrm{red}(\alpha)-2$. This can be seen as follows: since $\alpha[k-1]$ and $\alpha[k+1]$ are efficient snippets inside a branch or switch rectangle, we know that $\alpha[k-1]\neq \alpha[k+1]$ and that both snippets must be carried and have corner length one or three. Thus, $\alpha[k-1](0)$ and $\alpha[k+1](1)$ lie on the same side of a branch or switch rectangle. This implies that $\Hom(\alpha[k-1:k+2],1)$ is a bigon snippet. As the corner length of snippets inside a branch or switch rectangle is independent of their form, we further see that $\len_\textrm{corn}(\Hom(\alpha[k-1:k+2],1)=\len_\textrm{corn}(\alpha[k-1])$. As $\len_\textrm{corn}(\alpha[k+1])\geq 1$ and $\len_\textrm{corn}(\alpha[k])=0$ only if $\len_\textrm{corn}(\alpha[k+1])=3$, this implies that \[\len_\textrm{corn}(\Hom(\alpha[k-1:k+2],1)\leq \len_\textrm{corn}(\alpha[k-1:k+2])-2.\] Since no dual snippets are affected by the homotopy, the number of blockers of $\alpha$ and $\Hom(\alpha,k)$ coincides. Thus, we see that $\len_\textrm{red}(\Hom(\alpha,k))\leq \len_\textrm{red}(\alpha)-2$, giving \textit{1}.

Secondly, let us assume that $\alpha[k]$ is a bigon snippet of type $\mathbb{S}(h,h,0)$ or $\mathbb{S}(v,v,0)$. If $\alpha[k-1]=\alpha[k+1]$, then $\Hom(\alpha,k)$ consists of a single snippet that satisfies $\wind(\Hom(\alpha,k)) =  \wind(\alpha[k-1])$. Thus, if $\alpha[k-1] \subset R \in \mathcal{R}_\textrm{comp}$ is in efficient position, it has a winding number which is a non-trivial multiple of $|\partial^2 R|$. As $\wind(\Hom(\alpha,k)) =  \wind(\alpha[k-1])$, this implies that $\Hom(\alpha,k)$ is a peripheral curve consisting of a single snippet only, giving \textit{3}.

For the remainder of this proof we assume that $\alpha[k-1]\neq \alpha[k+1]$ and that both of these snippets are in efficient position. We distinguish the following two cases: either at least one of the snippets $\alpha[k-1]$ and $\alpha[k+1]$ has corner length $2s$, or both have a corner length that is strictly smaller than $2s$. We remark that in the latter case, efficient position implies that both snippets must be horizontal or vertical duals.

Let us first assume that at least one of the snippets $\alpha[k-1]$ and $\alpha[k+1]$ has corner length $2s$. Without loss of generality, we assume that this is the snippet $\alpha[k-1]$. Then $\len_\textrm{corn}(\Hom(\alpha[k-1:k+2],1)) \leq 2s \leq \len_\textrm{corn}(\alpha[k-1])$. We note that $\len_\textrm{corn}(\alpha[k])=3$ and $\len_\textrm{corn}(\alpha[k+1])\geq 1$. Thus, we see that \[\len_\textrm{corn}(\Hom(\alpha[k-1:k+2],1)) \leq \len_\textrm{corn}(\alpha[k-1:k+2])-4.\] As snippets of corner length $2s$ cannot be part of any blocker, the homotopy reduces the number of blockers of the underlying arc or curve by at most one. This implies that $\len_\textrm{red}(\Hom(\alpha,k))\leq \len_\textrm{red}(\alpha)-2$, giving \textit{1} and \textit{2} if $\Hom(\alpha[k-1:k+2],1)$ is in efficient position.

Let us now assume that the corner length of $\alpha[k-1]$ and $\alpha[k+1]$ is strictly smaller than $2s$. As remarked previously, efficient position implies that both of these snippets cut off regions of index zero on one of their sides. If they cut off a region of index zero on the same side, an index-argument shows that $\Hom(\alpha[k-1:k+2],1)$ is in efficient position, giving \textit{2}. Therefore, let us assume that they cut off the regions of index zero on different sides. In this case $\Hom(\alpha[k-1:k+2],1)$ is a bigon of type $\mathbb{R}(h,h)$ or $\mathbb{R}(v,v)$. Without loss of generality, we may further assume that $\len_\textrm{corn}(\alpha[k-1])\leq \len_\textrm{corn}(\alpha[k+1])$. Then the bigon cut off by $\Hom(\alpha[k-1:k+2],1)$ is contained inside the region of index zero cut off by $\alpha[k+1]$. However, the boundary of the bigon contains two components of $\partial \mathcal{R}- \partial^2 \mathcal{R}$ less than the boundary of the region of index zero cut off by $\alpha[k+1]$. Hence, we see that \[\len_\textrm{corn}(\Hom(\alpha[k-1:k+2],1))\leq \len_\textrm{corn}(\alpha[k+1])-2.\] Since $\len_\textrm{corn}(\alpha[k])=3$ and $\len_\textrm{corn}(\alpha[k+1])\geq 1$, this implies that \[\len_\textrm{corn}(\Hom(\alpha[k-1:k+2],1)) \leq \len_\textrm{corn}(\alpha[k-1:k+2])-6.\] As the homotopy reduces the number of blockers by at most two, we can again conclude that $\len_\textrm{red}(\Hom(\alpha,k))\leq \len_\textrm{red}(\alpha)-2$, which completes the proof of the lemma.
\end{proof}

As in the case of trigon homotopies, we see that homotopies for bigon snippets of type $\mathbb{B}(t,t)$, $\mathbb{S}(h,h,0)$, $\mathbb{S}(t,t,0)$, $\mathbb{S}(v,v,0)$, and $\mathbb{R}(v,v)$ can be computed in $\mathit{O}(|\chi(S)|)$ time:

\begin{lem} \label{Required time for easy bigon homotopies}
Suppose that $\alpha\subset S$ is a snippet-decomposed arc or curve satisfying $\len(\alpha)\geq 2$. Suppose that $\alpha[k] \subset \alpha_\textrm{trim}$ is a bad snippet of bigon type $\mathbb{B}(t,t)$, $\mathbb{S}(h,h,0)$, $\mathbb{S}(t,t,0)$, $\mathbb{S}(v,v,0)$, or $\mathbb{R}(v,v)$. Then $\Hom(\alpha,k)$ can be computed in $\mathit{O}(|\chi(S)|)$ time.
\end{lem}

\begin{proof}
By assumption, $\alpha$ is a snippet-decomposed arc or curve containing a bigon snippet $\alpha[k]$ of the required type.
We recall that $\alpha$ is determined by its cutting sequence and the winding numbers of the respective snippets. Lemma \ref{Easy Bigon bounds} implies that the subarc $\alpha[k-1:k+2]$ is replaced by the single snippet $a=\Hom(\alpha,k)[k-1]$ under the homotopy. Thus, the cutting sequence of $\Hom(\alpha,k)$ can be computed in $\mathit{O}(|\chi(S)|)$ time.

Assuming that arithmetic on winding numbers can be done in constant time, we claim that the winding number of $a$ can be calculated in $\mathit{O}(|\chi(S)|)$ time given the winding numbers of $\alpha$. For snippets of type $\mathbb{S}(t,t,0)$, $\mathbb{B}(t,t)$ or $\mathbb{R}(v,v)$ this is immediate since $a$ lies inside the tie neighbourhood. 

Thus, suppose that $\alpha[k]$ is of type $\mathbb{S}(h,h,0)$ or $\mathbb{S}(v,v,0)$.
If $\len(\alpha)=2$ and $\alpha[k-1]$ lies inside a peripheral annulus, then $\Hom(\alpha,k)$ consists of a single snippet only and $\wind(\Hom(\alpha,k)) =  \wind(\alpha[k-1])$. Suppose that $\len(\alpha)>2$ and that $\alpha[k-1]$ and thus $\alpha[k+1]$ lie inside a peripheral annulus. We can distinguish the following two cases: either $\alpha[k-1]$ and $\alpha[k+1]$ have empty intersection with $\partial S$, or at least one of their boundary points lies in $\partial S$. In the first case we see that
\begin{align*}
\wind(a) = \wind(\alpha[k-1])+\wind(\alpha[k+1]).
\end{align*}
In the second case, at least one of the boundary points of $a$ lies in $\partial S$. Thus, $\wind(a) = 0$. This implies that the winding numbers of $\Hom(\alpha,k)$ can be computed in $\mathit{O}(|\chi(S)|)$ time in all cases.
\end{proof}

\begin{lem}\label{Lemma B(S,t,v,1)}
Suppose that $\alpha \subset S$ is a snippet-decomposed arc or curve satisfying $\len(\alpha)\geq 2$. Suppose that $\alpha[k] \subset \alpha_\textrm{trim}$ is a bad snippet of bigon type $\mathbb{S}(t,v,1)$. Then $\len(\Hom(\alpha,k))=\len(\alpha)-1$. Furthermore, if $\alpha[k-1]$ is in efficient position, then $\Hom(\alpha,k)[k-1]$ is in efficient position or a trigon snippet of type $\mathbb{B}(h,t)$ or $\mathbb{R}(h,v)$. If $\alpha[k-1]$ and $\alpha[k+1]$ are in efficient position, then one of the following two statements holds.
\begin{enumerate}
\item $\Hom(\alpha[k-1:k+2],1)$ consists of one trigon snippet of type $\mathbb{B}(h,t)$ and one trigon snippet of type $\mathbb{R}(h,v)$. These are turning into the same direction. Further, $\len_\textrm{red}(\Hom(\alpha,k))\leq \len_\textrm{red}(\alpha)-4$.
\item $\Hom(\alpha[k-1:k+2],1)$ consist of one trigon snippet of type $\mathbb{B}(h,t)$ and one snippet in efficient position and $\len_\textrm{red}(\Hom(\alpha,k))\leq \len_\textrm{red}(\alpha)+ 2s-5$.
\end{enumerate}
\end{lem}

\begin{proof}
We first observe that $\len(\alpha) \geq 3$ as $\alpha[k-1]$ and $\alpha[k+1]$ are contained in a branch rectangle and a complementary region of $N$ respectively. As $\alpha[k]$ is of weight one, the bigon homotopy replaces the subarc $\alpha[k-1:k+2]$ by a subarc whose inside intersects $\partial \mathcal{R}$ exactly one (see Figure \ref{Homotopy of type B(S,t,v,1) picture}). Thus, we see that $\len(\Hom(\alpha[k-1:k+2],1))=2$. This implies that $\len(\Hom(\alpha,k))=\len(\alpha)-1$.

\begin{figure}[htbp] 
  \begin{minipage}[b]{0.49\linewidth}
    \centering
\begin{tikzpicture}[scale=0.35]
\draw [very thick] (-2,3) rectangle (4,-6);
\node at (1,4) {\tiny{$h$}};
\node at (10,1.5) {\tiny{$t$}};
\node at (3,-1.5) {\tiny{$v$}};
\draw [very thick] (4,3) rectangle (9,0);
\draw [very thick] (4,-3) rectangle (9,-6);
\draw [very thick, densely dashed, blue] plot[smooth, tension=.7] coordinates {(7,2.5)  (1,1.75) (1,-1.75)  (7,-2.5)};
\draw [very thick, densely dashed, cyan] plot[smooth, tension=.4] coordinates {(7,2.25) (4.5,1.75)  (4.5,-2) (7,-2.25)};
\end{tikzpicture}
    \caption{A homotopy of type $\mathbb{S}(t,v,1)$.} 
    \label{Homotopy of type B(S,t,v,1) picture}
  \end{minipage}
    \begin{minipage}[b]{0.49\linewidth}
    \centering
\begin{tikzpicture}[scale=0.35]
\draw [very thick] (-2,3) rectangle (4,-6);
\node at (1,4) {\tiny{$h$}};
\node at (10,1.5) {\tiny{$t$}};
\node at (2.5,-1.5) {\tiny{$v$}};
\draw [very thick] (4,3) rectangle (9,0);
\draw [very thick] (4,-3) rectangle (9,-6);
\draw [very thick, densely dashed, blue] plot[smooth, tension=.7] coordinates {(7,2.5)  (0,1.5) (0,-4.5)  (7,-5.5)};
\draw [very thick, densely dashed, cyan] plot[smooth, tension=.4] coordinates {(7,2) (4.75,1.5)  (4.75,-4.5) (7,-5)};
\end{tikzpicture}
    \caption{A homotopy of type $\mathbb{S}(t,t,2)$.} 
    \label{Homotopy of type B(S,t,t,2) picture}
  \end{minipage}
\end{figure}

Without loss of generality, we may assume that $\alpha[k]$ is turning right. If $\alpha[k-1]$ is in efficient position and embedded, it cuts off a region of non-positive index on either side. Under the homotopy, the index of this region increases by at most $1/4$. Hence, $\Hom(\alpha,k)[k-1]$ is a right-turning trigon snippet or in efficient position. If $\alpha[k-1]$ is in efficient position but not embedded, its winding number changes by at most one under the homotopy. As non-embedded snippets have a winding number whose modulus is strictly greater than two, the snippet $\Hom(\alpha,k)[k-1]$ will have a winding number whose modulus is greater than or equal to two. Thus, $\Hom(\alpha,k)[k-1]$ is in efficient position. We recall that $\alpha[k-1]$ lies inside a branch rectangle or complementary region. Thus, if $\alpha[k-1]$ is in efficient position, then $\Hom(\alpha,k)[k-1]$ is in efficient position or a trigon snippet of type $\mathbb{B}(h,t)$ or $\mathbb{R}(h,v)$.

Let us now assume that $\alpha[k-1]$ and $\alpha[k+1]$ are both in efficient position. Without loss of generality, we may assume that $\alpha[k-1]$ is the snippet that lies inside a branch rectangle. Following  previous arguments, $\Hom(\alpha,k)[k-1]$ is a trigon snippet inside the branch rectangle, and therefore of type $\mathbb{B}(h,t)$. The previous paragraph implies that $\Hom(\alpha,k)[k]$ is a trigon snippet or in efficient position. We note that $\Hom(\alpha,k)[k-1]$ and $\Hom(\alpha,k)[k]$ turn the same way.

Let us first assume that $\Hom(\alpha,k)[k]$ is a trigon snippet. Then $\alpha[k+1]$ is a horizontal dual and $\len_\textrm{corn}(\Hom(\alpha,k)[k]) = \len_\textrm{corn}(\alpha[k+1])-1$. We know that $\len_\textrm{corn}(\Hom(\alpha,k)[k-1]) = \len_\textrm{corn}(\alpha[k-1])$ and that $\len_\textrm{corn}(\alpha[k])=3$. Since no vertical duals of $\alpha$ are affected by the homotopy, this implies that \[\len_\textrm{red}(\Hom(\alpha,k))\leq \len_\textrm{red}(\alpha)-4.\] This gives \textit{1}.
If $\Hom(\alpha,k)[k]$ is in efficient position, its length is bounded by $2s$. As $\len_\textrm{corn}(\alpha[k+1])\geq 1$ and the number of blockers of the arc or curve does not decrease under the homotopy, we conclude that $\len_\textrm{red}(\Hom(\alpha,k))\leq \len_\textrm{red}(\alpha)+2s-5$, obtaining \textit{2}.
\end{proof}

\begin{rem}\label{Required time for bigon S(t,v,1) homotopy}
As in the case of the previously discussed types of local homotopies, we see that $\Hom(\alpha,k)$ can be computed in $\mathit{O}(|\chi(S)|)$ time if $\alpha[k]$ is a snippet of type $\mathbb{S}(t,v,1)$. This follows from the following two observations: firstly, Lemma \ref{Lemma B(S,t,v,1)} states that $\len(\Hom(\alpha[k-1:k+2],1))=2$. Thus, the cutting sequence can be adjusted in constant time. Secondly, at most one of the snippets of $\Hom(\alpha[k-1:k+2],1)$ can lie inside a peripheral annulus. Let us assume that this is the snippet $\Hom(\alpha,k)[k]$. Then $\wind(\Hom(\alpha,k)[k])=\wind(\Hom(\alpha[k+1])) \pm 1$, so assuming that arithmetic on winding numbers can be done in constant time, this implies that $\Hom(\alpha,k)$ can be computed in $\mathit{O}(|\chi(S)|)$.
\end{rem}

\begin{lem}\label{Lemma B(S,t,t,2)}
Suppose that $\alpha \subset S$ is a snippet-decomposed arc or curve satisfying $\len(\alpha)\geq 2$. Suppose that $\alpha[k] \subset \alpha_\textrm{trim}$ is a bad snippet of bigon type $\mathbb{S}(t,t,2)$. Then $\len(\Hom(\alpha,k))=\len(\alpha)$. Moreover, the following statements hold.
\begin{enumerate}
\item If $\alpha[k]$ turns right, then $\Hom(\alpha,k)[k]$ is a left vertical dual.
\item If $\len(\alpha)=2$, then $\Hom(\alpha,k)[k-1]$ is a bigon snippet of type $\mathbb{B}(h,h)$.
\item If $\len(\alpha)>2$, $\alpha[k]$ turns right, and $\alpha[k-1]$ is in efficient position, then $\Hom(\alpha,k)[k-1]$ is a right-turning trigon snippet inside a branch rectangle.
\item If $\len(\alpha)>2$ and $\alpha[k-1]$ and $\alpha[k+1]$ are in efficient position, then \[\len_\textrm{red}(\Hom(\alpha,k))\leq \len_\textrm{red}(\alpha)-2.\]
\end{enumerate}
\end{lem}

\begin{proof}
We prove the statements of the lemma in order.
As $\alpha[k]$ has weight two, the subarc $\alpha[k-1/3,k+4/3]$ is replaced by a subarc intersecting $\partial \mathcal{R}$ exactly twice. This implies that $\len(\Hom(\alpha,k))=\len(\alpha)$. Further, the snippet lying between these two intersection points is parallel to a component of $\partial_v N$, thus cuts off a region of index zero on the side where the bigon cut off by $\alpha[k]$ lies. Hence, if $\alpha[k]$ turns right, then $\Hom(\alpha,k)[k]$ is a left vertical dual, giving \textit{1}.

If $\len(\alpha)=2$, then $\alpha[k-1]$ is a carried snippet inside a branch rectangle (see Figure \ref{B(S,t,t,2) short}). The homotopy increases the index of the region cut off by $\alpha[k-1]$ by $1/2$, so $\Hom(\alpha,k)[k-1]$ is a bigon snippet of type $\mathbb{B}(h,h)$, giving \textit{2}.

\begin{figure}[htbp] 
    \begin{minipage}[b]{0.99\linewidth}
    \centering
\begin{tikzpicture}[scale=0.35]
\draw [very thick] (-2,3) rectangle (4,-6) node (v1) {};
\node at (1,4) {\tiny{$h$}};
\node at (-3.5,-1.5) {\tiny{$t$}};
\node at (3,-1.5) {\tiny{$v$}};
\draw [very thick] plot[smooth, tension=.7] coordinates {(4,3) (12.5,1.75) (12.5,-4.75) (v1)};
\draw [very thick] plot[smooth, tension=.7] coordinates {(4,-4) (10.5,-3.25) (10.5,0.25) (4,1)};
\draw [very thick] (7.5,-1.5) ellipse (0.5 and 0.5);
\draw [very thick, densely dashed, blue] plot[smooth cycle, tension=.7] coordinates {(4,2.5) (-0.5,1) (-0.5,-4) (4,-5.5) (11.5,-4) (11.5,1)};
\draw [very thick, densely dashed, cyan] plot[smooth cycle, tension=.5] coordinates { (4.5,1) (4.5,-4)  (7,-4.5) (11.5,-3.25) (11.5,0.25) (7,1.5)};
%\draw [very thick, densely dashed, green] plot[smooth cycle, tension=.5] coordinates {(5,-0.5) (5,-2.5) (6.5,-3) (10,-2.5) (10,-0.5) (6.5,0)};
\end{tikzpicture}
    \caption{A homotopy of type $\mathbb{S}(t,t,2)$ applied to a curve of snippet length two.} 
    \label{B(S,t,t,2) short}
      \end{minipage}
\end{figure}
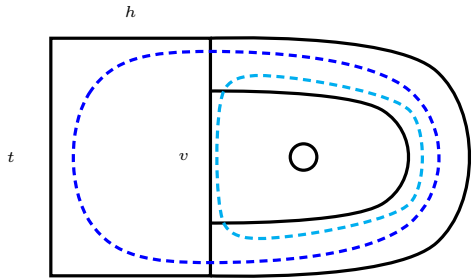

If $\len(\alpha)>2$ and $\alpha[k-1]$ is in efficient position, then $\alpha[k-1]$ must be a carried snippet inside a branch rectangle. Thus, if $\alpha[k]$ turns right, the index of the region cut off by $\alpha[k-1]$ on its right-hand side increases by $1/4$. This implies that $\Hom(\alpha,k)[k-1]$ is a right-turning trigon of type $\mathbb{B}(h,t)$, giving \textit{3}.

To prove the last claim let us assume that $\len(\alpha)>2$ and that $\alpha[k-1]$ and $\alpha[k+1]$ are in efficient position. Then
\begin{align*}
\len_\textrm{corn}(\alpha[k-1])=& \len_\textrm{corn}(\alpha[k+1])=\len_\textrm{corn}(\Hom(\alpha,k)[k-1])) \\=&\len_\textrm{corn}(\Hom(\alpha,k)[k+1]))=1.
\end{align*}
Moreover, we know that $\len_\textrm{corn}(\Hom(\alpha,k)[k])=\len_\textrm{corn}(\alpha)[k]-2$.
Thus, if $\len(\alpha)>2$ and $\alpha[k-1]$ and $\alpha[k+1]$ are in efficient position, we see that \[\len_\textrm{corn}(\Hom(\alpha,k))\leq \len_\textrm{corn}(\alpha)-2.\] As no vertical duals are affected by the homotopy, this gives \textit{4}.
\end{proof}

\begin{rem}\label{Required time for S(t,t,2) bigon homotopies}
Again, we see that $\Hom(\alpha,k)$ can be computed in $\mathit{O}(|\chi(S)|)$ time if $\alpha[k]$ is a snippet of type $\mathbb{S}(t,t,2)$. Firstly, Lemma \ref{Lemma B(S,t,t,2)} states that \[\len(\Hom(\alpha[k-1:k+2],1))=3.\] Thus, the cutting sequence can be adjusted in constant time. Secondly, the only snippet of $\Hom(\alpha[k-1:k+2],1)$ that can lie inside a peripheral annulus is the snippet $\Hom(\alpha[k-1:k+2],1)[1]$. Lemma \ref{Lemma B(S,t,t,2)} implies that this snippet is a vertical dual snippet. Thus, the modulus of its winding number is two, which implies that the winding numbers of $\Hom(\alpha,k)$ can be computed in constant time as well.
\end{rem}

\begin{lem}\label{Lemma B(R,h,h)}
Suppose that $\alpha \subset S$ is a snippet-decomposed arc or curve satisfying $\len(\alpha)\geq 2$. Suppose that $\alpha[k] \subset \alpha_\textrm{trim}$ is a bad snippet of bigon type $\mathbb{R}(h,h)$. Then the following statements hold.
\begin{enumerate}
\item $\len(\Hom(\alpha,k))\leq \len(\alpha)+s-2$.
\item If $\len(\alpha)=2$, then $\Hom(\alpha,k)$ is an inessential curve of length one or a curve containing a unique bad snippet of bigon type. In the latter case we know that $\len_\textrm{red}(\Hom(\alpha,k))\leq \len_\textrm{red}(\alpha)$ and that the bigon snippet does not intersect $\partial_h N$.
\item If $\len(\alpha)>2$, then $\Hom(\alpha[k-1:k+2],1)_\textrm{trim}$ consists of carried snippets only.
\item Assume that $\len(\alpha)>2$, that $\alpha[k]$ turns right, that $\alpha[k-1]$ is in efficient position, and that $\len(\Hom(\alpha[k-1:k+2],1))> 1$. Then $\Hom(\alpha,k)[k-1]$ is a right-turning trigon of type $\mathbb{B}(h,t)$, $\mathbb{S}(h,t,1)$ or $\mathbb{S}(h,t,3)$.
\end{enumerate}
Moreover, if $\len(\alpha)>2$ and $\alpha[k-1]$ and $\alpha[k+1]$ are both in efficient position, then one of the following two statements holds.
\begin{enumerate}[I.]
\item $\len(\Hom(\alpha[k-1:k+2],1))= 1$ and $\Hom(\alpha,k)[k-1]$ is a bigon snippet of type $\mathbb{B}(h,h)$ or $\mathbb{S}(h,h,0)$. Furthermore, it holds that \[\len_\mathrm{red}(\Hom(\alpha,k))\leq \len_\mathrm{red}(\alpha)-1.\]
\item $\len(\Hom(\alpha[k-1:k+2],1))> 1$ and $\Hom(\alpha[k-1:k+2],1)$ contains exactly two bad snippets which are trigon snippets turning into the same direction. Furthermore, it holds that $\len_\mathrm{red}(\Hom(\alpha,k))\leq \len_\mathrm{red}(\alpha)$.
\end{enumerate}
\end{lem}

\begin{proof}
We prove the statements of the lemma in order. Suppose that $B \subset R \in \mathcal{R}_\textrm{comp}$ is the bigon cut off by the snippet $\alpha[k]$. Then the homotopy increases the number of snippets by $|\partial B \cap \partial^2 \mathcal{R}|-1 \leq s-2$. Thus, we see that $\len(\Hom(\alpha,k))\leq \len(\alpha)+s-2$, giving \textit{1}.

Let us now assume that $\len(\alpha)=2$. Then either $\alpha[k](0)$ and $\alpha[k](1)$ lie on the same or on two different horizontal sides of a branch or switch rectangle (see Figure \ref{Short R(h,h) picture}). In the first case, $\Hom(\alpha,k)$ is an inessential curve of snippet length one. For the second case, that is if $\alpha[k-1]$ is a dual snippet, let us assume that $\alpha[k]$ is turning right. Then the region cut off by $\alpha[k-1]$ on the right-hand side looses two outward-pointing corners under the homotopy. Thus, $\Hom(\alpha,k)[k-1]$ is a bigon snippet whose boundary points lie inside the tie neighbourhood and \[\len_\textrm{corn}(\Hom(\alpha,k)[k-1])=\len_\textrm{corn}(\alpha[k-1]).\] Further, by definition of the corner length of the snippet $\alpha[k]$, we have that \[\len_\textrm{corn}(\Hom(\alpha,k))-\len_\textrm{corn}(\Hom(\alpha,k)[k-1])= \len_\textrm{corn}(\alpha[k]).\] Since no vertical duals are affected by the homotopy, this gives \textit{2}.

\begin{figure}[htbp] 
  \begin{minipage}[b]{0.99\linewidth}
    \centering
\begin{tikzpicture}[scale=0.25]
\draw [very thick] (-4,6) rectangle (0,-6);
\draw [very thick](0,3) -- (2,3);
\draw [very thick](0,6) -- (2,6);
\draw [very thick] (0,-2) rectangle (4,-6);
\draw [very thick] (4,-2) rectangle (8,-6) node (v1) {};
\draw [very thick] (4,-9) rectangle (4,-9);
\draw [thick, dotted] plot[smooth, tension=.7] coordinates {(8,-2) (11.5,-1.625) (16,0) (19,-1.5) (19,-6.5) (16,-8) (11.5,-6.375) (v1)};
\draw [thick, dotted] plot[smooth, tension=.7] coordinates {(8,-3) (11.4996,-2.8304) (17,-2) (17,-6) (11.523,-5.1619) (8,-5)};
\draw [thick, dotted] plot[smooth, tension=.7] coordinates {(2,3) (6,4) (9,6) (11,11)};
\draw [thick, dotted](11,11) -- (14,11);
\draw [thick, dotted] plot[smooth, tension=.7] coordinates {(14,11) (16,8) (19,7)};
\draw [very thick] (15,-4) ellipse (0.5 and 0.5);
\draw [very thick, densely dashed, blue] plot[smooth cycle, tension=.7] coordinates {(1,-2) (1.25,0) (2.75,0) (3,-2) (2.75,-4) (1.25,-4)};
\draw [very thick, densely dashed, blue](5,-2) -- (5,-6);
\draw [very thick, densely dashed, blue] plot[smooth, tension=.7] coordinates {(5,-6) (6,-7) (10,-7.5) (19.5,-9) (19.5,1) (10,-0.5) (6,-1) (5,-2)};
\node at (-5,0) {\tiny{$t$}};
\node at (-2,-7) {\tiny{$h$}};
\node at (-1,0) {\tiny{$v$}};
\draw [very thick](8,-2) -- (8.5,-2);
\draw [very thick](8,-3) -- (8.5,-3);
\draw [very thick](8,-5) -- (8.5,-5);
\draw [very thick](8,-6) -- (8.5,-6);
\end{tikzpicture}
    \caption{Homotopies of type $\mathbb{R}(h,h)$ for a curve of snippet length two.} 
    \label{Short R(h,h) picture}
  \end{minipage}
\end{figure}
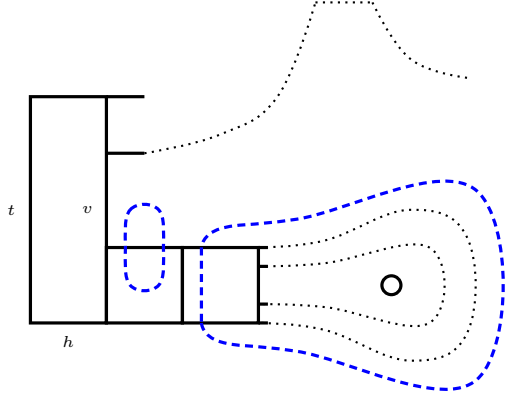

For the remainder of this proof we assume that $\len(\alpha)>2$. We distinguish two cases: either $\alpha[k](0)$ and $\alpha[k](1)$ lie in the same component of $\partial \mathcal{R}- \partial^2 \mathcal{R}$, or they do not (see Figure \ref{Homotopy of type B(R,h,h) picture}).

\begin{figure}[htbp] 
  \begin{minipage}[b]{0.99\linewidth}
    \centering
\begin{tikzpicture}[scale=0.35]
\draw [very thick] (-6,1) rectangle (-3.5,-3) node (v3) {};
\node at (4,1.75) {\tiny{$h$}};
\draw [very thick](-6,1) node (v1) {} -- (9.5,1) node (v2) {};
\draw [very thick] plot[smooth, tension=.7] coordinates {(v1) (-6.5,2) (-7,2.5)};
\draw [very thick] plot[smooth, tension=.7] coordinates {(v2) (10,2) (10.5,2.5)};
\draw [very thick] plot[smooth, tension=.7] coordinates {(-7,2.5) (-4.5,3.5) (-3,5)};
\draw [very thick] plot[smooth, tension=.7] coordinates {(10.5,2.5) (8,3.5) (6.5,5)};
\node at (-7.5,1.5) {\tiny{$v$}};
\draw [very thick] (4.5,1) rectangle (9.5,-3);
\draw [very thick, densely dashed, blue] plot[smooth, tension=.7] coordinates {(-5,-1.5)  (-4,2.5) (2,2.5)  (3,-1.5)};
\draw [very thick, densely dashed, cyan] plot[smooth, tension=.4] coordinates {(-4.75,-1.5) (-4.25,0.5)   (2.25,0.5) (2.75,-1.5)};
\draw [very thick,] (v3) rectangle (-1,1) node (v4) {};
\draw [very thick,] (v4) rectangle (1.5,0);
\draw [very thick,](-1.5,-1.5) -- (-0.5,-1.5);
\node at (-7.5,-1){\tiny{$t$}};
\draw [very thick,] (1.5,1) rectangle (4.5,-3);
\draw [very thick,](1,-1.5) -- (2,-1.5);
\draw [very thick, densely dashed, blue] plot[smooth, tension=.7] coordinates {(5.5,-1.5) (6,2.5) (8.5,2.5) (9,-1.5)};
\draw [very thick, densely dashed, cyan] plot[smooth, tension=.4] coordinates {(6,-1.5) (6.25,0.5) (8.25,0.5) (8.5,-1.5)};
\end{tikzpicture}
    \caption{Two different kinds of trigon snippets of type $\mathbb{R}(h,h)$.} 
    \label{Homotopy of type B(R,h,h) picture}
  \end{minipage}
\end{figure}
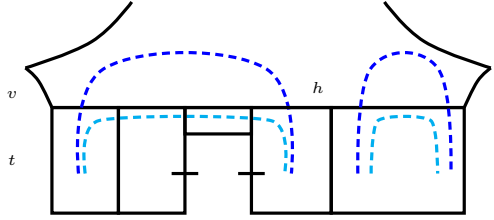

If $\alpha[k](0)$ and $\alpha[k](1)$ lie in the same component of $\partial \mathcal{R}- \partial^2 \mathcal{R}$, then $\Hom(\alpha[k-1:k+2],1)$ consists of a single snippet only, giving \textit{3}. If, in addition, $\alpha[k-1]$ and $\alpha[k+1]$ are both in efficient position, they must be ties of the same branch or switch rectangle. Thus, $(\Hom(\alpha[k-1:k+2],1))$ is a bigon snippet of type $\mathbb{B}(h,h)$ or $\mathbb{S}(h,h,0)$. Since 
\begin{align*}
& \len_\textrm{corn}(\Hom(\alpha[k-1:k+2],1)) = \len_\textrm{corn}(\alpha[k-1]) \\  \leq & \len_\textrm{corn}(\alpha[k-1:k+2] - \len_\textrm{corn}(\alpha[k+1],
\end{align*}
and no blockers are affected by the homotopy, we see that \[\len_\textrm{red}(\Hom(\alpha,k))\leq \len_\textrm{red}(\alpha)-1\] or $\len_\textrm{red}(\Hom(\alpha,k))\leq \len_\textrm{red}(\alpha)-3$ if $\Hom(\alpha,k)[k-1]$ is a snippet of type $\mathbb{B}(h,h)$ or $\mathbb{S}(h,h,0)$ respectively. This gives \textit{I}.

If $\alpha[k](0)$ and $\alpha[k](1)$ lie in different components of $\partial \mathcal{R}- \partial^2 \mathcal{R}$, then $\Hom(\alpha[k-1:k+2],1)_\textrm{trim}$ is empty or consists of snippets that are all parallel to $\partial_h N$. Thus, they are all carried, giving \textit{3}. If $\alpha[k-1]$ is in efficient position, it must be a tie of a branch or switch rectangle. Assuming that $\alpha[k]$ is turning right, the index of the region cut off by $\alpha[k-1]$ on its right-hand side increases by $1/4$. This implies that $\Hom(\alpha,k)[k-1]$ is a right-turning trigon snippet inside a branch or switch rectangle. We note that this trigon snippet cannot be of type $\mathbb{S}(h,v,2)$ since the boundary of the trigon contains exactly one point of $\partial ^2 \mathcal{R}$ less than the boundary of the region cut off by $\alpha[k-1]$. As $\alpha[k-1]$ is a tie inside a branch or switch rectangle, the boundary of the region cut off by $\alpha[k-1]$ contains two or four points of $\partial ^2 \mathcal{R}$. This gives \textit{4}.

We further note that $\len_\textrm{corn}(\alpha[k-1])=\len_\textrm{corn}(\Hom(\alpha,k)[k-1])$. Recalling the definition of the corner length of the snippet $\alpha[k]$, we see that \[\len_\textrm{corn}(\alpha[k-1:k+2],1)=\len_\textrm{corn}(\alpha[k-1:k+2]).\] As no dual verticals of $\alpha$ are affected by the homotopy, this gives \textit{II}.
\end{proof}

\begin{rem}\label{Required time for R(h,h) bigon homotopies}
Again, we see that $\Hom(\alpha,k)$ can be computed in $\mathit{O}(|\chi(S)|)$ time if $\alpha[k]$ is a snippet of type $\mathbb{R}(h,h)$. Lemma \ref{Lemma B(R,h,h)} implies that \[\len(\Hom(\alpha[k-1:k+2],1))\leq s\] and that $\Hom(\alpha[k-1:k+2],1) \subset N$. Thus, the cutting sequence can be adjusted in $\mathit{O}(|\chi(S)|)$ time, whereas the winding numbers remain unchanged.
\end{rem}

\begin{lem}\label{Lemma B(h,h,0)}
Suppose that $\alpha \subset S$ is a snippet-decomposed arc or curve satisfying $\len(\alpha)\geq 2$. Suppose that $\alpha[k] \subset \alpha_\textrm{trim}$ is a bad snippet of bigon type $\mathbb{B}(h,h)$. If $\len(\alpha)=2$, then $\len(\Hom(\alpha,k))=1$. Else, that is if $\len(\alpha)\geq 3$, we have that $\len(\Hom(\alpha,k))=\len(\alpha)-2$. If $\alpha[k-1]$ and $\alpha[k+1]$ are in efficient position, then one of the following three statements holds.
\begin{enumerate}
\item $\Hom(\alpha[k-1:k+2],1)$ is of type $\mathbb{R}(h,v)$, $\mathbb{R}(h,h)$, or $\mathbb{R}(v,v)$ and \[\len_\textrm{red}(\Hom(\alpha,k))\leq \len_\textrm{red}(\alpha)+1.\]
\item $\Hom(\alpha,k)$ is in efficient position.
\item $\Hom(\alpha,k)$ is a peripheral curve consisting of a single snippet only.
\end{enumerate}
\end{lem}

\begin{proof}
We prove the statements of the lemma in order.
Since $\alpha[k]$ is a bad snippet of type $\mathbb{B}(h,h)$, we know that $\alpha[k-1]$ and $\alpha[k+1]$ are contained in the same complementary region $R \in \mathcal{R}_\textrm{comp}$ (see Figure \ref{Homotopy of type B(B,h,h) picture}). This implies that the arc $\alpha[k-1:k+1]$ is replaced by a single snippet $\alpha'[k-1]$. If $\alpha[k-1]=\alpha[k+1]$ and this snippet is in efficient position, then $|\wind(\alpha[k-1])| = l \cdot n$ where $n=|\partial^2 R|$ and $l \in \mathbb{Z}-\{0\}$. As $\wind(\Hom(\alpha,k))=\wind(\alpha[k-1])$, we see that $\Hom(\alpha,k)$ is a peripheral curve consisting of a single snippet only, giving \textit{3}.

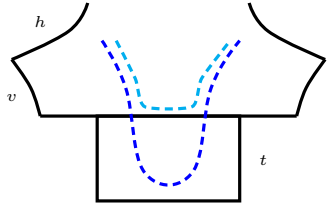
\begin{figure}[htbp] 
  \begin{minipage}[b]{0.99\linewidth}
    \centering
\begin{tikzpicture}[scale=0.25]
\draw [very thick] (-3.5,1.5) rectangle (4,-3);
\node at (-6.5,6.5) {\tiny{$h$}};
\draw [very thick](-6.5,1.5) node (v1) {} -- (7,1.5) node (v2) {};
\draw [very thick] plot[smooth, tension=.7] coordinates {(v1) (-7,3) (-8,4.5)};
\draw [very thick] plot[smooth, tension=.7] coordinates {(v2) (7.5,3) (8.5,4.5)};
\draw [very thick] plot[smooth, tension=.7] coordinates {(-8,4.5) (-5,6) (-4,7.5)};
\draw [very thick] plot[smooth, tension=.7] coordinates {(8.5,4.5) (5.5,6) (4.5,7.5)};
\draw [blue, densely dashed, very thick] plot[smooth, tension=.7] coordinates {(-3.25,5.5) (-2,3) (-1,-1.5)   (1.5,-1.5) (2.5,3) (4,5.5)};
\draw [cyan, densely dashed, very thick] plot[smooth, tension=.4] coordinates {(-2.5,5.5) (-1.5,3.5) (-1,2) (1.5,2) (2,3.5) (3.5,5.5)};
\node at (-8,2.5) {\tiny{$v$}};
\node at (5.24,-0.86) {\tiny{$t$}};
\end{tikzpicture}
    \caption{A homotopy of type $\mathbb{B}(h,h)$.} 
    \label{Homotopy of type B(B,h,h) picture}
  \end{minipage}
\end{figure}

For the remainder of this proof we assume that $\alpha[k-1]\neq \alpha[k+1]$ and that $\alpha[k-1]$ and $\alpha[k+1]$ are in efficient position. 
As in the case of bigon snippets of type $\mathbb{S}(h,h,0)$ in the proof of Lemma \ref{Easy Bigon bounds}, we distinguish the following two cases: either at least one of the snippets $\alpha[k-1]$ and $\alpha[k+1]$ has corner length $2s$, or both have a corner length that is strictly smaller than $2s$. We recall that in the latter case, efficient position implies that both snippets must be horizontal duals.

Let us first assume that at least one of the snippets $\alpha[k-1]$ and $\alpha[k+1]$ has corner length $2s$. Without loss of generality, we assume that this is the snippet $\alpha[k-1]$. Then \[\len_\textrm{corn}(\Hom(\alpha[k-1:k+2],1)) \leq 2s \leq \len_\textrm{corn}(\alpha[k-1]).\] We note that $\len_\textrm{corn}(\alpha[k])=1$ and $\len_\textrm{corn}(\alpha[k+1])\geq 1$. Thus, we see that \[\len_\textrm{corn}(\Hom(\alpha[k-1:k+2],1)) \leq \len_\textrm{corn}(\alpha[k-1:k+2])-2.\] As snippets of corner length $2s$ cannot be part of any blocker, the homotopy reduces the number of blockers of the underlying arc or curve by at most one. This implies that $\len_\textrm{red}(\Hom(\alpha,k))\leq \len_\textrm{red}(\alpha)$, giving \textit{1} and \textit{2}.

Let us now assume that the corner length of $\alpha[k-1]$ and $\alpha[k+1]$ is strictly smaller than $2s$. As remarked previously, efficient position implies that both of these snippets cut off regions of index zero on one of their sides. If they cut off a region of index zero on the same side, an index-argument shows that $\Hom(\alpha[k-1:k+2],1)$ is in efficient position. Therefore, let us assume that they cut off the regions of index zero on different sides. Then $\Hom(\alpha[k-1:k+2],1)$ is a bigon of type $\mathbb{R}(h,h)$. Without loss of generality, we may assume that $\len_\textrm{corn}(\alpha[k-1])\leq \len_\textrm{corn}(\alpha[k+1])$. This implies that \[\len_\textrm{corn}(\Hom(\alpha[k-1:k+2],1))\leq \len_\textrm{corn}(\alpha[k+1])-1.\] As $\len_\textrm{corn}(\alpha[k])=1$ and $\len_\textrm{corn}(\alpha[k+1])\geq 1$, we see that \[\len_\textrm{corn}(\Hom(\alpha[k-1:k+2],1)) \leq \len_\textrm{corn}(\alpha[k-1:k+2])-3.\] Since the homotopy decreases the number of blockers by at most two, we conclude that $\len_\textrm{red}(\Hom(\alpha,k))\leq \len_\textrm{red}(\alpha)+1$, giving \textit{1}.
\end{proof}

\begin{rem}\label{Required time for local homotopies}
Arguing as in the case of snippets of type $\mathbb{S}(h,h,0)$, we see that $\Hom(\alpha,k)$ can be computed in $\mathit{O}(|\chi(S)|)$ time if $\alpha[k]$ is a snippet of type $\mathbb{B}(h,h)$. 

Combining this with Lemma \ref{Required time for easy bigon homotopies} and Remarks \ref{Required time for trigon homotopies}, \ref{Required time for bigon S(t,v,1) homotopy}, \ref{Required time for S(t,t,2) bigon homotopies}, and \ref{Required time for R(h,h) bigon homotopies}, we conclude the following: suppose that $\alpha$ is a snippet-decomposed arc or curve with a bad snippet $\alpha[k] \subset \alpha_\textrm{trim}$ of arbitrary type. Then $\Hom(\alpha,k)$ can be computed in $\mathit{O}(|\chi(S)|)$.
\end{rem}

From the previous six lemmas about local bigon homotopies we deduce the following.

\begin{cor} \label{Bigon homotopies at end of arc}
Suppose that $\alpha \subset S$ is a snippet-decomposed arc satisfying $\len(\alpha) \geq 3$. We further assume that $\alpha_\textrm{trim}$ contains exactly one bad snippet, $\alpha[-2]$, which is a bigon snippet of arbitrary type. Then the following statements hold.
\begin{enumerate}
\item $\Hom(\alpha,-2)_\textrm{trim}$ contains at most one bad snippet, which must be of trigon type.
\item $\len(\Hom(\alpha,-2)_\textrm{trim})\leq \len(\alpha_\textrm{trim})+s-2$.
\item $\len_\textrm{red}(\Hom(\alpha,-2)_\textrm{trim})\leq \len_\textrm{red}(\alpha_\textrm{trim})+2s$.
\end{enumerate}
\end{cor}

\begin{proof}
The first two statements follow directly from Lemmas \ref{Easy Bigon bounds}, \ref{Lemma B(S,t,v,1)}, \ref{Lemma B(S,t,t,2)}, \ref{Lemma B(R,h,h)} and \ref{Lemma B(h,h,0)} when setting $k=-2$ and observing that $\alpha[k+1] \not\subset \alpha_\textrm{trim}$.

To prove that $\len_\textrm{red}(\Hom(\alpha,-2)_\textrm{trim})\leq \len_\textrm{red}(\alpha_\textrm{trim})+2s$, we note that $\Hom(\alpha,-2)_\textrm{trim}$ contains at most one blocker less than $\alpha_\textrm{trim}$. This follows from the fact that $\alpha[k+1]$ is not contained in $\alpha_\textrm{trim}$. If $\alpha[-3] \subset \alpha_\textrm{trim}$, we further know that $\len_\textrm{corn}(\alpha[-3])\geq 1$. If $\len(\Hom(\alpha[-3:],-2))=1$, this implies that 
\begin{align*}
& \len_\textrm{red} ( \Hom(\alpha,-2)_\textrm{trim}) =  \len_\textrm{red}(\alpha[1:-3]) \\ \leq & \len_\textrm{red}(\alpha_\textrm{trim}) -1 +2  \leq  \len_\textrm{red}(\alpha_\textrm{trim})+1.
\end{align*}
If $\len(\Hom(\alpha[-3:],-2))\neq 1$, $\alpha[-2]$ must be a bigon of type $\mathbb{S}(t,v,1)$, $\mathbb{S}(t,t,2)$, or $\mathbb{R}(h,h)$. Thus, no vertical duals of $\alpha$ are affected by the homotopy. As \[\len_\textrm{corn}(\Hom(\alpha[-3:],-2)) \leq \len_\textrm{corn}(\alpha[-3:])+2s,\] we see that $\len_\textrm{red}(\Hom(\alpha,-2)_\textrm{trim})\leq \len_\textrm{red}(\alpha_\textrm{trim})+2s$ in this case.
\end{proof}

\section{The algorithm \texttt{BigArc}}

Building on the results from the previous section, we now present the algorithm \texttt{BigArc}. This algorithm 
\begin{itemize}
\item takes as input a surface $S$ of positive complexity, a tie neighbourhood $N \subset S$ of a large train track in $S$, as well as an arc $\alpha \subset S$ whose inside contains at most one bad snippet, which must be its penultimate snippet, and
\item outputs an arc $\alpha'$ that is homotopic to $\alpha$ and in efficient position in $\alpha'[1:-1]$.
\end{itemize}

The formal statement of \texttt{BigArc} is given in Algorithm \ref{BigArc}.
Corollary \ref{Bigon homotopies at end of arc} shows that if $\alpha$ is an arc with a unique bad snippet, $\alpha[-2]$, of bigon type, then applying one local bigon homotopy is sufficient to obtain an arc that is a valid input for the algorithm \texttt{TrigArc}. Thus, efficient position for the inside of $\alpha$ is achieved by employing the algorithm \texttt{TrigArc}.

\begin{algorithm} 
    \SetKwInOut{Input}{Input}
    \SetKwInOut{Output}{Output}
	\DontPrintSemicolon
%    \textbf{Function} $\texttt{BigArc} (\alpha)$:\;
%    \BlankLine
    \Input{A surface $S$ of positive complexity, a tie neighbourhood $N \subset S$ of a large train track in $S$, and an arc $\alpha \subset S$ such that $\alpha[1:-1]$ contains at most one bad snippet, which must be the snippet $\alpha[-2]$.}
    \Output{An arc $\alpha''$ which is homotopic, relative its endpoints, to $\alpha$ and is in efficient position in its inside.}
    \BlankLine
    $\alpha'= \alpha$\;
    \If {$\len(\alpha')>2$ and $\alpha'[-2]$ is a bigon}
    {$\alpha'=\Hom(\alpha',-2)$}
    \Return $\texttt{TrigArc}(S,N,\alpha')$
    \caption{\texttt{BigArc} - Achieving efficient position for the inside of an arc that contains at most one bad snippet, which must be its penultimate snippet.}
    \label{BigArc}
\end{algorithm}

\begin{lem}\label{Functional correct BigArc}
The algorithm \texttt{BigArc} is correct. On an input $(S,N,\alpha)$, the algorithm halts in $\mathit{O}(\chi(S)^2(\len(\alpha_\textrm{trim})+s-1))$ time. For $\alpha''=\texttt{BigArc}(S,N,\alpha)$, we have that \[\len_\textrm{red}(\alpha''_\textrm{trim}) \leq \len_\textrm{red}(\alpha_\textrm{trim})+4s.\]
\end{lem}

\begin{proof}
The correctness of this algorithm follows from Corollary \ref{Bigon homotopies at end of arc} and Lemma \ref{TrigArc funtionally correct}. Remark \ref{Required time for local homotopies} implies that examining the last three snippets of the input arc and replacing them, if required, by $\Hom(\alpha[-3:],1)$, takes $\mathit{O}(|\chi(S)|)$ time. Further, any local bigon homotopy increases the snippet count of the underlying arc by at most $s-2$. Thus, Lemma \ref{TrigArc funtionally correct} implies that the algorithm halts within a further $\mathit{O}(\chi(S)^2 (\len(\alpha_\textrm{trim})+s-1))$ operations and that the reduced corner length of the output is bounded by $\len_\textrm{red}(\alpha_\textrm{trim})+2s+2s$.
\end{proof}

\begin{cor}\label{BigArc running time in reduced corner length}
For any input $(S,N,\alpha)$, the algorithm \texttt{BigArc} halts in \[\mathit{O}(|\chi(S)|^3 \cdot (\len_\textrm{red}(\alpha_\textrm{trim})+1))\] time.
\end{cor}

\begin{proof}
As $\alpha_\textrm{trim}$ contains at most one bad snippet, Lemma \ref{reduced corner length vs snippet length} implies that
\begin{align*}
\len(\alpha_\textrm{trim})+s-1 \leq \len_\textrm{red}(\alpha_\textrm{trim})+s,
\end{align*}
which is sufficient to prove the claim.
\end{proof}

%%%%%%%%%%%%%%%%%%%%%%%%%%%%%%%%%%%%%%%%%%%%%%%%%%%%%%%%%%%%%%%%%%%%%%%%%%%%%%%%%%%%%%%%%%%%%%%%%%%%%%%%%%%%%%%%%%%%%%%%%%%%%%%%%%%%%%%%%%%%%%%%%%%%%%%%%%%%%%%%%%%%%%%%%%%%%%%%%%%%%%%%%%%%%%%%%%%%%%%%%%%%%%%%%%%%%%%%%%%

\chapter{Obtaining efficient position for all snippets but one}

Building on the results from the previous two chapters, we now present the algorithm \texttt{ReduceToTwoBadSnippets} that ``almost'' yields efficient position for properly immersed arcs and curves.  For a formal statement of this algorithm we point the reader to Algorithm \ref{ReduceToTwoBadSnippets}. Suppose that $\alpha$ is a snippet-decomposed arc or curve. On input $\alpha$, $\texttt{ReduceToTwoBadSnippets}$ outputs an arc or curve $\alpha'$ homotopic to $\alpha$ such that $\alpha'[1:-1]$ is in efficient position. If $\alpha$ is a properly immersed, essential arc, this implies that $\alpha'$ is in efficient position. If $\alpha$ is a curve,  then $\alpha'$ can have up to two bad snippets: its very first and very last snippet. To reduce the number of bad snippets of the curve $\alpha'$ further, some additional work is needed. To decrease the number of bad snippets to be less than two, a simple modification of our previous discussions is sufficient. We present the corresponding algorithm \texttt{ReduceToOneBadSnippet} in this chapter. For a formal statement of this algorithm we point the reader to Algorithm \ref{ReduceToOneBadSnippet}. However, eliminating the last bad snippet turns out to be quite complex: When applying local homotopies as before, the number of bad snippets alternates between one and two. Moreover, these snippets might change their turning direction, so monitoring the number of right or left duals does not work. These issues will be resolved by applying local homotopies in a specific order, ensuring that the reduced corner length of the underlying curve decreases sufficiently often. As this discussion is rather delicate, it is postponed to the subsequent chapter.

Throughout this chapter, we assume $S = S_{g,b}$ to be a surface satisfying $\xi(S)=3g-3+b \geq 1$. We further assume that $\tau \subset S$ is a large train track and that $N=N(\tau) \subset S$ is a tie neighbourhood of $\tau$ in $S$. As usual, we do not distinguish between snippets and their strong snippet homotopy classes unless otherwise stated. For any snippet $a \subset R \in \mathcal{R}$, we always assume that $a$ has minimal self-intersection, intersects the one-skeleton of $\mathcal{R}$ perpendicularly and misses the corners of $\mathcal{R}$. We further recall that arcs and curves in $S$ are assumed to be self-transverse and transverse to $\partial \mathcal{R}$. Moreover, if $\alpha$ is an arc, we assume that $\alpha(0)$ and $\alpha(1)$ lie in $\partial \mathcal{R}$. Thus, arcs and curves in $S$ admit canonical decompositions into snippets.

\section{Reducing the number of bad snippets to two}

\begin{algorithm} 
    \SetKwInOut{Input}{Input}
    \SetKwInOut{Output}{Output}
	\DontPrintSemicolon
%    \textbf{Function} $\texttt{ReduceToTwoBadSnippets}(\alpha)$:\;
%    \BlankLine
    \Input{A surface $S$ of positive complexity, a tie neighbourhood $N \subset S$ of a large train track in $S$, and an arc or curve $\alpha \subset S$.}
    \Output{An arc or curve $\alpha''$ which is homotopic to $\alpha$ (relative its endpoints if applicable) such that $\alpha''[1:-1]$ is in efficient position.}
    \BlankLine
    $\alpha'= \alpha$\;
    \If{$\len(\alpha) \leq 2$}{\Return $\alpha$}
    \For{$k$ in $[3, \dots, \len(\alpha)-1]$}
    {
    $\alpha'= \texttt{BigArc} (S,N,\alpha'[:-\len(\alpha)+k])\cdot \alpha[-\len(\alpha)+k:]$\;
    }
    \Return $ \texttt{BigArc} (S,N,\alpha') $
    \caption{\texttt{ReduceToTwoBadSnippets} - Achieving efficient position for all but the first and last snippet of a given arc or curve.}
    \label{ReduceToTwoBadSnippets}
\end{algorithm}

\begin{lem}\label{Lemma ReduceToTwoBadSnippets}
The algorithm \texttt{ReduceToTwoBadSnippets} is correct. On an input $ (S,N,\alpha)$, the algorithm halts in $\mathit{O}(\chi(S)^4 \cdot \len(\alpha)^2)$ time. For \[\alpha''=\texttt{ReduceToTwoBadSnippets}(S,N,\alpha),\] we have that $\len_\textrm{red}(\alpha''_\textrm{trim}) \leq (\len(\alpha)-1)\cdot 6s$.
\end{lem}

\begin{proof}
We prove the correctness of the algorithm by going through the pseudocode line by line. By assumption on the input of the algorithm \texttt{ReduceToTwoBadSnippets}, we are given an arc or curve $\alpha$. If $\len(\alpha) \leq 2$, then $\alpha[1:-1]$ is empty. Hence, $\alpha$ is a valid output for \texttt{ReduceToTwoBadSnippets}. 

Let us now assume that $\len(\alpha)>2$. The statement of the for-loop is executed $\len(\alpha)-3$ many times. In the first iteration, that is for $k=3$, we apply \texttt{BigArc} to the subarc $\alpha[0:3]=\alpha'[0:3]$, whose inside consists of the single snippet $\alpha[1]$. Thus, $\alpha'[0:3]$ is a valid input for \texttt{BigArc}. By Lemma \ref{Functional correct BigArc}, the  arc $\texttt{BigArc}(S,N,\alpha'[0:3])$ is homotopic, relative its endpoints, to $\alpha[0:3]$ and is in efficient position in its inside. Therefore, the arc or curve $\alpha'=\texttt{BigArc}(S,N,\alpha'[0:3]) \cdot \alpha[3:]$ constitutes an improvement compared to the input arc or curve $\alpha$ as follows: even though the first snippet of $\alpha'$ might still be bad, we can now be sure that all the bad snippets of $\alpha'[1:]$ are contained in the subarc $\alpha'[-\len(\alpha)+2:]$. For $\alpha$, this is not true. There, we only know that all the bad snippets of $\alpha[1:]$ are contained in the subarc $\alpha[-\len(\alpha)+1:]$. Thus, after one iteration of the for-loop, we have shortened the subarc that is guaranteed to hold all bad snippets of $\alpha'$ but $\alpha'[0]$ by one snippet. In particular, we see that $\alpha'[:-\len(\alpha)+4]$ now contains at most one bad snippet in its inside, which is the snippet $\alpha'[-\len(\alpha)+2]= \texttt{BigArc}(S,N,\alpha[0:3])[-1]$. This implies that $\alpha'[:-\len(\alpha)+4]$ is a valid input for \texttt{BigArc} in the next iteration of the for-loop. By induction, all further arcs $\alpha'[:-\len(\alpha)+k]$ for $k$ in $[3, \dots, \len(\alpha)-1]$ are valid inputs for \texttt{BigArc}. We remark that we exit the for-loop once $\alpha'$ has at most three bad snippets: $\alpha'[0]$, $\alpha'[-2]$, and $\alpha'[-1]$. Thus, $\alpha'$ at this stage is a valid input for \texttt{BigArc}, which concludes the proof of the correctness of the algorithm \texttt{ReduceToTwoBadSnippets}. We remark that we excluded the last iteration of \texttt{BigArc} from the for-loop as $\alpha'[:0]$ equals, by Python convention, the empty subarc.

To analyse the running time of the algorithm \texttt{ReduceToTwoBadSnippets} we note that it takes at most $\mathit{O}(|\chi(S)|)$ time to advance to the for-loop in the pseudocode.
To analyse the number of operations required to execute the for-loop and the final call of $\texttt{BigArc}$, we recall that \[\len_\textrm{red}(\alpha_\textrm{trim}) \leq \len_\textrm{corn}(\alpha_\textrm{trim}) \leq 2s \cdot \len(\alpha_\textrm{trim}).\] Corollary \ref{BigArc running time in reduced corner length} implies that the first iteration of the for-loop is executed within $\mathit{O}(|\chi(S)|^3 \cdot (2s+1))$ time, while increasing the reduced corner length of $\alpha'[:-\len(\alpha)+3]_\textrm{trim}$ by at most $4s$. Assuming the maximum corner length of $2s$ for the last snippet of $\alpha'[:-\len(\alpha)+3]$, we see that \[\len_\textrm{red}(\alpha'[:-\len(\alpha)+4]_\textrm{trim}) \leq 2s+6s.\] From this we conclude that the next iteration of the for-loop is executed within $\mathit{O}(|\chi(S)|^3 \cdot (8s+1))$ time and yields a subarc that satisfies \[\len_\textrm{red}(\alpha'[:-\len(\alpha)+5]_\textrm{trim}) \leq 2s + 2 \cdot(6s).\] By induction, the last iteration of the for-loop terminates within \[\mathit{O}(|\chi(S)|^3 \cdot (2s+(\len(\alpha)-3-1)\cdot 6s))\] time and yields an arc $\alpha'[:-1]= \alpha'[:-\len(\alpha)+\len(\alpha)-1]$ that satisfies \[\len_\textrm{red}(\alpha'[:-1]_\textrm{trim}) \leq 2s + (\len(\alpha)-3)\cdot 6s.\] One final application of \texttt{BigArc} then returns an arc or curve $\alpha''$ of the desired form within $\mathit{O}(|\chi(S)| \cdot (2s+(\len(\alpha)-3)\cdot 6s))$ time, and we see that 
\[\len_\textrm{red}(\alpha''[1:-1]) \leq 2s + (\len(\alpha)-2)\cdot 6s.\]

Summarizing, this gives us a total of $\len(\alpha)-2$ many applications of \texttt{BigArc}, each being computed in $\mathit{O}(|\chi(S)|^3 \cdot (2s+(\len(\alpha)-3)\cdot 6s))$ time. Thus, the algorithm \texttt{ReduceToTwoBadSnippets} halts within
\begin{align*}
\mathit{O}(|\chi(S)|^3 \cdot (2s+(\len(\alpha)-3)\cdot 6s)) \cdot (\len(\alpha) -2) 
\end{align*}
time, which is bounded by $\mathit{O}(\chi(S)^4 \cdot \len(\alpha)^2)$ as $\len(\alpha)\geq 1$. 

Set $\alpha''=\texttt{ReduceToTwoBadSnippets}(S,N,\alpha)$. If $\alpha$ is an arc, then \[\alpha''_\textrm{trim}=\alpha''[1:-1],\] so
$\len_\textrm{red}(\alpha''_\textrm{trim}) \leq 2s + (\len(\alpha)-2)\cdot 6s$.
If $\alpha$ is a curve, then $\alpha''_\textrm{trim}=\alpha''$, so 
\begin{align*}
\len_\textrm{red}(\alpha''_\textrm{trim}) &\leq  \len_\textrm{red}(\alpha''[1:-1])+4s \\ & \leq 2s + (\len(\alpha)-2)\cdot 6s +4s \\ & \leq (\len(\alpha)-1)\cdot 6s.
\end{align*}
\end{proof}

\begin{cor}\label{Efficient position for arcs}
Suppose that $S = S_{g,b}$ is a surface satisfying $\xi(S)=3g-3+b \geq 1$. Suppose that $\tau \subset S$ is a large train track and $N=N(\tau)$ is a tie neighbourhood of $\tau$ in $S$. Suppose further that $\alpha \subset S$ is a properly immersed arc. Then $\texttt{ReduceToTwoBadSnippets}(S,N,\alpha)$ is in efficient position with respect to $N$ or consists of a single snippet. Moreover, $\texttt{ReduceToTwoBadSnippets}(S,N,\alpha)$ can be computed in polynomial time.
\end{cor}

\begin{proof}
Lemma \ref{No arcs of length 2} implies that there are no properly immersed arcs $\alpha \subset S$ whose snippet decomposition consists of exactly two snippets. As the first and last snippet of every arc $\alpha$ satisfying $\len(\alpha) \geq 3$ are in efficient position, the claim follows from Lemma \ref{Lemma ReduceToTwoBadSnippets}.
\end{proof}

\section{Reducing the number of bad snippets from two to one}

Suppose that $S$ is a surface of positive complexity, that $\tau \subset S$ is a large train track, and that $N=N(\tau)$ is a tie neighbourhood of $\tau$ in $S$. Suppose further that $\alpha \subset S$ is a properly immersed curve. In the last section we saw that $\texttt{ReduceToTwoBadSnippets}(S,N,\alpha)$ might still contain up to two bad snippets. In this section we explain how to obtain a curve that contains at most one bad snippet.

\subsection{The idea behind \texttt{ReduceToOneBadSnippet}}

Suppose that $\alpha$ is a properly immersed curve in $S$. If the curve \[\alpha'=\texttt{ReduceToTwoBadSnippets}(S,N,\alpha)\] contains two bad snippets, they will be the very first and very last snippet of $\alpha'$. Thus, if we cut $\alpha'$ open between its first and last snippet, we obtain an arc $\alpha''$ that is in efficient position in its inside. We then glue a copy of $\alpha[0]$ to the end of $\alpha''$. The resulting arc $\alpha'''=\alpha''\cdot \alpha[0]$ contains a unique bad snippet in its inside. As this is the penultimate snippet, the algorithm \texttt{BigArc} yields a homotopic arc which is in efficient position in its inside. By Remark \ref{local trigon homotopy fixes boundary} we know that the first two-thirds of $\alpha'''[0]$ and the last two-thirds of $\alpha'''[-1]$ remain unchanged when applying the algorithm \texttt{BigArc} to $\alpha'''$. Thus, we can remove the first and last half of the first and last snippet of $\texttt{BigArc}(S,N,\alpha''')$ respectively and glue the remaining arc along its boundary points to obtain a smooth curve that is homotopic to $\alpha$. This curve now contains at most one bad snippet. To formalize this process, we introduce some further notation.

\subsection{Subsets of snippets and gluings of arcs}
Suppose that $\alpha \subset S$ is a snippet-decomposed arc. Recall that for $0 \leq k < \len(\alpha)$ and $0 \leq \epsilon \leq \delta \leq 1$, we set \[\alpha[k]=\alpha[k:k+\epsilon] \cdot \alpha[k+\epsilon:k+\delta] \cdot \alpha[k+\delta:k+1].\]
Suppose that $\alpha[\epsilon:\delta]$ and $\alpha[|\alpha|-1+\epsilon:|\alpha|-1+\delta]$ are identical as parametrized arcs for some $0 \leq \epsilon < \delta \leq 1$. Let $\epsilon < \zeta < \delta$. By $\alpha[\zeta:|\alpha|-1+\zeta]/_{\sim}$ we denote the snippet-decomposed curve that we obtain by identifying the boundary points of $\alpha[\zeta:|\alpha|-1+\zeta]$ (see Figure \ref{Gluing of arc}). After a suitable reparametrisation, this is a smooth, properly immersed, snippet-decomposed curve in $S$.

\begin{figure}[htbp] 
  \begin{minipage}[b]{0.99\linewidth}
    \centering
\begin{tikzpicture}[scale=0.45]
\draw [very thick] (-4,4) rectangle (4,-2);
\draw [very thick, densely dashed, blue] plot[smooth, tension=.4] coordinates {(-1.9156,-1.9719) (-2.2807,2.5309) (-5.3795,2.9107) (-4.5,6) (4.5,5.5) (4.7853,-0.6321) (-0.7028,-0.8349) (-1.1687,4)};
\draw [very thick, densely dashed, cyan] plot[smooth cycle, tension=.4] coordinates {(-1.8291,3.0278)  (-4.7241,3.3863) (-3.9573,5.2666) (4,5) (4.5872,-1.1368) (-1.1321,-1.2028)};
\draw [very thick, red](-2.4,2) -- (-0.8,2);
\draw [very thick, red](-2.4,0) -- (-0.8,0);
\end{tikzpicture}
    \caption{Gluing of an arc whose first and last snippet have identical middle thirds. For clarity, small push-offs have been performed in the figure.} 
    \label{Gluing of arc}
  \end{minipage}
\end{figure}
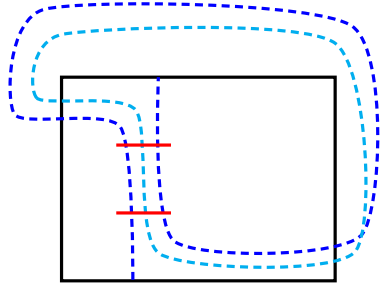

\subsection{The algorithm \texttt{ReduceToOneBadSnippet}}
As described earlier, the algorithm \texttt{ReduceToOneBadSnippet} reduces the number of bad snippets of a curve from two to at most one. For a formal statement of \texttt{ReduceToOneBadSnippet} we refer the reader to Algorithm \ref{ReduceToOneBadSnippet}.

\begin{algorithm} 
    \SetKwInOut{Input}{Input}
    \SetKwInOut{Output}{Output}
	\DontPrintSemicolon
%    \textbf{Function} $\texttt{ReduceToOneBadSnippet}(\alpha)$:\;
%    \BlankLine
    \Input{A surface $S$ of positive complexity, a tie neighbourhood $N \subset S$ of a large train track in $S$, and a curve $\alpha\subset S$ which contains at most two bad snippets $\alpha[0]$ and $\alpha[-1]$.}
    \Output{A curve $\alpha''$ that is homotopic to $\alpha$ such that $\alpha''$ contains at most one bad snippet.}
    \BlankLine
    \If{$\len(\alpha)<2$}{\Return $\alpha$}
    $\alpha'= \texttt{BigArc}(S,N,\alpha \cdot \alpha[0])$\; 
    \Return $\alpha'[1/2:\len(\alpha')-1/2]/_{\sim}$
    \caption{\texttt{ReduceToOneBadSnippet} - Reducing the number of bad snippets of a given curve to at most one.}
    \label{ReduceToOneBadSnippet}
\end{algorithm}

\begin{lem}\label{Lemma ReduceToOneBadSnippet}
The algorithm \texttt{ReduceToOneBadSnippet} is correct. On an input $(S,N,\alpha)$, the algorithm terminates in $\mathit{O}(\chi(S)^2(\len(\alpha)+s))$ time. For \[\alpha''=\texttt{ReduceToOneBadSnippet}(S,N,\alpha),\] we have that $\len_\textrm{red}(\alpha'') \leq \len_\textrm{red}(\alpha)+7s$.
\end{lem}

\begin{proof}
By our assumptions on the input $\alpha$, the arc $\alpha'=\alpha \cdot \alpha[0]$ is a valid input for the algorithm \texttt{BigArc} with a single bad snippet, $\alpha'[\len(\alpha)-1]$, as its penultimate snippet. We remark that $\alpha'[0]=\alpha'[\len(\alpha')-1]$. Remark \ref{local trigon homotopy fixes boundary} implies that the subarcs $\alpha'[:2/3]$ and $\alpha'[-2/3:]$ remain fixed under the algorithm \texttt{BigArc}. Set $\alpha'=\texttt{BigArc}(S,N,\alpha \cdot \alpha[0])$. Lemma \ref{Functional correct BigArc} implies that $\alpha'[0]$ and $\alpha'[-1]$ are the only snippets of $\alpha'$ that can be bad. We further know that their middle thirds coincide. Therefore, $\alpha'[1/2:\len(\alpha')-1/2]/_{\sim}$ is a well-defined, properly immersed, snippet-decomposed curve that contains at most one bad snippet. This proves the correctness of the algorithm \texttt{ReduceToOneBadSnippet}.

The bound on its running time follows from Lemma \ref{Functional correct BigArc}: since $\len(\alpha'_\textrm{trim})=\len(\alpha)-1$, we know that the algorithm \texttt{BigArc} on input $(S,N,\alpha \cdot \alpha[0])$ terminates within $\mathit{O}(\chi(S)^2 \cdot (\len(\alpha)+s-2))$ time. We further know that \[\len_\textrm{red}(\texttt{BigArc}(S,N,\alpha \cdot \alpha[0])[1:-1]) \leq \len_\textrm{red}(\alpha[1:])+4s.\] As $\len_\textrm{red}(\alpha[1:])\leq \len_\textrm{red}(\alpha)+4  \leq \len_\textrm{red}(\alpha)+s$, this implies that \begin{align*}
& \len_\textrm{red}(\texttt{ReduceToOneBadSnippet}(\alpha)) \\ \leq & \len_\textrm{red}(\texttt{BigArc}(S,N,\alpha \cdot \alpha[0])[1:-1]) + \len_\textrm{corn}(\texttt{BigArc}(S,N,\alpha \cdot \alpha[0])[0]) \\
\leq &\len_\textrm{red}(\alpha[1:])+4s +2s \\
\leq & \len_\textrm{red}(\alpha)+s+4s+2s \\
\leq & \len_\textrm{red}(\alpha)+7s.
\end{align*}
\end{proof}

\begin{cor}\label{ReduceToOneBadSnippet Reduced length}
\sloppy
For any input $(S,N,\alpha)$ of the required form, the algorithm \texttt{ReduceToOneBadSnippet} terminates in $\mathit{O}(|\chi(S)|^3 \cdot (\len_\textrm{red}(\alpha)+1))$ time.
\end{cor}

\begin{proof}
As $\alpha$ contains at most two bad snippets, Lemma \ref{reduced corner length vs snippet length} implies that
\begin{align*}
\len(\alpha)+s \leq \len_\textrm{red}(\alpha)+s+2 \leq \len_\textrm{red}(\alpha) + 2s,
\end{align*}
which is sufficient to prove the claim.
\end{proof}

\newpage

\chapter{Efficient position for curves with a single bad snippet} \label{Chapter last snippet}

In this chapter, we focus on the last major piece still missing for the proof for Theorem \ref{Polynomial time algorithm}.
Throughout this chapter, we assume that $S = S_{g,b}$ is a surface satisfying $\xi(S)=3g-3+b \geq 1$. We further assume that $\tau \subset S$ is a large train track and that $N=N(\tau) \subset S$ is a tie neighbourhood of $\tau$ in $S$.
Suppose that $\alpha \subset S$ is a properly immersed, almost efficient curve. If $\alpha$ contains a trigon snippet, we already know that we can achieve efficient position or shorten the curve to consist of a single snippet in $\mathit{O}(\chi(S)^2 \cdot \len(\alpha))$ time. However, if the bad snippet of $\alpha$ is not a trigon snippet, we have to undertake a delicate case analysis taking into account the various bigon snippet types.

In the following, we define a collection of algorithms whose input is an almost efficient curve $\alpha$ of a specified bigon type. Their output will be a curve $\alpha'$ homotopic to $\alpha$ such that one of the following three statements holds.
\begin{itemize}
\item $\alpha'$ is in efficient position or has snippet length one.
\item $\alpha'$ contains a unique bad snippet of trigon type.
\item $\alpha'$ contains a unique bad snippet of bigon type and satisfies \[\len_\textrm{red}(\alpha')<\len_\textrm{red}(\alpha).\]
\end{itemize}
Thus, we can obtain efficient position for an almost efficient curve $\alpha$ or shorten it to consist of a single snippet only by applying at most $\len_\textrm{red}(\alpha)+1$ many of these algorithms.

Lemma \ref{Easy Bigon bounds} implies that for bigon snippets of type $\mathbb{B}(t,t)$, $\mathbb{S}(h,h,0)$, $\mathbb{S}(t,t,0)$, $\mathbb{S}(v,v,0)$, or $\mathbb{R}(v,v)$, these algorithms only have to execute one bigon homotopy. However, the remaining four types of bigon snippets require some subtle case analyses. Applying a single bigon homotopy of one of these types might yield 
\begin{itemize}
\item a curve containing two trigon snippets (see page \pageref{Lemma B(S,t,t,2)}, Lemma \ref{Lemma B(S,t,t,2)}) or
\item a curve with a unique bigon snippet that has a greater reduced corner length (see page \pageref{Lemma B(h,h,0)}, Lemma \ref{Lemma B(h,h,0)}).
\end{itemize}
Thus, in the case of these four bigon snippet types, the algorithm will have to perform more than a single bigon homotopy to yield the desired results.
We begin our discussion with the more accessible cases of bigon snippets of type $\mathbb{S}(t,v,1)$ and $\mathbb{S}(t,t,2)$. Afterwards, we address the delicate cases of bigon snippets of type $\mathbb{R}(h,h)$ and $\mathbb{B}(h,h)$.
We conclude this chapter by presenting the algorithm \texttt{SingleBadSnippet} which takes as input an almost efficient curve and outputs a homotopic curve that consists of a single snippet only or is in efficient position.

\section{\texttt{WeightOneBigon}-algorithm for almost efficient curves}

We begin by studying almost efficient curves that contain a snippet of type $\mathbb{S}(t,v,1)$. The corresponding algorithm \texttt{WeightOneBigon} 
\begin{itemize}
\item takes as input a surface $S$ of positive complexity, a tie neighbourhood $N \subset S$ of a large train track in $S$, as well as an almost efficient curve $\alpha$ with a bad snippet of type $\mathbb{S}(t,v,1)$, and
\item outputs a curve homotopic to $\alpha$ which either contains a unique bad snippet of trigon type or contains a unique bad snippet of bigon type and has a shorter reduced corner length than $\alpha$.
\end{itemize}
For a formal statement of the algorithm \texttt{WeightOneBigon} we refer the reader to Algorithm \ref{WeightOneBigon_Curve}.

\begin{algorithm} 
    \SetKwInOut{Input}{Input}
    \SetKwInOut{Output}{Output}
	\DontPrintSemicolon
%    \textbf{Function} $\texttt{WeightOneBigon}(\alpha)$:\;
%    \BlankLine
    \Input{A surface $S$ of positive complexity, a tie neighbourhood $N \subset S$ of a large train track in $S$, and an almost efficient curve $\alpha \subset S$ that contains a bad snippet of type $\mathbb{S}(t,v,1)$.}
    \Output{A curve $\alpha'$ homotopic to $\alpha$ such that $\alpha'$ contains at most one bad snippet. If this is a trigon, then $\len_\textrm{red}(\alpha')\leq \len_\textrm{red}(\alpha)+ 2s-5$. If this is a bigon, then $\len_\textrm{red}(\alpha')\leq \len_\textrm{red}(\alpha)-4$.}
    \BlankLine
    Set $k$ such that $\alpha[k]$ is bad\;
    $\alpha'=\Hom(\alpha,k)$\;
    \If {$\alpha'[k-1]$ is a trigon of type $\mathbb{R}(h,v)$}
    {
    \Return $\Hom(\alpha',k-1)$
    }
    \If {$\alpha'[k]$ is a trigon of type $\mathbb{R}(h,v)$}
    {
    \Return $\Hom(\alpha',k)$
    }
    \Return $\alpha'$
    \caption{\texttt{WeightOneBigon} - Homotoping an almost efficient curve with a bad snippet of type $\mathbb{S}(t,v,1)$ into an input for $\texttt{TrigCurve}$ or into an almost efficient curve of shorter reduced corner length.}
    \label{WeightOneBigon_Curve}
\end{algorithm}

\begin{lem} \label{WeightOneBigon funtionally correct}
The algorithm \texttt{WeightOneBigon} is correct. On an input $(S,N,\alpha)$, the algorithm halts in $\mathit{O}(|\chi(S)|(\len_\textrm{red}(\alpha)+1))$ time.
\end{lem}

\begin{proof}
By assumption on the input of the algorithm \texttt{WeightOneBigon}, we are given an almost efficient curve $\alpha$ with a unique bad snippet $\alpha[k]$ of type $\mathbb{S}(t,v,1)$. As $\alpha[k]$ is a bigon of type $\mathbb{S}(t,v,1)$ we know that $\len(\alpha)>2$. As all other snippets of $\alpha$ are in efficient position, we know that $\alpha[k-1]$ and $\alpha[k+1]$ are in efficient position, too. Set $\alpha'=\Hom(\alpha,k)$. Lemma \ref{Lemma B(S,t,v,1)} implies that $\Hom(\alpha[k-1:k+2],1)$ consists of two snippets $\alpha'[k-1]$ and $\alpha'[k]$. Either both of these are trigon snippets, one of type $\mathbb{B}(h,t)$ and one of type $\mathbb{R}(h,v)$, or exactly one of them is a bad snippet, which then must be of type $\mathbb{B}(h,t)$.

Let us first assume that exactly one of the snippets $\alpha'[k-1]$ and $\alpha'[k]$ is a bad snippet. As all other snippets of $\alpha'$ are in efficient position, \texttt{WeightOneBigon} then returns a curve containing a unique trigon. Lemma \ref{Lemma B(S,t,v,1)} implies that \[\len_\textrm{red}(\alpha')\leq \len_\textrm{red}(\alpha)+ 2s-5.\] Thus, the algorithm terminates with a valid output.

Secondly, let us assume that both snippets $\alpha'[k-1]$ and $\alpha'[k]$ are bad. Lemma \ref{Lemma B(S,t,v,1)} implies that the curve $\alpha'$ satisfies $\len_\textrm{red}(\alpha')\leq \len_\textrm{red}(\alpha)-4$. Lemma \ref{Homotopy of type R(h,v)} implies that applying one local homotopy to the trigon of type $\mathbb{R}(h,v)$ does not increase the reduced corner length of the underlying curve. Suppose that $\alpha'[k-1]$ is the trigon of type $\mathbb{B}(h,t)$. Suppose further that $\alpha[k]$ turns right. Hence, $\alpha'[k-1]$ and $\alpha'[k]$ are both right-turning trigons. Applying one local homotopy of type $\mathbb{R}(h,v)$ to $\alpha$ at $\alpha'[k]$ increases the index of the region cut off by the trigon $\alpha'[k-1]$ by $1/4$ (see Figure \ref{Weight1Picture}). This follows from Lemma \ref{General observations trigon homotopies} and Lemma \ref{Homotopy of type R(h,v)}. Thus, we obtain an almost efficient curve $\alpha''$ with a bad snippet of type $\mathbb{B}(t,t)$ that satisfies \[\len_\textrm{red}(\alpha'') \leq \len_\textrm{red}(\alpha')\leq \len_\textrm{red}(\alpha)-4.\] This finishes the proof of the correctness of \texttt{WeightOneBigon}.

\begin{figure}[htbp] 
  \begin{minipage}[b]{0.99\linewidth}
    \centering
\begin{tikzpicture}[scale=0.45]
\draw [very thick] (-2,3) rectangle (4,-6);
\node at (1,4) {\tiny{$h$}};
\node at (10,1.5) {\tiny{$t$}};
\node at (3,-1.5) {\tiny{$v$}};
\draw [very thick] (4,3) rectangle (9,0);
\draw [very thick] (4,-3) rectangle (9,-6);
\draw [very thick, densely dashed, blue] plot[smooth, tension=.4] coordinates {(9,2.5)  (1,2) (1,-1.75)  (10.1891,-1.6821)(16.091,2.8492)};
\draw [very thick, densely dashed, cyan] plot[smooth, tension=.4] coordinates {(9.0208,2.2081) (4.5994,1.8084)  (4.7295,-1.6248) (10.4168,-1.2471)(15.7691,3.0611)};
\draw [thick, dotted] plot[smooth, tension=.7] coordinates {(9,0) (11,0.5) (14,2.5)};
\draw [thick, dotted](14,2.5) -- (16,1.5);
\draw [thick, dotted] plot[smooth, tension=.7] coordinates {(16,1.5) (15.591,-0.1652) (16.9193,-1.4337)};
\draw [very thick, densely dashed, green] plot[smooth, tension=.4] coordinates {(9.0045,1.8296) (5.4604,1.6664) (5.4076,0.4993) (9.0253,0.4577) (10.7119,0.8355) (13.9381,2.896)(14.724,2.6219)(15.4521,3.2669) };
\end{tikzpicture}
    \caption{A homotopy of type $\mathbb{S}(t,v,1)$ followed by a homotopy of type $\mathbb{R}(h,v)$.} 
    \label{Weight1Picture}
  \end{minipage}
\end{figure}

We now determine the running time of the algorithm \texttt{WeightOneBigon} on an input $(S,N,\alpha)$. First, recall that Lemma \ref{reduced corner length vs snippet length} implies that $\len(\alpha) \leq \len_\textrm{red}(\alpha)+1$. Therefore, it takes $\mathit{O}(|\chi(S)| (\len_\textrm{red}(\alpha)+1))$ many operations to determine the unique bad snippet, and further $\mathit{O}(|\chi(S)|)$ many operations to execute the remaining steps of the algorithm and replace the respective subarcs. Thus, the algorithm \texttt{WeightOneBigon} halts in $\mathit{O}(|\chi(S)|(\len_\textrm{red}(\alpha)+1))$ time.
\end{proof}

\section{\texttt{WeightTwoBigon}-algorithm for almost efficient curves}
Next, we consider almost efficient curves that contain a snippet of type $\mathbb{S}(t,t,2)$. The corresponding algorithm \texttt{WeightTwoBigon} 
\begin{itemize}
\item takes as input a surface $S$ of positive complexity, a tie neighbourhood $N \subset S$ of a large train track in $S$, as well as an almost efficient curve $\alpha$ with a bad snippet of type $\mathbb{S}(t,t,2)$, and
\item outputs a curve homotopic to $\alpha$ which is in efficient position, or consists of a single snippet only, or contains a unique bad snippet of trigon type.
\end{itemize}
For a formal statement of the algorithm \texttt{WeightTwoBigon} we refer the reader to Algorithm \ref{WeightTwoBigon_Curve}.

\begin{algorithm} 
    \SetKwInOut{Input}{Input}
    \SetKwInOut{Output}{Output}
	\DontPrintSemicolon
%    \textbf{Function} $\texttt{WeightTwoBigon}(\alpha)$:\;
%    \BlankLine
    \Input{A surface $S$ of positive complexity, a tie neighbourhood $N \subset S$ of a large train track in $S$, and an almost efficient curve $\alpha \subset S$ that contains a bad snippet of type $\mathbb{S}(t,t,2)$.}
    \Output{A curve $\alpha'''$ homotopic to $\alpha$ such that $\alpha'''$ contains at most one bad snippet. If there is a bad snippet, then this is a trigon snippet and $\len_\textrm{red}(\alpha''')\leq \len_\textrm{red}(\alpha)+ 2s$.}
    \BlankLine
    Set $k$ such that $\alpha[k]$ is bad\;
    \If{$\len(\alpha)=2$}
    {
    $\alpha'=\Hom(\alpha,k)$\;
    \Return $\Hom(\alpha',k-1)$
    }
    $\alpha'=\Hom(\alpha,k)$\;
    $\alpha''=\texttt{TrigArc}(S,N,\alpha'[k:]\cdot\alpha'[:k])$\;
    \If{$\alpha''[-1]$ is a bigon}
    {\Return $\Hom(\alpha''/_{\sim},-1)$}
    \Return $\alpha''/_{\sim}$
    \caption{\texttt{WeightTwoBigon} - Homotoping an almost efficient curve with a bigon of type $\mathbb{S}(t,t,2)$ into an input for $\texttt{TrigCurve}$.}
    \label{WeightTwoBigon_Curve}
\end{algorithm}

\begin{lem} \label{WeightTwoBigon funtionally correct}
The algorithm \texttt{WeightTwoBigon} is correct. On an input $(S,N,\alpha)$, the algorithm halts in $\mathit{O}(\chi(S)^2 \cdot (\len_\textrm{red}(\alpha)+1))$ time.
\end{lem}

\begin{proof}
By assumption on the input of the algorithm \texttt{WeightOneBigon}, we are given an almost efficient curve $\alpha$ with a unique bad snippet $\alpha[k]$ of type $\mathbb{S}(t,t,2)$. Without loss of generality, we assume that $\alpha[k]$ turns right. As all other snippets of $\alpha$ are in efficient position, we know that $\alpha[k-1]$ and $\alpha[k+1]$ are in efficient position. Set $\alpha'=\Hom(\alpha,k)$.

We have to distinguish two cases: either $\len(\alpha)=2$ or $\len(\alpha)>2$. If $\len(\alpha)=2$, then Lemma \ref{Lemma B(S,t,t,2)} implies that $\alpha'[k]$ is a left vertical dual and $\alpha'[k-1]$ is a bigon of type $\mathbb{B}(h,h)$ (see page \pageref{B(S,t,t,2) short}, Figure \ref{B(S,t,t,2) short}). Thus, $\Hom(\alpha',k-1)$ is a peripheral curve of snippet length one and therefore a valid output for \texttt{WeightTwoBigon}. We remark that \texttt{WeightTwoBigon} terminates in $\mathit{O}(|\chi(S)|)$ time in this case.

If $\len(\alpha)>2$, we know that $\alpha'[k]$ is a left vertical dual and that $\alpha'[k-1]$ and $\alpha'[k+1]$ are right-turning trigons. Thus, the arc that starts with the snippet $\alpha'[k]$ and ends with the snippet $\alpha'[k-1]$, that is, the arc $\alpha'[k:]\cdot\alpha'[:k]$, has a unique bad snippet of trigon type in its inside. Corollary \ref{Cor for bigons of weight 2} implies that $\alpha''=\texttt{TrigArc}(S,N,\alpha'[k:]\cdot\alpha'[:k])$ contains a unique bad snippet $\alpha''[-1]$, which is a trigon snippet or a bigon snippet. If $\alpha''[-1]$ is a bigon snippet, the adjacent two snippets must be left vertical duals, so an index-argument shows that $\Hom(\alpha''/_{\sim},-1)$ is in efficient position. If $\alpha''[-1]$ is a trigon snippet, then $\alpha''/_{\sim}$ is a smooth curve with a unique bad snippet of trigon type. 

To analyse the running time and length of the output curve when $\len(\alpha)>2$, we note that Lemma \ref{Lemma B(S,t,t,2)} implies that $\len_\textrm{red}(\alpha') \leq \len_\textrm{red}(\alpha)-2$. It requires at most $\mathit{O}(|\chi(S)|)$ many operations to replace $\alpha$ by $\alpha'$. Since $\alpha'[k-1]$ is a trigon, the reduced corner length of $\alpha'$ equals the reduced corner length of the arc $\alpha'[k:]\cdot\alpha'[:k]$. Lemma \ref{TrigArc in reduced corner length} implies that we obtain the arc $\alpha''$ within another $\mathit{O}(\chi(S)^2 \cdot (\len_\textrm{red}(\alpha)+1))$ many operations. The arc $\alpha''$ satisfies \[\len_\textrm{red}(\alpha'') \leq \len_\textrm{red}(\alpha'[k:]\cdot\alpha'[:k])+2s \leq \len_\textrm{red}(\alpha')+2s \leq \len_\textrm{red}(\alpha)+2s.\] As $\alpha''[-1]$ is a bad snippet, $\len_\textrm{red}(\alpha'')=\len_\textrm{red}(\alpha''/_{\sim})$. Thus, if the algorithm returns $\alpha''/_{\sim}$, that is, when $\alpha''/_{\sim}$ contains a unique trigon, then the reduced corner length of the output curve is bounded by $\len_\textrm{red}(\alpha)+2s$. Else, the algorithm returns a curve in efficient position within a further $\mathit{O}(|\chi(S)|)$ operations.

Summarizing, if $\len(\alpha)>2$, the algorithm \texttt{WeightTwoBigon} terminates within $\mathit{O}(\chi(S)^2 \cdot (\len_\textrm{red}(\alpha)+1))$ many operations, which concludes the proof of the lemma.
\end{proof}

\section{\texttt{AllButHorBigInCompOrBranch}-algorithm for almost efficient curves}

To simplify later discussions, we combine the algorithms \texttt{WeightOneBigon} and \texttt{WeightTwoBigon} with the trivial algorithms corresponding to the homotopies of type $\mathbb{B}(t,t)$, $\mathbb{S}(h,h,0)$, $\mathbb{S}(t,t,0)$, $\mathbb{S}(v,v,0)$, and $\mathbb{R}(v,v)$ into a single algorithm called \texttt{AllButHorBigInCompOrBranch}.

\begin{algorithm} \label{AllButHorBigInCompOrBranch}
    \SetKwInOut{Input}{Input}
    \SetKwInOut{Output}{Output}
	\DontPrintSemicolon
%    \textbf{Function} $\texttt{AllButHorBigInCompOrBranch}(\alpha)$:\;
%    \BlankLine
    \Input{A surface $S$ of positive complexity, a tie neighbourhood $N \subset S$ of a large train track in $S$, and an almost efficient curve $\alpha \subset S$ that contains a bad snippet of type $\mathbb{B}(t,t)$, $\mathbb{S}(h,h,0)$, $\mathbb{S}(t,t,0)$, $\mathbb{S}(v,v,0)$, $\mathbb{R}(v,v)$, $\mathbb{S}(t,t,2)$, or $\mathbb{S}(t,v,1)$.}
    \Output{A curve $\alpha'$ homotopic to $\alpha$ such that $\alpha'$ contains at most one bad snippet. If this unique bad snippet is a bigon snippet, then $\len_\textrm{red}(\alpha')\leq \len_\textrm{red}(\alpha)-2$. If it is a trigon snippet, then $\len_\textrm{red}(\alpha')\leq\len_\textrm{red}(\alpha)+2s$.}
    \BlankLine
    Set $k$ such that $\alpha[k]$ is bad\;
    \If {$\alpha[k]$ is of type $\mathbb{S}(t,v,1)$}
    {
    \Return $\texttt{WeightOneBigon}(S,N,\alpha)$
    }
    \If {$\alpha[k]$ is of type $\mathbb{S}(t,t,2)$}
    {
    \Return $\texttt{WeightTwoBigon}(S,N,\alpha)$
    }
    \Return $\Hom(\alpha,k)$
    \caption{\texttt{AllButHorBigInCompOrBranch} - Homotoping an almost efficient curve of bigon type that does not contain a bigon snippet of type $\mathbb{R}(h,h)$ or $\mathbb{B}(h,h)$ into an input for $\texttt{TrigCurve}$ or into an almost efficient curve of shorter reduced corner length.}
\end{algorithm}

\begin{lem} \label{AllButHorBigInCompOrBranch funtionally correct}
The algorithm \texttt{AllButHorBigInCompOrBranch} is correct. On an input $(S,N,\alpha)$, the algorithm halts in $\mathit{O}(\chi(S)^2 \cdot (\len_\textrm{red}(\alpha)+1))$ time.
\end{lem}

\begin{proof}
By assumption on the input of the algorithm \texttt{AllButHorBigInCompOrBranch}, we are given a surface $S$ of positive complexity, a tie neighbourhood $N$ of a large train track in $S$, and an almost efficient curve $\alpha$ with a unique bad snippet $\alpha[k]$ of type $\mathbb{B}(t,t)$, $\mathbb{S}(h,h,0)$, $\mathbb{S}(t,t,0)$, $\mathbb{S}(v,v,0)$, $\mathbb{R}(v,v)$, $\mathbb{S}(t,t,2)$, or $\mathbb{S}(t,v,1)$.
It requires at most $\mathit{O}(|\chi(S)|(\len_\textrm{red}(\alpha)+1))$ many operations to determine the index of this bad snippet. If this snippet is of type $\mathbb{S}(t,v,1)$ or $\mathbb{S}(t,t,2)$, the correctness of the algorithm and the bounds on the running time follow from Lemma \ref{WeightOneBigon funtionally correct} and Lemma \ref{WeightTwoBigon funtionally correct}. If the bad snippet is of type $\mathbb{B}(t,t)$, $\mathbb{S}(h,h,0)$, $\mathbb{S}(t,t,0)$, $\mathbb{S}(v,v,0)$, or $\mathbb{R}(v,v)$, Lemma \ref{Easy Bigon bounds} implies that a single homotopy is sufficient to obtain a curve of the desired shape that satisfies the required bounds on the reduced corner length.
\end{proof}

\section{\texttt{TwoTrigons}-algorithm for almost efficient curves}

It remains to understand how to handle curves that contain a unique bad snippet of type $\mathbb{R}(h,h)$ or $\mathbb{B}(h,h)$. As seen in the previous chapter, applying a single homotopy might not only increase the number of bad snippets, but can also increase the reduced corner length of the underlying curve. To solve these issues, a subtle case analysis is required. To reduce the number of cases appearing at one time, we first discuss how to proceed in the special case of curves that have exactly two bad snippets which are adjacent trigon snippets in the tie neighbourhood. The corresponding algorithm \texttt{TwoTrigons} 
\begin{itemize}
\item takes as input a surface $S$ of positive complexity, a tie neighbourhood $N \subset S$ of a large train track in $S$, as well as a curve with exactly two bad snippets, which are adjacent trigon snippets in the tie neighbourhood and turn into the same direction, and
\item outputs a curve that is homotopic to the input and contains at most one bad snippet. If this is a bigon snippet, then the reduced corner length has decreased by at least two.
\end{itemize}
For a formal statement of the algorithm \texttt{TwoTrigons} we refer the reader to Algorithm \ref{TwoTrigons}.

\begin{algorithm} 
    \SetKwInOut{Input}{Input}
    \SetKwInOut{Output}{Output}
	\DontPrintSemicolon
%    \textbf{Function} $\texttt{TwoTrigons}(\alpha)$:\;
%    \BlankLine
    \Input{A surface $S$ of positive complexity, a tie neighbourhood $N \subset S$ of a large train track in $S$, and a curve $\alpha \subset S$ that contains exactly two bad snippets $\alpha[0]$ and $\alpha[1]$. We require that $\alpha[0]$ is of type $\mathbb{S}(h,t,3)$ or $\mathbb{S}(h,t,1)$ and that $\alpha[1]$ is of type $\mathbb{B}(h,t)$. The two trigon snippets must turn into the same direction.}
    \Output{A curve $\alpha'$ homotopic to $\alpha$ such that $\alpha'$ contains at most one bad snippet. If this unique bad snippet is a bigon snippet, then $\len_\textrm{red}(\alpha')\leq \len_\textrm{red}(\alpha)-2$. If it is a trigon snippet, then $\len_\textrm{red}(\alpha')\leq\len_\textrm{red}(\alpha)+2s+1$.}
    \BlankLine
    \If{$\alpha[0]$ is of type $\mathbb{S}(h,t,3)$}{
    $\alpha'=\Hom(\alpha',1)$\;
    \If {$\alpha'[1]$ is in efficient position}
    {\Return $\alpha'$}
    $\alpha'=\Hom(\alpha',1)$\;
    \If{$\alpha'[0]$ is a bigon of type $\mathbb{S}(t,t,0)$}
    {\Return $\Hom(\alpha',0)$
    }
    $\alpha'=(\alpha'[0]\cdot \texttt{TrigArc}(S,N,\alpha'[1:]))/_\sim $\;
    \If{$\alpha'[-1]$ is a trigon of type $\mathbb{R}(h,v)$}
    {\Return $\Hom((\Hom(\alpha',-1)),0)$
    }
    \If{$\alpha'[1]$ is in efficient position}
    {\Return $\alpha'$}
    }
    \Return $\Hom(\Hom(\alpha',1),0)$
    \caption{\texttt{TwoTrigons} - Homotoping a curve with exactly two bad snippets which must be adjacent trigon snippets in the tie neighbourhood that turn into the same direction into an input for $\texttt{TrigCurve}$ or into a curve of shorter reduced corner length containing a single bigon.}
    \label{TwoTrigons}
\end{algorithm}

\begin{lem} \label{TwoTrigons funtionally correct}
The algorithm \texttt{TwoTrigons} is correct. On an input $(S,N,\alpha)$, the algorithm halts in $\mathit{O}(\chi(S)^2 \cdot (\len_\textrm{red}(\alpha)+1))$ time.
\end{lem}

\begin{proof}
By assumption on the input of the algorithm \texttt{TwoTrigons}, we are given a curve $\alpha$ that contains exactly two bad snippets, $\alpha[0]$ and $\alpha[1]$. We further know that $\alpha[0]$ is a bad snippet of type $\mathbb{S}(h,t,3)$ or $\mathbb{S}(h,t,1)$ and $\alpha[1]$ is of type $\mathbb{B}(h,t)$. Both turn into the same direction. Without loss of generality, we assume that they turn right.
We remark that $\len(\alpha)\geq 3$ as $\alpha[0]$ and $\alpha[1]$ are contained in different types of neighbourhood rectangles.

We begin by proving the correctness of the algorithm \texttt{TwoTrigons}.
Let us first assume that $\alpha[0]$ is of type $\mathbb{S}(h,t,3)$. Set $\alpha'=\Hom(\alpha,1)$. Since $\alpha[0]$ is of type $\mathbb{S}(h,t,3)$ and $\alpha[0]$ and $\alpha[1]$ turn right, the boundary of the trigon cut off by $\alpha[1]$ contains a corner of the adjacent complementary region. Thus, Lemma \ref{General observations trigon homotopies} implies that $\alpha[1]$ is not weakly snippet homotopic to $\alpha'[1]$, but $\alpha[0]$ is weakly snippet homotopic to $\alpha'[0]$. We note that $\alpha'[0]$ then must be of type $\mathbb{S}(h,v,2)$. As $\alpha[1]$ was in efficient position, we know that $\alpha'[1]$ is in efficient position or a right-turning trigon snippet of type $\mathbb{R}(h,v)$. We note that applying a local trigon homotopy to $\alpha$ at $\alpha[1]$ does not increase the reduced corner length of the underlying curve unless $\alpha'[1]$ is in efficient position. If $\alpha'[1]$ is in efficient position, the reduced corner length increases by at most $2s$ following Lemma \ref{Homotopy of type B(h,t,0)}. Thus, the algorithm either returns a curve $\alpha'$ with a unique bad snippet of type $\mathbb{S}(h,v,1)$ satisfying $\len_\textrm{red}(\alpha')\leq \len_\textrm{red}(\alpha)+2s$ or proceeds by applying a local trigon homotopy to $\alpha$ at $\alpha'[1]$. 

In the following, we set $\alpha''=\Hom(\alpha',1)$. Lemma \ref{Homotopy of type R(h,v)} implies that \[\len_\textrm{red}(\alpha'')=\len_\textrm{red}(\alpha')\leq \len_\textrm{red}(\alpha).\] If $\len(\alpha')=2$, then $\alpha''$ contains a unique bad snippet $\alpha''[0]$ of type $\mathbb{S}(t,t,0)$ (see Figure \ref{TwoTrigonShort}). Applying one local homotopy to the snippet of type $\mathbb{S}(t,t,0)$ reduces the reduced corner length by at least two according to Lemma \ref{Easy Bigon bounds} and yields a curve containing at most one bad snippet.

\begin{figure}[htbp] 
  \begin{minipage}[b]{0.49\linewidth}
    \centering
\begin{tikzpicture}[scale=0.45]
\draw [very thick] (-4,4) rectangle (0,-2);
\draw [very thick] (0,4) rectangle (4,2);
\draw [very thick] plot[smooth, tension=.7] coordinates {(0,0) (2,0) (6.5,0.5) (6.5,-2.5) (2,-2) (0.0013,-1.9903)};
\draw [very thick](1.5,0) -- (1.5,-2);
\draw [very thick](3.0076,0.1748) -- (3.0126,-2.1513);
\draw [very thick] plot[smooth, tension=.7] coordinates {(3,-0.5) (5.5,-0.5) (5.5,-1.5) (3,-1.5)};
\draw [very thick, densely dashed, blue] plot[smooth cycle, tension=.7] coordinates {(-2,-3.5) (-2,2) (2.017,3.1916) (3,1.5) (7.5,1.5) (7.5,-3.5) (2,-3.5)};
\draw [very thick, densely dashed, cyan] plot[smooth cycle, tension=.7] coordinates {(-1.5,-3) (-1.5,1) (1.75,1.75) (3,1) (7,1) (7,-3) (2,-3)};
\draw [very thick, densely dashed, green] plot[smooth cycle, tension=.7] coordinates {(0,-0.25) (-1.5,-0.5) (-1.5,-1.5) (0,-1.75) (2,-1.75) (6.25,-2.25) (6.25,0.25) (2,-0.25)};
\node at (-5,1) {\tiny{$t$}};
\node at (-2,5) {\tiny{$h$}};
\draw [very thick] (4.7,-1) ellipse (0.2 and 0.2);
\end{tikzpicture}
    \caption{Applying the algorithm \texttt{TwoTrigons} to a curve of snippet length three.} 
    \label{TwoTrigonShort}
  \end{minipage}
  \begin{minipage}[b]{0.49\linewidth}
    \centering
\begin{tikzpicture}[scale=0.45]
\draw [very thick] (-4,4) rectangle (0,-2) node (v1) {};
\draw [very thick] (0,4) rectangle (4,2);
\node at (-5,1) {\tiny{$t$}};
\node at (-2,5) {\tiny{$h$}};
\draw [very thick, densely dashed, blue] plot[smooth, tension=.7] coordinates {(-2.6,-4.1) (-2,2) (2,3) (3.1,1.6) (6.3,1.3) (7,-2)};
\draw [very thick, densely dashed, cyan] plot[smooth, tension=.7] coordinates {(-2.1,-4.1) (-1.4,0.3) (2.9,1.1) (6,0.8) (6.5,-2)};
\draw [very thick, densely dashed, green] plot[smooth, tension=.7] coordinates {(-1.7,-4.1) (-1,-0.8) (3.0872,-0.2718) (5.5449,-0.5149) (6,-2)};
\draw [thick, dotted] plot[smooth, tension=.7] coordinates {(4,0) (9,0)};
\draw [very thick](0,0) -- (4,0);
\draw [very thick] (0,-2) -- (4,-2);
\end{tikzpicture}
    \caption{Applying two trigon homotopies on an input of \texttt{TwoTrigons} which has snippet length at least four.} 
    \label{TwoTrigonLong}
  \end{minipage}
  \begin{minipage}[b]{0.99\linewidth}
    \centering
\begin{tikzpicture}[scale=0.45]
\draw [very thick] (-4,4) rectangle (0,-2);
\draw [very thick] (0,4) rectangle (4,2);
\node at (-5,1) {\tiny{$t$}};
\node at (-2,5) {\tiny{$h$}};
\draw [very thick] (0,0) rectangle (4,-2) node (v1) {};
\draw [thick, dotted] plot[smooth, tension=.7] coordinates {(v1) (6.9,-2.3) (8,-3)};
\draw [thick, dotted] plot[smooth, tension=.7] coordinates {(4,0) (7,0) (8.5,0)};
\node at (7.7,-3.7) {\tiny{$v$}};
\draw [very thick, densely dashed, blue] plot[smooth, tension=.7] coordinates {(8.6,-5.1) (6.2,-3.8) (-2,-3.6) (-2.2,-0.8) (3,-0.4) (8.4,-0.4)};
\draw [very thick, densely dashed, cyan] plot[smooth, tension=.7] coordinates {(8.4,-0.8) (3,-0.8) (-1.4,-1) (-1.4,-1.7) (6.2,-1.9) (8.7,-3.2) (8.3,-4.1) (8.8,-5)};
\draw [very thick](6.5,-5.5) -- (9,-3) -- (10.5,-4.5) -- (8,-7) -- cycle;
\draw [very thick](6,-5) -- (7,-6);
\draw [very thick](6.5,-4.5) -- (7,-5);
\draw [very thick](8.5,-2.5) -- (9.5,-3.5);
\draw [very thick] plot[smooth, tension=.7] coordinates {(8.5,-3.5) (8,-3)};
\end{tikzpicture}
    \caption{A curve with two adjacent trigons of which one lies in a complementary region and one inside a switch rectangle.} 
    \label{TwoTrigonLongwithcomplement}
  \end{minipage}
\end{figure}

If $\len(\alpha')>2$, then $\alpha''[1:]$ is an arc that contains at most one bad snippet, which must be of trigon type (see Figure \ref{TwoTrigonLong}). We note that this trigon can be the snippet $\alpha''[1]$. Applying the algorithm $\texttt{TrigArc}$ to $\alpha''[1:]$ yields an arc which is in efficient position in its inside. As a single trigon homotopy changes at most one weak snippet homotopy type of the adjacent snippets, $\texttt{TrigArc}(S,N,\alpha''[1:])[0]$ or $\texttt{TrigArc}(S,N,\alpha''[1:])[-1]$ are in efficient position. This implies that either $\texttt{TrigArc}(S,N,\alpha''[1:])$ is in efficient position, or that $\texttt{TrigArc}(S,N,\alpha''[1:])[-1]$ is a trigon inside a complementary region, or that $\texttt{TrigArc}(\alpha''[1:])[0]$ is a trigon in a branch rectangle. We note that the reduced corner length of the arc $\alpha''[1:]$ increases under $\texttt{TrigArc}$ only if $\texttt{TrigArc}(S,N,\alpha''[1:])$ is in efficient position, and then by at most $2s$. This can be seen as follows: if $\alpha''[1]$ or $\alpha''[-1]$ are the unique bad snippet of the arc $\alpha''[1:]$, then \texttt{TrigArc} simply returns $\alpha''[1:]$ and thus does not increase the reduced corner length. Else, the arc $\alpha''[1:]$ can be artificially prolonged by one snippet at the beginning and one snippet at the end to yield an arc that contains a unique trigon snippet in its inside and contains the entire arc $\alpha''[1:]$ in its inside. Thus, applying \texttt{TrigArc} to this prolonged arc increases the reduced corner length only if efficient position is achieved on the inside or if the weak snippet homotopy type of the first or last snippet is altered. As the algorithm \texttt{TrigArc} applied to the arc $\alpha''[1:]$ would have terminated in the latter two cases already, this implies that the reduced corner length of $\texttt{TrigArc}(S,N,\alpha''[1:])$ has increased only if $\texttt{TrigArc}(S,N,\alpha''[1:])$ is in efficient position.

If $\texttt{TrigArc}(S,N,\alpha''[1:])$ is in efficient position, the curve \[\alpha''[0] \cdot \texttt{TrigArc}(S,N,\alpha''[1:])/_\sim\] contains a unique trigon snippet of type $\mathbb{S}(h,t,1)$. Its reduced corner length is bounded by $\len_\textrm{red}(\alpha'')+2s$, hence also by $\len_\textrm{red}(\alpha)+2s$. 

In the second case, that is, if $\texttt{TrigArc}(S,N,\alpha''[1:])[-1]$ is a trigon snippet inside a complementary region, applying on trigon homotopy to \[\texttt{TrigArc}(S,N,\alpha''[1:])[-1]\] yields a curve that contains a unique bigon of type $\mathbb{S}(t,t,0)$ (see Figure \ref{TwoTrigonLongwithcomplement}). Applying one further homotopy to the snippet $\mathbb{S}(t,t,0)$ reduces the reduced corner length by at least two and yields a curve containing at most one bad snippet. Thus, the algorithm outputs a curve with a unique bad snippet whose reduced corner length is less than or equal to $\len_\textrm{red}(\alpha)-2$ in this case.

Else, that is if $\texttt{TrigArc}(S,N,\alpha''[1:])[0]$ is a trigon snippet in a branch rectangle, the first snippet of the curve $\alpha'''= \alpha''[0] \cdot \texttt{TrigArc}(S,N,\alpha''[1:])/_\sim$ must be a right-turning snippet of type $\mathbb{S}(h,t,0)$, whereas its second snippet, the snippet $\texttt{TrigArc}(S,N,\alpha''[1:])[0]$ is a trigon snippet inside a branch rectangle. We note that the algorithm exists the first if-statement and the given curve is of the second possible input type for the algorithm \texttt{TwoTrigons}. We recall that $\len_\textrm{red}(\alpha''') \leq \len_\textrm{red}(\alpha)$. Then $\Hom(\alpha''',1)$ contains a unique bad snippet of type $\mathbb{S}(h,h,0)$ and satisfies \[\len_\textrm{red}(\Hom(\alpha''',1)) \leq \len_\textrm{red}(\alpha''').\] Lemma \ref{Easy Bigon bounds} implies that the curve $\alpha''''=\Hom(\Hom(\alpha''',1),0)$ contains at most one bad snippet. If it does, the curve satisfies \[\len_\textrm{red}(\alpha'''') \leq \len_\textrm{red}(\Hom(\alpha''',1))-2 \leq \len_\textrm{red}(\alpha''')-2 \leq \len_\textrm{red}(\alpha)-2.\] Thus, the correctness of the algorithm \texttt{TwoTrigons} follows.

Lastly, let us analyse the running time of the algorithm \texttt{TwoTrigons}. Checking whether $\alpha[0]$ is of type $\mathbb{S}(h,t,3)$ and, if applicable, replacing the subarc $\alpha[:3]$ by the subarc $\Hom(\alpha[:3],1)$ can be done within $\mathit{O}(|\chi(S)|)$ many operations. Similarly, evaluating any further if-statements in the pseudocode of the algorithm \texttt{TwoTrigons} and executing one or two local homotopies can be done within $\mathit{O}(|\chi(S)|)$ many operations. As $\len_\textrm{red}(\alpha''[1:])\leq \len_\textrm{red}(\alpha)$, executing the algorithm \texttt{TrigArc} on the input $(S,N,\alpha''[1:])$ requires at most $\mathit{O}(\chi(S)^2 \cdot (\len_\textrm{red}(\alpha)+1))$ many operations. Hence, the entire algorithm terminates within $\mathit{O}(\chi(S)^2 \cdot (\len_\textrm{red}(\alpha)+1))$ many operations.
\end{proof}

\section{\texttt{HorBigonInCompWide}-algorithm for almost efficient curves}

In the following we say that a snippet of type $\mathbb{R}(h,h)$ is \emph{wide} or \emph{narrow} if the boundary of the bigon contains at least one point or does not contain any points of $\partial ^2 \mathcal{R}$ respectively. For examples of wide and narrow bigon snippets of type $\mathbb{R}(h,h)$ we refer the reader to Figure \ref{Homotopy of type B(R,h,h) picture} on page \pageref{Homotopy of type B(R,h,h) picture}.

We now use the results from the previous section to define the algorithm \texttt{HorBigonInCompWide} that
\begin{itemize}
\item takes as input a surface $S$ of positive complexity, a tie neighbourhood $N \subset S$ of a large train track in $S$, as well as an almost efficient curve with a wide snippet of type $\mathbb{R}(h,h)$, and
\item outputs a curve that is homotopic to the input and contains at most one bad snippet. If this is a bigon snippet, then the reduced corner length has decreased by at least two.
\end{itemize}
For a formal statement of the algorithm \texttt{HorBigonInCompWide} we refer the reader to Algorithm \ref{HorBigonInCompWide}.

Strictly speaking it is not necessary to differentiate between wide and narrow horizontal bigons and to present an extra algorithm which deals with wide horizontal bigons. However, it reduces the number of case distinctions that appear simultaneously in the arguments.

To simplify notation, we introduce the notion of reverse snippets and reverse arcs or curves. We recall that snippets and snippet-decomposed arcs and curves carry an orientation induced by the standard parametrization of the interval $[0,1]$.

\begin{defn}
Suppose that $a \subset  R \in \mathcal{R}$ is a snippet. By $\rev(a)$ we denote the snippet $a$ with its orientation reversed. Similarly, for any snippet-decomposed arc or curve $\alpha \subset S$, we denote by $\rev(\alpha)$ the arc or curve $\alpha$ with its orientation reversed.
\end{defn}
Suppose that $\alpha\subset S$ is a snippet-decomposed arc or curve of length $\len(\alpha)=k$. Then $\rev(\alpha)[i]=\rev(\alpha[k-i-1])$.

%In the algorithm below I once assume that $\alpha[k-1]$ is not strongly snippet homotopic to $\alpha[k+1]$. No idea why?

\begin{algorithm}
    \SetKwInOut{Input}{Input}
    \SetKwInOut{Output}{Output}
	\DontPrintSemicolon
%    \textbf{Function} $\texttt{HorBigonInCompWide}(\alpha)$:\;
%    \BlankLine
    \Input{A surface $S$ of positive complexity, a tie neighbourhood $N \subset S$ of a large train track in $S$, and an almost efficient curve $\alpha \subset S$ that contains a wide bigon of type $\mathbb{R}(h,h)$. We further require that $\len(\alpha)>2$.}
    \Output{A curve $\alpha''''$ homotopic to $\alpha$ such that $\alpha''''$ contains at most one bad snippet. If this unique bad snippet is a bigon snippet, then $\len_\textrm{red}(\alpha'''')\leq \len_\textrm{red}(\alpha)-2$. If it is a trigon snippet, then $\len_\textrm{red}(\alpha'''')\leq\len_\textrm{red}(\alpha)+2s+1$.}
    \BlankLine
    Set $k$ such that $\alpha[k]$ is bad\;
    $\alpha'=\Hom(\alpha,k)$\;
    $\alpha''=\texttt{TrigArc}(S,N,\alpha'[k:]\cdot \alpha'[:k-1])$\;
    \If {$\alpha''$ is in efficient position}
    {
    \Return $\alpha''\cdot \alpha'[k-1]/_\sim$
    }
    \If {$\alpha''[-1]$ is of type $\mathbb{R}(h,v)$}{
    $\alpha'''=\Hom(\alpha''\cdot \alpha'[k-1]/_\sim,-2)$\;
    \Return $\texttt{AllButHorBigInCompOrBranch}(S,N,\alpha''')$
    }
    \If {$\alpha'[k-1]$ is of type $\mathbb{S}(h,t,1)$ or $\mathbb{S}(h,t,3)$}{
    $\alpha'''=\alpha'[k-1] \cdot \alpha''/_\sim$\;
    \Return $\texttt{TwoTrigons}(S,N,\alpha''')$
    }
    $\alpha'''=\rev(\alpha''[0]) \cdot \rev(\alpha'[k-1]) \cdot \rev(\alpha''[1:]) / _\sim$\;
    \Return $\rev(\texttt{TwoTrigons}(S,N,\alpha'''))$
    \caption{\texttt{HorBigonInCompWide} - Homotoping an almost efficient curve with a wide bigon of type $\mathbb{R}(h,h)$ into an input for $\texttt{TrigCurve}$ or into an almost efficient curve of bigon type and shorter reduced corner length.}
    \label{HorBigonInCompWide}
\end{algorithm}

\begin{lem} \label{HorBigInComplWide funtionally correct}
The algorithm \texttt{HorBigInComplWide} is correct. On an input $(S,N,\alpha)$, the algorithm halts in $\mathit{O}(\chi(S)^2 \cdot (\len_{\textrm{red}}(\alpha)+1))$ time.
\end{lem}

\begin{proof}
By assumption on the input of the algorithm \texttt{HorBigonInCompWide}, we are given a surface $S$ of positive complexity, a tie neighbourhood $N \subset S$ of a large train track in $S$, as well as an almost efficient curve $\alpha$ with a wide horizontal bigon snippet $\alpha[k]$ of type $\mathbb{R}(h,h)$. We can assume that $\len(\alpha)>2$.
We remark that it requires at most $\mathit{O}(|\chi(S)|(\len_\textrm{red}(\alpha)+1))$ many operations to determine the index of this bad snippet.

We prove the correctness of the algorithm by going through its pseudocode line by line.
As $\len(\alpha)>2$, the curve $\alpha'=\Hom(\alpha,k)$ has two trigon snippets inside the tie neighbourhood which turn the same way. Lemma \ref{Lemma B(R,h,h)} implies that \[\len_{\rm\textrm{red}}(\alpha')\leq \len_\textrm{red}(\alpha).\] Since $\Hom(\alpha[k-1:k+2],1)[1:-1]$ is in efficient position, we further know that one of these trigon snippets is the snippet $\alpha'[k-1]$, whereas the other trigon snippet might be adjacent to $\alpha'[k-1]$ or separated from it by a number of carried snippets of $\alpha'$ (see page \pageref{Homotopy of type B(R,h,h) picture}, Figure \ref{Homotopy of type B(R,h,h) picture}). The previous section provides and algorithm to handle curves that contain two adjacent trigon snippets inside the tie neighbourhood. Thus, the next step in the algorithm aims to homotope the curve into such a form. As $\alpha'[k-1]$ is one of the two bad snippets of trigon type of $\alpha'$, the arc $\alpha'[k:] \cdot \alpha'[:k-1]$ contains exactly one bad snippet. This is of trigon type. We note that the inside of $\alpha'[k:] \cdot \alpha'[:k-1]$ might already be in efficient position. In any case, the arc $\alpha''=\texttt{TrigArc}(S,N,\alpha'[k:] \cdot \alpha'[:k-1])$ then is in efficient position in its inside. If the entire arc $\alpha''$ is in efficient position, then $\alpha'' \cdot \alpha'[k-1]/_\sim$ contains a unique bad snippet of trigon type. Else, either the snippet $\alpha''[0]$ or the snippet $\alpha''[-1]$ are a bad snippet of trigon type. We note that in this case, the reduced corner length of $\alpha''$ equals the reduced corner length of $\alpha'[k:] \cdot \alpha'[:k-1]$ following previously applied arguments as in the proof of Lemma \ref{TwoTrigons funtionally correct}.

If $\alpha''[-1]$ is a trigon, it must be a trigon of type $\mathbb{R}(h,v)$ (see page \pageref{TwoTrigonLongwithcomplement}, Figure \ref{TwoTrigonLongwithcomplement}). This follows from the fact that the snippet $\alpha'[k-2]$ intersects the trigon $\alpha'[k-1]$ along $\partial_h N$. Hence, $\alpha'[k-2]$ and therefore $\alpha''[-1]$ lie inside a complementary region. Applying one homotopy to the snippet $\alpha''[-1]$ of type $\mathbb{R}(h,v)$ does not increase the reduced corner length but yields a curve with a unique snippet of bigon type that lies inside the tie neighbourhood. This implies that \[\texttt{AllButHorBigInCompOrBranch}(S,N,\Hom(\alpha''\cdot \alpha'[k-1]/_\sim,-2)\] is a curve that contains at most one bad snippet. If this bad snippet is a bigon, its reduced corner length has decreased by at least two. If the curve contains a unique trigon, the reduced corner length has increased by at most $2s+1$.

If $\alpha''[0]$ is a trigon, it must be a trigon inside the tie neighbourhood which is turning into the same direction as $\alpha'[k-1]$. Hence, we can use the algorithm \texttt{TwoTrigons} to obtain the desired results. Due to our special set-up of the algorithm \texttt{TwoTrigons} we have to make sure to feed it a curve whose first snippet lies inside a switch rectangle. Hence, if $\alpha'[k-1]$ lies inside a switch rectangle, we can immediately apply the algorithm \texttt{TwoTrigons} and $\texttt{TwoTrigons}(S,N,\alpha'[k-1]\cdot \alpha'')$ yields a curve whose reduced corner length has decreased (respectively increased) by at least two (respectively at most $2s+1$) if it contains a bigon (respectively trigon). If, on the other hand, the snippet $\alpha'[k-1]$ lies inside a branch rectangle, $\alpha''[0]$ is a snippet inside a switch rectangle. Thus, the curve $\alpha'''=\rev(\alpha''[0]) \cdot \rev(\alpha'[k-1]) \cdot \rev(\alpha''[1:])/_\sim$ is a valid input for \texttt{TwoTrigons}, and $\rev(\texttt{TwoTrigons}(S,N,\alpha'''))$ is homotopic to $\alpha$ and of the required form. This concludes the proof of the correctness of the algorithm \texttt{HorBigInComplWide}.

To analyse the running time of the algorithm \texttt{HorBigInComplWide} on an input $(S,N,\alpha)$, we note that any evaluation of the if-statements and algorithms can be executed within $\mathit{O}(\chi(S)^2 \cdot (\len_\textrm{red}(\widetilde{\alpha})+1))$ many operations, where $\widetilde{\alpha}$ is the relevant version of the curve homotopic to $\alpha$. As $\alpha'[k-1]$ is a bad snippet, we know that \[\len_\textrm{red}(\alpha')=\len_\textrm{corn}(\alpha'[k-1])+\len_\textrm{red}(\alpha'[k:]\cdot \alpha'[:k-1]).\] Together with the fact that the reduced corner length of $\widetilde{\alpha}$ never exceeds the reduced corner length of $\alpha$ unless efficient position is achieved or a unique bad snippet of trigon type is created, this implies that the entire algorithm terminates within $\mathit{O}(\chi(S)^2 \cdot (\len_\textrm{red}(\alpha)+1))$ many operations.
\end{proof}

\section{\texttt{HorBigonInComp}-algorithm for almost efficient curves}

Building on the results from the previous section, we now present the algorithm \texttt{HorBigonInComp} that handles almost efficient curves with a wide or narrow horizontal bigon. For a formal statement of \texttt{HorBigonInComp} we refer the reader to Algorithm \ref{HorBigonInComp}.

\begin{algorithm} 
    \SetKwInOut{Input}{Input}
    \SetKwInOut{Output}{Output}
	\DontPrintSemicolon
%    \textbf{Function} $\texttt{HorBigonInComp}(\alpha)$:\;
%    \BlankLine
    \Input{A surface $S$ of positive complexity, a tie neighbourhood $N \subset S$ of a large train track in $S$, and an almost efficient curve $\alpha \subset S$ that contains a snippet of type $\mathbb{R}(h,h)$.}
    \Output{A curve $\alpha'$ homotopic to $\alpha$ such that $\alpha'$ contains at most one bad snippet. If this unique bad snippet is a bigon snippet, then $\len_\textrm{red}(\alpha')<\len_\textrm{red}(\alpha)$. If it is a trigon snippet, then $\len_\textrm{red}(\alpha')\leq\len_\textrm{red}(\alpha)+2s+1$.}
    \BlankLine
    Set $k$ such that $\alpha[k]$ is bad\;
    \If {$\len(\alpha)=2$}
    {
    \Return $\texttt{AllButHorBigInCompOrBranch}(S,N,\Hom(\alpha,k))$
    }
    \If {$\Hom(\alpha,k)[k-1]$ is a bigon of type $\mathbb{B}(h,h)$ or $\mathbb{S}(h,h,0)$}
    {
    \Return $\Hom(\alpha,k)$
    }
    \Return $\texttt{HorBigInComplWide}(S,N,\alpha)$
    \caption{\texttt{HorBigonInComp} - Homotoping an almost efficient curve with a snippet of type $\mathbb{R}(h,h)$ into an input for $\texttt{TrigCurve}$ or into an almost efficient curve of bigon type and shorter reduced corner length.}
    \label{HorBigonInComp}
\end{algorithm}

\begin{lem} \label{HorBigInCompl funtionally correct}
The algorithm \texttt{HorBigInCompl} is correct. On an input $(S,N,\alpha)$, the algorithm halts in $\mathit{O}(\chi(S)^2 \cdot (\len_\textrm{red}(\alpha)+1))$ time.
\end{lem}

\begin{proof}
By assumption on the input of the algorithm \texttt{HorBigonInComp}, we are given a surface $S$ of positive complexity, a tie neighbourhood $N \subset S$ of a large train track in $S$, as well as an almost efficient curve $\alpha \subset S$ with a bad snippet $\alpha[k]$ of type $\mathbb{R}(h,h)$. We remark that it requires at most $\mathit{O}(|\chi(S)|(\len_\textrm{red}(\alpha)+1))$ many operations to determine the index of this bad snippet.

If $\len(\alpha)=2$, then $\alpha[k-1]$ must be a tie of branch or switch rectangle (see page \pageref{Short R(h,h) picture}, Figure \ref{Short R(h,h) picture}). We see that $\Hom(\alpha,k)[k-1]$ is a bigon inside the tie neighbourhood that does not intersect $\partial_h N$. Thus, it is a valid input for the algorithm \texttt{AllButHorBigInCompOrBranch}. Since $\len_\textrm{red}(\alpha)=\len_\textrm{red}(\Hom(\alpha,k))$, we know that \[\len_\textrm{red}(\texttt{AllButHorBigInCompOrBranch}(S,N,\Hom(\alpha,k))) \leq \len_\textrm{red}(\alpha)-2\] and that the output curve contains a bigon.

If $\len(\alpha)>2$ and $\alpha[k]$ is a narrow bigon, then $\alpha[k-1]$ and $\alpha[k+1]$ are both ties of the same branch or switch rectangle. Hence, $\Hom(\alpha,k)$ contains a unique bigon of type $\mathbb{B}(h,h)$ or $\mathbb{S}(h,h,0)$ and satisfies $\len_\textrm{red}(\Hom(\alpha,k))\leq \len_\textrm{red}(\alpha)-1$ by Lemma \ref{Lemma B(R,h,h)}.

\sloppy
Else, we have that $\len(\alpha)>2$ and $\alpha[k]$ is a wide bigon. Thus, $\texttt{HorBigInComplWide}(S,N,\alpha)$ is of the desired form following the results in the previous section.

As all algorithms employed terminate within $\mathit{O}(\chi(S)^2 \cdot (\len_\textrm{red}(\alpha)+1))$ time, the algorithm \texttt{HorBigInCompl} terminates within $\mathit{O}(\chi(S)^2 \cdot (\len_\textrm{red}(\alpha)+1))$ many operations.
\end{proof}

\section{\texttt{HorBigonInBranch}-algorithm for almost efficient curves}

In this section we see that the discussion about almost efficient curves containing a snippet of type $\mathbb{B}(h,h)$ naturally reduces to the case of almost efficient curve containing a snippet of type $\mathbb{R}(h,h)$. For a formal statement of the corresponding algorithm \texttt{HorBigonInBranch} we refer the reader to Algorithm \ref{HorBigonInBranch}.

\begin{algorithm} 
    \SetKwInOut{Input}{Input}
    \SetKwInOut{Output}{Output}
	\DontPrintSemicolon
%    \textbf{Function} $\texttt{HorBigonInBranch}(\alpha)$:\;
%    \BlankLine
    \Input{A surface $S$ of positive complexity, a tie neighbourhood $N \subset S$ of a large train track in $S$, and an almost efficient curve $\alpha \subset S$ that contains a snippet of type $\mathbb{B}(h,h)$.}
    \Output{A curve $\alpha''$ homotopic to $\alpha$ such that $\alpha''$ contains at most one bad snippet. If this bad snippet is a bigon snippet, then $\len_\textrm{red}(\alpha'')<\len_\textrm{red}(\alpha)$. If it is a trigon snippet, then $\len_\textrm{red}(\alpha'')\leq\len_\textrm{red}(\alpha)+2s+1$.}
    \BlankLine
    Set $k$ such that $\alpha[k]$ is bad\;
    \If {$\len(\alpha)=2$ or $\len_\mathrm{corn}(\alpha[k-1])=2s$ or $\len_\mathrm{corn}(\alpha[k+1])=2s$}
    {
    \Return $\Hom(\alpha,k)$
    }
    \If {$\alpha[k-1]$ and $\alpha[k+1]$ are vertical duals that turn the same way}
    {
    \Return $\Hom(\alpha,k)$
    }
    \If {$\len_\mathrm{corn}(\alpha[k-1])>1$ and $\len_\mathrm{corn}(\alpha[k+1])>1$}
    {
    \Return $\Hom(\alpha,k)$
    }
    \If {$\len_\mathrm{corn}(\alpha[k-1])=1$ and $\len_\mathrm{corn}(\alpha[k+1])=1$}
    {
    \If {$\alpha[k-3:k]$ or $\alpha[k+1:k+4]$ is not a blocker}
    {\Return $\Hom(\alpha,k)$}
    $\alpha'=\Hom(\Hom(\alpha,k),k-1)$\;
    \Return $\Hom(\alpha',k-2)$\;
    }
    \Return $\texttt{HorBigonInComp}(S,N,\Hom(\alpha,k))$
    \caption{\texttt{HorBigonInBranch} - Homotoping an almost efficient curve with a bigon snippet of type $\mathbb{B}(h,h)$ into an input for $\texttt{TrigCurve}$ or into a curve of shorter reduced corner length containing a single bigon snippet.}
    \label{HorBigonInBranch}
\end{algorithm}

\begin{lem} \label{HorBigInBranch funtionally correct}
The algorithm \texttt{HorBigInBranch} is correct. On an input $(S,N,\alpha)$, the algorithm halts in $\mathit{O}(\chi(S)^2 \cdot (\len_\textrm{red}(\alpha)+1))$ time.
\end{lem}

\begin{proof}
By assumption on the input of the algorithm \texttt{HorBigInBranch}, we are given a surface $S$ of positive complexity, a tie neighbourhood $N \subset S$ of a large train track in $S$, as well as an almost efficient curve $\alpha$ that contains a bigon snippet $\alpha[k]$ of type $\mathbb{B}(h,h)$. We remark that it requires at most $\mathit{O}(|\chi(S)|(\len_\textrm{red}(\alpha)+1))$ many operations to determine the index of this bad snippet.
 
The structure of the algorithm and of our proof is arranged according to the different types of the snippets $\alpha[k-1]$ and $\alpha[k+1]$. We recall that all snippets of $\alpha$ but $\alpha[k]$ are in efficient position. The following list provides a complete case distinction.
\begin{itemize}
\item $\len(\alpha)=2$.
\item $\len_\textrm{corn}(\alpha[k-1])$ or $\len_\textrm{corn}(\alpha[k+1])$ equal $2s$.
\item $\alpha[k-1]$ and $\alpha[k+1]$ are vertical duals. In this case they are
\begin{itemize}
\item vertical duals turning into the same direction (see Figure \ref{TwodualsSameway}) or
\item vertical duals turning into different direction (see Figures \ref{Bothcornerlengthgreaterthanone}-\ref{ExactlyOnehascornerlengthone}).
\end{itemize}
\end{itemize}
If $\alpha[k-1]$ and $\alpha[k+1]$ are vertical duals turning into different directions, we can further distinguish the following disjoint cases.
\begin{itemize}
\item $\len_\textrm{corn}(\alpha[k-1])>1$ and $\len_\textrm{corn}(\alpha[k+1])>1$ (see Figure \ref{Bothcornerlengthgreaterthanone}).
\item $\len_\textrm{corn}(\alpha[k-1])=1$ and $\len_\textrm{corn}(\alpha[k+1])=1$ (see Figure \ref{Bothcornerlengthone}).
\item Exactly one of the snippets $\alpha[k-1]$ and $\alpha[k+1]$ has a corner length greater than one (see Figure \ref{ExactlyOnehascornerlengthone}).
\end{itemize}

\begin{figure}[htbp] 
  \begin{minipage}[b]{0.49\linewidth}
    \centering
\begin{tikzpicture}[scale=0.45]
\draw [very thick](-4,-1.5) {} -- (-4,-3) {};
\draw [very thick] plot[smooth, tension=.7] coordinates { (-4,-3) (1,-2.5) (6,-3)};
\draw [very thick](6,-1.5) node (v3) {} -- (6,-3);
\draw [very thick] plot[smooth, tension=.7] coordinates {(-4,-1.5) (-1,-0.5) (-0.2,0.8)};
\draw [very thick] plot[smooth, tension=.7] coordinates {(v3) (3,-0.5) (2.2,0.8)};
\node at (-4.625,-2.25) {{\tiny{$v$}}};
\node at (1,-3.5) {{\tiny{$h$}}};
\node at (6.5175,-2.3319) {{\tiny{$v$}}};
\draw [very thick, densely dashed, blue] plot[smooth, tension=.7] coordinates {(-3.4175,-1.3542) (-3,-3.5) (-1.151,-3.5001) (-0.6504,-1.2708) (2.5,0)};
\draw [very thick] (-4,0) rectangle (-5.5,-4.5);
\draw [very thick](-4,-4.5) -- (-1.5,-4.5);
\draw [very thick] plot[smooth, tension=.7] coordinates {(-4,0) (-2.761,0.3436) (-2.1173,0.8793)};
\end{tikzpicture}
    \caption{A trigon snippet of type $\mathbb{B}(h,h)$ that is adjacent to two vertical duals that turn the same way.} 
    \vspace{24pt}
    \label{TwodualsSameway}
  \end{minipage}
  \begin{minipage}[b]{0.49\linewidth}
    \centering
\begin{tikzpicture}[scale=0.45]
\draw [very thick](-4,-1.5) {} -- (-4,-3) {};
\draw [very thick] plot[smooth, tension=.7] coordinates { (-4,-3) (1,-2.5) (6,-3)};
\draw [very thick](6,-1.5) node (v3) {} -- (6,-3);
\draw [very thick] plot[smooth, tension=.7] coordinates {(-4,-1.5) (-1,-0.5) (-0.2,0.8)};
\draw [very thick] plot[smooth, tension=.7] coordinates {(v3) (3,-0.5) (2.2,0.8)};
\node at (-4.625,-2.25) {{\tiny{$v$}}};
\node at (2.0665,-2.935) {{\tiny{$h$}}};
\node at (6.5175,-2.3319) {{\tiny{$v$}}};
\draw [very thick] (-4,0) rectangle (-5.5,-4.5);
\draw [very thick](-4,-4.5) -- (-2.5,-4.5);
\draw [very thick] plot[smooth, tension=.7] coordinates {(-4,0) (-2.761,0.3436) (-2.1173,0.8793)};
\draw [very thick](-2.9192,0.2912) -- (-2.6468,-1.1139);
\draw [very thick](-2.5105,-2.7784) -- (-2.5,-4.5);
\draw [very thick, densely dashed, blue] plot[smooth, tension=.7] coordinates {(-3.7905,-1.4515) (-3.2776,-2.3611)  (-0.4818,-2.0925) (-0.011,-3.4229) (0.7514,-3.3685) (0.78,-1.919) (-0.1635,-1.3753)  (-2.4315,-1.7905) (-3.0389,-1.2555)};
\draw [very thick, densely dashed, blue] plot[smooth, tension=.7] coordinates {(2.603,-0.1772) (2.2803,-1.2861) (4.3627,-2.0513) (4.6463,-3.4892) (5.5073,-3.4822) (5.5991,-1.8987)  (3.7449,-1.3593) (3.6499,-0.8172)};
\draw [very thick](-1,-2.6113) -- (-1,-3);
\draw [very thick](4.9724,-0.8706) -- (4.8902,-1.201);
\draw [very thick](3.4761,-0.3348) -- (3.2656,-0.6395);
\draw [very thick] (6,-0.5) rectangle (7.5,-4);
\draw [very thick](6,-4) -- (4,-4) -- (3.9948,-2.7264);
\end{tikzpicture}
    \caption{Examples of trigon snippets of type $\mathbb{B}(h,h)$ that are adjacent to two vertical duals which turn into different directions and both have corner length greater than one.} 
    \label{Bothcornerlengthgreaterthanone}
  \end{minipage}
  \begin{minipage}[b]{0.49\linewidth}
    \centering
\begin{tikzpicture}[scale=0.45]
\draw [very thick](-4,-1.5) {} -- (-4,-3) {};
\draw [very thick] plot[smooth, tension=.7] coordinates { (-4,-3) (1,-2.5) (6,-3)};
\draw [very thick](6,-1.5) node (v3) {} -- (6,-3);
\draw [very thick] plot[smooth, tension=.7] coordinates {(-4,-1.5) (-1,-0.5) (-0.2,0.8)};
\draw [very thick] plot[smooth, tension=.7] coordinates {(v3) (3,-0.5) (2.2,0.8)};
\node at (-4.625,-2.25) {{\tiny{$v$}}};
\node at (1,-3.5) {{\tiny{$h$}}};
\draw [very thick] (-4,0) rectangle (-5.5,-4.5);
\draw [very thick](-4,-4.5) -- (-2,-4.5);
\draw [very thick] plot[smooth, tension=.7] coordinates {(-4,0) (-2.761,0.3436) (-2.1173,0.8793)};
\draw [very thick](-1.9934,-2.7402) -- (-2,-4.5);
\draw [very thick] (-5.5,0) rectangle (-7,-4.5);
\draw [thick, dotted](-7,0) -- (-9,0) -- (-9,1.5);
\draw [thick, dotted] plot[smooth, tension=.7] coordinates {(-9,1.5) (-8,1.75) (-7.0339,3.0759)};
\draw [very thick, densely dashed, blue] plot[smooth, tension=.7] coordinates {(-7.7209,1.9687) (-4.014,0.9474)  (-3.5,-3.5) (-2.5,-3.5)  (-2.854,1.0563)  (-7.1283,2.8261)};
\draw [very thick](-2.3119,0.6362) -- (-1.2654,-0.6498);
\end{tikzpicture}
    \caption{A trigon snippet of type $\mathbb{B}(h,h)$ that is adjacent to two vertical duals which turn into different directions and which both have corner length one.} 
    \vspace{12pt}
    \label{Bothcornerlengthone}
  \end{minipage}
    \begin{minipage}[b]{0.49\linewidth}
    \centering
\begin{tikzpicture}[scale=0.45]
\draw [very thick](-4,-1.5) {} -- (-4,-3) {};
\draw [very thick] plot[smooth, tension=.7] coordinates { (-4,-3) (1,-2.5) (6,-3)};
\draw [very thick](6,-1.5) node (v3) {} -- (6,-3);
\draw [very thick] plot[smooth, tension=.7] coordinates {(-4,-1.5) (-1,-0.5) (-0.2,0.8)};
\draw [very thick] plot[smooth, tension=.7] coordinates {(v3) (3,-0.5) (2.2,0.8)};
\node at (-4.625,-2.25) {{\tiny{$v$}}};
\node at (1,-3.5) {{\tiny{$h$}}};
\draw [very thick, densely dashed, blue] plot[smooth, tension=.7] coordinates {(-3.4175,-1.3542) (-3,-3.5) (-1,-3.5) (-0.6504,-1.2708) (-0.9146,-0.459)};
\draw [very thick] (-4,0) rectangle (-5.5,-4.5);
\draw [very thick](-4,-4.5) -- (0,-4.5);
\draw [very thick] plot[smooth, tension=.7] coordinates {(-4,0) (-2.761,0.3436) (-2.1173,0.8793)};
\draw [very thick](-2.5814,0.4189) -- (-1.951,-0.9049);
\draw [very thick](0.0042,-2.538) -- (0,-4.5);
\end{tikzpicture}
    \caption{A trigon snippet of type $\mathbb{B}(h,h)$ that is adjacent to two vertical duals which turn into different directions and of which exactly one has corner length one.} 
    \label{ExactlyOnehascornerlengthone}
  \end{minipage}
\end{figure}

We note that the corner length of any vertical dual snippet must be greater than zero as vertical duals are in efficient position.
We go through each of these cases and the corresponding if-statements of the pseudocode step by step. We remark that each of the if-statements concludes with a return-command. Thus, earlier tested properties of the curve can be assumed to be false.

If $\len(\alpha)=2$, then $\alpha[k-1]$ is an efficient snippet inside a peripheral complementary region winding non-trivially around a boundary component. Hence, $\Hom(\alpha,k)$ is a peripheral curve of snippet length one.

If $\len_\textrm{corn}(\alpha[k-1])=2s$ or $\len_\textrm{corn}(\alpha[k+1])=2s$, then $\alpha[k-1]$ or $\alpha[k+1]$ are in efficient position but not part of any blockers of the curve $\alpha$. We recall that any snippet in efficient position has non-zero corner length. Lemma \ref{Lemma B(h,h,0)} implies that $\len(\Hom(\alpha[k-1:k+2],1))=1$. Thus, we see that \[\len_\textrm{corn}(\Hom(\alpha[k-1:k+2],1))\leq 2s < 2s+1+1\leq \len_\textrm{corn}(\alpha[k-1:k+2]).\] 
Unless $\len_\textrm{corn}(\alpha[k-1])=1$ or $\len_\textrm{corn}(\alpha[k+1])=1$, this implies that $\len_\textrm{red}(\Hom(\alpha,k))$ is strictly smaller than $\len_\textrm{red}(\alpha)$ as the number of blockers decreases by at most one under the homotopy. Thus, of whatever type the snippet $\Hom(\alpha[k-1:k+2],1)$ is in this case, the curve $\Hom(\alpha,k)$ is a valid output for the algorithm \texttt{HorBigonInBranch}. On the other hand, if $\len_\textrm{corn}(\alpha[k-1])=1$ or $\len_\textrm{corn}(\alpha[k+1])=1$, we see that $\Hom(\alpha[k-1:k+2],1)$ must be a trigon or in efficient position following an index argument, and \[\len_\textrm{red}(\Hom(\alpha,k)) \leq \len_\textrm{red}(\alpha).\] Thus, the curve $\Hom(\alpha,k)$ is a valid output for the algorithm \texttt{HorBigonInBranch} in this case, too.

If $\alpha[k-1]$ and $\alpha[k+1]$ are vertical duals turning into the same direction, and index argument shows that $\Hom(\alpha,k)$ is in efficient position (see Figure \ref{TwodualsSameway}).

From this point onwards we can assume that $\alpha[k-1]$ and $\alpha[k+1]$ are vertical duals that turn into different directions (see Figures \ref{Bothcornerlengthgreaterthanone}-\ref{ExactlyOnehascornerlengthone}). The points $\alpha[k-1](0)$ and $\alpha[k+1](1)$ lie on the same component of $\partial_h N$. Thus, $\Hom(\alpha,k)$ is an almost efficient curve containing a snippet of type $\mathbb{R}(h,h)$. Set $\beta=\Hom(\alpha[k-1:k+2,1])$. For the remainder of this proof we assume that $\len_\textrm{corn}(\alpha[k-1])\leq \len_\textrm{corn}(\alpha[k+1])$. This implies that \[\len_\textrm{corn}(\beta)\leq \len_\textrm{corn}(\alpha[k-1:k+2])-\len_\textrm{corn}(\alpha[k-1])-1-1.\] This can be seen as follows: Let us denote the boundary of the index-zero region cut off by $\alpha[k+1]$ by $B_{\alpha[k+1]}$ and the boundary of the region cut off by the bigon snippet $\beta$ by $B_\beta$. Then $(B_{\alpha[k+1]} \cap \partial \mathcal{R})-\partial_v N$ consists of two components, of which one contains the point $\alpha[k-1](0)$ as well as the set $B_\beta \cap \partial \mathcal{R}$. We see that \[\len_\textrm{corn}(\beta) \leq \len_\textrm{corn}(\alpha[k+1])-1.\] As $\len_\textrm{corn}(\alpha[k])=1$, this implies that \[\len_\textrm{corn}(\beta)\leq \len_\textrm{corn}(\alpha[k-1:k+2])-\len_\textrm{corn}(\alpha[k-1])-2.\]

If $\len_\textrm{corn}(\alpha[k+1])>1$, then one of the components of $(B_{\alpha[k+1]} \cap \partial \mathcal{R})-\partial_v N$ must contain a component of the horizontal boundary of a branch rectangle. If this is the component that does not contain $B_\beta \cap \partial \mathcal{R}$, then \[\len_\textrm{corn}(\beta) \leq \len_\textrm{corn}(\alpha[k+1])-2.\] This implies that \[\len_\textrm{corn}(\beta)\leq \len_\textrm{corn}(\alpha[k-1:k+2])-\len_\textrm{corn}(\alpha[k-1])-3.\] Similarly, if this component of the horizontal boundary of a branch rectangle lies in the component of $B_\beta \cap \partial \mathcal{R}$ that contains $\alpha[k-1](0)$, then $B_\beta \cap \partial \mathcal{R}$, which is confined by the points $\alpha[k+1](1)$ and $\alpha[k-1](0)$, contains one less component of the horizontal boundary of a branch rectangle, namely the one that intersects $\partial_v N$. This, too, implies that $\len_\textrm{corn}(\beta) \leq \len_\textrm{corn}(\alpha[k+1])-2$, thus \[\len_\textrm{corn}(\beta)\leq \len_\textrm{corn}(\alpha[k-1:k+2])-\len_\textrm{corn}(\alpha[k-1])-3.\]

If $\len_\textrm{corn}(\alpha[k-1])>1$ and $\len_\textrm{corn}(\alpha[k+1])>1$, this discussion implies that $\len_\textrm{corn}(\beta)< \len_\textrm{corn}(\alpha[k-1:k+2])-4$ (see Figure \ref{Bothcornerlengthgreaterthanone}). As the homotopy decreases the number of blockers by at most two, we see that $\len_\textrm{red}(\Hom(\alpha,k))< \len_\textrm{red}(\alpha)$.

If $\len_\textrm{corn}(\alpha[k-1])=1$ and $\len_\textrm{corn}(\alpha[k+1])>1$, the previous discussion implies that $\len_\textrm{corn}(\Hom(\alpha[k-1:k+2],1])\leq \len_\textrm{corn}(\alpha[k-1:k+2])-4$. Since at most two blockers are affected by the homotopy, we obtain that \[\len_\textrm{red}(\Hom(\alpha,k))\leq \len_\textrm{red}(\alpha),\] where $\Hom(\alpha,k)$ is a curve containing a wide horizontal bigon. Lemma \ref{HorBigInCompl funtionally correct} then implies that $\texttt{HorBigonInComp}(S,N,\Hom(\alpha,k))$ is a curve which contains at most one bad snippet. If this is a bigon, then the reduced corner length decreased by at least one compared to the reduced corner length of the original curve $\alpha$.

Lastly, let us assume that $\len_\textrm{corn}(\alpha[k-1])=1$ and $\len_\textrm{corn}(\alpha[k+1])=1$ (see Figure \ref{Bothcornerlengthone}). Then $\beta$ is a narrow bigon of type $\mathcal{R}(h,h)$ whose boundary points lie in the same component of the horizontal boundary of a branch rectangle. We note that \[\len_\textrm{corn}(\Hom(\alpha,k))=\len_\textrm{corn}(\alpha)-3.\] Hence, if at most one of the subarcs $\alpha[k-3:k]$ or $\alpha[k+1:k+4]$ was a blocker, then $\len_\textrm{red}(\Hom(\alpha,k))\leq \len_\textrm{red}(\alpha)-1$ and $\Hom(\alpha,k)$ is a valid output for \texttt{HorBigonInBranch}. Else, that is if the subarcs $\alpha[k-3:k]$ and $\alpha[k+1:k+4]$ are both blockers, then $\len_\textrm{red}(\Hom(\alpha,k))\leq \len_\textrm{red}(\alpha)+1$. As $\alpha[k-2]$ and $\alpha[k+2]$ are ties of the same branch rectangle, applying one further local homotopy does not reduce the number of blockers of the underlying curve but reduces the corner length by one. This implies that \[\len_\textrm{red}(\Hom(\Hom(\alpha,k),k-1)) \leq \len_\textrm{red}(\alpha).\]  Since $\alpha[k-1]$ and $\alpha[k+1]$ were part of some blocker, we know that $\alpha[k-3]$ and $\alpha[k+3]$ are not the same snippet as they are vertical duals turning in different directions. Moreover, Lemma \ref{Length of blocker} implies that $\len_\textrm{corn}(\alpha[k-3])\geq 5$ and $\len_\textrm{corn}(\alpha[k+3])\geq 5$. Set $\alpha'= \Hom(\Hom(\Hom(\alpha,k),k-1),k-2)$. Then \[\len_\textrm{red}(\alpha') = \len_\textrm{red}(\alpha'[k-2:] \cdot \alpha'[:k-3]) + \len_\textrm{corn}(\alpha'[k-3]).\] Reasoning as in the second to last paragraph we see that
\begin{align*}
\len_\textrm{corn}(\alpha'[k-3]) & \leq \len_\textrm{corn}(\Hom(\Hom(\alpha,k),k-1)[k-3:k])-(5+1+5) \\ & \leq \len_\textrm{corn}(\Hom(\alpha,k)[k-3:k+2])-1-11 \\ & \leq \len_\textrm{corn}(\alpha[k-3:k+4])-3-12.
\end{align*}
As $\alpha'[k-2:] \cdot \alpha'[:k-3]=\alpha[k+4:]\cdot\alpha[:k-3]$, we know that \[\len_\textrm{red}(\alpha'[k-2:] \cdot \alpha'[:k-3])=\len_\textrm{red}(\alpha[k+4:]\cdot\alpha[:k-3]).\] Therefore,
\begin{align*}
\len_\textrm{red}(\alpha') & \leq \len_\textrm{red}(\alpha[k+4:]\cdot\alpha[:k-3]) + \len_\textrm{corn}(\alpha[k-3:k+4])-15 \\
& \leq \len_\textrm{red}(\alpha[k+4:]\cdot\alpha[:k-3]) + \len_\textrm{red}(\alpha[k-3:k+4])+4 -15 \\
& \leq \len_\textrm{red}(\alpha)+4+4-15 \\
& \leq \len_\textrm{red}(\alpha)-7.
\end{align*}
Thus, we see that $\alpha'= \Hom(\Hom(\Hom(\alpha,k),k-1),k-2)$ contains a unique bad snippet of type $\mathbb{R}(h,h)$ and satisfies $\len_\textrm{red}(\alpha')\leq \len_\textrm{red}(\alpha)-7$. Therefore, it is a valid output for the algorithm \texttt{HorBigonInBranch}. This completes the proof of the correctness of the algorithm \texttt{HorBigonInBranch}.

To analyse the running time of the algorithm \texttt{HorBigonInBranch} on an input $(S,N,\alpha)$, we note that within $\mathit{O}(|\chi(S)|\cdot(\len_\textrm{red}(\alpha)+1))$ time all if-statements in the pseudocode can be checked and all single local homotopies can be executed. If $\len_\textrm{corn}(\alpha[k-1])=1$ and $\len_\textrm{corn}(\alpha[k+1])>1$, then $\len_\textrm{red}(\Hom(\alpha,k) \leq \len_\textrm{red}(\alpha)$. Therefore, $\texttt{HorBigonInComp}(S,N,\Hom(\alpha,k))$ terminates within $\mathit{O}(\chi(S)^2 \cdot (\len_\textrm{red}(\alpha)+1))$ many operation. Thus, on input $\alpha$, the entire algorithm \texttt{HorBigonInBranch} halts in $\mathit{O}(\chi(S)^2 \cdot (\len_\textrm{red}(\alpha)+1))$ time.
\end{proof}

\section{The algorithm \texttt{SingleBadSnippet}}

Combining the algorithms from the previous sections, we can finally present an algorithm that achieves efficient position for an almost efficient curve or shortens it until it consists of a single snippet.

\begin{algorithm} 
    \SetKwInOut{Input}{Input}
    \SetKwInOut{Output}{Output}
	\DontPrintSemicolon
%    \textbf{Function} $\texttt{SingleBadSnippet}(\alpha)$:\;
%    \BlankLine
    \Input{A surface $S$ of positive complexity, a tie neighbourhood $N \subset S$ of a large train track in $S$, and a curve $\alpha \subset S$ that contains at most one bad snippet.}
    \Output{A curve $\alpha'$ homotopic to $\alpha$ such that $\alpha'$ is in efficient position or consists of a single snippet.}
    \BlankLine
     $\alpha'= \alpha$\;
    \While{$\len(\alpha')>1$}
    {
    \If {There is a bad snippet $\alpha[k]$ of type $\mathbb{B}(t,t)$, $\mathbb{S}(h,h,0)$, $\mathbb{S}(t,t,0)$, $\mathbb{S}(v,v,0)$, $\mathbb{R}(v,v)$, $\mathbb{S}(t,t,2)$, or $\mathbb{S}(t,v,1)$}
    {$\alpha'=\texttt{AllButHorBigInCompOrBranch}(S,N,\alpha')$}
    \ElseIf {There is a bad snippet $\alpha[k]$ of type $\mathbb{R}(h,h)$}
    {$\alpha'=\texttt{HorBigInComp}(S,N,\alpha')$}
     \ElseIf {There is a bad snippet $\alpha[k]$ of type $\mathbb{B}(h,h)$}
    {$\alpha'=\texttt{HorBigInBranch}(S,N,\alpha')$}
    \If {There is a bad snippet $\alpha[k]$ that is a trigon}
    {\Return $\texttt{TrigCurve}(S,N,\alpha')$}
    \If {$\alpha'$ is in efficient position}
    {\Return $\alpha'$}
    }
    \Return $\alpha'$  
    \caption{\texttt{SingleBadSnippet} - Homotoping a curve with at most one bad snippet into efficient position or shortening it to consist of a single snippet.}
    \label{SingleBadSnippet}
\end{algorithm}

\begin{lem} \label{SingleBadSnippet funtionally correct}
The algorithm $\texttt{SingleBadSnippet}$ is correct. On an input $(S,N,\alpha)$, the algorithm halts in $\mathit{O}(|\chi(S)|^3 \cdot (\len_\textrm{red}(\alpha)^2+1))$ time.
\end{lem}

\begin{proof}
By assumption on the input of the algorithm \texttt{SingleBadSnippet}, we are given a surface $S$ of positive complexity, a tie neighbourhood $N \subset S$ of a large train track in $S$, as well as an almost efficient curve $\alpha$. Let us first prove the correctness of the algorithm \texttt{SingleBadSnippet}.
If $\len(\alpha)>1$, then the unique bad snippet must be of bigon or trigon type following the classification result of Lemma \ref{Bad snippet types in long enough curves and arcs}. If it is a bigon snippet, then exactly one of the algorithms \texttt{AllButHorBigInCompOrBranch}, \texttt{HorBigInComp}, or \texttt{HorBigInBranch} will be executed once. The resulting curve either contains a single bigon snippet and has a smaller reduced corner length than $\alpha$, or contains a unique trigon snippet, is in efficient position, or of snippet length one. We furthermore know that the reduced corner length of the curve increases by at most $2s+1$ if the curve contains a trigon snippet. Thus, if the curve does not contain a bigon snippet, the algorithm terminates within $\mathit{O}(\chi(S)^2\cdot(\len_\textrm{red}(\alpha)+2s+2))$ further operations following Corollary \ref{Trig Curve Corollary}. Else, the algorithm enters the next iteration of the while-loop. Since the reduced corner length of the underlying curves is decreasing by at least one during every iteration, there can be at most $\len_\textrm{red}(\alpha)+1$ many iterations of the while-loop before the algorithm terminates and outputs a curve that is in efficient position or consists of a single snippet only. This proves the correctness of the algorithm $\texttt{SingleBadSnippet}$.

To analyse the running time of the algorithm $\texttt{SingleBadSnippet}$ on an input $(S,N,\alpha)$ we first remark that it requires at most $\mathit{O}(|\chi(S)|(\len_\textrm{red}(\alpha)+1))$ many operations to determine the index of the unique bad snippet of $\alpha$. We further recall that the reduced corner length of the underlying curve decreases by at least one after every iteration of the while-loop. Thus, the while-loop is executed at most $\len_\textrm{red}(\alpha)+1$ many times. The snippet length of the underlying curve when evaluation the if-statements is bounded by $\len_\textrm{red}(\alpha)+1$ as the curve has at most one bad snippet. Hence, searching for the single bad snippet takes $\mathit{O}(|\chi(S)|\cdot (\len_\textrm{red}(\alpha)+1))$ time. Executing a single bigon routine requires $\mathit{O}(\chi(S)^2\cdot (\len_\textrm{red}(\alpha)+1))$ time. Thus, every iteration of the while-loop requires at most $\mathit{O}(|\chi(S)|^3\cdot (\len_\textrm{red}(\alpha)+1))$ many operations. This implies that the entire algorithm $\texttt{SingleBadSnippet}$ halts in $\mathit{O}(|\chi(S)|^3\cdot (\len_\textrm{red}(\alpha)^2+1))$ time.
\end{proof}

\newpage
%%%%%%%%%%%%%%%%%%%%%%%%%%%%%%%%%%%%%%%%%%%%%%%%%%%%%%%%%%%%%%%%%%%%%%%%%%%%%%%%%%%%%%%%%%%%%%%%%%%%%%%%%%%%%%%%%%%%%%%%%%%%%%%%%%%%%%%%%%%%%%%%%%%%%%%%%%%%%%%%%%%%%%%%%%%%%%%%%%%%%%%%%%%%%%%%%%%%%%%%%%%%%%%%%%%%%%%%%%%%%%

\chapter{Polynomial-time efficient position}

In this chapter we present the very last algorithm of this thesis, the algorithm \texttt{EfficientPosition}. As input, \texttt{EfficientPosition} takes a surface $S$ of positive complexity, a tie neighbourhood $N \subset S$ of a large train track in $S$, and an arc or curve $\alpha$, and outputs an arc or curve $\alpha'$ homotopic to $\alpha$ such that $\alpha'$ is in efficient position or satisfies $\len(\alpha')=1$. In the latter case, previous results imply that $\alpha$ must be inessential or peripheral. For a formal statement of \texttt{EfficientPosition} we refer the reader to Algorithm \ref{EfficientPosition}. We will see that \texttt{EfficientPosition} halts in polynomial time and thus is sufficient to provide a proof for Theorem \ref{Polynomial time algorithm}.

\begin{algorithm} 
    \SetKwInOut{Input}{Input}
    \SetKwInOut{Output}{Output}
	\DontPrintSemicolon
%    \textbf{Function} $\texttt{EfficientPosition}(\alpha)$:\;
%    \BlankLine
    \Input{A surface $S$ of positive complexity, a tie neighbourhood $N \subset S$ of a large train track in $S$, and an arc or curve $\alpha \subset S$.}
    \Output{An arc or curve $\alpha'$ homotopic, relative to endpoints if applicable, to $\alpha$ such that $\alpha'$ is in efficient position or consists of a single snippet only.
   }
    \BlankLine
    $\alpha'=\alpha$\;
    \If {$\alpha'$ is an arc}{
    \Return $\texttt{ReduceToTwoBadSnippets}(S,N,\alpha')$}
    \If {$\len(\alpha')>2$}{
    $\alpha'=\texttt{ReduceToTwoBadSnippets}(S,N,\alpha')$}
    \If {$\len(\alpha')>1$}{
    $\alpha'=\texttt{ReduceToOneBadSnippet}(S,N,\alpha')$\;
    $\alpha'=\texttt{SingleBadSnippet}(S,N,\alpha')$}
    \Return $\alpha'$
    \caption{\texttt{EfficientPosition} - Achieving efficient position for essential and non-peripheral arcs and curves or shortening inessential or peripheral arcs and curves so that they consist of a single snippet.}
    \label{EfficientPosition}
\end{algorithm}

\begin{lem} \label{EfficientPositionForCurves funtionally correct}
The algorithm $\texttt{EfficientPosition}$ is correct. On an input $(S,N,\alpha)$, the algorithm halts in $\mathit{O}(\chi(S)^4 \cdot \len(\alpha)^2)$ time.
\end{lem}

\begin{proof}
The correctness of the algorithm follows from the correctness of its subroutines and Corollary \ref{Efficient position for arcs}. 

We now discuss the running time of the algorithm \texttt{EfficientPosition}. If the input $\alpha$ is an arc, then the bounds on the running time follow from Lemma \ref{Lemma ReduceToTwoBadSnippets}. 
Thus, suppose that the input $\alpha \subset S$ is a curve. Lemma \ref{Lemma ReduceToTwoBadSnippets} implies that the algorithm $\texttt{ReduceToTwoBadSnippets}$ terminates in $\mathit{O}\big(\chi(S)^4\cdot \len(\alpha)^2)$ time on input $(S,N,\alpha)$. 
\sloppy
We set $\alpha'=\texttt{ReduceToTwoBadSnippets}(S,N,\alpha)$. Lemma \ref{Lemma ReduceToTwoBadSnippets} further implies that \[\len_\textrm{red}(\alpha') \leq (\len(\alpha)-1)\cdot 6s.\] Corollary \ref{ReduceToOneBadSnippet Reduced length} states that the algorithm $\texttt{ReduceToOneBadSnippet}$ terminates within $\mathit{O}(|\chi(S)|^3 \cdot (\len_\textrm{red}(\alpha')+1))$ time on input $(S,N,\alpha')$. Thus, the algorithm $\texttt{ReduceToOneBadSnippet}$ terminates within $\mathit{O}((\chi(S))^4 \cdot (\len(\alpha)+1))$ time on input $(S,N,\alpha')$. 

Set $\alpha''=\texttt{ReduceToOneBadSnippet}(\alpha')$. Lemma \ref{Lemma ReduceToOneBadSnippet} implies that \[\len_\textrm{red}(\alpha'') \leq (\len(\alpha)-1)\cdot 6s +10s \leq (\len(\alpha)+1)\cdot 6s .\] Lastly, Lemma \ref{SingleBadSnippet funtionally correct} implies that the algorithm \texttt{SingleBadSnippet} terminates within $\mathit{O}(|\chi(S)|^3  \cdot (\len_\textrm{red}(\alpha'')^2+1))$ time on input $(S,N,\alpha'')$. As $\len_\textrm{red}(\alpha'') \leq (\len(\alpha)+1)\cdot 6s$, we see that the algorithm \texttt{SingleBadSnippet} terminates within $\mathit{O}(|\chi(S)|^3  \cdot (\len(\alpha)^2+1))$ time on input $(S,N,\alpha'')$. Adding the running times of the three main subroutines, this implies that the algorithm \texttt{EfficientPosition} terminates in $\mathit{O}(\chi(S)^4  \cdot (\len(\alpha)^2+1))$ time. Since $\len(\alpha)\geq 1$, we know that $\len(\alpha)^2+1 \leq 2 \len(\alpha)^2$. Thus, \texttt{EfficientPosition} terminates in $\mathit{O}(\chi(S)^4  \cdot \len(\alpha)^2)$ time.
\end{proof}

%%%%%%%%%%%%%%%%%%%%%%%%%%%%%%%%%%%%%%%%%%%%%%%%%%%%%%%%%%%%%%%%%%%%%%%%%%%%%%%%%%%%%%%%%%%%%%%%%%%%%%%%%%%%%%%%%%%%%%%%%%%%%%%%%%%%%%%%%%%%%%%%%%%%%%%%%%

\begin{thm}\label{Polynomial time algorithm}
There is an algorithm that takes as input
\begin{itemize}
\item a surface $S=S_{g,b}$ satisfying $\xi(S)=3g-3+b \geq 1$,
\item a tie neighbourhood $N=N(\tau)$ of a large train track $\tau$ in $S$, and
\item a properly immersed arc or curve $\alpha$ given via its snippet decomposition with respect to $N$,
\end{itemize}
and outputs an arc or curve $\alpha'$ homotopic to $\alpha$, relative to endpoints in the case of arcs, such that
\begin{itemize}
\item $\alpha'$ is in efficient position with respect to $N$ or
\item $\alpha'$ has snippet length one.
\end{itemize}
This algorithm terminates in $\mathit{O}(\chi(S)^4  \cdot \len(\alpha)^2)$ time. Moreover, the output $\alpha'$ has snippet length one if and only if $\alpha$ is inessential or peripheral. If $\alpha$ is peripheral, the boundary component that $\alpha$ is homotopic to, as well as the corresponding power, can be read off from the snippet-decomposition of $\alpha'$.
\end{thm}

\begin{proof}
Lemma \ref{Efficient position for curves implies essential and non-peripheral} and Lemma \ref{Efficient position for arcs implies essential} imply that arcs and curves in efficient position are essential and non-peripheral. Further, Lemma \ref{arcs or curves of length one} implies that properly immersed arcs and curves of snippet length one are inessential or peripheral. Thus, the existence of the desired algorithm as well as the bounds on its running time follow from Lemma \ref{EfficientPositionForCurves funtionally correct}. 
We recall that the cutting sequence and winding number of arcs and curves of snippet length one are sufficient to determine whether they are inessential or peripheral, and, in the latter case, to determine which power of the boundary component they are homotopic to. This concludes the proof of the theorem.
\end{proof}

The mathematical content of Theorem \ref{Polynomial time algorithm} can be rephrased as follows.

\begin{cor}\label{Existence}
Suppose that $S=S_{g,b}$ is a surface satisfying $\xi(S)=3g-3+b \geq 1$. Suppose that $\tau \subset S$ is a large train track and $N=N(\tau)$ is a tie neighbourhood of $\tau$ in $S$. Let $\alpha \subset S$ be an essential, non-peripheral, properly immersed multiarc or multicurve given via its snippet decomposition. Then efficient position for $\alpha$ with respect to $N$ exists and can be obtained in $\mathit{O}(\chi(S)^4  \cdot \len(\alpha)^2)$ time.
\end{cor}
As noted previously, this is the converse to Lemmas \ref{Efficient position for curves implies essential and non-peripheral} and \ref{Efficient position for arcs implies essential}.

\newpage

%%%%%%%%%%%%%%%%%%%%%%%%%%%%%%%%%%%%%%%%%%%%%%%%%%%%%%%%%%%%%%%%%%%%%%%%%%%%%%%%%%%%%%%%%%%%%%%%%%%%%%%%%%%%%%%%%%%%%%%%%%%%%%%%%%%%%%%%%%%%%%%%%%%%%%%%%%%%%%%%%%%%%%%%%%%%%%%%%%%%%%%%%%%%%%%%%%%%%%%%%%%%%%%%%%%%%%%%%%%%%%

\newpage

%
%\section*{Acknowledgments}  It is a pleasure to thank my mentor, 
%his/her name, for ....
%\bibliographystyle{plainnat}
%\bibliographystyle{alpha}
\bibliographystyle{alpha}
\bibliography{literatur}

%\begin{verbatim}\citep{abraham_etal}
% \end{verbatim}
% for a parenthetical citation \citep{abraham_etal},\\
%
% \begin{verbatim}\citep*{MTW}
% \end{verbatim}
% for a full list of authors use a * parenthetical citation \citep*{MTW},\\
% \\
%!!!!!!!!!!!!!!!
\end{document}